\documentclass[11 pt, reqno]{article}

\usepackage{amsmath,amsfonts,amssymb,amsthm}
\usepackage[colorlinks=true,hyperindex=true]{hyperref}
\usepackage{cancel,bbm}
\usepackage{mathabx} 
\usepackage{comment}
\usepackage{enumitem}
\usepackage{cite}

\usepackage{chngcntr}
\counterwithin*{equation}{section}

\usepackage[a4paper, left=2cm, right=2cm, top=2 cm, bottom= 2 cm]{geometry}
\usepackage{color}
\usepackage{fancyhdr}
\usepackage{latexsym}

\newtheorem{ccounter}{ccounter}[section]
\newtheorem{thm}[ccounter]{Theorem}
\newtheorem{lem}[ccounter]{Lemma}
\newtheorem{cor}[ccounter]{Corollary}
\newtheorem{defn}[ccounter]{Definition}
\newtheorem{prop}[ccounter]{Proposition}
\newtheorem{ass}[ccounter]{Assumption}
\newtheorem{ex}[ccounter]{Example}

\def\bet{\begin{thm}}
\def\eet{\end{thm}}
\def\bel{\begin{lem}}
\def\eel{\end{lem}}
\def\bas{\begin{ass}}
\def\eas{\end{ass}}
\def\bec{\begin{cor}}
\def\eec{\end{cor}}
\def\bed{\begin{defn}}
\def\eed{\end{defn}}
\def\bep{\begin{prop}}
\def\eep{\end{prop}}
\def\beq{\begin{equation}}
\def\eeq{\end{equation}}
\def\proof{\noindent {\bf Proof.}\ \ }
\def\bea{\begin{equation*}}
\def\eea{\end{equation*}}
\def\tr{\mathrm{tr}}
\def\bex{\begin{ex}}
\def\eex{\end{ex}}
\def\sumAthi{\sum^{\mathcal{A}_3, (i)}}
\def\fl{{f_{\eta_\ell}}}
\def\C{\mathcal{C}}
\def\fc{\mathrm{fc}}
\def\sc{\mathrm{sc}}
\def\Chat{\hat{\mathcal{C}}}
\def\gf{\gamma^{(\mathfrak{f})}}
\def\sig{\sigma}
\def\Galp{\mathcal{G}_\alpha}
\def\zstar{z^*}

\def\gsc{\gamma^{(\mathrm{sc})}}
\def\gV{\gamma^{(V)}}
\def\sumEi{\sum^{\mathcal{E}, (i)}}
\def\sumEic{\sum^{\mathcal{E}^c, (i)}}
\def\xhat{\hat{x}}
\def\yhat{\hat{y}}
\def\epst{{\varepsilon_\tau}}
\def\sumEti{\sum^{\mathcal{E}_2, (i)}}
\def\sumEthi{\sum^{\mathcal{E}\backslash\mathcal{E}_2, (i)}}
\def\sumEtci{\sum^{\mathcal{E}^c_2, (i)}}
\def\R{\mathcal{R}}
\def\F{\mathcal{F}}
\def\rr{\mathbb{R}}

\def\cc{\mathbb{C}}
\def\zz{\mathbb{Z}}

\def\i{\mathrm{i}}
\def\UBth{\mathcal{U}^{\mathcal{B}_3}}
\def\zhat{\hat{z}}
\def\e{\rm{e}}

\def\del{\partial}
\def\d{\mathrm{d}}
\def\e{\mathrm{e}}

\def\Re{\mathrm{Re}}
\def\supp{\mathrm{supp}\,}
\def\Im{\mathrm{Im}\mbox{ }}
\def\pp{\mathbb{P}}
\def\ee{\mathbb{E}}
\def\eps{\varepsilon}

\def\sumit{{\sum_i}^{'}}

\renewcommand\leq\varleq
\renewcommand\geq\vargeq
\def\D{\mathcal{D}}

\def\z1{\zeta_1}
\def\z2{\zeta_2}

\def\z{\zeta}

\def\O{\mathcal{O}}

\def\g{\gamma}

\def\A{\mathcal{A}}
\def\L{\mathcal{L}}
\def\1{\boldsymbol{1}}

\def\sumAqsi{\sum^{\A_{q_*}, (i)}}
\def\sumAqsci{\sum^{\A_{q_*}^c, (i)}}
\def\sumsthi{\sum^{\A_{q_*} \backslash \A_2, (i)}}
\def\equald{\overset{d}{=}}
\def\UBhat{\mathcal{U}^{ \hat{B} }}

\def\E{\mathcal{E}}

\def\I{\mathcal{I}}

\def\ystar{y^*}
\def\xstar{x^*}
\def\tilB{\tilde{B}}

\def\L{L}

\def\tilV{\tilde{V}}
\def\calG{\mathcal{G}}
\def\benr{\begin{enumerate}[label=(\roman*)]}
\def\eenr{\end{enumerate}}

\def\remark{\noindent{\bf Remark. }}

\def\rhofc{\rho_{\mathrm{fc}}}
\def\hatx{\hat{x}}

\def\haty{\hat{y}}

\def\xione{\xi^{(1)}}
\def\xitwo{\xi^{(2)}}

\def\B{\mathcal{B}}

\def\S{\mathcal{S}}
\def\L{\mathcal{L}}
\def\US{\mathcal{U}^{\mathcal{S}}}

\def\UL{\mathcal{U}^{\mathcal{L}}}

\def\G{\mathcal{G}}

\def\mfc{m_{\mathrm{fc}}}

\def\UB{\mathcal{U}^{B}}
\def\fa{\mathfrak{a}}

\def\mfct{m_{\mathrm{fc}, t}}
\def\rhofct{\rho_{\mathrm{fc}, t}}
\def\om{\omega}

\def\nn{\mathbb{N}}

\def\rhosc{\rho_{\mathrm{sc}}}

\def\rhofcto{\rho_{\mathrm{fc}, t_0}}

\def\L{G}

\def\tilz{\tilde{z}}
\def\hatz{\hat{z}}
\def\sumAi{\sum^{\A, (i) }}
\def\sumAci{\sum^{\A^c, (i)}}
\def\sumAti{\sum^{\A_2, (i) }}

\def\Ihat{\hat{\mathcal{I}}}

\def\I{\mathcal{I}}

\def\L{\mathcal{L}}

\def\moj{m_N^{(j)}}

\def\G{G}
\def\g{g}
\def\hatB{\hat{B}}

\def\Ii{\mathcal{I}_i}
\def\Ici{\mathcal{I}^{c}_i}
\def\tilg{\tilde{\gamma}}

\def\sumoj{\sum^{(j)}}

\def\Goj{G^{(j)}}

\def\bGoj{\bar{G}^{(j)}}

\def\dvarp{\partial_{\bar{z}}\tilde{\varphi}_N}

\begin{document}
\title{Deformed GOE}


\begin{table}
\centering
\begin{tabular}{c}
\multicolumn{1}{c}{\Large{\bf Fixed energy universality of Dyson Brownian motion }}\\
\\
\\
\end{tabular}
\begin{tabular}{c c c c c}
Benjamin Landon$^1$ & & Philippe Sosoe$^2$ & & Horng-Tzer Yau$^3$\\
\\
\multicolumn{5}{c}{ \small{$^1$Department of Mathematics} } \\
 \multicolumn{5}{c}{ \small{Massachusetts Institute of Technology} } \\
 \\
 \multicolumn{5}{c}{ \small{$^2$Department of Mathematics} } \\
 \multicolumn{5}{c}{ \small{Cornell University} } \\
 \\
  \multicolumn{5}{c}{ \small{$^3$Department of Mathematics} } \\
 \multicolumn{5}{c}{ \small{Harvard University} } \\
 \\
 
\small{blandon@mit.edu} & & \small{ps934@cornell.edu} & & \small{htyau@math.harvard.edu}  \\
\\
\end{tabular}
\\
\begin{tabular}{c}
\multicolumn{1}{c}{\today}\\
\\
\end{tabular}

\begin{tabular}{p{15 cm}}
\small{{\bf Abstract:} We consider Dyson Brownian motion for classical values of $\beta$ with deterministic initial data $V$.  We prove that the local eigenvalue statistics coincide with the GOE/GUE in the fixed energy sense after time $t \gtrsim 1/N$ if the density of states of $V$ is bounded above and below down to scales $\eta \ll t$ in a window of size $L \gg \sqrt{t}$.  Our results imply that fixed energy universality holds for essentially any random matrix ensemble for which averaged energy universality was previously known.  Our methodology builds on the homogenization theory developed in \cite{homogenization} which reduces the microscopic problem to a mesoscopic problem.   As an auxiliary result we prove a mesoscopic central limit theorem for linear statistics of various classes of test functions for classical Dyson Brownian motion. }
\end{tabular}
\end{table}
{
\hypersetup{linkcolor=black}
\tableofcontents
}
\section{Introduction} \label{sec:int}

{\let\thefootnote\relax\footnote{The work of B.L. is partially supported by NSERC.  The work of H.-T. Y. is partially supported by NSF Grant DMS-1307444,  DMS-1606305 and a Simons Investigator award.}}In the pioneering work \cite{wigner}, Wigner introduced what are known as the Wigner random matrix ensembles.  These ensembles consist of $N \times N$ real symmetric ($\beta=1$) or complex Hermitian ($\beta=2$) random matrices $W = (w_{ij})$ whose entries are centered and independent (up to the symmetry constraint $W=W^*$) with variance
\beq
\ee [ (w_{ij})^2 ] = \frac{ 1 + \delta_{ij} }{N}, \quad \beta=1, \qquad \ee[ |w_{ij}|^2 ] = \frac{1}{N} , \quad \beta =2.
\eeq
If the $w_{ij}$'s are independent real (resp., complex) Gaussians then the ensemble is called the Gaussian Orthogonal Ensemble (resp., Gaussian Unitary Ensemble) (GOE/GUE).  Wigner conjectured that in the limit $N \to \infty$ the local eigenvalue statistics are universal in that they depend only on the symmetry class of the matrix ensemble (real symmetric or complex Hermitian) and are otherwise independent of the underlying distribution of the matrix entries.   After Wigner's seminal work, Gaudin, Dyson and Mehta explicitly calculated the eigenvalue correlation functions in the Gaussian cases.

Mehta formalized the universality conjecture in the book \cite{mehta2004random} and stated that the correlation functions of general Wigner matrices should coincide with the GOE/GUE in the limit $N \to \infty$.  There are several possible topologies in which this convergence could hold. Perhaps the most natural topology to consider is pointwise convergence of the correlation functions.  However,  this cannot hold for random matrix ensembles with discrete entries.  One suitable topology is that of vague convergence of the correlation functions around an energy $E$, which we will call {\it fixed energy universality}.  A weaker topology can be constructed by averaging over energies in a small window near $E$ and asking for vague convergence of the energy-averaged quantities.  We will call this {\it averaged} or unfixed energy universality.  Finally, one can also ask for the vague convergence of the eigenvalue gaps with a fixed label (i.e., vague convergence of the random variable $\lambda_{N/2+1} - \lambda_{N/2}$) which we call {\it gap universality}.

There has recently been spectacular progress in proving the Wigner-Dyson-Mehta conjecture for a wide variety of random matrix ensembles.  Bulk universality for Wigner matrices of all symmetry classes was proven in the works \cite{renyione, erdos2012spectral,erdos2010bulk,erdos2011universality,Gap,erdos2012rigidity}.  Parallel results were established in certain cases in \cite{tao2010random,tao2011random}, with the key result being a ``four moment comparison theorem."  In this paper we are interested in the  robust three-step approach to universality formulated and developed in the works \cite{renyione, erdos2012spectral,erdos2010bulk,erdos2011universality,Gap,erdos2012rigidity}.  This approach consists of:

\begin{enumerate}
\item A high probability estimate of the eigenvalue density down to the almost-optimal scale $\eta \sim N^\eps/N$.  This establishes eigenvalue {\it rigidity}; that is, the bulk eigenvalues are close to their expectations
\beq
| \lambda_i - \ee [ \lambda_i ] | \leq \frac{ N^\eps}{N}
\eeq
with overwhelming probability.  Moreover, the expectations are determined by the quantiles of the macroscopic eigenvalue density.
\item Proving bulk universality for random matrix ensembles with a small additive Gaussian component.  This is usually established by studying the rate of convergence of Dyson Brownian motion to local equilibrium.
\item A comparison or stability argument comparing a given random matrix ensemble to one with a small Gaussian component.
\end{enumerate}

For complex Hermitian ensembles, Step 2 can be established by using an explicit algebraic formula, the Br\'ezin-Hikami formula,  to analyze the correlation functions.  This idea was used by Johansson \cite{johansson2001universality}  and Ben Arous-Peche \cite{ben2005universality} who established bulk universality for ensembles with an order $1$ Gaussian component, 
 i.e., establishing that the time to equilibrium is at most order $1$ in this case.    The optimal time to equilibrium in the second step,  i.e., for $t \gtrsim 1/N$, was established in \cite{erdos2010bulk}  where the Br\'ezin-Hikami formula and  estimates from the local semicircle law were the key tools.  
In this special algebraic case, the second step yields fixed energy universality and so the WDM conjecture was established for complex Hermitian matrices in this strong sense \cite{erdos2010bulk, tao2010random,yautaovu}. 

An analogue of the Br\'ezin-Hikami formula is unknown in the real symmetric case and therefore the approach \cite{erdos2010bulk} could not be extended to real symmetric Wigner ensembles.  A new approach based on the local relaxation flow of {\it Dyson Brownian motion} (DBM) was developed in the works \cite{localrelaxation,erdos2011universality,erdos2012bulk}.   DBM is defined by applying an independent (up to the symmetry constraint $H =H^*$)  Ornstein-Uhlenbeck process to every matrix element; Dyson computed the flow on the eigenvalues and found that they satisfy a closed system of stochastic differential equations.  The approach of \cite{localrelaxation,erdos2011universality,erdos2012bulk} is based on this representation and  applies to all symmetry classes as well as sample covariance matrices and sparse ensembles.  However, it yields only  averaged energy universality, albeit with the averaging taken over a very small window.

The second step was finally completed for real symmetric Wigner ensembles in the sense of fixed energy in \cite{homogenization}.  By developing a sophisticated homogenization theory for a discrete parabolic equation derived from DBM, the authors proved that after a time $t = o (1)$, the local statistics of Dyson Brownian motion started from a Wigner ensemble coincide with that of the GOE in the fixed energy sense.  As the third step in the three-step strategy described above is insensitive to the mode of convergence of the correlation functions, this proved the Wigner-Dyson-Mehta conjecture in the fixed energy sense for all symmetry classes.

The time to equilibrium proven in the work \cite{homogenization} is relatively long, $t \sim N^{-\eps}$ for a small $\eps >0$, which moreover depends on the choice of test function.  This limits the applicability of the work \cite{homogenization} in proving fixed energy universality for other ensembles.  For example, it does not imply fixed energy universality for sparse random graphs, for which averaged energy universality is known \cite{huang, ziliang, erdos2012spectral,renyione}.  Moreover, the approach relies on the fact that the global eigenvalue density  of the initial data is given by the semicircle law.


The analysis of DBM developed in the works \cite{localrelaxation,erdos2011universality,erdos2012bulk} is in some sense global as it relies on the fact that DBM with initial data a Wigner matrix will follow the semicircle law.  
In the work  \cite{Gap}  the correlation functions were expressed as time averages of random walks in a random environment.   This allows for a local analysis of the dynamics and various tools from PDE (such as H\"older regularity via the di-Giorgi-Nash-Moser method) and stochastic analysis can be applied.


In the work \cite{landonyau} the time to equilibrium of DBM for a wide class of initial data (going beyond the Wigner class) was studied (see also \cite{ES} for related results).  For random matrix ensembles that have a local density down to scales $\eta \gtrsim 1/N$, it was proven that the time to local equilibrium is $t \gtrsim 1/N$, in the sense of both averaged energy and gap universality.

There have been several recent works extending the Wigner-Dyson-Mehta conjecture beyond the class of Wigner matrices, such as to sparse random graphs \cite{hl,huang,roland1,roland2,erdos2012spectral,renyione,ziliang}, matrices with correlated entries \cite{ziliang2,ajankicor, ajankigaus}, deformed Wigner ensembles \cite{Kevin1,Kevin2}, certain classes of band matrices \cite{bourgadeband} and the general Wigner-type matrices of \cite{ajanki1,ajanki2,ajanki3}.  These works generally follow the three-step strategy outlined above.   In many of these cases, the works \cite{landonyau, ES} essentially complete the second step of this approach.  As the results \cite{landonyau, ES} imply averaged energy universality, 
any work relying on \cite{landonyau, ES} for the second step establishes the Wigner-Dyson-Mehta conjecture in only the averaged energy sense.

In the current work we establish that the time to local equilibrium for DBM is $t \gtrsim N^{-1}$ for a wide class of initial data, in the fixed energy sense.  The main assumption on the initial data is that the density of states is bounded above and below down to scales $\eta \ll t$ in a window of size $L \gg \sqrt{t}$.   As a consequence, fixed energy universality is established for essentially all random matrix ensembles for which previously only averaged energy universality could be proven.

One of the key insights of \cite{homogenization} is that the difference of two coupled DBM flows obeys a discrete nonlocal parabolic equation.  One of the main results of \cite{homogenization} is a homogenization theory for this parabolic equation.  This theory shows that the solution of the discrete parabolic equation is given by the discretization of the continuum limit, and this reduces the problem of microscopic statistics to an easier mesoscopic problem.

Our approach follows the same high-level strategy in that we couple two DBM flows and develop a homogenization theory for the resulting parabolic equation.  The generator of the parabolic equation of \cite{homogenization} is hard to control.  To deal with this we modify the coupling of \cite{homogenization} and introduce a continuous interpolation.  This gives us a family of parabolic equations whose generators have better properties.  

The homogenization theory of \cite{homogenization} was based around a Duhamel expansion and estimating the coefficients of the generator.  The short range part of the generator is quite singular and was controlled using an energy estimate and the discrete Di-Giorgi-Nash-Moser theorem of \cite{Gap}.  This caused some restriction on the time to equilibrium that could be proven.

Our method is based around the standard $\ell^2$-energy method  
and a discrete Sobolev inequality. 
%
%
The energy method gives us an estimate on the time average of the discrete $\dot{H}^{1/2}$ norm of the difference between the fundamental solution of the discrete equation and its continuum limit.  This allows us to get a time-averaged $\ell^\infty$ estimate on the fundamental solution of the discrete parabolic equation via a discrete Sobolev inequality.  We then use the semigroup property to remove the time average.  

In order to carry this out one needs a good ansatz for comparison with the discrete fundamental solution.  We substitute the particle location coming from the DBM into the fundamental solution of the continuum limit.  With this approach a martingale term, as well as other errors of lower order, arises in the energy method, but we are able to control them using heat kernel bounds for our process.  This ansatz first appeared in \cite{que} and has been used independently in \cite{bourgadeprep} to analyze extremal gap statistics of Wigner ensembles.

We find that the limiting hydrodynamic equation is a fairly simple nonlocal parabolic equation describing a symmetric jump process on $\rr$.  The heat kernels of such processes have been studied recently in, e.g., \cite{CKK, CKK2} and we partially rely on their work in our analysis of the limiting equation.


Our homogenization theory has an advantage over \cite{homogenization} in that our estimates hold with overwhelming probability (i.e., $\geq 1- N^{-D}$ for any large $D>0$).  The homogenization theory \cite{homogenization} relied on certain level repulsion estimates and as a consequence the main estimates were only known to hold with polynomially high probability (i.e., $\geq 1- N^{-\eps}$ for some small $\eps >0$).  While this is not  significant to the application of universality, we believe that this improvement is important for future applications.  For example, if one wishes to study the maximal eigenvalue gap in the bulk of generalized Wigner matrices, then one can use the homogenization result given here together with a union bound over order $N$ eigenvalues.  This approach has been carried out in the work \cite{LLM}.  Note that the result of \cite{homogenization} would not be sufficient to study this spectral statistic. 

 Moreover our method is robust in that it essentially relies only on rigidity; in many random matrix ensembles optimal level repulsion estimates (on which the previous methods \cite{homogenization, landonyau, ES} relied) are not known and can be hard to establish.

 
 As mentioned above, the homogenization theory reduces the microscopic problem of fixed energy universality to a problem involving linear mesoscopic statistics.  
 Central limit theorems for mesoscopic linear statistics of Wigner matrices were established first in certain scales in \cite{boutet1,boutet2,meso1} and then down to the almost-optimal scale $\eta = N^\eps/N$ in \cite{meso2}.  Mesoscopic statistics of compactly supported test functions for the special case of $\beta=2$ for DBM with deterministic initial data was established in \cite{DJ}. 
The analysis in \cite{DJ} relied on the Br\'ezin-Hikami formula special to the $\beta=2$ case and cannot be applied here.  Moreover the test function coming from the homogenization theorem is not of compact support - only its derivative is - and has no spatial decay, which presents a serious complication.  The mesoscopic results \cite{boutet1, boutet2, meso1, meso2, DJ} all apply only to functions with either compact support or at least some spatial decay as $|x| \to \infty$.
 
   In the present work we establish a mesoscopic central limit theorem for DBM for a certain class of non-compactly supported test functions which have no spatial decay.   As an aside, we remark that if our methods are restricted to the compactly supported case, then we can prove that if the scale of the function is less than $t$, then the linear statistic coincides with the GOE.  Here, in the compactly supported case one can remove the restrictions that we encounter in the non-compactly supported case.  This is an extension of some of the results of \cite{DJ} to $\beta=1$. 
    Our main interest in a mesoscopic central limit theorem is to analyze  the statistic coming from the homogenization theory, and so we only settles for a  few remarks concerning test functions of compact support - see Section \ref{sec: mesosection}.
   


The works \cite{landonyau, ES} establishing averaged energy universality for DBM relied heavily on the discrete Di-Giorgi-Nash-Moser theorem of \cite{Gap}.  As a consequence, the rate of convergence was somewhat non-explicit.  While in this work we do not attempt to derive optimal error bounds, our result improves on \cite{landonyau, ES} in the sense that our bounds can be quantified explicitly in terms of the parameters of the model.

\subsection{Applications}

Our homogenization theory also allows us to establish universality of the space-time DBM process --- that is, up to an explicit deterministic shift in space and a re-scaling in space and time, the multitime correlation functions of DBM coincide with the GOE/GUE.   

Averaged energy universality for general one-cut $\beta$-ensembles was established first in \cite{bourgade2012bulk,bourgade2014universality,bourgade2014edge,Gap}.  Further results for bulk universality for multi-cut potentials were established in \cite{shcherbinabeta,transport}.  Fixed energy universality for one-cut $C^4$ potentials was announced in \cite{ICM} and can be proven using the methods of \cite{homogenization}.  Previously,  Shcherbina had established fixed energy universality for analytic potentials in the multi-cut case in \cite{shcherbinabeta}.  For completeness we sketch how our methods can be adapted to re-prove the results of \cite{ICM, homogenization} and moreover establish a polynomial error estimate.

\subsubsection{Fixed energy universality for general random matrix ensembles}

Many of the recent works on universality of general random matrix ensembles have relied on the aforementioned three-step strategy to proving universality.  In the second step these works relied on \cite{landonyau, ES} for universality for the Gaussian divisible ensembles.  The works \cite{landonyau, ES} provided both gap universality and averaged energy universality for the Gaussian divisible ensembles; consequently this form of universality has been proven, for example, for the adjacency matrices of sparse random graphs \cite{hl,huang,roland1,roland2,erdos2012spectral,renyione,ziliang}, matrices with correlated entries \cite{ziliang2,ajankicor, ajankigaus} and the general Wigner-type matrices of \cite{ajanki1,ajanki2,ajanki3}.  By instead relying on the current work, fixed energy universality is established for all of these ensembles.

\subsubsection{Eigenvalue interval probabilities}

 Fixed energy universality has several other consequences which we now outline.  It establishes the existence of the local density of states on microscopic scales as well as the universality of the Jimbo-Miwa-Mori-Sato formula for the gap probability.  In addition, it implies universality of the distribution of the smallest singular value of various random matrix ensembles, including the adjacency matrices of a wide variety of sparse random graphs which is of interest in computer science.  
 
 \subsubsection{Invertibility of symmetric random matrices}

The invertibility problem in random matrix theory is typically divided into two components \cite{vershsymmetric}. The first is whether a random matrix is invertible with high probability, and the second is to determine the typical size of the norm of the inverse, or size of the smallest singular value.  A motivating problem of the former type is the conjecture that an iid Bernoulli matrix is singular with probability less than $(2 +  o (1) )^{-N}$.  Komlos \cite{komlosdet} first proved that the singularity probability is vanishing.  An exponential bound was first obtained in \cite{kks} and later improved in \cite{taosing, bvw}.

The size of the inverse is related to the condition number which plays a crucial role in applied linear algebra.  For example, the condition number controls the complexity or numerical accuracy in solving the linear equation $Ax = b$.  Von Neumann and his collaborators speculated \cite{vninvert} that the least singular value satisfies $s_{\mathrm{min}} (A) \sim N^{-1}$ for matrices with iid entries.

By now a large literature has emerged on the invertibility problem for both symmetric and iid ensembles.  We refer  to the surveys \cite{rvsurvey,vusurvey,nguyensurvey} and the references therein, and mention only a few specific results placing the present work in context.  The invertibility of dense Erd\H{o}s-R\'enyi graphs was established in \cite{ctv} and was later extended to the sparse regime in \cite{cv}.  The first estimate of the form $s_{\mathrm{min}} (H) \sim N^{-1}$ for symmetric matrices was first obtained in \cite{vershsymmetric}.   In the iid case, it is even known that the distribution of the (properly rescaled) smallest singular value at the hard edge is universal \cite{taovuleast}.  However, in the sparse regime little is known about the size of the smallest singular value in the symmetric case.

It is an open conjecture that the adjacency matrix of a random $d$-regular graph is invertible with high probability for $d \geq 3$ \cite{cv,vusurvey,frieze}.  This problem is of interest in universal packet recovery \cite{courtade}.  In the case of random $d$-regular {\it directed} graphs, substantial progress has been made by \cite{cook,tiki}.   It is also conjectured that the adjacency matrices of more general sparse graphs outside the Erdos-Renyi class should be invertible with high probability \cite{cv}.

The invertibility of many classes of random matrices is in fact a corollary of previous works by two of the current authors \cite{landonyau, huang} as well as others \cite{ziliang,ziliang2,roland1,roland2,ajanki1,ajankicor}.  These classes include, for example, adjacency matices of random regular graphs, matrices with correlated entries and general sparse random matrices.   To fix ideas we consider the $d$-regular random graph with adjacency matrix $A$, where 
\beq \label{eqn:dregime}
d \geq N^\eps .
\eeq
The invertibility of $A + G$ where $G$ is a small GOE component follows from Section 5 of \cite{landonyau} as well as the local law of \cite{roland1}.  The comparison methods of \cite{huang,roland2} then allow the invertibility to be transferred back to $A$.  This proves the conjecture of \cite{cv,frieze,vusurvey} in the regime \eqref{eqn:dregime}.  This methodology extends to the other random matrix ensembles considered in \cite{huang}, as well as those considered in  \cite{ziliang,ziliang2,ajanki1,ajankicor} as long as $0$ lies in the bulk of the spectrum.

This strategy  yields an additional effective estimate on the size of the inverse in all cases, which was previously known only in the non-sparse regime.   That is, there is a $c>0$ so that for all sufficiently small $\eps >0$,
\beq
\mathbb{P} [ ||A^{-1} || \geq N^{\eps}/N ] \leq C (\eps) N^{- c \eps}.
\eeq
For example, one can take $A$ to be the (properly rescaled so that the spectrum lies in a window of order $1$) adjacency matrix of a sparse $d$-regular or Erd\H{o}s-R\'enyi graph.

The current work goes beyond this and establishes universality of the smallest singular value of many random matrix ensembles.  Our work implies, for example, that for each $t\in \rr$,
\beq
\lim_{N \to \infty} \left| \mathbb{P} [ || A^{-1} || \geq t N ] - \mathbb{P} [ || H^{-1} || \geq t N ] \right| =0,
\eeq
where $A$ is the (again, properly rescaled) adjacency matrix of a random $d$-regular graph and $H$ is a GOE matrix (again for $d$ in the regime \eqref{eqn:dregime}).


\subsection{Overview}

The remainder of the paper is as follows.  In Section \ref{sec:2} we introduce precisely our model and state our main results, applications and auxilliary results.   Section \ref{sec:homog} contains the main part of the homogenization results and Section \ref{sec:finitespeed} contains the proofs of certain a-priori bounds on the heat kernel.  We study and prove regularity of the limiting continuum equation in Section \ref{sec:hydro}.  In Section \ref{sec: mesosection} we prove our results on mesoscopic linear statistics for DBM.  Section \ref{sec:fixed} contains the proof of fixed energy universality using the homogenization theory and the central limit theorem for mesoscopic linear statistics.  In Section \ref{sec:beta} we sketch the proof of fixed energy universality for $\beta$-ensembles.

\noindent{\bf Acknowledgements.} B.L. thanks Jiaoyang Huang for useful discussions.

\section{Model} \label{sec:2}

Let $V$ be a deterministic diagonal matrix and let $W$ be a standard GOE matrix.  We consider the following model
\beq \label{eqn:Htdef}
H_t =  V + \sqrt{t} W.
\eeq
We make the following assumptions on $V$.  
\bed
Let $\G = \G_N$ and $\g = \g_N$ be $N$-dependent parameters.  For definiteness we assume that there is a $\delta >0$ s.t.  $ N^{-\delta} \geq \g \geq N^\delta/N$ and $\G \leq N^{ - \delta}$.  This $\delta$ will not be important in the method or the main results.  We say that $V$ is $(\g, \G)$-regular if 
\beq \label{eqn:immVassump}
c \leq \Im [ m_V (E + \i \eta ) ] \leq C
\eeq
for $|E| \leq \G$ and $\g \leq \eta \leq 10$, and if there is a  $C_V>0$ s.t.
\begin{equation}\label{eqn: CVdef}
||V|| \leq N^{C_V}.
\end{equation}
\eed
\remark  The assumption \eqref{eqn: CVdef} is technical and can be removed with some minor work.  We omit this from the current paper.  

We will be considering times satisfying $g N^{\sigma} \leq t \leq N^{- \sigma} \G^2$.  We also introduce here the frequently used notation $z = E + \i \eta$ for  $E, \eta \in \rr$.

\subsection{Free convolution} \label{sec:fc}
In this section we introduce the free convolution.   The semicircle law is given by
\beq
\rhosc (E) := \frac{1}{ 2 \pi } \1_{ \{ |E| \leq 2 \} }\sqrt{ 4 - E^2 }.
 \eeq
 It describes the limiting eigenvalue density of the GOE.  The eigenvalue density of $H_t$ does not follow the semicircle law and is given by a free convolution.   We define the free convolution of $V$ with the semicircle law at time $t$ via its Stieltjes transform which we denote by $\mfct$.  The function $\mfct$ is defined as the unique solution to 
\beq \label{eqn:mfcdef}
\mfct (z) = \frac{1}{N} \sum_{i=1}^N \frac{1}{  V_i - z - t\mfct (z) } , \qquad \Im  [ \mfct (z) ] ,\quad \eta  \geq0 .
\eeq
  The free convolution law is defined by
\beq
\rhofct (E) : = \lim_{ \eta \downarrow 0} \frac{1}{\pi} \Im [ \mfct (E + \i \eta ) ].
\eeq
The free convolution is well-studied.  For example,  it is known that a unique solution to \eqref{eqn:mfcdef} exists and that $\rhofct$ is analytic on the interior of its support.  We refer to \cite{biane} for further details.  We will also denote the free convolution law at time $t$ by $\rhofct (E) := \rho_V \boxplus \rho_{\sc, t}$.

\subsection{Fixed energy universality} \label{thm: fixed-energy}

Let $p_{H_t}^{(N)}$ denote the symmetrized eigenvalue density of $H_t$.  The $k$-point correlation functions are defined by
\beq
p_{H_t}^{(k)} ( \lambda_1, \cdots , \lambda_k ) := \int p_{H_t}^{(N)} ( \lambda_1, \cdots , \lambda_N ) \d \lambda_{k+1} \cdots \d \lambda_N.
\eeq
The corresponding objects for the GOE are denoted $p^{(N)}_{GOE}$ and $p^{(k)}_{GOE}$.
The following is our main result which states that the $k$-point correlation functions of $H_t$ converges to those of the GOE in the fixed energy sense.
\bet \label{thm:fixed} Let $V$ be a deterministic $(\g, \G)$-regular diagonal matrix.   Let $\sigma >0$ and let
\beq
g N^{\sigma} \leq t \leq N^{- \sigma} \G^2.
\eeq
Let $0 < q < 1$ and let $|E| \leq q \G$.  There is a constant $\kappa >0$ so that the following holds.  For every $k$ and smooth test function $O \in C^\infty_c ( \rr^k )$ there is a constant $C>0$ such that
\begin{align}
\bigg| & \frac{1}{ ( \rhofct (E) )^k} \int O ( \alpha_1, ..., \alpha_k ) p_{H_t}^{(k)} \left( E + \frac{ \alpha_1}{ N \rhofct (E) } , \cdots E + \frac{ \alpha_k }{ N \rhofct (E) } \right) \d \alpha_1 \cdots \d \alpha_k \notag\\
- &\frac{1}{( \rhosc(E) )^k} \int O ( \alpha_1, ..., \alpha_k ) p_{GOE}^{(k)} \left( E + \frac{ \alpha_1}{ N \rhosc (E) } , \cdots E + \frac{ \alpha_k }{ N \rhosc (E) } \right) \d \alpha_1 \cdots \d \alpha_k  \bigg| \leq C N^{- \kappa}
\end{align}
\eet
\noindent{\bf Remark.} We will state and prove our results only for the real symmetric case. However, our methods and results hold also for the complex Hermitian case with essentially only notational changes. 
\subsubsection{Applications to other ensembles}

Theorem \ref{thm:fixed} implies fixed energy universality for a wide variety of ensembles appearing in random matrix theory.  Recall the three-step strategy to proving universality for random matrix ensembles outlined in the introduction.  Many recent works in random matrix theory used the results of \cite{landonyau, ES} to complete the second step.  The input of \cite{landonyau, ES} is to provide universality of DBM started from the chosen random matrix ensemble in either the fixed gap sense or averaged energy sense.  The third step is relatively insensitive to the type of universality proven in the second step.  Therefore, if one uses Theorem \ref{thm:fixed} instead of \cite{landonyau, ES} one can prove fixed energy universality for following ensembles.
\begin{enumerate}
\item Sparse random matrix ensembles such as the adjacency matrices of random graphs  \cite{huang, roland1, roland2, renyione, erdos2012spectral,ziliang, hl}
\item The general Wigner-type matrices of \cite{ajanki1, ajanki2, ajanki3}.
\item Matrices with correlated entries \cite{ziliang2, ajankicor,ajankigaus}.
\item Deformed Wigner ensembles  \cite{Kevin1, Kevin2}.
\end{enumerate}
Lastly, while fixed energy universality of generalized Wigner matrices was settled in \cite{homogenization}, our methods yield a polynomial rate of convergence which was previously unknown.

\subsection{Further results}
\subsubsection{Multitime correlation functions}
The eigenvalues $\lambda_i$ of $H_t$ at each fixed time are equal in distribution to the unique strong solution of the system of the following system of SDEs, known as Dyson Brownian motion:
\beq \label{eqn:dlambda}
\d \lambda_i = \sqrt{ \frac{2}{N} } \d B_i + \frac{1}{N} \sum_j \frac{1}{ \lambda_i - \lambda_j } \d t
\eeq
with initial data $\lambda_i (0) = V_i$.  Theorem \ref{thm:fixed} implies that at each fixed time, the correlation functions of $\{ \lambda_i (t) \}_i$ coincide with the GOE.  Our methods also allow us to consider multitime correlation functions.  For simplicity we just state the result for two times $t_a < t_b$.  One can also consider any finite set of times $t_a, ..., t_k$.  Given two times $t_a < t_b$ let $p_{t_a, t_b} (\lambda_1 (t_a) , \cdots \lambda_N (t_a), \lambda_1 (t_b ), \cdots \lambda_N ( t_b) )$ denote the symmetrized density of $\{ \lambda_i (t_a), \lambda_j (t_b) \}_{i, j}$.  The multitime $k$-point correlation function is defined by
\begin{align}
&p^{(k)}_{t_a, t_b} ( \lambda_ 1 ( t_a ) , \cdots , \lambda_k (t_a ), \lambda_1 (t_b ) , \cdots , \lambda_k (t_b ) ) \notag\\
:= &\int p_{t_a, t_b} (\lambda_1 (t_a) , \cdots \lambda_N (t_a), \lambda_1 (t_b ), \cdots \lambda_N ( t_b) ) \d \lambda_{k+1} (t_a) \cdots \d \lambda_{N} ( t_a) \d \lambda_{k+1} ( t_ b) \cdots \d \lambda_{N} (t_b).
\end{align}
Denote the analogous object for the GOE by $p^{(k)}_{t_a, t_b, GOE} $ (i.e., start the process $\lambda_i$ from the GOE ensemble).  Fix an energy $|E (t_a)| \leq q\G$ and define $E(t)$ for $t>t_a$ by
\beq
\del_t E = \Re [ \mfct (E) ] 
\eeq
\bet \label{thm:multitime}
Let $V$  be as above and let $\g N^{\sigma} \leq t_a \leq N^{-\sigma} \G^2$.  Let $O$ be a smooth compactly supported test function.    There is a constant $\kappa >0$ so that for any $t_a \leq t_b \leq N^{\kappa}/N$ we have
\begin{align}
\bigg|  &\int \frac{ O ( \alpha_1, \cdots \alpha_k, \beta_1, \cdots \beta_k )}{( \rho_{\fc, t_a} (E (t_a) ) )^k ( \rho_{\fc, t_b} ( E(t_b) ) )^k} p_{t_a, t_b}^{(k)} \left( E (t_a) + \frac{ \vec{ \alpha}}{ N \rho_{\fc, t_a} (E (t_a ) ) } , E (t_b) + \frac{ \vec{\beta}}{ N \rho_{\fc, t_b } (E (t_b ) ) }  \right) \d \vec{\alpha} \d \vec{\beta} \notag\\
&-\int \frac{ O ( \alpha_1, \cdots \alpha_k, \beta_1, \cdots \beta_k ) }{ ( \rhosc (E' ) )^{2k} } p_{t_a, c_t t_b, GOE}^{(k)} \left( E' + \frac{ \vec{ \alpha}}{ N \rhosc (E') }  , E' + \frac{ \vec{\beta}}{ N \rhosc (E')  }  \right) \d \vec{\alpha} \d \vec{\beta} \bigg| \leq N^{- \kappa},
\end{align}
for any fixed $E' \in (-2, 2)$.  Above the constant  $c_t$ is defined by $c_t := ( \rhosc (0) / \rho_{\fc, t_b} (E (t_b ) ) )^2$.
\eet
\remark One can replace $\rho_{\fc, t_b } (E (t_b ) )   $ by $\rho_{\fc, t_a } (E (t_a ) )  $ as the difference is $o (1)$.

\subsubsection{Jimbo-Miwa-Mori-Sato formula}

Once one establishes fixed energy universality for a random matrix ensemble, it is a standard argument to determine the distribution of the number of eigenvalues in a interval of size $c/N$. 
More precisely, Theorem \ref{thm:fixed} implies that for intervals $I_1, \cdots, I_k$ and integers $n_1, \cdots n_k$ the probability
\beq
\pp \left[  \left| \left\{ \lambda_i \in E + \frac{ I_j}{ N \rho (E) } \right\} \right| = n_j, 1 \leq j \leq k \right]
\eeq
converges to that of the GOE where $\rho (E)$ is the eigenvalue density of the ensemble under consideration.  

For example, for the adjacency matrices of a class of sparse random graphs, we have
\beq \label{eqn:gapprob}
\lim_{N \to \infty} \pp \left[ \left| \left\{ \lambda_i \in \frac{ [0, t]}{ N \pi \rho_{\sc} (0) } \right\} \right| = 0 \right] = E_1 (0, t)
\eeq
where $E_1$ is an explicit function of a solution to Painlev\'e equation. 

\subsubsection{Invertibility of symmetric random matrices}
The result \eqref{eqn:gapprob} provides explicit information on the distribution of the size of the inverse of various random matrix ensembles.  For example, \eqref{eqn:gapprob} implies that for the adjacency matrices $A$ of sparse Erd\H{o}s-R\'enyi and $d$-regular graphs we have for every $t>0$,
\beq \label{eqn:invsize}
\lim_{ N \to \infty} \left| \pp \left[ || A^{-1} || \geq N t \right] - \pp \left[ || H^{-1} || \geq N t \right] \right| = 0
\eeq
where $H$ is a GOE matrix.  From previous results in the literature \cite{huang,landonyau,roland2,roland1} it is easily deduced that the adjacency matrix of a sparse Erd\H{o}s-R\'enyi or $d$-regular graph is invertible with high probability.  The result \eqref{eqn:invsize} is finer, in that it demonstrates that the limiting distribution of the size of the inverse, or equivalently, the size of the smallest singular value of $A$, is universal.


\subsubsection{Fixed energy universality for $\beta$-ensembles} \label{sec:betaresult}
Our methods also imply fixed energy universality for a class of $\beta$-ensembles.  A $\beta$-ensemble is a measure on the simplex $\lambda_1 \leq \cdots \leq \lambda_N$ with probability density proportional to
\beq
\e^{ - \beta N \sum_k \frac{1}{2} V ( \lambda_k ) + \beta \sum_{ i < j } \log | \lambda_j  - \lambda_i | }.
\eeq
We assume that $V$ is a $C^4$ real function with second derivative bounded below and growth condition
\beq
V(x) > (2 + \alpha ) \log (1 + |x|)
\eeq
for all large $x$ and an $\alpha >0$.   The averaged density of the empirical spectral measure converges weakly to a continuous function $\rho_V$, the equilibrium density with compact support.  We assume that $\rho_V$ is supported on a single interval $[A, B]$ and that $V$ is regular in the sense of \cite{generic}.  We denote the $k$-point correlation functions by $p^{(k)}_V$ and those for the Gaussian $\beta$-ensemble (for $V(x) = x^2/2$) by $p^{(k)}_\beta$.

Under these conditions fixed energy universality was announced in \cite{ICM} and can be proven using the methods of \cite{homogenization}.  Previously, M. Shcherbina established fixed energy universality for multi-cut analytic $\beta$-ensembles in \cite{shcherbinabeta}.   
 The following result is an improved version of the result in \cite{ICM} in that it provides an error estimate $N^{- \kappa}$ to the fixed energy universality. Similarly to \cite{ICM}, 
our methodology is  based on the homogenization idea initiated in   \cite{ICM, homogenization}.  

\bet
Let $V$ be as above and assume $\beta \geq 1$.  Let $E \in (A, B)$ and $E' \in (-2, 2)$.  Let $O$ be a smooth test function. There is a $\kappa >0$ such that
\begin{align}
\bigg| \frac{1}{ ( \rho_V (E) )^k} &\int O ( \alpha_1, ..., \alpha_k ) p_{V}^{(k)} \left( E + \frac{ \alpha_1}{ N \rho_V (E) } , \cdots E + \frac{ \alpha_k }{ N \rho_V (E) } \right) \d \alpha_1 \cdots \d \alpha_k \notag\\
- & \frac{1}{ ( \rhosc (E') )^k} \int O ( \alpha_1, ..., \alpha_k ) p_{\beta}^{(k)} \left( E' + \frac{ \alpha_1}{ N \rhosc (E') } , \cdots E' + \frac{ \alpha_k }{ N \rhosc (E') } \right) \d \alpha_1 \cdots \d \alpha_k  \bigg| \leq C N^{- \kappa}
\end{align}
\eet

\remark It is also possible to deduce analogous results for multitime correlation functions in the following sense.  If one modifies \eqref{eqn:dlambda} to
\beq
\d \lambda_i = \sqrt{ \frac{2}{N} } \d B_i + \frac{1}{N} \sum_j \frac{1}{ \lambda_i - \lambda_j } \d t - \frac{ V' ( \lambda_i )}{2} \d t
\eeq
then the $\beta$-ensemble with potential $V$ is left invariant by this flow.  One can prove that the multitime correlation functions coincide with the process \eqref{eqn:dlambda} started from a Gaussian $\beta$-ensemble.

\subsubsection{Mesoscopic statistics for DBM}
Our methodology of proving fixed energy universality reduces the microscopic problem to a problem involving mesoscopic linear statistics.  In order to complete the proof of fixed energy universality we are forced to calculate mesoscopic statistics for DBM.  Mesoscopic statistics have received some attention in the literature recently and we therefore state our result as it may be of independent interest.
\bet \label{thm: mesostatement}
Let $\varphi$ be a smooth test function satisfying
\beq
 \varphi' (x) = 0, \qquad |x| >  C t', \qquad | \varphi^{(k)} | \leq C  / (t')^k, \quad k=0, 1, 2,
\eeq
where $t' = N^{\alpha}/N$.
  Let $V$ be $(\g, \G)$-regular and let $|E| \leq q \G$.  Let $t = N^{\om}/N$ satisfy $\g N^{\sigma} \leq t \leq \G^2 N^{-\sigma}$.  Assume that $\alpha >0$ satisfies $ \om/2 < \alpha < \om$.  Then the mesoscopic statistic
\beq \label{eqn:mesoresult}
\sum_i \varphi ( \lambda_i )
\eeq
converges weakly to a Gaussian.  If $\varphi$ is not compactly supported, then the variance is bounded below by $c | \log (t'/t)|$.
\eet
\remark Our results are more general --- see Section \ref{sec: mesosection}.  We calculate the characteristic function with an explicit rate of convergence in a growing neighborhood of the origin.

If $\varphi$ is compactly supported, then with some modifications of our methods one can remove the unnatural restriction $\alpha > \om/2$.  As the above theorem will suffice in our application to fixed energy universality we do not provide the details.



\subsection{Local law and rigidity} \label{sec:rigid}
In this section we recall the local law for $H_t$.  These a-priori estimates are the key technical input of our methods.  For times of order $1$ the local law was established in \cite{Kevin1,Kevin2}.  The argument was adapted to short times in \cite{landonyau}.  The empirical Stieltjes transform of $H_t$ will be denoted by
\beq
m_N (z) := \frac{1}{N} \sum_i \frac{1}{ \lambda_i -z}.
\eeq

Under the above hypotheses we have the following rigidity and local law estimates.   We need some notation.  For any $0 < q < 1$ let
\beq
\I_q := [ - q \G, q \G ].
 \eeq
  Let $\eps >0$ and $0 < q < 1$. We  consider the spectral domain
\begin{align}
\D_{\eps, q} &=  \left\{ z = E + \i \eta : E \in \I_q,  N^{10C_V} \geq \eta \geq N^\eps/N \right\} \notag\\
&\cup \left\{ z : D + \i \eta : |E| \leq N^{10 C_V} ,  N^{10 C_V} \geq \eta \geq c \right\}
\end{align}
We have
\bet \label{thm:localrev} Fix $\eps >0$ and $0 < q < 1$.  Let $\sigma >0$ be such that $ \g N^{\sigma} \leq t \leq N^{-\sigma} \G^2$. For any $D>0$ and $\delta >0$ we have
\beq
\pp \left[ \sup_{z \in D_\eps } 
\left| m_N (z) - \mfct (z) \right| \geq \frac{ N^\delta}{N \eta} \right] \leq C N^{-D}.
\eeq
\eet
We fix now a certain index set.  Let $0 < q < 1$.  Let
\beq \label{eqn:cqset}
\C_q := \{ i : V_i \in \I_q \}.
\eeq

\subsubsection{Classical eigenvalue locations}
Given a probability measure $\rho (x) \d x$ and matrix size $N$, we define the classical eigenvalues $\gamma_i = \gamma_{i, N}$ in the following manner.  If $N$ is even then
\beq
\gamma_i = \inf \left\{ x : \int_{-\infty}^x \rho (E ) \d E \geq \frac{i+1}{N} \right\}
\eeq
and if $N$ is odd then
\beq
\gamma_i = \inf \left\{ x : \int_{-\infty}^x \rho (E ) \d E \geq \frac{i+1/2}{N} \right\}.
\eeq
We denote the classical eigenvalue locations of the free convolution law at times $t$ by $\gamma_i (t)$ and the classical eigenvalue locations of the semicircle law by $\gsc_i$.  The above definition is slightly nonstandard, but we take it so that
\beq
\gsc_{\lceil N/2 \rceil} = 0,
\eeq
which will turn out to be convenient later. 


\subsubsection{Rigidity estimates}
We have the following rigidity result for the eigenvalues.
\bet \label{thm:rigidity} Fix $0 < q < 1$ and let $t$ be as above.  For any $\eps >0$ and $D>0$ we have
\beq
\pp \left[ \sup_{ i \in \C_q } | \lambda_i(t) - \gamma_i(t) | \geq \frac{N^{\eps}}{N} \right] \leq N^{-D}.
\eeq
We also have
\beq \label{eqn:lambdaV}
\pp \left[ \sup_j | \lambda_j - V_j | \geq 3 \sqrt{t} \right] \leq N^{-D}.
\eeq
\eet

From the above theorem we see that for any $q_1$ with $q < q_1 < 1$ we have for $N$ large enough,
\beq
i \in \C_q \implies \gamma_i (t) \in \I_{q_1} .
\eeq

\subsection{Proof strategy}

In this section we give an overview of the strategy of the proof of fixed energy universality.  The DBM flow starting from $V$ is given by the SDE
\beq \label{eqn:sketchdxi}
\d x_i = \sqrt{\frac{2}{N} } \d B_i + \frac{1}{N} \sum_j \frac{1}{ x_i - x_j } \d t 
\qquad x_i (0) = V_i ,
\eeq
where the $\d B_i$ are standard Brownian motions.
\begin{enumerate}
\item {\bf Regularization}.  In the next step we will couple the DBM \eqref{eqn:sketchdxi} to an auxilliary process.  Before this coupling, we first allow the DBM \eqref{eqn:sketchdxi} to run freely for an initial time interval of length $t_0$, where $t_0$ satisfies the compatibility conditions $\g \ll t_0 \ll \G^2$.  This is needed for several reasons.  Firstly, after $t_0$, we can apply the results \cite{landonyau} which state that rigidity holds wrt the free convolution law. Secondly, this regularizes the DBM flow in the sense that the free convolution will be regular on this scale; for example $|\rhofc' (E) | \leq C / t_0$.

\item{\bf Matching and coupling}.  For times $t \geq t_0$ we couple the DBM flow to another DBM flow started from an independent GOE ensemble.  That is, we define the process
\beq \label{eqn:sketchdyi}
\d y_i (t) = \sqrt{ \frac{2}{N} } \d B_i + \frac{1}{N} \sum_j \frac{1}{ y_i - y_j } \d t
  \eeq
  where initially $y_i (t_0)$ is distributed as a GOE ensemble independent from $\{ x_i \}_i$.  The point is that the Brownian motions in \eqref{eqn:sketchdxi} and \eqref{eqn:sketchdyi} are the same.     This idea first appeared in \cite{homogenization}.  Moreover, we re-scale and shift the DBM flow so that the classical eigenvalue locations match those of the semicircle law near a chosen energy $E$.  Due to the regularity of the free convolution law, we can match up to $\sqrt{ N t_0}$ eigenvalues.  This matching implies that
  \beq \label{eqn:cc2}
  x_i (t) - y_i (t) \sim \frac{\log(N)}{N},
  \eeq
for eigenvalues that are near the spectral energy $E$. 
At this point we are now regarding $t_0$ as a fixed a-priori scale on which the DBM flow is regular.  By running the coupling for times $t$ satisfying $t_0 \leq t \leq t_0 + t_1$ with $t_1 \ll t_0$, the DBM flow will not see the non-matching eigenvalues.

\item{\bf Discrete parabolic equation}.  The difference $w_i (t):= x_i (t) - y_i (t)$ satisfies the parabolic equation 
\beq
\del_t w_i = ( \L w )_i 
\eeq
where
\beq \label{eqn:Lsketch}
( \L w )_i := \frac{1}{N} \sum_j \frac{ w_j - w_i}{ (x_i - x_j ) ( y_i - y_j ) } .
\eeq
As $N \to \infty$, a natural limit for this equation is
\beq \label{eqn:fsket}
\del_t f (x) = \int \frac{ f(y) - f(x) }{ (x-y)^2} \rhofct (y) \d y \sim \int \frac{ f(y) - f(x) }{ (x-y)^2}  \d y 
\eeq
In order to justify the replacement of $\rhofct$ by a constant we will use its regularity on the scale $t_0$ and a short-range approximation of the DBM flow.  We omit the details in this simple sketch.  
\item{\bf Homogenization theory}. We may now write
\beq \label{eqn:xiyisk}
x_i (t_0 + t) - y_i (t_0 + t) = \sum_{ j} \UL_{ij} (t_0, t_0 + t) (x_j (t_0) - y_j (t_0 ) ),
\eeq
where $\UL$ is the semigroup for the equation \eqref{eqn:Lsketch}. We need to develop a homogenization theory in order to calculate the matrix elements $\UL_{ij}$.  We let $p_t (x, y)$ denote the fundamental solution of \eqref{eqn:fsket}.   
Let
\beq \label{eqn:ffsketch}
f_i (t) := p_t (x_i (t), \gamma_a )
\eeq
and
\beq \label{eqn:uusketch}
g_i (t) := N \UL_{i,a} (t_0+t, t_0 ).
\eeq
Our main calculation is
\beq
\d || f-g ||_2^2 = - c \langle (f-g), \L (f-g ) \rangle \d t + \d M_t + \frac{o(1)}{  t^2} \d t,
\eeq
where $M$ is a martingale.  Integrating this inequality in time (and dropping the time average for simplicity) we will obtain
\beq
\langle (f - g ), \L  (f-g)  \rangle \leq \frac{ o(1)}{ t^2} + \frac{1}{t} || g(0) - f(0) ||_2^2
\eeq
The second term on the RHS is not well-defined, as $f(0)$ is a delta function (as a distribution on $\rr$) and $g(0)$ is a discrete delta function.  In order to make sense of this quantity,  we introduce an additional regularization to the initial data $f(0), g(0)$.  We omit the details from this sketch but say that this roughly corresponds to convolving $f(0)$ and $g(0)$ with a mollifier which lives on a regularization scale $s_r$ (we will only use this notation of $s_r$ in this sketch and it is absent from the remainder of the paper).  This mollification allows us to take $||g(0) - f(0) ||_2^2 = o( t^{-1} )$.  By choosing the regularization scale $s_r \ll t$, we will also see that the regularization does not significantly affect the final value $f(t)$ and $g(t)$ (as the regularization/mollification scale $s_r$ is shorter than the natural scale $t$ of these functions).

We obtain
\beq
\langle (f(t) - g(t) ), \L  (f(t)-g(t))  \rangle \leq \frac{ o(1)}{ t^2}.
\eeq
Using a discrete $\dot{H}^{1/2} - \ell^\infty$ Sobolev inequality, this inequality will imply
\beq \label{eqn:ellinfsketch}
|| g_i   (t) - f_i (t) ||_\infty \leq \frac{N^{-\mathfrak{b}}}{t},
\eeq
for some positive $\mathfrak{b} >0$,
\item {\bf Cut-offs}.  The natural size of $g_i = N \UL_{ia} (t_0, t_0 + t)$ is $1/t$ for indices $i$ near $a$.  Hence, the estimate \eqref{eqn:ellinfsketch} determines the object $\UL_{ia}$ beyond its natural scale and we can use it to control the terms in the sum \eqref{eqn:xiyisk} for $j$ near $i$; that is, we can control approximately $o(N^{\mathfrak{b}} N t)$ terms using \eqref{eqn:ellinfsketch} and \eqref{eqn:cc2}.

For terms satisfying $N t N^{\mathfrak{b}/2} < |i-j| < \sqrt{N t_0 }$ we have the estimate \eqref{eqn:cc2} as well as the a-priori upper bound for $\UL$,
\beq \label{eqn:ULbdsketch}
N \UL_{ij} (t_0+t, t_0  ) \leq N^\eps p_t ( \gamma_i, \gamma_j ) \leq N^\eps \frac{ t}{ t^2 + (\gamma_i - \gamma_j)^2}
\eeq
for any small $\eps >0$.  This allows us to control the contribution to the sum \eqref{eqn:xiyisk} for terms in this range.

Finally we have to deal with the contribution of $j$ so that $|j-i | > \sqrt{N t_0}$.  Here, we do not have the estimate \eqref{eqn:cc2} due to lack of sufficient regularity of the initial data $x_i (0)$, and the fact that the decay \eqref{eqn:ULbdsketch} is not fast enough to counteract the growth of the LHS of \eqref{eqn:cc2} as $j$ moves further from $i$.  Instead, we will modify the processes  \eqref{eqn:sketchdxi} and \eqref{eqn:sketchdyi} and replace them by certain {\it short-range approximations}.  For the purposes of this sketch we will not define the approximations precisely.  We will just say that the local law and rigidity estimates allow us to replace the long-range contribution in  \eqref{eqn:sketchdxi} and \eqref{eqn:sketchdyi} of terms with $|i-j| > \ell$ large with a deterministic drift term.  Here $\ell$ is an additional scale chosen larger than $N t$.  This modifies the operator $\mathcal{L}$ to only allow jumping between sites $|i-j|< \ell$.  The behavior of the kernal $\UL$ is then modified to decay exponentially for $|i-j | > \sqrt{N t} \ell$ (i.e., simple random walk behavior in $1$ dimension with step-size $\ell$).  This latter property enables the cut-off of the non-matching terms where the estimate \eqref{eqn:cc2} fails.

\item {\bf Mesoscopic linear statistics}.  The homogenization theory proves that there is a smooth function $\zeta_t : \rr \to \rr$ so that (up to errors)
\beq
x_i (t_0+t) - y_i (t_0 + t) = \sum_{ |i-j| \leq N t} \zeta_t ( \gamma_i - \gamma_j ) ( x_j (t_0) - y_j (t_0) ) =: \zeta_x - \zeta_y.
\eeq
Roughly, $ \zeta_t (\gamma_i - \gamma_j ) \sim p_t ( \gamma_i, \gamma_j )$ where $p_t$ is the fundamental solution introduced above (to the PDE \eqref{eqn:fsket}). 
This reduces the microscopic problem to a simpler mesoscopic one.   For this mesoscopic observable, we calculate the characteristic function and prove that 
\beq \label{eqn:mesosketch}
 \left| \psi  (\lambda ) \right| := \left| \ee [ \e^{ \i \lambda \zeta_x  } ] \right| \leq \e^{ -\lambda^2 c_x \log (N) } + N^{- \eps}
\eeq
 for some $\eps >0$ and a constant $c_x >0$.
 
 We remark that we will use a slightly different convention for $\zeta_t$ in the full proof than in the simple sketch given here.  The precise definition of $\zeta_t$ is given Section \ref{sec:homog}.   To differentiate the two conventions we use $\zeta(x, t)$ instead of $\zeta_t (x)$ in the rest of the paper.
\item {\bf Fourier cut-off}. We now proceed similarly to \cite{homogenization}.  It suffices to consider for smooth test functions $Q$ sums of terms of the form,
\beq
\ee [ Q ( N ( x_{i} - E), N ( x_{j} - x_i )  ) ] .
\eeq
The homogenization theory shows that
\beq \label{eqn:Qxsketch}
\ee [ Q ( N ( x_i - E), N(x_j - x_i )) ] = \ee  [ Q  ( N (y_i - E + (\zeta_x - \zeta_y ) ) , N (y_j - y_i) ) ] + o (1).
\eeq
By Fourier duality we have
\beq
\ee  [ Q  ( N (y_i - E + (\zeta_x - \zeta_y ) ) , N (y_j - y_i) ) ] = \int  \ee[ \hat{Q} ( \lambda, N ( y_i - y_j ) ) \e^{ \i \lambda N ( y_i - E + \zeta_x- \zeta_y ) } ] \psi  ( \lambda )   \d \lambda.
\eeq
Here $\hat{Q}$ denotes the Fourier transform of $Q$ in the first variable.  By \eqref{eqn:mesosketch} we can cut off the Fourier support of $\hat{Q}$ in the range $| \lambda | \geq \delta$ for any small fixed $\delta >0$.  Therefore it suffices to consider observables $Q$ with Fourier support contained in $| \lambda | \leq \delta$.

\item {\bf Reverse heat-flow}.  Running the same argument with a third ensemble $z$ distributed as the GOE shows that
\beq \label{eqn:Qzsketch}
\ee[ Q ( N (z_i - E), N ( z_j - z_i) ) ] = \ee [  Q ( N (y_i - E  + ( \zeta_z - \zeta_y )), N ( y_j - y_i) ) ] + o (1).
\eeq
As in \cite{homogenization} we see that from \eqref{eqn:Qxsketch} and \eqref{eqn:Qzsketch} that fixed energy universality will follow if we can prove that the function
\beq
F (a ) :=\ee [  Q ( N (y_i - E  + ( a - \zeta_y )), N ( y_j - y_i) ) ]
\eeq
is approximately constant, for $Q$ a function of small Fourier support.  The argument to prove this is the same as in \cite{homogenization}.  What is new is that we have analyzed the mesoscopic statistic $\zeta_x$ and used it to complete the Fourier cut-off in the previous step.  In \cite{homogenization} a Fourier cut-off was also used, but only for $\delta$ a large constant; here, $\delta$ is allowed to be any small constant.  In \cite{homogenization} this caused some restriction in the following argument on how small $t$ can be in proving fixed energy universality.  Here this restriction is removed due to the Fourier cut-off $\delta$.

We would like to prove that $F(a)$ is constant. Define $F_h (a) := F(a +h) - F(a)$.  By translation invariance of the local GOE statistics we know that $\ee[ F_h ( \zeta_z ) ] = o (1)$.  We will prove that $\zeta_z$ is close to a Gaussian with variance $c_z  \log (N)$.  In order to conclude that $F_h$ is small we run the reverse heat flow argument of \cite{homogenization}.  We see that
\beq
\left| \hat{F}_h ( \lambda ) \e^{  - c_z \log (N) \lambda^2} \right| \leq N^{- \eps}
\eeq
for some $\eps >0$, {\it independent} of the $\delta$ chosen above.  By the Fourier support restriction on $Q$ we see that $ | \hat{F}_h ( \lambda ) | \leq N^{ \delta^2 c_z } N^{ - \eps}$.   Hence for $\delta$ small enough we get that $\hat{F}_h = o (1)$ and we conclude that $F_h$ is small.  This proves fixed energy universality.

\end{enumerate}

\subsection{Notation}
We will use the following notion of overwhelming probability.
\bed
We say that an event $\F$ holds with overwhelming probability if for any $D>0$ we have $\pp [ \F^c] \leq N^{-D}$ for large enough $N$.  If we have a family of events $\{ \F (u) \}_u$ then we will say that $\{ \F (u) \}_u$ holds with overwhelming probability if $\sup_u \pp [ \F^c (u) ] \leq N^{-D}$ for large enough $N$.
\eed
For two positive $N$-dependent quantities $a_N$ and $b_N$ we say that $a_N \asymp b_N$ if there are constants $c$ and $C$ s.t. $c a_N \leq b_N \leq C a_N$.  

In our work we use $C$ to denote a  positive constant that can change from line to line.  The constant $C$ will typically only depend on the constants appearing in the assumptions on $V$.

For $A, B \in \rr$ we denote
\beq
[[A, B]] := [A, B] \cap \zz.
\eeq


\section{Homogenization} \label{sec:homog}

In this section we prove a homogenization result for DBM.   This reduces the problem of fixed energy universality of the model $H_t$ to a problem involving mesoscopic statistics.  Given a real symmetric matrix $M$ with eigenvalues $\lambda_1 (M)\leq \cdots \leq \lambda_N (M) $, we define Dyson Brownian motion with $\beta=1$ and initial data $M$ 
  to be the process satisfying
\beq
\d \lambda_i = \sqrt{ \frac{2}{N}} \d B_i + \frac{1}{N} \sum_{j \neq i} \frac{1}{ \lambda_i - \lambda_j } \d t , \qquad \lambda_i (0) = \lambda_i (M)
\eeq
At each fixed time $t \geq 0$, the particles $\{ \lambda_i (t) \}_i $ are equal in distribution to the eigenvalues of the matrix $M + \sqrt{t} W$,  where $W$ is a GOE matrix independent of $M$.  It is well-known that there is a unique strong solution to the above system of SDEs and the sample paths are continuous a.s.  Recall that we want to study the eigenvalues of the matrix $H_t :=V + \sqrt{t} W$.  We will do this by studying the DBM flow for $t$ in the regime 
\beq
t_0 \leq t \leq t_0 + t_1.
\eeq
Here $t_0$ and $t_1$ are times defined by $t_i = N^{\om_i}/N$ and $\om_1 < \om_0$.  The time $t_0$ satisfies $ \g N^{\sigma} \leq t_0 \leq N^{-\sigma} \G^2.$

  Our study requires the choice of an index $i_0 \in \C_q$, with $\C_q$ defined in \eqref{eqn:cqset}.   We will compare eigenvalues near $i_0$ to the GOE. At time $t_0$ the eigenvalue density of $H_{t_0}$ is given by the free convolution law $\rho_{\fc, t_0}$ as defined in Section \ref{sec:fc}.  In this section we are going to assume that 
\beq
\gamma_{i_0} (t_0) = 0, \qquad \rho_{\fc, t_0} (0) = \rhosc(0). 
\eeq
In applications of the homogenization theorem this will be implemented by a re-scaling and shift of $V$ and a re-scaling of time.  For times $t \geq t_0$ we define $x_i (t)$ to be the solution of 
\beq \label{eqn:xdef}
\d x_i (t) := \sqrt{\frac{2}{N}}  \d B_i  + \frac{1}{N} \sum_{j \neq i} \frac{1}{ x_i (t) - x_j (t)} \d t
\eeq
with initial data $x_i (t_0) = \lambda_i ( H_{t_0} )$.

We now introduce the coupled GOE process. For times $t \geq t_0$ define $y_i (t)$ as the solution to
\beq \label{eqn:ydef}
\d y_i (t) =  \sqrt{\frac{2}{N}}  \d B_i  + \frac{1}{N} \sum_{j \neq i} \frac{1}{ y_i  (t) -y_j (t)} \d t
\eeq
where the initial data $y_i (t_0)$ are the eigenvalues of a GOE matrix independent of $\{ x_i (t) \}_i$.  Above, the Brownian motions are the same as those appearing in \eqref{eqn:xdef}.  At times $t \geq t_0$ the particles $\{ y_i (t) \}_i$ are distributed as
\beq
\{ y_i (t) \}_i \equald \left\{ \lambda_i  \left( \sqrt{ 1 + (t- t_0 ) } W' \right) \right\}_i.
\eeq
At times $t \geq t_0$ the $y_i (t)$ satisfy a rigidity estimate with respect to the classical eigenvalue locations $\sqrt{ 1 + (t-t_0 ) } \gsc_i$ where $\gsc_i$ denote the classical eigenvalue locations of the semicircle law $\rhosc$.  For our purposes we adopt the convention $N/2 :=\lceil N/2 \rceil$ as well as $\gsc_{N/2} =0$.

In order to state the following theorem we introduce a function $\zeta(x, t) : \rr^2 \to \rr$.  It is defined in terms of a fundamental solution $p_t (x, y)$ to a specific non-local integral equation that we will compare our process to.  This non-local integral equation is  \eqref{eqn: Ketadef} below; its definition requires a cut-off which will be introduced over the next part of the proof, so we defer the definition for now.  For the time being it will suffice to just assert the existence of this function $\zeta(x, t)$ and summarize some of its properties  in Proposition \ref{prop:zetaprop} below.  In terms of $p_t(x, y)$ we have,
\beq \label{eqn:zetarev1}
\zeta (x, t) = \frac{1}{N} p_t (0, x/\rhosc (0) )
\eeq

The following theorem is the main result of this section.
\bet \label{thm:mainhomog}
Fix $t_0 = N^{\om_0}/N$ satisfying $g N^\sigma \leq t_0 \leq N^{ - \sigma} \G^2$ for $\sigma >0$.  Let $t_1 = N^{\om_1}/N$ with $0 < \om_1 < \om_0/ 2$. Let
\beq
t_2 = \max \{ N^{-\om_1/15} t_1 , N^{ - 4 ( \om_0/2-\om_1)/3} t_1 \}.
\eeq
 Let $0 < \eps_b < \min \{ ( \om_0/2 - \om_1)/3, \om_1/60 \}$.  Let $i_0 \in \C_q$.  Assume that
\beq \label{eqn:oldrev1}
\gamma_{i_0} (t_0) = 0, \qquad \rho_{\fc, t_0} (0) = \rhosc (0).
\eeq
With overwhelming probability we have the following estimates.  For every $|u| \leq t_2$ and $|i| \leq N t_1 N^{\eps_b}$ we have
\begin{align} \label{eqn:oldrev2}
& \left( x_{i_0+i} (t_0 + t_1 + u ) -  \gamma_{i_0 } ( t_0 + t_1 + u ) \right)- y_{ N/2+i} (t_0 + t_1 + u ) \notag\\
= & \sum_{ |j| \leq N^{\om_1 + (\om_0/2-\om_1)/3 } } \zeta \left(\frac{i-j}{N}, t_1 \right) \left[  x_{i_0+j} (t_0 )  - y_{ N/2+j} (t_0   )  \right]+\frac{N^\eps}{N} \O \Bigg( \frac{ N^{\om_1/3} }{ N^{\om_0/6}}+   \frac{1}{ N^{\om_1/60}}\Bigg)
\end{align}
\eet

Theorem \ref{thm:mainhomog} will be a consequence of Theorem \ref{thm:xyref} below.   For the mesoscopic statistic $\zeta (x, t_1)$ we have the following properties.
\bep \label{prop:zetaprop}
The function $\zeta (x, t)$ satisfies for $0 \leq t \leq 1$ and $x \in \rr$ the following.  We have,
\beq
\int \zeta (x, t) \d x = 1, \qquad 0 \leq \zeta (x, t) \leq C \frac{ t}{ x^2 + t^2} 
\eeq
and
\beq
\left| ( \del_x)^k \zeta (x, t) \right| \leq \frac{ C}{ t^k} \frac{ t}{ x^2 + t^2 } , \qquad k=1, 2, 3.
\eeq
\eep

\subsection{Re-indexing} \label{sec:resc}
In this subsection we are going to make some assumptions which will greatly simplify notation.  
In Appendix \ref{a:reindex} we present an argument which reduces the general case to these assumptions.  

Let $i_0$ be as in Theorem \ref{thm:mainhomog}.  We assume that $N$ is odd and that $i_0 = (N+1)/2$.  
Note that with our convention, $\gsc_{i_0} = 0$. 

Presently the eigenvalues are labelled by the integers $[[1, N]]$.  We re-label the eigenvalues so that they are indexed by $[[-(N-1)/2, (N-1)/2]]$.  The eigenvalues are then $x_{-(N-1)/2} \leq x_{-(N-3)/2}  \leq \cdots \leq x_{(N-1)/2} $ and $i_0 = 0$.  We furthermore have that $\gamma_0 (t_0 ) = 0$.

We adopt these assumptions for the remainder of Section \ref{sec:homog}, apart from Section \ref{sec:xythmproof} which is where we prove Theorem \ref{thm:mainhomog} and must therefore unravel the re-labelling.  We also adopt this convention for Section \ref{sec:finitespeed}.  It will not be used in the other sections.  

\subsection{Interpolation}

\label{sec:interpolation}

In the work \cite{homogenization}, the parabolic equation satisfied by the differences $u_i := x_i - y_i$ was directly considered.  The jump rates of the generator of this equation are hard to control as they involve both of the differences $(x_i - x_j)$ and $(y_i - y_j)$.  In this paper we define a continuous interpolation which allows us to consider a family of parabolic equations whose generators are easier to control.

We now introduce this interpolation.  For $0 \leq \alpha \leq 1$, we define $z_i (t, \alpha )$ as the solution to 
\beq \label{eqn:dzi}
\d z_i (t, \alpha ) = \sqrt{ \frac{2}{N}} \d B_i +  \left( \frac{1}{N} \sum_{j} \frac{1}{ z_i (t, \alpha ) - z_j (t, \alpha ) } 
 \right) \d t, \qquad z_i (0, \alpha) = \alpha x_i (t_0) + (1- \alpha ) y_i (t_0 ).
\eeq
Note that $z_i (t, 0) = y_i (t_0 + t)$ and $z_i (t, 1) = x_i (t_0 + t)$.  Note that we have effectively introduced a time shift which sets $t_0 = 0$.  In the remainder of Section \ref{sec:homog} we will refer to $\{ z_i (t , \alpha ) \}_i$ as ``particles,'' instead of using the terminology of eigenvalues.  Below we will introduce some other processes with similar notation that we will also refer to as particles. 

We will soon see that like the $x_i$ and $y_i$, the $z_i (t, \alpha )$ satisfy a rigidity estimate.  However, we have some freedom in choosing the measure with which to construct the classical particle (eigenvalue) locations (note that here we are still referring to quantiles of a measure with which the empirical density of the $z_i (t, \alpha)$ but in accordance with calling the $z_i (t, \alpha)$ particles we will call them classical particle locations).  One choice is the free convolution of the empirical measure of the initial data $z_i (0, \alpha )$ with the semicircle law.  However, this law is somewhat singular for short times $t$, and does not reflect the fact that at time $t_0$, the particles $x_i (t_0)$ satisfy a rigidity estimate wrt $\rhofc (t_0)$ which has some regularity properties (e.g., $|\rhofc' (t_0) | \leq C / t_0$).

To compensate for this, we construct a measure $\nu (\d x, \alpha )$ that has a  density near $0$ that is at least as smooth as $\rhofc (t_0)$.   The construction is described in detail in Appendix \ref{a:free}, and for now we just sketch its construction.  This measure constructed is random but has good properties with overwhelming probability.  

We need $\nu ( \d x, \alpha )$ to satisfy two properties which motivate its construction.  Firstly, we would like it to have a smooth density near $0$ which is at least as regular as $\rhofc (t_0)$.  Secondly, we need the initial data $z_i (0, \alpha )$ to be approximated by $\nu ( \d x, \alpha )$ down to the optimal scale $\eta \gtrsim 1/N$, so that at later times $t$, the particles $z_i (t, \alpha )$ follow the free convolution of $\nu ( \d x , \alpha )$ with the semicircle distribution.  

We now sketch the construction of the measure $\nu ( \d x, \alpha )$; complete details are given in Appendix \ref{a:free}.  The construction requires the choice of a parameter $0 < q^* <1$ which we now fix.  This $q^*$ is the same as that which appears in Appendix \ref{a:free}.  In the interval $[-q^* \G, q^* \G]$,  one can construct, using the inverse function theorem, a density whose quantiles (in this case defined by starting the integration of the density from $0$) equal $\alpha \gamma_i (t_0) + (1-\alpha ) \gsc_i$.  This density has as good regularity properties as $\rhofc (t_0)$.  Since $z_i (0, \alpha )$ satisfy a rigidity estimate in the interval $[-q^* \G, q^* \G]$, this density gives the required approximation for $z_i$ in this interval.  To approximate the $z_i$ outside the interval $[-q^* \G, q^* \G]$ we can take $\nu ( \d x , \alpha )$ to consist of a Dirac delta mass at each $z_i (t, \alpha )$ such that $|z_i (t, \alpha ) | > q^* \G$.  Since the delta functions are outside of the interval $[-q^* \G, q^* \G]$ we do not  affect the regularity inside this interval.  Clearly $\nu ( \d x , \alpha )$ gives a good approximation to $z_i (t, \alpha)$.

We now record formally some of the needed properties of the measure $\nu ( \d x , \alpha )$.  Again, we mention that the explicit construction appears in Appendix \ref{a:free}.   One of the key properties will be that, although the measure $\nu ( \d x, \alpha )$ and its free convolution are {\it random}, for $\alpha=0, 1$ the quantitative properties inside $|x| \leq q^* \G$ coincide with $\rhofct$ and $\rhosc$ and are {\it deterministic}.  In particular, certain quantiles of the measure at later times are deterministic up to an error term.  This is summarized in Lemma \ref{lem:newquant} below.

Let $k_0$ be the largest index so that 
\beq
|\gamma_{k_0} (t_0 ) |  \leq q^* \G, \qquad |\gamma_{- k_0 }  (t_0) |\leq q^* \G , \qquad | \gsc_{-k_0} | = | \gsc_{k_0} | \leq q^* \G.
\eeq
Note that $k_0 \asymp N \G$.  The measure $\nu ( \d x, \alpha )$ has a nonvanishing density on the interval
\beq
\Galp :=[ \alpha \gamma_{-k_0} (t_0) + (1 - \alpha ) \gsc_{ - k_0} (t_0) , \alpha \gamma_{k_0} (t_0) + (1- \alpha ) \gsc_{k_0} ] 
\eeq
but has a singular part which may overlap with an $o ( \G)$ portion of $\Galp$ at its boundary; for any $0 < q < 1$, $\nu (\d x, \alpha )$ is purely a.c. on $q \Galp$.

We now define $\rho (E, t, \alpha)$ to be the free convolution of $\rho (E, 0, \alpha) \d E := \nu ( \d E , \alpha )$ and the semicircle distribution with Stieltjes transform $m (z, t, \alpha )$.  We have abused notation here slightly as at time $t=0$, the measure $\nu (\d x, \alpha )$ is the sum of an absolutely continuous part and a sum of delta functions and it is only for times $t>0$ that it has a true density.  However, this will not affect anything as whenever we write $\rho (E, 0, \alpha)$ we only be referring to $E$ near $0$ where the measure $\nu (\d x, \alpha )$ is purely a.c. 

The following holds for the free convolutions. We defer the proof to Appendix \ref{a:free}. 
\bel \label{lem:interprho}
Let $\delta >0$.  All of the following holds for $|E| \leq N^{-\delta} t_0$ and $t \leq N^{-\delta} t_0$, and $N^\delta/N \leq \eta \leq 10$, and with overwhelming probability.  For the Stieltjes transform we have
\beq
| m (t, E, \alpha ) - m (t, 0 , \alpha ) - ( m (t, E, 0) - m (t, 0, 0) ) | \leq C \log(N) \left( \frac{|E|}{t_0} + \frac{t}{t_0} \right)
\eeq
We have
\beq \label{eqn:mderalph}
| \del_z m (z, t, \alpha ) | \leq \frac{C}{t_0}.
\eeq
For the free convolution laws we have
\beq
\left| \frac{\d }{ \d E} \rho (E, t, \alpha ) \right| \leq \frac{C}{t_0}, \qquad 
\rho (0, 0, \alpha ) = \rho (0, 0, 0) = \rhosc (0),
\eeq
and
\beq
\left| \rho(E, t, \alpha ) - \rhosc (0) \right| \leq C \log(N) \frac{ |E| + t}{t_0}.
\eeq
Moreover, for $0 < q < 1$ and $E \in q \Galp$, $N^\delta/N \leq \eta \leq 10$,
\beq
| m (z, t, \alpha ) | \leq C \log (N), \qquad c \leq \Im [ m (t, z, \alpha ) ] \leq C.
\eeq
\eel
The classical particle locations of the measure $\rho(E, t, \alpha )$ are denoted by $\gamma_i (t, \alpha)$ and are defined by
\beq
\frac{1}{2} + \frac{i}{N} =  \int_{-\infty}^{ \gamma_i (t, \alpha ) } \rho (E, t, \alpha ) \d E.
\eeq
We have that $\gamma_0 (0, \alpha ) = 0$ by the definition of $\rho (E, 0, \alpha)$.
We will also need to relate $\gamma_i (t, 0)$ and $\gamma_i (t , 1)$ back to $\gsc_i$ and $\gamma_i (t)$, respectively.
\bel \label{lem:newquant}
We have for any $\eps >0$ and $\om_1 < \om_0/2$ with overwhelming probability,
\beq \label{eqn:newq1}
\sup_{ 0 \leq t \leq 10 t_1} | \gamma_0 (t, 1) - \gamma_0 (t_0 + t) | \leq \frac{ N^\eps}{N} \frac{ N^{\om_1}}{ N^{\om_0/2}}
\eeq
and
\beq \label{eqn:newq2}
\sup_{ 0 \leq t \leq 10 t_1} | \gamma_0 (t, 0) - 0 | \leq \frac{ N^\eps}{N} \frac{ N^{\om_1}}{ N^{\om_0/2}}.
\eeq
For $|j|, |k| \leq N^{\om_0/2}$ we have with overwhelming probability,
\beq \label{eqn:newgaps}
\gamma_k (t, \alpha ) - \gamma_j (t, \alpha ) = \frac{1}{N} \frac{k-j}{\rhosc(0) } + \O \left( \frac{1}{N} \right),
\eeq
for any $t \leq 10 t_1$, with $\om_1 \leq \om_0/2$ and $0 \leq \alpha \leq 1$.
\eel
We again defer the proof to Appendix \ref{a:free}.

We  define the empirical Stieltjes transforms by
\beq
m_N (z, t, \alpha ) = \frac{1}{N} \sum_{i=1}^N \frac{ 1}{ z_i (t, \alpha ) -z }.
\eeq
The measures $\rho (E, t, \alpha )$ are $\alpha$ dependent.  We are eventually going to introduce some short-range and long-range cut-offs, and differentiate certain objects in $\alpha$.  As the cut-offs are inherently discrete (they involve the particle indices), we want to choose them independent of $\alpha$ to interact nicely with the differentiation.  This requires the introduction of the following index sets.  Let $k_1$ be the largest index so that
\beq
\bigcup_{0 \leq \alpha \leq 1 } [ \alpha \gamma_{-k_1} (t_0) + (1 - \alpha ) \gsc_{ - k_1} (t_0) , \alpha \gamma_{k_1} (t_0) + (1- \alpha ) \gsc_{k_1} ] \subseteq \bigcap_{ 0 \leq \alpha \leq 1 } \calG_\alpha \cap \{ - \calG_\alpha \} 
\eeq
where $- \F := \{ x : -x \in \F \}$.

Finally, for any $0  < q < 1$, we define
%
\beq
\Chat_q = \{ j: |j| \leq q k_1 \}.
\eeq
The cardinality $\Chat_q$ satisfies $| \Chat_q | \asymp q \G N$.  This definition is just so that particles $z_k ( t, \alpha )$, for $k \in \Chat_q$ have nice qualitative properties uniformly in $\alpha$ (the constants degenerate as $q\to 1$).  For example, the optimal rigidity estimate holds for any $k \in \Chat_q$.  Moreover, the classical particle locations $\gamma_k (t, \alpha )$ for $k \in \Chat_q$ will all be contained in a symmetric interval $[-q' \G, q' \G]$ for some $q'$ on which all the densities $\rho(E, t, \alpha )$ have good properties.

We will also use tacitly that for any $q$ we have for $j,k \in \Chat_{q}$ that
\beq
c \frac{ |j-k|}{N} \leq | \gamma_{j} (t, \alpha ) - \gamma_k (t, \alpha ) | \leq C \frac{|j-k|}{N}.
\eeq
We have the following rigidity and local law estimates.
\bel \label{lem:rig} Let $\eps >0$, $\delta >0$, $\delta_1 >0$, $D>0$ and $0 < q < 1$.  We have
\beq \label{eqn:opt}
\pp\left[ \sup_{0 \leq t \leq N^{-\delta_1 } t_0} \sup_{ i \in \Chat_q } \sup_{ 0 \leq \alpha \leq 1}  |z_i (t, \alpha ) - \gamma_i (t, \alpha ) | \geq \frac{N^{\eps}}{N} \right] \leq N^{-D}.
\eeq
We have also, 
\beq \label{eqn:semi}
\pp \left[ \sup_{ N^{\delta}/N \leq \eta \leq 10}  \sup_{ 0 \leq t \leq N^{-\delta_1 }t_0} \sup_{0 \leq \alpha \leq 1 } \sup_{ E \in q \Galp }  | m_N (z, t, \alpha) - m (z, t, \alpha ) | \geq \frac{N^\eps}{ N \eta} \right] \leq N^{-D}.
\eeq
\eel
We defer the proof to Appendix \ref{a:free}.

\subsubsection{Reformulation of homogenization} \label{sec:homogref}

We have proven already a few estimates about the processes $\{ z_i (t, \alpha ) \}_i$ introduced above.  In order to clarify what properties of the processes we are using we are going to reformulate the relevent estimates as \emph{hypotheses} and use them to prove a theorem.  This theorem will then be used to prove Theorem \ref{thm:mainhomog}.   Lemmas \ref{lem:interprho} and \ref{lem:rig} imply that the hypotheses \ref{item:zgoe}-\ref{item:supz} are satisfied by the processes $\{ z_i (t, \alpha)\}_i$ considered above.

The point of this is to give a black box result which can be applied to other DBM processes.  Essentially the two key inputs are the rigidity estimates and the regularity of the free convolution measure that the particles are being compared to.  Here, we have proven the rigidity estimates using the fact that the process comes from a matrix model.  In other settings, e.g., non-classical $\beta$, such methods may not be available and the rigidity could be proven by other means.  The regularity of the measure $\rho_{\fc, t_0}$  describing the empirical eigenvalue distribution of $x_i (t_0)$ comes from the fact that we were able to run the DBM for a regularization period before running the coupling; however, the initial data under consideration in different contexts may obey this regularity condition due to other reasons.  In such cases, the following reformulation will apply.



Our starting point is that there are two processes $z_i (t, 0)$ and $z_i (t, 1)$ which satisfy \eqref{eqn:dzi} (for $\alpha=0, 1$, of course) with initial data which we denote by $z_i (0, 0)$ and $z_i (0, 1)$.  Given these two processes, we then construct the interpolating processes $z_i (t, \alpha)$ by \eqref{eqn:dzi}, except that the initial data is $z_i (0, \alpha ) = \alpha z_i (0, 1) + (1-\alpha ) z_i (0, 0)$.  
We assume the following on the processes $z_i (t, \alpha)$.  
\benr
\item \label{item:zgoe}$z_i (0, 0)$ is a GOE ensemble.
\item \label{item:rho} There is a law $\rho(E)$ with Stieltjes transform $m(z)$ and parameters $t_0 = N^{\om_0}/N$ and $\G$ with $t_0 \leq N^{-\sigma} \G^2$ so that the following hold.
\beq
c \leq \rho (E) \leq C, \qquad \left| \del_z^k m (z) \right| \leq C/ (t_0)^k, \qquad |E| \leq \G, \quad k=1, 2.
\eeq
The classical particle locations of $\rho (E)$ satisfy $\gamma_0  = 0$ and $\rho (0) = \rhosc (0)$.  
\item \label{item:estsref} 
We need also measures that give the empirical particle density of the $z_i (t, \alpha)$.  For this, we proceed the same as in Section \ref{sec:interpolation}, with $\rho (E)$ from the previous assumption taking the place of $\rho_{\fc, t_0}$.  That is, we fix a $0 < q^* < 1$, and first construct $\nu ( dE, \alpha)$ as a density on $[- q^* \G, q^* \G]$, add to it some delta functions and then take the free convolution.  We denote the resulting measure by $\rho (E,t, \alpha ) d E$.  We then use much of the same notation as introduced in Section \ref{sec:interpolation}; we have the classical particle locations $\gamma_i (t, \alpha)$, Stieltjes transform $m(z, t, \alpha)$ and index sets $\Chat_q$ and energy windows $\Galp$.

The assumption is then the following.  We assume the following rigidity and local laws,
\beq \label{eqn:rigref}
\pp \left[ \sup_{0 \leq t \leq N^{-\delta_1} t_0 } \sup_{ i \in \Chat_q} \sup_{ 0 \leq \alpha \leq 1 } | z_i (t, \alpha ) - \gamma_i (t, \alpha ) | \geq \frac{N^\eps}{N} \right] \leq C N^{-D}
\eeq
and
\beq \label{eqn:lawref}
\pp \left[ \sup_{ 0 \leq t \leq N^{-\delta_1} t_0 } \sup_{ N^{\delta-1} \leq \eta \leq 10} \sup_{ 0 \leq \alpha \leq 1 }  \sup_{ E \in q \Galp } | m_N (z, t, \alpha ) - m (z, t, \alpha ) | \geq \frac{ N^\eps}{ N \eta } \right] \leq C N^{-D}
\eeq
for any $\delta, \eps, \delta_1, D>0$ and $0 < q < 1$.
\item \label{item:supz} There is a $C>0$ so that for any $D >0$, 
\beq \label{eqn:supzref}
\pp \left[ \sup_{ 0 \leq t \leq t_0 }  \sup_i | z_i (t, 1) | \geq N^C \right] \leq N^{-D}.
\eeq
for large enough $N$.
\eenr

Since the only properties of $\rho_{\fc, t_0}$ used in the proofs of Lemma \ref{lem:interprho} and  \eqref{eqn:newgaps} are those assumed in \ref{item:rho}, the estimates hold for new measures $\rho (E, t, \alpha ) d E$ and their quantiles as introduced in \ref{item:estsref}.  We record this in the following lemma.
\bel \label{lem:rev1}
Under the above set-up and assuming \ref{item:zgoe}-\ref{item:supz}, the estimates of Lemma \ref{lem:interprho} and the estimate \eqref{eqn:newgaps} 
 hold for the measures $\rho (E, t, \alpha)$ and the classical particle locations $\gamma_i (t, \alpha)$. 
\eel

 Under the above assumptions we will prove the following, from which Theorem \ref{thm:mainhomog} will be deduced.

\bet \label{thm:xyref} Let $z_i (t, 0)$ and $z_i (t, 1)$ be defined as at the start of Section \ref{sec:homogref}. 
Suppose that assumptions \ref{item:zgoe}-\ref{item:supz} hold.  Let $0 < \eps_b < \eps_a$.  Let $t_2 := t_1 N^{-\eps_2}$ with $\om_1 - \eps_1 >0$.  Let $\eps >0$.   There is a parameter $\ell = N^{\om_\ell}$ so that the following holds (see \eqref{eqn:ccdef} below for its definition). There is an event with overwhelming probability on which the following holds.  For any $|i| \leq N t_1 N^{\eps_b}$ and $|u| \leq t_2$ we have
\begin{align}
&z_i (t_1 + u, 1 ) - z_i (t_1 + u, 0)= \left( \gamma_0 (t_1 + u, 1) - \gamma_0 (t_1 +u, 0) \right) \notag\\
+& \sum_{ |j| \leq N t_1 N^{\eps_a} } \zeta \left( \frac{i-j}{N}, t_1 \right) ( z_j (0, 1) - z_j (0, 0) ) \notag\\
+ &\frac{N^\eps}{N} \O \left( N^{\om_1} \left( \frac{ N^{\om_A}}{N^{\om_0}} + \frac{1}{ N^{\om_\ell}} + \frac{1}{ \sqrt{ N \G } } \right) + \frac{1}{ N^{\eps_a}}+ N^{\eps_2+\eps_a} \left( \frac{ ( N t_1 )^2}{ \ell^2} + \frac{1}{ ( N t_1 )^{1/10} } \right) + N^{\eps_a-\eps_2/2 } \right)
\end{align}
\eet
\remark 
In the set-up considered here, there is no assumption about whether or not the initial data $z_i (0, 1)$ satisfy the local law with respect to $\rho (E)$.  In the case that this is true, the analog of \eqref{eqn:newq1} holds.  That is, consider the free convolution of $\rho (E)$ with the semicircle distribution at time $t$ and its quantiles $\gamma_i (t).$  Note that this is in general different than $\rho (E, t, 1)$ and $\gamma_i (t, 1)$.  Under the assumption that the local law holds for $z_i (0, 1)$ with respect to $\rho (E)$ and the regularity assumption \ref{item:rho} we see that the estimate \eqref{eqn:newq1} for the new quantities $\gamma_0 (t, 1)$ and $\gamma_0 (t)$.  On the other hand, the estimate \eqref{eqn:newq2} holds for the new quantity $\gamma_0 (t, 0)$ as $z_i (0, 0)$ obey a local law with respect to the semicircle distribution.  Hence, the classical particle locations appearing above can be replaced by deterministic counterparts which do not depend on the realization of the initial data $z_i (0, 1)$, under the assumption of a local law for the initial data $z_i (0, 1)$. \qed


\subsection{Short-range DBM} \label{sec:shortrange}
The centered $\tilz_i$'s are given by
\beq
\tilz_i (t, \alpha ) := z_i (t, \alpha ) - \gamma_0 (t, \alpha ).
\eeq
We also define the classical locations of the centered $\tilz_i (t, \alpha)$ by
\beq
\tilg_i (t, \alpha ) = \gamma_i (t, \alpha)- \gamma_0 (t, \alpha ).
\eeq
Note that $\gamma_i (t, \alpha )$ satisfies (see \cite{Kevin2})
\beq
\del_t \gamma_i (t, \alpha ) = - \Re [ m ( \gamma_i (t, \alpha ), t, \alpha ) ] 
\eeq
and so for $i \in \Chat_q$ we have
\beq
| \del_t \gamma_i (t, \alpha ) | \leq C \log (N)
\eeq
by Lemma \ref{lem:rev1}.  The $\tilz_i (t, \alpha)$ satisfy the equations
\beq \label{eqn:tilz}
\d \tilz_i (t, \alpha ) = \frac{\d B_i } { \sqrt{N}} + \left( \frac{1}{N} \sum_j \frac{1}{ \tilz_i (t, \alpha ) - \tilz_j (t, \alpha ) }  + \Re [ m (\gamma_0 (t, \alpha ) ,t, \alpha ) ]
 \right) \d t
\eeq
We introduce the following cut-off dynamics for the $\{z_i\}_i$.  Its definition will use the parameters
\beq \label{eqn:ccdef}
\om_\ell >0, \qquad \om_A > 0, \qquad 0 < q_* <  1.
\eeq
Before defining the short-range approximation we outline the role of each of these parameters in the definition.  The parameter $\ell = N^{\om_\ell}$ is the most fundamental.  It is the ``range'' of the short-range approximation.  In our short-range dynamics we allow particles $i$ and $j$ to interact iff $|i-j| \leq \ell$.  In order for this approximation to be effective we need $\ell \gg N t_1$; that is $\ell$ must exceed the ``range'' of DBM which is $t_1$.  

For particles $i$ near $0$ we can use the rigidity estimates to replace the long-range part of the dynamics (i.e., the force coming from particles $j$ s.t.  $|i-j| > \ell$) by a deterministic quantity.  Since the free convolution law is regular the dependence of this deterministic quantity on the particle index $i$ is smooth; we can therefore replace it by something independent of the particle index $i$ if $|i| \leq N^{\om_A}$, as long as we choose $\om_A$ to be smaller than the regularity scale of the free convolution law which is governed by $t_0$.

Finally, since the rigidity estimates only hold for particles near $0$ we need to make a different cut-off for particles away from $0$; this is the role of $q_*$.  We will not make a short range cut-off for particles $i \notin \Chat_{q_*}$.

We now turn to the definition of the short-range approximation.  We first introduce some notation. Let $\om_\ell >0$ be as above. For each $0 < q < 1$, define the short range index set $\A_q$ by
\begin{align}
\A_q := \{ (i, j) : |i - j | \leq N^{\om_\ell} \} \cup \{ (i, j) : ij >0, i \notin \Chat_q, j \notin \Chat_q \}.
\end{align}

We introduce the following notation.  Let
\beq
\sum^{\mathcal{A}_q, (i)}_j := \sum_{ j : (i, j) \in \A_q }, \qquad \sum^{\mathcal{A}^c_q, (i)}_j := \sum_{j : (i, j) \notin \A_q }.
\eeq
For a fixed $i$ let $j_{\leq, i}$ be the smallest index s.t. $(i, j_{\leq, i}) \in \A_{q_*}$ and $j_{\geq, i}$ be the largest index s.t. $(i, j_{\geq, i}) \in \A_{q_*}$.  Then define the interval
\beq
\Ii (\alpha, t) = [ \gamma_{j_{\leq, i}} (\alpha, t) , \gamma_{j_{\geq, i}} (\alpha, t) ].
\eeq
The interval $\Ii ( \alpha, t)$ corresponds to the classical spatial locations of the particles $j$ that are allowed to interact with particle $i$.

Let $\om_A>0$ be as above.  We define the short-range approximation $\hatz_i (t, \alpha )$ as the solution to the following system of SDEs.  For $|i| \leq N^{\om_A}$ let
\begin{align} \label{eqn:zhat1}
\d \hatz_i (t, \alpha ) &= \sqrt{ \frac{2}{ N} } \d B_i+ \frac{1}{N} \sumAqsi_j \frac{1}{ \hatz_i (t, \alpha ) - \hatz_j (t, \alpha ) } \d t
\end{align}
and for $|i| > N^{\om_A}$ let
\begin{align} \label{eqn:zhat2}
\d \hatz_i (t, \alpha ) &= \sqrt{ \frac{2}{N} } \d B_i + \frac{1}{N} \sumAqsi_j \frac{1}{ \hatz_i (t, \alpha ) - \hatz_j (t, \alpha ) } \d t + \frac{1}{N} \sumAqsci_{ j } \frac{1}{ \tilz_i (t, \alpha ) - \tilz_j (t, \alpha ) } \d t \notag\\
&
+ \Re [ m (t, \gamma_0 (\alpha, t ), \alpha ) ] \d t 
 \d t.
\end{align}

The initial condition is $\hatz_i (0, \alpha ) = \tilz_i (0, \alpha)$.   Like the $\tilz_i (t, \alpha)$, the $\hatz_i (t, \alpha )$ retain the ordering $\hatz_i (t, \alpha ) < \hatz_{i+1} (t, \alpha)$ for all positive times.  The above parameters are chosen so that
\beq
0<\om_1 < \om_\ell < \om_A < \om_0/2.
\eeq
The following lemma shows that the $\hatz_i$'s are a good approximation for the $z_i$'s - that is $\hatz_i = \tilz_i + o (1/N)$ with overwhelming probability.
\bel \label{lem:shortrange} Let $\hatz_i (t, \alpha)$ be defined as in \eqref{eqn:zhat1}-\eqref{eqn:zhat2} and $\tilz_i (t, \alpha )$ be defined as in \eqref{eqn:tilz}.  Let $\eps >0$ and $D >0$.  Then we have
\beq
\pp \left[ \sup_{ 0 \leq t \leq t_1 } \sup_{i } \sup_{ 0 \leq \alpha \leq 1 } | \hatz_i (t, \alpha ) - \tilz_i (t, \alpha ) | \geq N^{\eps} t_1  \left( \frac{ N^{\om_A}}{N^{\om_0}} + \frac{1}{  N^{\om_\ell}} + \frac{1}{ \sqrt{ N \G }}    \right) \right] \leq N^{-D},
\eeq
for large enough $N$.

\eel
\proof Define $w_i (t, \alpha ) := \hatz_i (t, \alpha ) - \tilz_i (t, \alpha )$.  The $w_i$ satisfy the equations
\beq
\del_t w_i = \sumAqsi_j \hatB_{ij} (w_j - w_i ) + A_i
\eeq
where
\beq
\hatB_{ij} = \frac{1}{N} \frac{1}{ ( \hatz_i (t, \alpha ) - \hatz_j ( t, \alpha ) ) ( \tilz_i (t, \alpha ) - \tilz_j (t, \alpha ) ) }.
\eeq
The error term $A_i$ satisfies $A_i = 0$ for $|i| > N^{\om_A},$ and for $|i| \leq N^{\om_A}$ it is given by
\begin{align}
A_i &= \frac{1}{N} \sumAqsci_j \frac{1}{ \tilz_i (t, \alpha ) - \tilz_j (t, \alpha ) } + \Re [ m ( \gamma_0 ( t , \alpha ) , t, \alpha )]  \notag\\
&= \left( \frac{1}{N} \sumAqsci_j \frac{1}{ z_i (t, \alpha) - z_j (t, \alpha) }  - \int_{ \I^c_i (t, \alpha ) } \frac{ \rho (x, t, \alpha ) \d x }{ z_i  (t, \alpha) - x }  \right)   \notag\\
&+ \left(\int_{ \I^c_i (t, \alpha ) } \frac{ \rho (x, t, \alpha ) \d x }{ z_i (t, \alpha) - x }  - \int_{ \I^c_i (t, \alpha ) } \frac{ \rho (x, t, \alpha ) \d x }{ \gamma_i (t, \alpha) - x } \right)  \notag \\
&+ \left( \Re[ m ( \gamma_0 (t, \alpha ), t, \alpha ) ] - \Re [ m ( \gamma_i(t, \alpha ), t, \alpha ) ] \right) + \left( \int_{ \I_i } \frac{ \rho (x, t, \alpha ) \d x }{ \gamma_i (t, \alpha) - x } \right) \notag\\
&
=: E_1 + E_2 + E_3 + E_4 
\label{eqn:xi1}
\end{align}
The proof of the lemma is as follows.  Since both $\tilz_i$ and $\hatz_i$ are ordered the kernel $\hatB_{ij}$ are the coefficients of a jump process on $[[- (N-1)/2, (N-1)/2]]$.  Hence the semigroup $\UBhat$ is a contraction on every $\ell^p$ space.  Since at time $t=0$ we have $\hatz_i (0, \alpha ) = \tilz_i (0, \alpha )$, we have $w (t) = \int_0^t \UBhat (s, t) A (s) \d s$ by the Duhamel formula.  Therefore, 
\beq \label{eqn:wxiduh}
||w (t) ||_\infty \leq t \sup_{ 0 \leq s \leq t } || A (s) ||_\infty.
\eeq
The remainder of the proof consists of estimating $A_i$ using the rigidity estimates.

Let $\eps >0$.  For the remainder of the proof we work on the event that the estimates \eqref{eqn:rigref} and \eqref{eqn:lawref} of Section \ref{sec:homogref} \ref{item:estsref} hold with this $\eps$ and a $q$ satisfying $q_* < q < 1$, and a small $\delta >0$ to be determined and large $D>0$.   We take $\delta_1$ to satisfy $\om_0 - \delta_1 > \om_1$.  

 We fix $\eta >N^{2 \delta}/N$ satisfying $\eta \ll \G$.   We write the term $E_1$ as 
\begin{align}
E_1 &= \frac{1}{N} \sumAqsci_j \frac{1}{z_i (t, \alpha) - z_j (t, \alpha)  } - \int_{\Ici (t, \alpha ) } \frac{ \rho (x, t, \alpha ) \d x }{ z_i (t, \alpha) - x } \label{eqn:E1dec} \\
&= \Bigg( \frac{1}{N} \sumAqsci_{ j \in \Chat_q} \frac{1}{ z_i - z_j } - \int_{\Ihat_1 }  \frac{ \rho (x, t, \alpha ) \d x } { z_i -x } \Bigg) + \Bigg( \frac{1}{N} \sumAqsci_{ j \notin \Chat_q} \left(  \frac{1}{z_i - z_j} - \frac{1}{ z_i -z_j - \i \eta } \right)  \Bigg) \notag\\
&+ \left( \int_{ \Ihat_2} \left( \frac{1}{ z_i - x - \i \eta } - \frac{1}{z_i - x  } \right) \rho (x, t, \alpha ) \d x \right) + \left( m_N (z_i + \i \eta ) - m (z_i + \i \eta , t, \alpha ) \right) \notag\\
&+  \Bigg( \int_{ \Ihat_3 } \frac{ \rho (x, t, \alpha ) }{ z_i -x - \i \eta } - \frac{1}{N} \sum_{j \in \Chat_q} \frac{1}{ z_i - z_j + \i \eta } \Bigg) \notag\\
&=: F_1 + F_2 + F_3 + F_4 + F_5.
\end{align}
Above, the intervals $\Ihat_1, \Ihat_2$ and $\Ihat_3$ are defined as follows. They depend on $i, \alpha$ and $t$ but we suppress this for notational simplicity.   Let $a>0$ be the first index not in $\Chat_q$.  Define $\Ihat_1 :=[\gamma_{-a} (t, \alpha ), \gamma_{a} (t, \alpha ) ] \backslash \I_i (t, \alpha )$.  We define $\Ihat_3 := [ \gamma_{-a} (t, \alpha ) , \gamma_a (t, \alpha ) ]$.  Lastly, $\Ihat_2 := \Ihat_3^c$.  Before estimating each of the $F_k$ let us explain the motivation for the above decomposition.  First we remark that since $|i| \leq N^{\om_A}$ the interval $\I_i ( t, \alpha )$  has length $ | \I_i ( t, \alpha ) | \asymp \ell/N$ and is contained in $[\gamma_{-a} (t, \alpha ), \gamma_{a} (t, \alpha )]$.  Moreover, $| \gamma_{\pm a} (t, \alpha ) - \gamma_i (t, \alpha ) |\asymp \G$.   We want to use rigidity to estimate \eqref{eqn:E1dec}.  However, we only know that rigidity holds for the particles in $\Chat_q$.  Hence we break up the terms in \eqref{eqn:E1dec} into two parts.  The first is $F_1$ which is estimated using rigidity.  The remaining particles are distance at least $\G$ from $\hatz_i (t, \alpha )$ and we can use the local law \eqref{eqn:lawref} on a scale $\eta \ll \G$ to estimate this contribution.  The terms $F_2 - F_5$ just correspond to some gymnastics to rewrite these particles in a form that can be estimated by the local law \eqref{eqn:lawref}.

  The term $F_1$ is estimated using rigidity;  using \eqref{eqn:rigref} we easily see that  $|F_1 | \leq C N^{\eps}/N^{\om_\ell}.$  For $F_2$ we use the fact that the restriction $|i| \leq N^{\om_A}$ and $j \notin \Chat_q$ enforces that $|z_i -z_j | \geq c \G$.  Since $\eta \ll \G$ we may bound $F_2$ by
\beq
|F_2| \leq \eta C \sum_{ j} \frac{1}{ (z_i - z_j )^2 + c \G^2 } \leq C \frac{\eta}{ \G} \Im [ m_N ( z_i + \i c \G ) ] \leq C \frac{ \eta }{ \G}.
\eeq
By similar reasoning we get  $| F_3 | \leq C \eta/ \G$.   For $F_4$ we use the local law estimate \eqref{eqn:lawref} and get  $|F_4 | \leq N^{\eps}  / ( N \eta ) $.   Lastly for $F_5$ we use the optimal rigidity estimate \eqref{eqn:rigref} and get
\beq
|F_5 | \leq \frac{N^{\eps}}{N \eta} C \left( \Im [ m (z_i + \i \eta, t, \alpha ) ] + \Im [ m_N (z_i + \i \eta ) ] \right) \leq C\frac{N^{\eps}}{N \eta} .
\eeq
Hence,
\beq
|E_1| \leq C N^{\eps} \left( \frac{1}{N^{\om_\ell}} + \frac{ \eta }{ \G } + \frac{1}{ N \eta } \right) \leq C N^{\eps} \left( \frac{1}{N^{\om_\ell}} + \frac{1}{ \sqrt{ N \G } } \right)
\eeq
where we optimized and chose $\eta = ( \G / N )^{1/2}$ (and chose $\delta >0$ small enough to allow this choice).

We can bound
\beq
|E_2 | = \left|  \int_{ \I^c_i (t, \alpha ) } \frac{ \rho (x, t, \alpha ) \d x }{ z_i - x }  - \int_{ \I^c_i (t, \alpha ) } \frac{ \rho (x, t, \alpha ) \d x }{ \gamma_i - x }  \right| \leq C \frac{N^\eps}{N^{\om_\ell}} \Im [ m (\gamma_i + \i c N^{\om_\ell}/N , t, \alpha ) ] \leq C \frac{N^{\eps}}{N^{\om_\ell}}.
\eeq
For $E_3$ we use \eqref{eqn:mderalph} and get $|E_3| \leq C N^{\om_A} / N^{\om_0}$.   We now have to bound $E_4$.  Note that $\I_i$ is almost symmetric about $\gamma_i (t, \alpha)$.  We have
\beq
|\gamma_{i + k} - \gamma_i |= | \gamma_{i-k} - \gamma_i | \left( 1 + \O \left( \frac{k}{N t_0 } \right) \right),
\eeq
and so
\beq
|E_4| \leq \left| \int_{ \I_i } \frac{ \rho (x, t, \alpha ) - \rho (\gamma_i, t, \alpha ) }{ \gamma_i - x } \d x \right| + C \frac{ N^{\om_\ell}}{ N t_0  } \leq C \frac{ N^{\om_{\ell}}}{N t_0 }.
\eeq
This proves that with overwhelming probability we have 
\beq
|| A ||_\infty \leq C N^{\eps} \left( \frac{ N^{\om_A}}{N^{\om_0}} + \frac{1}{  N^{\om_\ell}}  + \frac{1}{ \sqrt{ N \G }}    \right).
\eeq
This yields the claim via Duhamel's formula \eqref{eqn:wxiduh}. \qed

\subsection{Derivation of parabolic equation} \label{sec:deriveequation}
Define now
\beq
u_i := \del_\alpha \hatz_i (t, \alpha ).
\eeq
The $u_i$ satisfy the equation
\beq \label{eqn:ueqn}
\del_t u_i = \sumAqsi_j B_{ij} (u_j - u_i ) + \xi_i =: - ( \B u )_i + \xi_i
\eeq
where
\beq
B_{ij} = \frac{1}{N} \frac{1}{ (\hatz_i - \hatz_j )^2}
\eeq
and $\xi_i = 0$ for $|i| \leq N^{\om_A}$ and for $|i| > N^{\om_A}$,
\beq
\xi_i = \frac{1}{N} \sumAci_j \frac{\del_\alpha \tilz_i - \del_\alpha \tilz_j}{ (\tilz_i - \tilz_j  )^2} + \del_\alpha \left( \Re [ m (t, \gamma_0 (\alpha, t), \alpha ) ] \right). 
\eeq
Moreover, the initial data is 
\beq
u_i (0, \alpha ) = \alpha z_i (0, 1)  + (1-\alpha ) z_i (0, 0) = \alpha \hatz_i (0, 1) + (1- \alpha ) \hatz_i (0, 0).
\eeq
Note that for $\alpha_1$ and $\alpha_2$, the differences $\tilde{u}_i :=\tilz_i (\alpha_1 ) - \tilz_i ( \alpha_2)$ satisfy
\beq
\del_t \tilde{u}_i = \frac{1}{N} \sum_j \frac{ \tilde{u}_j - \tilde{u}_i }{ ( \tilz_i ( \alpha_1 ) - \tilz_j (\alpha_1 )) ( \tilz_i ( \alpha_2 ) - \tilz_j (\alpha_2 )) } 
\eeq
with $\tilde{u}_i (0) = ( z_i (0, 1) - z_i (0, 0) ) ( \alpha_1 - \alpha_2)$.  Since $| z_i (0, 0) | + |z_i (0, 1)| \leq N^{C}$ with overwhelming probability for some $C>0$ by \eqref{eqn:supzref} we see that
\beq
|| \del_\alpha \tilz (\alpha, t) ||_\infty \leq N^C
\eeq
for $0 \leq t \leq 1$ with overwhelming probability.  It is not hard to see that
\beq \label{eqn:xiestimate}
| \xi_i | \leq \1_{ \{ |i| > N^{\om_A} \} } N^C
\eeq
for some $C>0$ with overwhelming probability.

The parabolic equation \eqref{eqn:ueqn} is the key starting point.  We will treat $\xi_i$ as an error term.  Since it vanishes for indices $|i| \leq N^{\om_A}$ and the operator $\B$ involves jumps only for particles distance $N^{\om_\ell} \ll N^{\om_A}$ apart we expect that $\xi_i$ will have a negligible contribution near $0$.  This is in fact true as we will see below.

\subsection{The kernel $\UB$} \label{sec:ubref}
At this point essentially the entire remainder of Section \ref{sec:homog} and all of Section \ref{sec:finitespeed} are concerned only with properties of the semigroup $\UB$ of the kernel $\B$.  The semigroup $\UB$ depends on the $\zhat_i (t, \alpha)$.  The method that we are going to present for analyzing the semigroup $\UB$ is more general and works for any semigroup whose kernel consists of random coefficients satisfying a system of SDEs with certain properties and certain a-priori bounds.

In this section we will pass to a more general set-up involving the hypotheses \ref{item:rho2}-\ref{item:fiji} below.  The set-up consists of a semigroup $\UB$ and kernel $\B$ with coefficients $B_{ij} = N^{-1} ( \hatz_i (t, \alpha ) - \hatz_i (t, \alpha ) )^{-2}$.  Here we are abusing notation slightly and re-using $\hatz_i$, $\B$, $\UB$, etc.  
  In the next few subsections and in Section \ref{sec:finitespeed} we will then use these hypotheses to derive various facts about $\UB$.  The main result about $\UB$ is Theorem \ref{thm:homogref} which will be proven at the end of Section \ref{sec:actualhomog}.   After proving Theorem \ref{thm:homogref}, we will return to the previous set-up and use Theorem \ref{thm:homogref} to prove Theorem \ref{thm:xyref} in Section \ref{sec:xyrefthmproof}.

We let $\hatz_i (t, \alpha )$ (we will leave in the $\alpha$ notation even though it is unnecessary - for the next few sections $\alpha$ should be regarded as fixed) be the solution to
\beq
\d \hatz_i (t, \alpha ) = \sqrt{ \frac{2}{N} } \d B_i + \frac{1}{N} \sumAqsi_j \frac{1}{ \hatz_i - \hatz_j } \d t + \1_{ \{ |i| \leq N^{\om_A} \} } F_i \d t+ \1_{ \{ |i| > N^{\om_A} \}} J_i \d t
\eeq
where $F_i$ and $J_i$ are adapted bounded processes.    The parameters $\om_A$, $\om_\ell$ and $q_*$ are the same as before, and $\A_{q_*}$ is defined as above.  Previously we also introduced the index $k_0$ in the definition of $\Chat_q := \{ i : |i| \leq q k_0 \}$.  Here, we take $k_0$ as a given parameter in the set-up and assume that $k_0 \asymp N \G$.   Let $\rho( E, t, \alpha ) \d E$ be measures with densities on $|E| \leq q \G$ for any $0 < q < 1$ (here, we just sent $\G \to c \G$ in order to simplify notation).  Suppose that the following hold.
\begin{enumerate}[label=(\Roman*)]
\item \label{item:rho2} We have $\rho (0, 0, \alpha ) = \rhosc( 0)$ and $\gamma_0 (0, \alpha ) = 0$ and 
\beq
c \leq \rho (E, t, \alpha ) \leq C, \quad \left| \del_E \rho (E, t, \alpha ) \right| \leq \frac{ C }{ (t_0)}, \quad \left| \del_t \rho (E, t, \alpha ) \right| \leq \frac{C}{t_0} \quad |E| \leq q\G , \quad 0 \leq t \leq 10 t_1.
\eeq
Moreover, $\G^2 \geq N^\sigma t_0 = N^{\om_0}/N$ for some $\sigma >0$ and $\om_0 >0$.  We assume that the classical particle locations $\gamma_i ( \alpha, t)$ satisfy
\beq
i \in \Chat_q \implies \gamma_i (\alpha, 0) \in [-\G q', \G q' ]
\eeq
for some $0 < q' < 1$ depending on $q$.  We also assume
\beq
| \del_t \gamma_i (t, \alpha ) | \leq C \log (N)
\eeq
for $i \in \Chat_q$.
\item \label{item:zrigref} We have the rigidity estimate
\beq \label{eqn:zrigref}
\pp \left[ \sup_{ i \in \Chat_q } \sup_{ 0 \leq t \leq 10 t_1 } | \zhat_i (t, \alpha ) - \gamma_i (t, \alpha ) | \geq \frac{ N^\eps}{N} \right] \leq N^{-D}
\eeq
for any $\eps, D>0$ and $0 < q < 1$.
\item \label{item:fiji} For the terms $F_i$ and $J_i$ we have for some fixed $C_J >0$ and $\om_F >0$ and every $0 < q < 1$,
\beq \label{eqn:Jiest2}
\pp \left[ \sup_{i \in \Chat_q}  \sup_{ 0 \leq t \leq 10 t_1} |J_i | \geq C_J \log (N)\right] \leq N^{-D}
\eeq
and
\beq \label{eqn:Jiest}
\pp \left[  \sup_i \sup_{ 0 \leq t \leq 10 t_1} |J_i | \geq N^{C_J}\right] \leq N^{-D}
\eeq
and
\beq \label{eqn:Fiest}
\pp \left[ \sup_i \sup_{ 0 \leq t \leq 10 t_1} |F_i | \geq \frac{N^\eps}{ N^{ \om_F}} \right] \leq N^{-D}
\eeq
for any $\eps, D>0$.
\end{enumerate}

\remark In our case $F_i =0$ but we have added it for the following reason.  In our set-up we have $F_i=0$ because we  differentiated the short-range approximation to arrive at the parabolic equation \eqref{eqn:ueqn}.  In other applications it is conceivable that one would like a homogenization result for the full process $\d \tilz_i$.  This is covered by the above set-up by rigidity - in this case $\om_F = \om_\ell$ (i.e., the long-range $z_i - z_j$ terms cancel with $\del_t \gamma = - \Re [ m ( \gamma ) ]$).  Finally, while the $\om_\ell$ appearing in the definition of the $\zhat_i$ and $\B$ are the same, this is not crucial as extra terms can just be absorbed into the $F_i$ term using rigidity and the smoothness of the density $\rho (E, t, \alpha )$.

With $\zhat_i$ satisfying \ref{item:rho2}-\ref{item:fiji} we will consider the operator
\beq
( \B u )_i := \sumAqsi_j B_{ij} (u_i - u_j )
\eeq
with semigroup $\UB$.  Before we write down the main result about the semigroup $\UB$ we record some estimates on it.  First, we have the following finite speed of propogation estimate.
\bel \label{lem:fs1}
Let $0 \leq s \leq t\leq t_1$.  Let $q_*$ and $\om_\ell$ be in the definition of the short-range set $\A_{q_*}$ for $\B$.  Suppose that \ref{item:rho2}-\ref{item:fiji} hold.  Let $0 < q_1 < q_2 < q_*$ be given.  Let $D >0$ and $\eps >0$.   For each $\alpha$ there is an event $\F_\alpha$ with probability $\pp [ \F_\alpha ] \geq 1 - N^{-D}$ on which the following estimates hold. If $i \in \Chat_{q_2}$ and $0 \leq s \leq t \leq 10 t_1$, then
\beq \label{eqn:fscl}
| \UB_{ji} (s, t, \alpha) | \leq \frac{1}{N^D}, \quad |i-j| > N^{\om_{\ell} + \eps}.
\eeq
If $i \notin \Chat_{q_2}$ and $j \in \Chat_{q_1}$ and $0 \leq s \leq t \leq 10 t_1$ then
\beq \label{eqn:fsfar}
| \UB_{ji} (s, t, \alpha) | \leq \frac{1}{N^D}.
\eeq
\eel
Lemma \ref{lem:fs1} is an immediate consequence of Theorem \ref{thm:finitespeed}.  Similar estimates appeared earlier in \cite{que} and our proof follows closely the one appearing there.

 Lemma \ref{lem:fs1} contains two estimates.    The first \eqref{eqn:fscl} is almost-optimal in that the kernel decays quickly when $|i-j| \gtrsim \ell$, where $\ell$ is the range of the jump kernel. Its proof requires the optimal rigidity estimate.  We also need the second estimate \eqref{eqn:fsfar} which is weaker but holds for particles $i$ for which rigidity does not hold.

We also have the following estimate for the kernel $\UB$ which says that for the purposes of upper bounds we can think of $\UB \sim t/ (x^2 + t^2 )$.
\bel \label{lem:prof1} Let $q_*$ be from the definition of $\A_{q_*}$ in the definition of $\B$.  Suppose that \ref{item:rho2}-\ref{item:fiji} hold.  Let $0 < q_1 < q_*$, $D>0$ and $\eps >0$.  For each $\alpha$ there is an event with $\pp [ \F_\alpha ] \geq 1 - N^{-D}$ on which the following estimates hold.  For $i, j \in \Chat_{q_1}$ and $0 \leq s \leq t \leq 10 t_1$ we have
\beq\label{eqn:decay}
\left| \UB_{ij} (s, t) \right| \leq \frac{N^\eps}{N} \frac{|t-s|\vee N^{-1}}{ (( i-j)/N)^2 +( |t-s|\vee N^{-1})^2 }
\eeq
\eel
\remark Note that the above estimate will not hold for $\UB$ if $(s-t) \gg N^{\om_\ell}/N$.

The proof of the above estimate is deferred to the next section and is stated there as Theorem \ref{thm:sec4prof}.  Roughly, the proof consists of the following steps.  First we derive the general estimate 
\beq
|\UB_{ij} (0, t) | \leq \frac{C}{Nt}
\eeq
using the Nash method.  This argument is similar to that in \cite{Gap} - it is slightly different as we only have a short range operator $\B$ living on the scale $N^{\om_\ell}$, but this does not affect things as long as $t \ll N^{\om_\ell}/N$.  Then we decompose
\beq
\B = \S + \R
\eeq
where $\S$ is a short-range operator on the scale $\ell_2 \sim Nt$.  (Although $\B$ is already a short-range operator, we are interested in time scales $Nt \ll \ell$, where $\ell$ is the scale that $\B$ lives on - hence we must make a further long range/short range decomposition of $\B$).  We then prove finite speed estimates for $\S$ and use this together with a Duhamel expansion to derive the estimate.

\subsection{Homogenization of $\UB$} \label{sec:actualhomog}
In this section we will prove that $\UB$ is given by a deterministic quantity, plus random corrections of lower order.  This is the main calculation of Section \ref{sec:homog}.
Fix $\eps_B >0$ s.t.
\beq
\om_A - \eps_B > \om_\ell,
\eeq
 and let 
\beq
|a| \leq N^{\om_A-\eps_B}.
\eeq 
 We consider a solution $w$ of the equation
\beq \label{eqn:wdef}
\del_t w_i = - (\B w )_i, \qquad w_i (0) = N \delta_a (i).
\eeq
Let $\mu$ be the counting measure on $[[- (N-1)/2, (N-1)/2]]$ normalized to have mass $1$.    We introduce the $\ell^p$-norms
\beq
|| u||^p_p := \int |u_i |^p \d \mu (i), \qquad ||u||_\infty = \sup_i |u_i|.
\eeq

The particle density is smooth on the scale $t_0$.  Our operator $\B$ instead lives on the scale $N^{\om_\ell}/N \ll t_0$, and we are working with times $t \ll \ell/N \ll t_0$, and so our solutions $w_i$ will never see the density fluctuations.  Hence it makes sense to compare $w$ with the solution (on $\rr$) of
\beq \label{eqn: Ketadef}
\del_t f (x) = \int_{|x-y| \leq \eta_\ell} \frac{ f(y) - f(x) }{ (x-y)^2 } \rho_{\sc} (0) \d y
\eeq
where
\beq
\eta_\ell := \frac{N^{\om_\ell}}{N \rho_{\sc} (0) }.
\eeq
Let $p_t (x, y)$ be the kernel of the above equation.  We define the ``flat'' classical eigenvalue/particle locations by
\beq
\gf_j := \frac{ j}{ N \rho_{\sc} (0) }.
\eeq
Note that for $|j| \leq N^{\om_0/2}$ we have
\beq
| \gf_j - \tilg_j (t, \alpha ) | \leq \frac{ C}{N}.
\eeq
The main result of this section is the following.
\bet \label{thm:homogref}
Suppose that \ref{item:rho2}-\ref{item:fiji} of Section \ref{sec:ubref} hold.  Fix an index $|a| \leq N^{\om_A - \eps_B }$.  Let $i$ satisfy $|i-a| \leq \ell/10$.  Let $t_1$ be as above and let
\beq
t_2 := N^{-\eps_2} t_1
\eeq
for $\om_1 - \eps_2 >0$.  Let $\eps >0$ and $D>0$.  There is an event $\F_\alpha$ with $\pp [ \F_\alpha] \geq 1 - N^{-D}$ on which the following estimate holds.  For every $u$ with $|u| \leq t_2$ we have
\beq
\left| \UB_{ia} (0, t_1 + u ) - \frac{1}{N} p_{t_1} ( \gf_i, \gf_a ) \right| \leq N^\eps \frac{N^{\eps_2}}{ N t_1 } \left\{ \frac{ ( N t_1)^2}{ \ell^2} + \frac{1}{ ( N t_1 )^{1/10}} + \frac{ 1}{ N^{\om_F/3}} \right\} + N^\eps \frac{ N^{-\eps_2/2}}{ N t_1 }.
\eeq
\eet

In the remainder of Section \ref{sec:actualhomog} we will work under the assumption that \ref{item:rho2}-\ref{item:fiji} hold.   The proof of the following lemma is deferred to Section \ref{sec:hydro}.
\bel \label{lem: Ketalemma} Let $\eps_1 >0$ and $D_1>0$.  
We have for $N^{-D_1} \leq t \leq N^{-\eps_1} \eta_\ell$,
\beq \label{eqn: Keta-shortrange}
p_t (x, y) \leq C \frac{t}{ (x-y)^2 + t^2 }.
\eeq
For any $\eps_2 >0$ if $|x-y| > N^{\eps_2} \eta_\ell$ and $N^{-D_1} \leq t \leq N^{-\eps_1} \eta_\ell$,
\beq \label{eqn: Keta-expdecay}
p_t (x, y) \leq \frac{1}{N^{D_2}}
\eeq
for any $D_2>0$.  

For spatial derivatives we have, for $ N^{-D_1} \leq t \leq N^{-\eps_1} \eta_\ell$,
\beq
p^{(k)}_t (x, y) \leq \frac{C}{ t^k} p_t (x, y) + \frac{1}{N^{D_2}}, \label{eqn: Kderivbd}
\eeq
and
\beq
p^{(k)}_t (x, y) \leq \frac{1}{N^D}
\eeq
for any $D_2$ if $|x-y| > N^{\eps_2} \eta_\ell$.

For the time derivative we have for $ N^{-D_1} \leq t \leq N^{-\eps_1} \eta_\ell$,
\beq
| \del_t p_t (x, y) | \leq \frac{C}{ x^2 + y^2} + N^{-D_2}.
\eeq
\eel
\remark The short time cut-off $t \geq N^{-D_1}$ is technical.  In our application we will only take $p_t$ with $t \geq N^{-1}$.

We need to introduce two auxilliary scales $s_0$ and $s_1$.  They will satisfy
\beq
N^{-1} \ll s_0 \ll s_1 \ll t_1 \ll t_0.
\eeq

Define now
\beq \label{eqn:fdef}
f (x, t) = \sum_j \frac{1}{N} p_{s_0+t-s_1} (x, \gf_j )  w_j (s_1 ),
\eeq
and
\beq
f_i (t) := f ( \hatz_i (t, \alpha ), t).
\eeq
We are going to compare $w_i (t)$ to $f_i (t)$.  A more natural choice would perhaps be $f_i (t) = p_t ( \hatz_i, \gf_a)$.  We explain here the motivation for the above choice of $f_i$ and the introduction of the scales $s_1$ and $s_0$.   Our method relies on differentiating the $\ell^2$ norm of the difference $w_i - f_i$, and then integrating it back.  We therefore require an estimate on the  $\ell^2$ norm  of the difference at the beginning endpoint of the time interval over which the integration occurs.  One choice could be $t=0$ for the start point of this interval.  However, $w_i$ is quite singular at this point so we allow it to evolve for a short time $s_1$ before comparing $f_i$ to $w_i$.  At this point one might want to take $f (t) = p_{t-s_1} \star w (s_1)$.  However at $t=s_1$, the kernel $p_{t-s_1}$ is a $\delta$-function and so this convolution operation does not make sense.  We therefore introduce the regularizing scale $s_0$.  By the standard energy estimate $w (s_1)$ has some smoothness on the scale $s_1$.  This allows us to control the $\ell^2$ distance between $w(s_1)$ and its convolution with the approximate $\delta$-function $p_{s_0}$.

An additional technical complication is that the standard energy estimate involves a time average and so we will have to average the startpoint $s_1$ over the interval $[s_1, 2 s_1 ]$.

We have the normalization condition 
\beq \label{eqn:ptnormfull}
\sum_j \frac{1}{N} p_t (\hatz_i, \gf_j) = 1 + \O ( (Nt)^{-1} )
\eeq
and also for $\ell_1 \gg Nt$,
\beq \label{eqn:ptnorm}
\sum_{|i-j| \leq \ell_1} \frac{1}{N} p_t (\hatz_i, \gf_j ) = 1 + \O  ( (Nt)^{-1} ) + \O ( (Nt)/\ell_1 )
\eeq
and
\beq \label{eqn:ptl1dec}
\sum_{ |i-j| > \ell_1 } \frac{1}{N} p_t (\hatz_i, \gf_j) \leq C \frac{Nt}{\ell_1}
\eeq
The following lemma provides an estimate on the $\ell^2$ norm of the difference $w-f$.  The error is in terms of the scales $s_0, \ell$ and $s_1$ as well as a quantity which can be controlled via the standard energy estimate for $w$.
\bel \label{lem:initell2} Let $w$ be as in \eqref{eqn:wdef} and $f$ as in \eqref{eqn:fdef}.  For any $\eps_1 >0$ and $\eps_2 >0$ and $D>0$ there is for each $\alpha$ an event $\F_\alpha$ with $\pp [ \F_\alpha ] \geq 1 - N^{-D}$ on which the following holds,
\begin{align}
& ||w (s_1) - f (s_1) ||_2^2 \notag\\
\leq &s_0C  \sum_{|i|\leq N^{\om_A-\eps_B}+N^{\om_\ell+\eps_2}} \sum_{|i-j| \leq \ell} \frac{ (w_i (s_1) - w_j (s_1 ))^2}{ (i-j)^2} + N^{\eps_1} \left( \frac{1}{ (N s_0)^2} + \frac{ (N s_0)^2}{\ell^2} \right) \frac{1}{s_1}.
\end{align}
\eel
\proof For notational simplicity let
\beq
\sumit :=  \sum_{|i|\leq N^{\om_A-\eps_B}+N^{\om_\ell+\eps_1}}.
\eeq
We also drop the argument $s_1$ and write $w_i = w_i (s_1)$, $f_i = f_i (s_1)$.  With overwhelming probability we have,
\begin{align}
\frac{1}{N} \sum_i (w_i -f_i )^2 &= \frac{1}{N}  \sumit \left( w_i - \sum_j \frac{p_{s_0} (\hatz_i, \gf_j) }{N}w_j \right)^2 + \O ( N^{-D} ) \notag \\
&\leq \frac{C}{N} \sumit \left( \sum_{|j-i| \leq \ell }\frac{ p_{s_0} (\hatz_i, \gf_j ) }{N}(w_i - w_j ) \right)^2 + C \left( \frac{1}{ (N s_0)^2 } + \frac{(N s_0)^2}{\ell^2} \right) ||w||_2^2  \label{eqn:ffref1}
\end{align}
In the first line we used the decay estimates from Lemmas \ref{lem:fs1} and  \ref{lem: Ketalemma} to change the sum from $\sum$ to $\sum'$.   In the  inequality we used the normalization condition \eqref{eqn:ptnorm}, as well as \eqref{eqn:ptl1dec} which together with Young's inequality shows that
\beq
\frac{1}{N} \sum_i \left( \sum_{|j-i|>\ell} \frac{1}{N} p_t (\hatz_i, \gf_j) w_j \right)^2 \leq C \frac{ (Nt)^2}{\ell^2} ||w||_2^2.
\eeq
We then bound the first term in \eqref{eqn:ffref1} by
\beq
\frac{1}{N}  \sumit \left( \sum_{|j-i| \leq \ell }\frac{ p_{s_0} (\hatz_i, \gf_j ) }{N}(w_i - w_j ) \right)^2 \leq \frac{1}{N} \sumit\left( \sum_{|j-i| \leq \ell }\frac{ p^2_{s_0} (\hatz_i, \gf_j )|i-j|^2}{N^2}\right) \left( \sum_{|j-i|\leq \ell} \frac{ (w_i - w_j )^2}{|i-j|^2} \right).
\eeq
We have
\begin{align}
 \sum_{j : |j-i| \leq \ell }\frac{ p^2_{s_0} (x_i, \gf_j )|i-j|^2}{N^2} \leq C N \int \frac{ (s_0)^2 x^2}{ (x^2 + (s_0)^2 )^2} \d x \leq C N s_0.
\end{align}
Lastly we can estimate the $\ell^2$ norm of $w$ using Lemma \ref{lem:prof1} by
\beq
||w||_2^2 \leq \frac{N^{\eps_1}}{N} \sum_i \frac{(s_1)^2}{ ((i-a)/N)^2+(s_1)^2)^2} \leq \frac{N^{\eps_1}}{s_1}.
\eeq
These inequalities yield the claim. 
\qed

The main calculation of the homogenization theorem is the following.  We use the Ito lemma to differentiate $||w-f||_2^2$.  Roughly what we find is that
\beq
\d ||w-f||_2^2= - ||w-f||^2_{\dot{H}^{1/2}}  \d t+ \mbox{lower order}
\eeq
where the lower order terms contains a martingale term as well as other errors$^1$\footnote{$^1$  A similar idea was independently discovered in a forthcoming work by Jun Yin and Antti Knowles \cite{kyprep}.}. 
Integrating this back gives us control over the homogeneous $\dot{H}^{1/2}$ norm of $w-f$.

\bel  \label{lem:maincalc} 
Let $w$ be as in \eqref{eqn:wdef} and $f$ as in \eqref{eqn:fdef} with parameters $s_1$ and $s_0$.  For $t \geq s_1$ we can write the Ito differential of $||w (t) - f(t) ||_2^2$ in the form
\beq \label{eqn:dl2}
\d \frac{1}{N} \sum_i (w_i - f_i)^2 = -  \langle w (t)-f (t), \B (w (t)-f (t)) \rangle \d t +  X_t \d t + \d M_t
\eeq
where $M_t$ is a martingale and $X_t$ is a process implicitly defined by the above equality.  We have the following estimates for $X_t$ and $M_t$.
Let $\eps >0$ and $ D>0$ be given.   For each $\alpha$ there is an event $\F_\alpha$ with $\pp [ \F_\alpha ] \geq 1 - N^{-D}$ on which the following estimates hold. For $s_1 \leq t \leq 9 t_1$ we have for $X_t$
\begin{align} \label{eqn:xtest}
|X_t|& \leq \frac{1}{5} \langle w-f, \B (w-f) \rangle \notag\\
&+  \frac{C}{ t+s_1} \frac{N^\eps}{ (t-s_1 +s_0 ) } \bigg\{  \frac{1}{ \sqrt{ N (t-s_1+s_0 ) } } + \frac{1}{ N^{\om_F}} \bigg\}.
\end{align}
 For any $u_1$ and $u_2$ with  $ 9 t_1 > u_2 > u_1 \geq s_1$ we have
\beq \label{eqn:martest}
\left| \int_{u_1}^{u_2} \d M_t \right| \leq  \frac{N^\eps}{N} \frac{1}{ (u_1+s_1)^{3/2} } \frac{1}{ (u_1-s_1+s_0)^{1/2} }.
\eeq
\eel
\proof The estimates
\beq \label{eqn:fders}
f'(t, \hatz_i ) \leq \frac{C}{ (t-s_1+s_0 ) }f(t, \hatz_i) + N^{-D}\qquad f''(t, \hatz_i ) \leq \frac{C}{(t-s_1+s_0)^2} f(t, \hatz_i) + N^{-D}
\eeq
and
\beq
f(t, \hatz_i) \leq \frac{C}{ (t-s_1+s_0)}
\eeq
are immediate corollaries of Lemma \ref{lem: Ketalemma}.  Since $||w (s_1) ||_\infty \leq C (s_1)^{-1}$ we   obtain
\beq \label{eqn:finf}
f(t, \hatz_i ) \leq \frac{C}{ t + s_1}
\eeq
In the calculations below we implicitly use that for any $\eps_1 >0$,
\beq \label{eqn:fidecay}
|f_i| \leq \frac{1}{N^D}  , \qquad w_i \leq \frac{1}{N^D}, \qquad \mbox{if } |i| \geq |a| + N^{\om_{\ell} + \eps_1}
\eeq
with overwhelming probability by the finite speed estimates of Lemma \ref{lem:fs1} above.  For example by our assumption on $a$ this holds for 
\beq
|i| >  N^{\om_A-\eps_B} + N^{\om_\ell+\eps_1}, \quad \mbox{or} \quad |i| > N^{\om_A}.
\eeq
  It will be convenient to use the notation  $f_i^{(k)} := f^{(k)} (\hatz_i ).$  It is clear that similar estimates to \eqref{eqn:fidecay} hold for the derivatives $f'_i$ and $f''_i$ and $(\del_t f)_i$.

We calculate by the Ito formula,
\begin{align}
&d \frac{1}{N} \sum_i (w_i - f_i )^2\notag\\
 = &\frac{2}{N} \sum_i (w_i - f_i) \left[ \del_t w_i \d t - (\del_t f) (t, \hatz_i ) \d t - f'(t, \hatz_i ) \d \hatz_i -f''(t, \hatz_i ) \frac{\d t}{2 N} \right] + (f' (t, z_i ) )^2 \frac{\d t}{ 2 N}.
\end{align}

Let us start with the Ito terms.  Using \eqref{eqn:fders} and \eqref{eqn:finf} we obtain
\beq \label{eqn:ito1}
\left| \frac{1}{N^2} \sum_i (f'_i)^2 \right| \leq \frac{C}{N (t+s_0-s_1)^2} \frac{1}{t+s_1} \frac{1}{N} \sum_i f_i \leq \frac{C}{N (t+s_0-s_1)^2} \frac{1}{t+s_1}
\eeq
and similarly,
\begin{align} \label{eqn:ito2}
\Big | \frac{1}{N^2} \sum_i (w_i - f_i  ) f''_i \Big | 
\leq \frac{C}{N (t+s_0-s_1)^2} \frac{1}{t+s_1} \frac{1}{N}\sum_i w_i + f_i  
\leq \frac{C}{N (t+s_0-s_1)^2} \frac{1}{t+s_1} .
\end{align}

We write
\begin{align}
& \frac{1}{N} \sum_i (w_i - f_i ) (\del_t w_i - (\del_t f)_i )= \frac{1}{N} \sum_i (w_i - f_i ) \Big ( \frac{1}{N}\sumAqsi_j \frac{w_j-w_i}{(\hatz_i - \hatz_j )^2} -\frac{1}{N} \sumAqsi_j \frac{ f_j - f_i}{(\hatz_i -\hatz_j)^2} \Big ) \label{eqn:timedermain}\\
&+\frac{1}{N} \sum_i (w_i - f_i ) \Bigg( \frac{1}{N} \sumAqsi_j \frac{f_j -f_i}{ (\hatz_i - \hatz_j )^2} - \int_{|\hatz_i - y | \leq \eta_\ell } \frac{ f (y) - f (\hatz_i )}{ (\hatz_i -y)^2} \rho_{\sc} (0) \d y \Bigg) \label{eqn:timederer}
\end{align}
The term \eqref{eqn:timedermain} equals
\beq
 \frac{1}{N} \sum_i (w_i - f_i ) \Bigg( \frac{1}{N}\sumAqsi_j \frac{w_j-w_i}{(\hatz_i - \hatz_j )^2} -\frac{1}{N} \sumAqsi_j \frac{ f_j - f_i}{(\hatz_i -\hatz_j)^2} \Bigg) = - \frac{1}{2} \langle w-f, \B (w-f) \rangle. \label{eqn:timederminequals}
\eeq
Note that this term is negative and is the first term appearing on the RHS of \eqref{eqn:dl2}.  It will be used to account for terms on which we cannot use rigidity.   
For $0< \om_{\ell, 2} < \om_\ell$ define
\begin{align}
\A_2 := \{ (i, j) : |i - j | \leq N^{\om_{\ell, 2}} \} \cup \{ (i, j) : ij >0, i, j \notin \Chat_{q_*} \}.
\end{align}
We write the term \eqref{eqn:timederer} as
\begin{align}
\frac{1}{N} &\sum_i (w_i - f_i ) \left( \frac{1}{N} \sumAqsi_j \frac{f_j -f_i}{ (\hatz_i - \hatz_j )^2} - \int_{| \hatz_i - y | \leq \eta_\ell} \frac{ f (y) - f (\hatz_i )}{ (\hatz_i -y)^2} \rhosc (0) \d y \right) \notag\\
&= \frac{1}{N} \sum_i (w_i - f_i ) \left(\frac{1}{N} \sumAti_j \frac{f_j -f_i}{ (\hatz_i - \hatz_j )^2} \right) \label{eqn:timederer1} \\
&+\frac{1}{N} \sum_i (w_i - f_i ) \left( \frac{1}{N} \sumsthi_j \frac{f_j -f_i}{ (\hatz_i - \hatz_j )^2} - \int_{ |\hatz_i - y | \leq \eta_\ell} \frac{ f (y) - f (\hatz_i )}{ (\hatz_i -y)^2} \rhosc(0) \d y \right) \label{eqn:timederer4}
\end{align}
We first deal with \eqref{eqn:timederer1}.  We will later use rigidity to deal with \eqref{eqn:timederer4}. 
Write
\beq
v_i := w_i - f_i.
\eeq
Using a second order Taylor expansion for $f_i$ we have for $|i| \leq N^{\om_A}$,
\beq \label{eqn:aa1}
\frac{1}{N} \sumAti_j \frac{f_j - f_i }{ (\hatz_j - \hatz_i )^2} = \frac{1}{N} \sumAti_j \frac{ f_i'}{ \hatz_j - \hatz_i } + \O \left( \frac{ N^{\om_{\ell,2}}}{ N (t-s_1+s_0)^2 } \frac{1}{ (t+s_1)} \right) .
\eeq
To estimate the remainder term we used the fact that $||f''||_\infty \leq C (t-s_1+s_0)^{-2} (t+s_1)^{-1}$ as well as the fact that since $ |i| \leq N^{\om_A}$, the cardinality of $\{j : (j, i )\in\A_2 \}$ is less than $C N^{\om_{\ell,2}}$.  For $|i| > N^{\om_A}$ we just use
\beq \label{eqn:aa2}
\frac{1}{N} \sumAti_j \frac{f_j - f_i }{ (\hatz_j - \hatz_i )^2} = \frac{1}{N} \sumAti_j \frac{ f_i'}{ \hatz_j - \hatz_i } + \O ( N^8 ).
\eeq
Using \eqref{eqn:aa1} and \eqref{eqn:aa2} and the estimate \eqref{eqn:fidecay} we can write the term \eqref{eqn:timederer1} as
%
\begin{align}
\frac{1}{N} \sum_i (w_i - f_i ) \left(\frac{1}{N} \sumAti_j \frac{f_j -f_i}{ (\hatz_i - \hatz_j )^2} \right) &=\frac{1}{N^2} \sum_{i} v_i \sumAti_j \frac{f'_i}{ \hatz_i -\hatz_j } \label{eqn:timederer2} + \O \left( \frac{ N^{\om_{\ell,2}}}{ N (t-s_1+s_0)^2 } \frac{1}{ t+s_1} \right) .
\end{align}
We then write the first term on the RHS of \eqref{eqn:timederer2} as
\beq \label{eqn:timederer3}
\frac{1}{N^2} \sum_{i} v_i \sumAti_j \frac{f'_i}{ \hatz_i -\hatz_j } = \frac{1}{2} \frac{1}{N^2} \sum_{(i, j) \in \A_2} \frac{ (v_i -v_j ) f'_i + v_j (f'_i - f'_j ) }{ \hatz_j - \hatz_i }.
\eeq
Using again the estimates \eqref{eqn:fidecay} we bound the second term by
\begin{align}
\left| \frac{1}{2} \frac{1}{N^2} \sum_{(i, j) \in \A_2} \frac{  v_j (f'_i - f'_j ) }{ \hatz_j - \hatz_i } \right| &\leq  \left| \frac{1}{2} \frac{1}{N^2} \sum_{|i-j| \leq N^{\om_{\ell,2} }, |j| \leq 2 N^{\om_A}} \frac{  v_j (f'_i - f'_j ) }{ \hatz_j - \hatz_i } \right| + N^{-D} \notag \\
&\leq C \frac{ N^{\om_{\ell,2}}}{ N (t-s_1+s_0)^2 } \frac{1}{ t+s_1} .
\end{align}
The second inequality used \eqref{eqn:fders}.  We bound the first term on the RHS of \eqref{eqn:timederer3} using Schwarz by
\begin{align}
\left| \frac{1}{2} \frac{1}{N^2} \sum_{(i, j) \in \A_2} \frac{ (v_i -v_j ) f'_i  }{ \hatz_j - \hatz_i } \right| &\leq \frac{1}{10} \frac{1}{N^2} \sum_{(i, j) \in \A_2 } \frac{ (v_i -v_j )^2}{ (\hatz_i - \hatz_j )^2} + \frac{C}{N^2} \sum_{i} (f_i')^2 \sumAti_j 1 \notag\\
&\leq \frac{1}{10} \frac{1}{N^2} \sum_{(i, j) \in \A_2 } \frac{ (v_i -v_j )^2}{ (\hatz_i - \hatz_j )^2}  +  C \frac{ N^{\om_{\ell,2}}}{ N (t-s_1+s_0)^2 } \frac{1}{ t+s_1 } \label{eqn:timederer10}
\end{align}
where we used again the decay estimate \eqref{eqn:fidecay}, the estimate for $f'_i$ and the fact that the cardinality of the set $\{j : (i, j ) \in \A_2 \}$ is bounded by $C N^{\om_{\ell,2}}$ for $|i| \leq N^{\om_A}$.  The first term will be absorbed into the $\langle (w-f),  \B (w-f) \rangle $ term.  

In summary, the estimates \eqref{eqn:timederer2}-\eqref{eqn:timederer10}  prove that for the term \eqref{eqn:timederer1} we have
\begin{align} \label{eqn:timederer12}
 \left| \frac{1}{N} \sum_i (w_i - f_i ) \left(\frac{1}{N} \sumAti_j \frac{f_j -f_i}{ (\hatz_i - \hatz_j )^2} \right) \right| \leq\frac{1}{10} \frac{1}{N^2} \sum_{(i, j) \in \A_2 } \frac{ (v_i -v_j )^2}{ (\hatz_i - \hatz_j )^2}  +  C \frac{ N^{\om_{\ell,2}}}{ N (t-s_1+s_0)^2 } \frac{1}{ t+s_1}.
\end{align}

In order to complete the bound of \eqref{eqn:timederer} we need to estimate \eqref{eqn:timederer4}.  This term will be estimated by rigidity. Due to the decay estimates \eqref{eqn:fidecay} we can safely ignore the terms with $|i| > N^{\om_A}$; i.e., for $|i| > N^{\om_A}$ the term inside the brackets is estimated by
\beq \label{eqn:aa5}
 \left| \frac{1}{N} \sumsthi_j \frac{f_j -f_i}{ (\hatz_i - \hatz_j )^2} - \int_{ |\hatz_i - y | \leq \eta_\ell } \frac{ f (y) - f (\hatz_i )}{ (\hatz_i -y)^2} \rhosc(0) \d y \right| \leq N^{10}.
\eeq
For the terms with $|i| \leq N^{\om_A}$ we use the rigidity estimates \eqref{eqn:zrigref} of Section \ref{sec:ubref} \ref{item:zrigref}.  We write
\begin{align}
 &\left( \frac{1}{N} \sumsthi_j \frac{f_j -f_i}{ (\hatz_i - \hatz_j )^2} - \int_{ |y -\hatz_i | \leq \eta_\ell }  \frac{ f (y) - f (\hatz_i )}{ (\hatz_i -y)^2} \rhosc(0)  \d y \right)  \notag\\
 =  &\left( \frac{1}{N} \sumsthi_j \frac{f_j -f_i}{ (\hatz_i - \hatz_j )^2} - \int_{ \eta_{\ell,2} \leq |y -\hatz_i | \leq \eta_\ell }  \frac{ f (y) - f (\hatz_i )}{ (\hatz_i -y)^2} \rhosc(0)  \d y \right) \label{eqn:aa3} \\
 -&\left(  \int_{|y -\hatz_i | \leq \eta_{\ell,2} }  \frac{ f (y) - f (\hatz_i )}{ (\hatz_i -y)^2} \rhosc(0)  \d y \right),\label{eqn:aa4}
\end{align}
where $\eta_{\ell,2} = N^{\om_{\ell,2}} / ( N \rhosc(0) )$.
The term \eqref{eqn:aa3} is estimated using rigidity by
\beq \label{eqn:aa6}
\left| \frac{1}{N} \sumsthi_j \frac{f_j -f_i}{ (\hatz_i - \hatz_j )^2} - \int_{ \eta_{\ell,2} \leq |y -\hatz_i | \leq \eta_\ell }  \frac{ f (y) - f (\hatz_i )}{ (\hatz_i -y)^2} \rhosc(0)  \d y \right|\leq  \frac{CN^\eps}{ N^{\om_{\ell,2}}} \frac{1}{ t-s_1+s_0} \frac{1}{ t + s_1} .
\eeq
The term \eqref{eqn:aa4} is estimated using a second order Taylor expansion. We write it as
\beq \label{eqn:aa7}
 \int_{|y -\hatz_i | \leq \eta_{\ell,2} }  \frac{ f (y) - f (\hatz_i )}{ (\hatz_i -y)^2} \rhosc(0)  \d y = \int_{|y -\hatz_i | \leq \eta_{\ell,2} }  \frac{ f' (\hatz_i )}{ (\hatz_i -y)} \rhosc(0)  \d y + \O \left( \frac{ N^{\om_{\ell,2}} }{ N ( t-s_1+s_0)^2}  \frac{1}{ t+s_1}  \right).
\eeq

We then have that
\beq
\int_{|y -\hatz_i | \leq \eta_{\ell,2} }  \frac{ f' (\hatz_i )}{ (\hatz_i -y)} \rhosc(0)  \d y = 0.
\eeq
Combining \eqref{eqn:fidecay} with \eqref{eqn:aa5} for the terms with $|i| >N^{\om_A}$ and then \eqref{eqn:aa6} and \eqref{eqn:aa7} for the remaining terms yields the following estimate for \eqref{eqn:timederer4}.
\begin{align} \label{eqn:timederer11}
&\left| \frac{1}{N} \sum_{ i }  (w_i - f_i ) \left( \frac{1}{N} \sumsthi_j \frac{f_j -f_i}{ (\hatz_i - \hatz_j )^2} - \int_{ |y - \hatz_i | \leq \eta_\ell } \frac{ f (y) - f (\hatz_i )}{ (\hatz_i -y)^2} \rhosc(0) \d y \right) \right| \notag\\
\leq & \frac{C N^\eps}{ N^{\om_{\ell,2}}} \frac{1}{ t-s_1+s_0} \frac{1}{ t + s_1} +C \frac{ N^{\om_{\ell,2}} }{ N ( t-s_1+s_0)^2}  \frac{1}{ t+s_1 } .
\end{align}

In summary, we see that \eqref{eqn:timederminequals} and the estimates \eqref{eqn:timederer12} and \eqref{eqn:timederer11} imply
\begin{align} \label{eqn:timedereqy}
\frac{1}{N} \sum_i (w_i - f_i ) (\del_t w_i - (\del_t f)_i )  = - \frac{1}{2} \langle w-f, \B (w-f) \rangle + Y_t
\end{align}
where 
\beq \label{eqn:ytest} 
|Y_t| \leq \frac{1}{10} \langle w-f, \B (w-f ) \rangle + \frac{C N^\eps}{ N^{\om_{\ell,2}}} \frac{1}{ t-s_1+s_0} \frac{1}{ t + s_1} +C \frac{ N^{\om_{\ell,2}} }{ N ( t-s_1+s_0)^2}  \frac{1}{ t+s_1 } .
\eeq

The remaining term to deal with is
\begin{align}
\frac{1}{N} \sum_i (w_i -f_i )  f'_i \d \hatz_i &= \d M_t + \frac{1}{N} \sum_i (w_i -f_i )f'_i \bigg\{  \frac{1}{N} \sumAti_j \frac{1}{ \hatz_i - \hatz_j } + \frac{1}{N} \sumsthi_j \frac{1}{\hatz_i -\hatz_j} \notag\\
&+ \1_{ \{ |i| \leq N^{\om_A}\}} F_i + \1_{ \{ |i | > N^{\om_A} \} } J_i\bigg\} \d t. \label{eqn:dz1}
\end{align}
The martingale term is
\beq
d M_t= \frac{1}{N}\sum_i (w_i - f_i ) f'_i \sqrt{ \frac{2}{N} } \d B_i
\eeq
which we estimate later.
The first non-martingale term appearing on the RHS of \eqref{eqn:dz1} is identical to \eqref{eqn:timederer3} (we comment here that they actually appear with the same sign and so do not cancel as one might hope) and so we have, proceeding as above,
\begin{align}
\left| \frac{1}{N} \sum_i (w_i -f_i )f'_i \frac{1}{N} \sumAti_j \frac{1}{ \hatz_i - \hatz_j } \right| \leq \frac{1}{10} \langle (w-f ), \B (w-f) \rangle  +  C \frac{ N^{\om_{\ell,2}}}{ N (t-s_1+s_0)^2 } \frac{1}{ t+s_1 }.
\end{align}
Using \eqref{eqn:fidecay} and \eqref{eqn:Jiest} we can drop all terms in \eqref{eqn:dz1} with $|i| > N^{\om_A}$; i.e., we have with overwhelming probability
\begin{align}
&\bigg| \sum_{ |i| >N^{\om_A} } (w_i - f_i ) f_i' \bigg\{  \frac{1}{N} \sumsthi_j \frac{1}{\hatz_i -\hatz_j} + J_i   \bigg\} \bigg| \leq \frac{1}{N^D}.
\end{align}
For the $F_i$ terms we have by \eqref{eqn:Fiest} with overwhelming probability,
\beq
\left| \sum_{ |i| \leq N^{\om_A}} (w_i - f_i ) f'_i F_i \right| \leq  C \frac{N^\eps}{ N^{\om_F} ( t- s_1 + s_0 )( t +  s_1 ) }
\eeq
For $|i| \leq N^{\om_A}$, since
\beq
0 = \int_{ \eta_{\ell,2} \leq |y- \hatz_i | \leq \eta_\ell } \frac{1}{ \hatz_i - y } \rhosc (0) \d y,
\eeq
we have by the rigidity estimate \eqref{eqn:zrigref}
\begin{align}
\left| \frac{1}{N} \sumsthi_j \frac{1}{\hatz_i -\hatz_j}  \right| &= \left| \frac{1}{N} \sumsthi_j \frac{1}{\hatz_i -\hatz_j}  -  \int_{ \eta_{\ell,2} \leq |y- \hatz_i | \leq \eta_\ell } \frac{1}{ \hatz_i - y } \rhosc (0) \d y \right| \leq C \frac{ N^\eps}{ N^{\om_{\ell,2}}}.
\end{align}
Hence, ignoring the martingale term (and slightly abusing notation) we have obtained the following bound for \eqref{eqn:dz1}.
\begin{align}
\left| \frac{1}{N} \sum_i (w_i-f_i ) f_i' \d \hatz_i \right| &\leq \frac{ \langle (w-f), \B (w-f) \rangle}{10} + \frac{C}{ t+s_1 } \frac{N^\eps}{ (t-s_1+s_0 ) } \bigg( \frac{N^{\om_{\ell,2}}}{N (t-s_1+s_0)} + \frac{1}{ N^{\om_{\ell,2}} } + \frac{1}{ N^{\om_F}} \bigg)\label{eqn:dzrig}
\end{align} 
for any $\eps >0$.  The equality \eqref{eqn:dl2} and estimate \eqref{eqn:xtest}  follow from \eqref{eqn:ito1}, \eqref{eqn:ito2}, \eqref{eqn:timedereqy}, \eqref{eqn:ytest} and \eqref{eqn:dzrig}, after optimizing and choosing $N^{\om_{\ell,2} } \asymp \sqrt{ N (t-s_1  + s_0 ) }$.

The quadratic variation of the martingale term satisfies
\beq
\d \langle M \rangle_t = \frac{1}{N^3} \sum_i (w_i -f_i)^2 (f_i')^2 \d t \leq \frac{C}{N^2} \frac{1}{ (t+s_1)^3} \frac{1}{ (t-s_1-s_0)^2} \d t
\eeq
with overwhelming probability.  Hence by the BDG inequality,
\begin{align}
\ee \left[\sup_{u_2 : 9 t_1 \geq u_2 \geq u_1 } \left| \int_{u_1}^{u_2} \d M_t \right|^p \right] \leq C_p \frac{1}{N^p} \frac{1}{ (u_1+s_1)^{3p/2}} \frac{1}{ (u_1-s_1+s_0)^{p/2} } 
\end{align}
and so
\beq
 \sup_{u_2 : 9 t_1 \geq u_2 \geq u_1 } \left| \int_{u_1}^{u_2} \d M_t \right| \leq \frac{N^\eps}{N} \frac{1}{ (u_1+s_1)^{3/2} } \frac{1}{ (u_1-s_1+s_0)^{1/2} }
\eeq
with overwhelming probability.  A simple argument using a union bound over $u_1$ in a set of cardinality at most $N^2$ extends this estimate to all $s_1 < u_1 < u_2 < 9 t_1$.  This yields \eqref{eqn:martest}. \qed

Lemma \ref{lem:maincalc} yields the following corollary, after integration in $t$.  Note that the boundary term at $2 t_1$ appears with a negative sign so we can drop it from the RHS of \eqref{eqn:rev1} below. 
\bec \label{cor:maincalc}  Let $w$ be as in \eqref{eqn:wdef} and $f$ as in \eqref{eqn:fdef} with parameters $s_1$ and $s_0$.
Let $\eps >0$ and $D>0$.  For each $\alpha$ there is an event $\F_\alpha$ with $\pp [ \F_\alpha ] \geq 1 - N^{-D}$  on which the following holds.
\begin{align} \label{eqn:rev1}
 \int_{s_1}^{2 t_1} \langle (w-f), \B (w-f ) \d s \rangle \leq ||(w-f) (s_1  ) ||_2^2 + \frac{N^\eps}{s_1} \left\{ \frac{ 1}{ (N s_0)^{1/2} }  + \frac{1}{N^{\om_F}}\right\}  
\end{align}
\eec

Putting together the last two lemmas yields the following homogenization theorem.  It is essentially Theorem \ref{thm:homogref} but with a time average.  In the next subsection we will remove the time average.
\bet \label{thm:techhomog}
Let $a$ and $i$ satisfy
\beq
|a| \leq N^{\om_A-\eps_B} , \qquad
|i - a| \leq \ell/10.
\eeq
For any $\eps >0$ and $D>0$ there is an event  $\F_{\alpha}$ with $\pp [ \F_{\alpha} ] \geq 1 - N^{-D}$ on which
\begin{align}
&\frac{1}{t_1} \int_0^{t_1} \left( \UB_{ia}(0, t_1+u) - \frac{1}{N} p_{t_1+u} (\gf_i, \gf_a ) \right)^2 \d u \leq   \frac{  N^{ \eps}}{ ( N t_1 )^2 } \bigg\{  \frac{ ( N t_1 )^4}{ \ell^4} + \frac{ s_1^2}{t_1^2}  + \frac{t_1}{s_1} \left( \frac{1}{ ( N s_0 )^{1/2} } + \frac{1}{ N^{\om_F}}+  \frac{s_0}{s_1}    \right) \bigg\} 
\end{align}
\eet
\proof   Define $w$ and $f$ as in \eqref{eqn:wdef} and \eqref{eqn:fdef}, except replace $s_1$ by an auxilliary $s_1' \in [s_1, 2 s_1]$.  The reason for doing this is that will eventually have to average $s_1'$ over $[s_1 , 2 s_1]$

 We estimate for $0\leq u \leq t_1$,
\begin{align} \label{eqn:reva1}
\left( \UB_{ia}(0, t_1+u) - \frac{1}{N} p_{t_1+u} (\gf_i, \gf_a) \right)^2 &\leq C \left( \frac{1}{N} w_{t_1+u} (i) - \frac{1}{N} f_{t_1+u} (i) \right)^2 \\
& + C \left( \frac{1}{N} p_{t_1+u} (\gf_i , \gf_a) - \frac{1}{N} p_{t_1+u-s_1'+s_0} ( \hatz_i, \gf_a ) \right)^2 \label{eqn:up1}\\
&+ C \left( \frac{1}{N} p_{t_1+u-s_1'+s_0} (\hatz_i, \gf_a ) -\frac{1}{N} f_{t_1+u} (i) \right)^2 \label{eqn:up2}
\end{align}
The terms \eqref{eqn:up1} and \eqref{eqn:up2} are estimated using essentially the regularity of $p_t (x, y)$.  The remaining term \eqref{eqn:reva1} is estimated using the last corollary.

We can estimate the term \eqref{eqn:up1} using the results from Lemma \ref{lem: Ketalemma} and the optimal rigidity estimate \eqref{eqn:zrigref} from Section \ref{sec:ubref} \ref{item:zrigref}.  We obtain,
\begin{align}
&\left( \frac{1}{N} p_{t_1+u} (\gf_i , \gf_a) - \frac{1}{N} p_{t_1+u-s_1'+s_0} ( \hatz_i, \gf_a ) \right)^2 \leq  C \left( \frac{1}{N} p_{t_1+u} (\gf_i , \gf_a) - \frac{1}{N} p_{t_1+u} ( \hatz_i, \gf_a ) \right)^2 \notag\\
+&C\left( \frac{1}{N} p_{t_1+u} (\gf_i , \gf_a) - \frac{1}{N} p_{t_1+u-s_1'+s_0} ( \gf_i, \gf_a ) \right)^2 \leq \frac{N^\eps}{ (N t_1)^4} + \frac{C}{ (N t_1)^2} \frac{ s_1^2}{t_1^2}.
\end{align}
For \eqref{eqn:up2} we have, using the normalization $N^{-1} \sum w (j) = 1$, the profile from Lemma \ref{lem:prof1}, and the estimates from Lemma \ref{lem: Ketalemma},
\begin{align}
&\left(\frac{1}{N} p_{t_1+u-s_1'+s_0} ( \hatz_i, \gf_a ) -\frac{1}{N} f_{t_1+u} (i) \right)^2\notag\\
= &\left(\frac{1}{N^2} \sum_j w_{s_1'} (j) \left( p_{t_1+u-s_1'+s_0} ( \hatz_i , \gf_a ) - p_{t_1+u-s_1'+s_0} (\hatz_i, \gf_j ) \right) \right)^2 \notag\\
\leq& N^\eps\left( \frac{1}{N^2} \sum_j \frac{s_1}{ (|j-a|/N)^2+s_1^2} \frac{1}{t_1} \left[ \left( \frac{|j-a|}{N t_1} \right) \wedge 1 \right] \right)^2 \notag\\
\leq& C N^\eps \left( \frac{1}{N t_1} \frac{1}{t_1}\int_{|x| \leq t_1} \frac{s_1|x| }{x^2 +s_1^2} \d x \right)^2+ C N^\eps \left( \frac{1}{N t_1} \int_{ |x| > t_1} \frac{s_1}{s_1^2+x^2} \d x\right)^2 \leq \frac{C N^{2\eps}}{ (N t_1)^2} \frac{ s_1^2}{t_1^2}.
\end{align}
Above, we used the fact that
\beq
\left| p_{t_1+u-s_1'+s_0} ( \hatz_i , \gf_a ) - p_{t_1+u-s_1'+s_0} (\hatz_i, \gf_j ) \right| \leq \frac{C}{t_1} \min\left\{ \frac{ |j-a|}{N t_1}, 1 \right\} .
\eeq
We now deal with \eqref{eqn:reva1}. 
We estimate
\begin{align}
&\left( \frac{1}{N} w_{t_1+u} (i) - \frac{1}{N} f_{t_1+u} (i) \right)^2 \notag\\
  \leq & C \left( \frac{1}{N} w_{t_1+u} (i)-\frac{1}{N} \frac{1}{\ell} \sum_{|j-i| \leq \ell }w_{t_1+u} (j)  - \frac{1}{N} f_{t_1+u} (i) + \frac{1}{N} \frac{1}{\ell} \sum_{|j-i| \leq \ell }f_{t_1+u} (j) \right)^2  \notag\\
  + &C\left( \frac{1}{N} \frac{1}{\ell} \sum_{|j-i| \leq \ell }w_{t_1+u} (j) -\frac{1}{N} \frac{1}{\ell} \sum_{|j-i| \leq \ell }f_{t_1+u} (j) \right)^2
\end{align}
We apply the Sobolev inequality of  Lemma \ref{lem:sobrev} to the difference of the sequences on $\{ k: |i-k| \leq \ell \}$ (with the $N$ in Lemma \ref{lem:sobrev} being $\ell$) 
\beq
\left \{ w_{t_1+u} (k) - \frac{1}{N} \frac{1}{\ell} \sum_{|j-i| \leq \ell }w_{t_1+u} (j) \right\}_k, \left\{ f_{t_1+u} (k) - \frac{1}{N} \frac{1}{\ell} \sum_{|j-i| \leq \ell }f_{t_1+u} (j)\right\}_k
\eeq
which now have mean $0$.  We find,
\begin{align}
\left( \frac{1}{N} w_{t_1+u} (i) - \frac{1}{N} f_{t_1+u} (i) \right)^2 \leq &\frac{N^\eps}{N^2} \langle (w-f) (t_1+u), \B (w-f) (t_1+u) \rangle \notag \\
+ &C\left( \frac{1}{N} \frac{1}{\ell} \sum_{|j-i| \leq \ell }w_{t_1+u} (j) -\frac{1}{N} \frac{1}{\ell} \sum_{|j-i| \leq \ell }f_{t_1+u} (j) \right)^2 .
\end{align}
We used the fact that $|z_i - z_j | \geq N^{ - \eps/4} |i-j|/N$. 
We have
\beq
\frac{1}{N} \frac{1}{\ell} \sum_{|j-i| \leq \ell }w_{t_1+u} (j) = \frac{1}{\ell} + \O \left( \frac{ N t_1}{\ell^2 }\right).
\eeq
Similarly, (using \eqref{eqn:ptnormfull}) we have
\beq
\frac{1}{N} \frac{1}{\ell} \sum_{|j-i| \leq \ell }f_{t_1+u} (j) = \frac{1}{\ell} + \O \left( \frac{1}{\ell N t_1} +\frac{ N t_1}{\ell^2} \right).
\eeq
Hence,
\begin{align}
\left( \frac{1}{N} w_{t_1+u} (i) - \frac{1}{N} f_{t_1+u} (i) \right)^2 \leq &\frac{N^\eps}{N^2} \langle (w-f) (t_1+u), \B (w-f) (t_1+u) \rangle \notag + C\frac{1}{ \ell^2 (N t_1)^2} +C \frac{ (N t_1)^2}{ \ell^4}
\end{align}

   We now apply Corollary \ref{cor:maincalc}  to obtain that there is an event with overwhelming probability (which depends on the choice of $s_1'$), such that
\begin{align}
\int_{s_1'}^{2 t_1} \langle (w-f), \B (w-f ) \rangle \d s \leq \frac{N^\eps}{s_1}\left( \frac{1}{ ( N s_0 )^{1/2} } + \frac{1}{ N^{\om_F}} \right) + || (w-f) (s_1') ||_2^2 .
\end{align}
We can average over $s_1' \in [s_1, 2 s_1]$ (even though the event described above is $s_1'$-dependent, since each holds with overwhelming probability and $\UB$ and $p_t$ are bounded, we can apply Lemma \ref{lem:fubini}) and obtain that with overwhelming probability,
\begin{align}
&\frac{1}{t_1} \int_{0}^{t_1} \left( \UB_{t_1+u}(i, a) - \frac{1}{N} p_{t_1+u} ( \gf_i, \gf_a ) \right)^2 \d u \leq \frac{1}{N^2 t_1 s_1} \int_{0}^{s_1}  || (w-f) (s_1 + u ) ||_2^2 \d u  \notag\\
+& \frac{C N^\eps}{ ( N t_1)^2} \left(\frac{1}{ ( N t_1 )^2 } + \frac{ s_1^2}{ t_1^2} + \frac{1}{ \ell^2} + \frac{ ( N t_1)^4}{ \ell^4}  + \frac{ t_1}{ s_1 ( N s_0 )^{1/2} } + \frac{ t_1}{s_1 N^{\om_F}}  \right) .
\end{align} 
Note that on the RHS the choice of $f$ itself has an $s_1 + u$ dependence. 
By Lemma \ref{lem:initell2} we have
\begin{align}
& \int_{0}^{s_1}  || (w-f) (s_1 + u ) ||_2^2 \d u \leq N^\eps \left( \frac{1}{ ( N s_0 )^2} + \frac{ ( N s_0 )^2}{ \ell^2} \right) \notag\\
+ & C s_0 \int_{0}^{s_1} \sum_{ |i| \leq N^{\om_A} } \sum_{ |i-j| \leq \ell} \frac{ ( w_i (s_1 + u ) - w_j (s_1 + u ) )^2}{ (i-j)^2} \d u .
\end{align}
With overwhelming probability we have
\begin{align}
\int_{0}^{s_1} \sum_{ \substack { |i| \leq N^{\om_A}  \\ |i-j| \leq \ell}} \frac{ ( w_i (s_1 + u ) - w_j (s_1 + u ) )^2}{ (i-j)^2} \d u &\leq N^\eps \int_{0}^{s_1}  \langle w , \B w (s_1 +u ) \rangle \d u  \leq N^\eps ||w (s_1)||_2^2  \leq \frac{ C N^{2 \eps} }{ s_1}
\end{align}
where in the second inequality we used the standard energy estimate $\del_t ||w||_2^2 = -  \langle w, \B w \rangle$.  The claim now follows after simplifying the errors.  In particular we use,
\beq
\frac{1}{ \ell^2} \leq \frac{ ( N s_0 )^2}{ \ell^2} \leq \frac{ s_0}{s_1}, \qquad \frac{1}{ ( N t_1 )^2 } \leq \frac{t_1}{s_1} \frac{1}{ ( N s_0 )^{1/2} }
 \eeq \qed

\subsubsection{Removal of time average}
Let $\eps_2 > 0$ and let 
\beq
t_2 := t_1 N^{- \eps_2}.
\eeq
In this section we show how to remove the time average in Theorem \ref{thm:techhomog}.    More precisely, we prove the following theorem.  It is deduced from Theorem \ref{thm:techhomog} using only the fact that $\UB$ is a semigroup, the decay properties of $\UB$ given by Lemma \ref{lem:prof1} and the regularity of $p_t (x, y)$.
\bet \label{thm:notime}Let $a$ satisfy 
\beq
|a| \leq N^{\om_A - \eps_B}/2
\eeq
and $i$ satisfy
\beq
|i-a | \leq \frac{ \ell}{20}.
\eeq
For any $\eps >0$ and $D >0$ there is an event $\F_\alpha$ with $\pp [ \F_\alpha ] \geq 1 - N^{-D}$ on which
\begin{align}
& \left| \UB_{ia} (0, t_1 + 2 t_2)- \frac{1}{N}p_{t_1 } ( \gf_i, \gf_a ) \right| \notag\\
\leq & C N^\eps \frac{N^{\eps_2}}{ N t_1 }  \bigg\{  \frac{ s_1^2}{ t_1^2} + \frac{ ( N t_1 )^4}{ \ell^4 } + \frac{t_1}{s_1} \left( \frac{1}{ ( N s_0 )^{1/2} } + \frac{ 1}{ N^{\om_F}}+ \frac{ s_0}{s_1} \right) \bigg\}^{1/2} \notag\\
+& \frac{N^\eps}{ N t_1} N^{-\eps_2/2}.
\end{align}

\eet
\proof Theorem \ref{thm:techhomog} implies that we have with overwhelming probability
\begin{align} \label{eqn:rta1}
&\frac{1}{t_2} \int_0^{t_2} \left|  \UB_{jk}(0, t_1+u) - \frac{1}{N} p_{t_1 + u } ( \gf_j , \gf_k ) \right| \d u \notag \\
\leq &N^\eps \frac{ N^{\eps_2}}{ N t_1 }   \bigg\{  \frac{ s_1^2}{ t_1^2} + \frac{ ( N t_1 )^4}{ \ell^4 } + \frac{t_1}{s_1} \left( \frac{1}{ ( N s_0 )^{1/2} } + \frac{1}{ N^{\om_F}}+ \frac{ s_0}{s_1} \right) \bigg\}^{1/2},
\end{align}
for  $k \leq N^{\om_A-\eps_B}$ and $|j-k| \leq \ell/10$.  For notational simplicity let us denote
\beq
\Phi :=  N^{\eps_2} \bigg\{  \frac{ s_1^2}{ t_1^2} + \frac{ ( N t_1 )^4}{ \ell^4 } + \frac{t_1}{s_1} \left( \frac{1}{ ( N s_0 )^{1/2} } + \frac{1}{N^{\om_F}}+ \frac{ s_0}{s_1} \right) \bigg\}^{1/2}
\eeq
By the semigroup property we can write for any $0 \leq u \leq t_2$,
\beq
\UB_{a i} (0, t_1 + 2 t_2 ) = \sum_j \UB_{a j } ( t_1 + u, t_1 + 2 t_2 ) \UB_{j i } ( 0, t_1 + u ) 
\eeq
and so we can take an average over $u$ and obtain
\beq
\UB_{a i} (0, t_1 + 2 t_2 ) = \sum_j  \frac{1}{t_2} \int_0^{t_2}  \UB_{a j } ( t_1 + u, t_1 + 2 t_2 ) \UB_{j i } ( 0, t_1 + u )  \d u.
\eeq
We now rewrite the RHS as
\begin{align}
\UB_{a i} (0, t_1 + 2 t_2 ) &= \sum_j  \frac{1}{t_2} \int_0^{t_2}  \UB_{a j } ( t_1 + u, t_1 + 2 t_2 ) \left( \UB_{j i } ( 0, t_1 + u ) - \frac{1}{N} p_{t_1 + u } ( \gf_j, \gf_i ) \right) \d u \label{eqn:time1} \\
&+ \sum_j  \frac{1}{t_2} \int_0^{t_2}  \UB_{a j } ( t_1 + u, t_1 + 2 t_2 ) \frac{1}{N} \left(  p_{t_1 + u } ( \gf_j, \gf_i ) - p_{t_1 + u } (\gf_a, \gf_i ) \right) \d u \label{eqn:time2} \\
&+ \sum_j  \frac{1}{t_2} \int_0^{t_2}  \UB_{a j } ( t_1 + u, t_1 + 2 t_2 ) \left( \frac{1}{N} p_{t_1 + u } ( \gf_a, \gf_i ) \right) \d u
\end{align}
From Lemma \ref{lem:prof1}, we have the estimate
\beq
| \UB_{aj} ( t_1 + u, t_1 + 2 t_2 )  | \leq \frac{1}{N} \frac{ N^\eps t_2}{ ( (a-j)/N )^2 + t_2^2 } 
\eeq
from which we see that for any $ \delta > 0$,
\beq
\sum_{ j : |j-a | > N t_2 N^{\delta} } | \UB_{aj} (t_1 + u, t_1 + 2 t_2 ) | \leq \frac{ N^\eps}{ N^{\delta}}
\eeq
and also
\beq
\sum_{ j : |j-a|  \leq N t_2 N^{\delta} }  \UB_{aj} (t_1 + u, t_1 + 2 t_2 ) = 1 +  N^\eps \O \left( \frac{1}{N^\delta } \right).
\eeq
Fix a $\delta >0$ s.t. $\delta < \eps_2$.  We also have the estimate
\beq
\left| \UB_{ji} (0, t_1 + u ) \right| + \left| \frac{1}{N} p_{t_1+u} (\gf_j, \gf_i ) \right| \leq \frac{ N^\eps}{ N t_1 }.
\eeq
We use these to estimate the term \eqref{eqn:time1} by
\begin{align}
&\left| \sum_j  \frac{1}{t_2} \int_0^{t_2}  \UB_{a j } ( t_1 + u, t_1 + 2 t_2 ) \left( \UB_{j i } ( 0, t_1 + u ) - \frac{1}{N} p_{t_1 + u } ( \gf_j, \gf_i ) \right) \d u \right| \notag\\
\leq & \sum_{j : |j-a| > N t_2 N^{\delta} } \frac{1}{t_2} \int_0^{t_2}  \UB_{a j } ( t_1 + u, t_1 + 2 t_2 ) \frac{N^\eps }{ N t_1 } \d u  \notag\\
+ & \sum_{ j : |j - a | \leq N t_2 N^{\delta} } \frac{1}{N} \frac{ N^\eps t_2}{ (( j-a )/N )^2 + t_2^2 } \frac{1}{t_2} \int_0^{t_2}  \left| \UB_{j i } ( 0, t_1 + u ) - \frac{1}{N} p_{t_1 + u } ( \gf_j, \gf_i ) \right| \d u \notag\\
\leq & \frac{ N^{2\eps}}{ N t_1 N^{\delta}} + \frac{ N^{2\eps}}{ N t_1} \Phi.
\end{align}
Note that we are allowed to apply the estimate \eqref{eqn:rta1} because $|j-a| \leq N t_2 N^\delta \ll \ell$ which implies $|j-i| < \ell/10$.  For $ |j-a | \leq N^{\delta} ( N t_2 ) $ we have the estimate
\beq
 \frac{1}{N} \left|  p_{t_1 + u } ( \gf_j, \gf_i ) - p_{t_1 + u } (\gf_a, \gf_i ) \right| \leq N^\eps \frac{1}{ N t_1 } \frac{ N t_2 N^{\delta}}{ N t_1 }.
\eeq
Therefore we can estimate \eqref{eqn:time2} by 
\begin{align}
& \left| \sum_j  \frac{1}{t_2} \int_0^{t_2}  \UB_{a j } ( t_1 + u, t_1 + 2 t_2 ) \frac{1}{N} \left(  p_{t_1 + u } ( \gf_j, \gf_i ) - p_{t_1 + u } (\gf_a, \gf_i ) \right) \d u \right| \notag\\
\leq & \sum_{j : |j-a | > N t_2 N^{\delta}}  \frac{1}{t_2} \int_0^{t_2}  \UB_{a j } ( t_1 + u, t_1 + 2 t_2 ) \frac{N^{\eps}}{ N t_1 } \d u \notag\\
+ & \sum_{j: |j-a | \leq N t_2 N^{\delta } }  \frac{1}{t_2} \int_0^{t_2}  \UB_{a j } ( t_1 + u, t_1 + 2 t_2 ) N^\eps \frac{1}{ N t_1 } \frac{ N t_2 N^{\delta}}{ N t_1 } \d u \notag\\
&\leq  N^{2 \eps } \frac{ 1}{ N t_1 }  \left( \frac{1}{ N^\delta} +  \frac{ N^\delta}{ N^{\eps_2}}\right).
\end{align}
Lastly, since $\sum_j \UB_{aj } (t_1 + u, t_1 + 2 t_2 ) = 1$ we get
\begin{align}
\sum_j  \frac{1}{t_2} \int_0^{t_2}  \UB_{a j } ( t_1 + u, t_1 + 2 t_2 ) \left( \frac{1}{N} p_{t_1 + u } ( \gf_a, \gf_i ) \right) \d u &= \frac{1}{t_2} \int_0^{t_2}  \left( \frac{1}{N} p_{t_1 + u } ( \gf_a, \gf_i ) \right) \notag\\
&= \frac{1}{N} p_{t_1 + 2 t_2 } (\gf_a, \gf_i ) + \frac{1}{ N t_1 } \O \left( N^{-\eps_2} \right).
\end{align}
Here we used
\beq
\frac{1}{N} \left| p_{t_1+u} ( \gf_i , \gf_a ) - p_{t_1} ( \gf_i, \gf_a ) \right| \leq C \frac{N^{-\eps_2}}{ N t_1 }.
 \eeq
 This yields the claim after taking $\delta = \eps_2/2$.
  \qed
  
  At this point we just have to choose the parameters $s_0$ and $s_1$ to conclude the homogenization result for $\UB$.
  
  \noindent{\bf Proof of Theorem \ref{thm:homogref}}.   We use the result of Theorem \ref{thm:notime} and just make a choice of $s_0$ and $s_1$.  First we optimize over $s_0$ and take
  \beq
  ( N s_0 ) = ( N s_1 )^{2/3}.
  \eeq
  The error then simplifies to
  \beq
 N^\eps \frac{ N^{\eps_2}}{ N t_1} \left\{ \frac{ s_1^2}{t_1^2} + \frac{ ( N t_1 )^4}{ \ell^4 } + \frac{t_1}{s_1} \left( \frac{1}{ N^{\om_F}} + \frac{1}{ ( N s_1)^{1/3} } \right) \right\}^{1/2} + N^\eps \frac{ N^{-\eps_2/2} }{ N t_1}.
  \eeq
To optimize over $s_1$ we have two cases.  If $\om_F \geq \om_1 3/10$ then we take $N s_1 = ( N t_1 )^{9/10}$.  If $\om_F \leq \om_1 3/10$ then we take $N s_1 = N t_1 N^{-\om_F/3}$.  The error simplifies to
\beq
N^\eps \frac{ N^{\eps_2}}{ N t_1} \left\{  \frac{ ( N t_1 )^4}{ \ell^4 } + \frac{1}{( N t_1)^{1/5}} + \frac{1}{ N^{\om_F 2/3} } \right\}^{1/2} + N^\eps \frac{ N^{-\eps_2/2} }{ N t_1}.
\eeq
This is the claim. \qed

\subsection{Completion of proof of Theorem \ref{thm:xyref}} \label{sec:xyrefthmproof}
First of all we see by the definitions of $\tilz_i$ and $\hatz_i$ (see \eqref{eqn:tilz} for the former and \eqref{eqn:zhat1}-\eqref{eqn:zhat2} for the latter) and  Lemma \ref{lem:shortrange} that with overwhelming probability,
\begin{align}
z_i ( t_1 + u, 1 ) - z_i ( t_1 + u, 0 ) &= \left( \tilz_i ( t_1 + u , 1) - \tilz_i ( t_1 + u , 0) \right) + \left(\gamma_0 ( t_1 + u, 1 ) - \gamma_0 (t_1 + u , 0) \right) \notag\\
&= \left( \hatz_i ( t_1 + u, 1 ) - \hatz_i ( t_1 + u, 0 ) \right) + \left(\gamma_0 ( t_1 + u, 1 ) - \gamma_0 ( t_1 + u , 0) \right)  \notag\\
&+ \O \left( N^\eps t_1 \left( \frac{ N^{\om_A}}{ N^{\om_0}} + \frac{1}{ N^{\om_\ell}} + \frac{1}{ \sqrt{ N \G } } \right) \right).
\end{align}
Note that the first equation is just by definition - the classical particle locations are defined together with $\tilde{z}_i$ at the start of Section \ref{sec:shortrange}.

With $u_i = \del_\alpha \hatz_i$ we have
\beq
\left( \hatz_i (t_1 + u , 1) - \hatz_i ( t_1 + u, 0 ) \right)  = \int_0^1 u_i (t_1 + u, \alpha ) \d \alpha.
\eeq
Recall $u$ satisfies $\del_t u ( \alpha ) = \B (\alpha ) u ( \alpha) + \xi ( \alpha)$ with $\xi ( \alpha)$ defined as in Section \ref{sec:deriveequation}, and initial data
\beq
u_i ( 0, \alpha ) = \alpha \hat{z}_i (0, 1) + (1- \alpha) \hat{z}_i (0, 0).
\eeq
  By the bound \eqref{eqn:xiestimate}, the assumption \eqref{eqn:supzref} and Lemma \ref{lem:fs1} we see that for any small $\delta_B >0$
\beq
\sup_{|i| \leq N^{\om_A - \delta_B}, |u| \leq t_1} \left| u_i (t_1 + u, \alpha ) - v_i (t_1 + u, \alpha ) \right| \leq \frac{1}{N^{10}}
\eeq
with overwhelming probability where $v_i (\alpha)$ is defined by
\beq
\del_t v (t, \alpha) = \B (\alpha )  v (t, \alpha), \qquad v_i (0, \alpha)) = \1_{ \{ |i| \leq N^{\om_A} \} } u_i (0, \alpha).
\eeq
Fix an $\eps_a >0$ and consider the solution
\beq
\del_t w (\alpha) = \B (\alpha) w (\alpha), \qquad w_i (0, \alpha)) = \1_{ \{ |i| \leq ( N t_1 ) N^{\eps_a} \} } u_i (0, \alpha).
\eeq
Since $|u_i (0, \alpha) | \leq N^\eps/N$ for any $|i| \leq N^{\om_0/2}$ with overwhelming probability, we see by Lemma \ref{lem:prof1} that for $|i| \leq N t_1 N^{\eps_b}$ with $\eps_b < \eps_a$,
\begin{align}
\left| v_i (t_1 + u , \alpha) - w_i (t_1 + u , \alpha) \right| &\leq \left| \sum_{ N^{\om_1+\eps_a} < |j| \leq N^{\om_A}  } \UB_{ij} (0, t_1 + u, \alpha ) u_j (0, \alpha) \right| \notag\\
&\leq \frac{ N^\eps}{N} N t_1 \sum_{ |j| > N^{\om_1 + \eps_a} } \frac{1}{ (i - j)^2 } \leq C \frac{ N^\eps}{ N N^{\eps_a}}.
\end{align}
Therefore,
\begin{align}
z_i (t_1 + u, 1) - z_i (t_1 + u, 0) &= \left( \gamma_0 (t_1 + u, 1) - \gamma_0 (t_1 + u, 0 ) \right) \notag\\
&+ \int_0^1 \sum_{ |j| \leq N t_1 N^{\eps_a} } \UB_{ij} (0, t_1 + u, \alpha) ( z_j (0, 1) - z_j (0, 0) ) \d \alpha \notag\\
&+ \frac{N^\eps}{N} \O \left( N^{\om_1} \left( \frac{ N^{\om_A}}{N^{\om_0}} + \frac{1}{ N^{\om_\ell}} + \frac{ 1}{ \sqrt{ N \G } } \right)  + \frac{1}{N^{\eps_a}} \right)
\end{align}
with overwhelming probability for $|i| \leq N t_1 N^{\eps_b}$.  Note that we used Lemma \ref{lem:fubini}, and that $u (\alpha), v(\alpha), w ( \alpha)$ are all bounded by $N^C$ on an $\alpha$-independent event of overwhelming probability.  Theorem \ref{thm:xyref} now follows from an application of Theorem \ref{thm:homogref} with $\om_F = \infty$. \qed

\subsection{Proof of Theorem \ref{thm:mainhomog}} \label{sec:xythmproof}

Theorem \ref{thm:xyref} implies, after re-writing $z_i (t, 1)$ in terms of $x_i$, that with overwhelming probability we have,
\begin{align}
& \left(  x_{i_0+i} (t_0 + t_1 + u ) -  \gamma_{i_0 } ( t_0 + t_1 + u ) \right)- y_{ N/2+i} (t_0 + t_1 + u ) \notag\\
= & \sum_{ |j| \leq N t_1 N^{\eps_a } } \zeta (N^{-1}(i-j), t_1 ) \left[  \left(  x_{i_0+j} (t_0 + t_1 ) -  \gamma_{i_0 } ( t_0 + t_1 ) \right)  - y_{ N/2+j} (t_0 + t_1  )  \right] \notag\\
+&\frac{N^\eps}{N} \O \left( N^{\om_1} \left( \frac{ N^{\om_A}}{N^{\om_0}} + \frac{1}{ N^{\om_\ell}} + \frac{1}{ \sqrt{ N \G } } \right) + \frac{1}{N^{\eps_a}}+ N^{\eps_2+\eps_a} \left( \frac{ ( N t_1 )^2}{ \ell^2} + \frac{1}{ ( N t_1 )^{1/10} } \right) + N^{\eps_a-\eps_2/2 } + \frac{N^{\om_1}}{N^{\om_0/2}} \right),
\end{align}
for any $|i| \leq N t_1 N^{\eps_b}$ and $|u| \leq t_2$.   Recall that $\zeta$ was defined in \eqref{eqn:zetarev1} (and then $p_t$ as the fundamental solution to \eqref{eqn: Ketadef}).  We also used Lemma \ref{lem:newquant} to replace the classical particle locations from the interpolating measures with those coming from the original free convolution $\rhofct$ and the semicircle law. We choose $\om_A = (\om_\ell+\om_0/2)/2$.  
The error simplifies to
\beq
\frac{N^{2\eps}}{N} \O \left( N^{\om_1} \left(\frac{1}{ N^{\om_\ell}}  \right) + \frac{1}{N^{\eps_a}}+ N^{\eps_2+\eps_a} \left( \frac{ ( N t_1 )^2}{ \ell^2} + \frac{1}{ ( N t_1 )^{1/10} } \right) + N^{\eps_a-\eps_2/2 }  \right).
\eeq
There are two cases.  First if $\om_1 \geq 10 \om_0 / 21 $ then the error simplifies to 
\beq
\frac{N^{2\eps}}{N} \O \left( N^{\om_1} \left(\frac{1}{ N^{\om_\ell}}  \right) + \frac{1}{N^{\eps_a}}+ N^{\eps_2+\eps_a} \left( \frac{ ( N t_1 )^2}{ \ell^2}  \right) + N^{\eps_a-\eps_2/2 }  \right).
\eeq
In this case we then take $\eps_2 = 4 ( \om_\ell - \om_1 ) / 3 $ and then $\eps_a = ( \om_\ell - \om_1 )/3$.  The error simplifies to 
\beq
\frac{ N^{3 \eps} }{N} \O \left(  \frac{ N^{ \om_1/3} }{ N^{\om_\ell/3}} \right) .
\eeq
We then take $\om_\ell = \om_0/2 - \eps$ and so the error is 
\beq
\frac{ N^{4 \eps} }{N} \O \left(  \frac{ N^{ \om_1/3} }{ N^{\om_0 /6}} \right)  \leq  \frac{ N^{4 \eps} }{N} \O \left(  \frac{ N^{ \om_1/3} }{ N^{\om_0 /6}}  + \frac{1}{ N^{\om_1/60}} \right)
\eeq
In the case $\om_1 < 10 \om_0 / 21 $ we take $\om_\ell = 21 \om_1 / 20 < \om_0/2$.  The error simplifies to 
\beq
\frac{ N^{ 2 \eps }}{N} \O \left( \frac{1}{ N^{\om_1/20}} + \frac{1}{ N^{\eps_a}} + N^{ \eps_2 + \eps_a} \frac{1}{ N^{\om_1/10}} + N^{ \eps_a - \eps_2/2} \right)
\eeq
Choose $\eps_2 = \om_1/15$ and $\eps_a = \om_1/60$.  The error then simplifies to
\beq
\frac{N^{3\eps}}{N} \O \left(   \frac{1}{N^{\om_1/60}}   \right) \leq \frac{N^{4\eps}}{N} \O \left( \frac{1}{N^{\om_1/60}}  + \frac{N^{ \om_1/3} }{ N^{\om_0 /6}}   \right).
\eeq
\qed

\section{Finite speed estimates} \label{sec:finitespeed}
\subsection{Estimate for short-range operator} \label{sec:expfs}

In this section we work in the set-up of Section \ref{sec:ubref} and assume that \ref{item:rho2}-\ref{item:fiji} hold.  Let $\hatz_i$ be defined as in that section. Fix a parameter $\ell_3 = N^{\om_{\ell,3}}$ satisfying
\beq
0 < \om_{\ell,3} \leq \om_\ell.
 \eeq
 For the current Section \ref{sec:expfs} we fix a $0 < q < 1$ and let $\A_3$ be the set
\begin{align}
\A_3 := \{ (i, j) : |i - j | \leq N^{\om_{\ell,3}} \} \cup \{ (i, j) : ij >0, i \notin \Chat_q, j \notin \Chat_q \}.
\end{align}
Define the operator 
\beq
(\B_3 u )_i := \sumAthi_j \frac{1}{N} \frac{u_j - u_i }{ (\hatz_j - \hatz_i )^2}.
\eeq


We want to prove the following theorem.  Lemma \ref{lem:fs1} is an immediate consequence.   The method is based on that appearing in \cite{que}.
\bet  \label{thm:finitespeed}
Let $\ell_3$ as above.  Let $D_1, D_2 >0$, and let $\eps >0$.  Let $0 < q_3 < q$.  We assume that Section \ref{sec:ubref} \ref{item:rho2}-\ref{item:fiji} hold.  Fix a time $u \leq 10 t_1 $.  There is an event $\F ( \alpha, u )$ s.t. $\pp [ \F ( \alpha, u ) ] \geq 1 - N^{-D_1}$ for $N$ large enough (independent of $\alpha$) such that all of the following estimates hold.  For every $u \leq s \leq t \leq \left( u + 2 \ell_3/N \right) \wedge 10 t_1$ we have the estimate
\beq
\UBth_{ba} (s, t) \leq \frac{1}{ N^{D_2}}.
\eeq
provided one of the following three criteria holds.
\begin{enumerate}[label=(\roman*)]
\item $a \in \Chat_{q_3}$ and $|a-b | > N^{\om_{\ell_3}+\eps}$. 
\item $b \in \Chat_{q_3}$ and $|a-b| > N^{\om_{\ell,3}+\eps}$.
\item $a \notin \Chat_{q_3}$, $b \notin \Chat_{q_3}$ and $ab <0$.
\end{enumerate}
Hence, the same estimate holds for any $0 \leq s \leq t \leq 10 t_1$ that satisfy $t-s \leq \ell_3/N$, with overwhelming probability.
\eet

In the proof of Theorem \ref{thm:finitespeed} we take $u=0$ for notational simplicity.  The first step in proving Theorem \ref{thm:finitespeed} is to establish the estimate for $s=0$, which is the content of the following lemma.  We then use the semigroup property to extend the estimate to all $s$.

\bel \label{lem:prefs} Fix $0 < q_3 < q$.  Let $\eps >0$ and $D_1 , D_2 >0$.  Assume that Section \ref{sec:ubref} \ref{item:rho2}-\ref{item:fiji} hold.  There is an event $\F_\alpha$ with $\pp [ \F_\alpha ] \geq 1 - N^{-D_1}$ on which the following estimates hold.  For every $0  \leq t \leq \ell_3/N \wedge 10 t_1$ we have the estimate
\beq
\UBth_{ba} (0, t) \leq \frac{1}{ N^{D_2}}.
\eeq
provided one of the following three criteria holds.
\begin{enumerate}[label=(\roman*)]
\item $a \in \Chat_{q_3}$ and $|a-b | > N^{\om_{\ell_3}+\eps}$. 
\item $b \in \Chat_{q_3}$ and $|a-b| > N^{\om_{\ell,3}+\eps}$.
\item $a \notin \Chat_{q_3}$, $b \notin \Chat_{q_3}$ and $ab <0$.
\end{enumerate}
\eel
\proof Define for $t \geq 0$, $f_i (t) := \UB_{ia} (0, t)$; i.e.,  $f_i$ satisfies the equation
\beq
\del_t f_i = ( \B_3 f)_i, \qquad f_i (0) = \delta_a.
\eeq
WLOG, take $\eps >0$ s.t. $\om_{\ell,3} + \eps < \om_0/2$. We can assume $a \geq 0$. Fix $q_4$ satisfying $q_3 < q_4 < q$.  It then suffices to prove the statement for the following two cases.
\begin{enumerate}
\item $a \in \Chat_{q_4}$, $|a-b| > N^{\om_{\ell,3} + \eps}$.
\item $a \notin \Chat_{q_4}$, $b \in \Chat_{q_3}$ or $b <0$.
\end{enumerate}
Let us first consider the case $a \in \Chat_{q_4}$.   Let $\nu >0$ and define
\beq
\phi_k := \e^{ \nu \psi ( \hatz_k  ( t, \alpha) - \tilg_a (t, \alpha ))/2 }
\eeq
where $\psi$ is the following smooth function.  Fix a scale $\ell_4 = N^{\om_{\ell,4}}>0$ and a $\delta_1 >0$ s.t.
\beq
\delta_1 + \om_\ell < \om_0, \qquad  0 < \om_{\ell,4} \leq \om_{\ell,3}.
\eeq
Assume $
0 < \eps < \delta_1.
$
We choose $\psi$ s.t. 
\beq
\psi (x) = -x , \qquad |x| \leq   \frac{N^{\delta_1+ \om_\ell}}{N},
\eeq
and 
\beq
\psi ( x ) = \mp \left( \frac{N^{\delta_1+\om_\ell}}{N} + \frac{\ell_4}{2N}\right), \qquad \pm x > \frac{N^{\delta_1+\om_\ell}}{N} + \frac{\ell_4}{N}
\eeq
We can choose $\psi$ so that $| \psi'| \leq 1$ and $| \psi''| \leq C N/\ell_4$.

Our proof is based around a Gronwall argument and we will need to take an expectation of a martingale.  For this we need to introduce the following stopping time $\tau_r$.  Let $q < q_r < 1$ and $\eps_r > 0$ with $\eps_r < \eps/100$.  Let $\tau_i$, $i=1,2,3$ be the stopping time
\begin{align}
\tau_1 &:= \inf \{ t > 0 : \exists i  \in \Chat_{q_r} : | \hatz_i (t ) - \gamma_i (t) | > N^{\eps_r}/N \} \notag\\
\tau_2 &:= \inf \{ t >0 : \exists i \in \Chat_{q_r} : | J_i | > C_J \log (N)\} \notag \\ 
\tau_3 &:= \inf \{ t > 0 : \exists i : |F_i | \geq \log (N) \} .
\end{align}
We set $\tau_r := \tau_1 \wedge \tau_2 \wedge \tau_3 \wedge (10 t_1)$.  We know that $\tau_r = 10 t_1$ with overwhelming probability by the assumptions \ref{item:zrigref} and \ref{item:fiji} of Section \ref{sec:ubref}.

Define now $v_k(t) = \phi_k f_k$ and $F = \sum_k v^2_k (t) \1_{ \{ \tau_r > 0 \} }$.  By the same calculation as in \cite{que} we obtain
\begin{align}
\d F (t) &= - \frac{1}{2} \sum_{ (i, j) \in \A_3 } \frac{1}{N} \frac{ ( v_k - v_j )^2}{ (\hatz_k - \hatz_j )^2 } \d t \label{eqn:dissip} \\
&- \sum_{(j, k) \in \A_3} \frac{1}{N} \frac{1}{ ( \hatz_j - \hatz_k )^2 } \left[ \frac{ \phi_k }{ \phi_j } + \frac{ \phi_j }{ \phi_k } -2 \right] v_k v_j \d t \label{eqn:fserr1} \\
&+\sum_k \nu v_k^2 \psi_k'  \d ( \hatz_k - \tilg_a ) \label{eqn:fserr2}\\
&+ \sum_k v_k^2 \left( \nu^2 ( \psi_k' )^2 + \nu \psi''_k \right) \frac{ \d t}{N} \label{eqn:fserr3}.
\end{align}
We now deal with each term individually, applying Gronwall at the end of the proof.  In the remainder of the argument we work on times $t < \tau_r$.  We start with \eqref{eqn:fserr1}.  Fix $q_5$ satisfying $q_4 < q_5 < q$.  By rigidity and choice of $\psi$ we have that the term
\beq
\frac{\phi_k }{ \phi_j } + \frac{ \phi_j }{ \phi_k } -2, \quad (j, k) \in \A_3
\eeq
vanishes unless $j, k \in \Chat_{q_5}$.  In this case by rigidity and the fact that $| \psi'| \leq 1$ we have that
\beq
\left| \frac{\phi_k }{ \phi_j } + \frac{ \phi_j }{ \phi_k } -2 \right| \leq  C \nu^2 | \hatz_j - \hatz_k |^2
\eeq
as long as we choose $\nu$ so that $
\nu \ell_3 \leq C N $. 
Hence,
\beq
\left| \sum_{(j, k) \in \A_3} \frac{1}{N} \frac{1}{ ( \hatz_j - \hatz_k )^2 } \left[ \frac{ \phi_k }{ \phi_j } + \frac{ \phi_j }{ \phi_k } -2 \right] v_k v_j  \right| \leq C \frac{ \nu^2}{N} \ell_3 \sum_k v_k^2.
\eeq
Above we used the fact that the cardinality of the set $\{ j : (j, k) \in \A_3 \}$ is bounded by $C \ell_3$ if $k \in \Chat_{q_5}$.
The Ito terms \eqref{eqn:fserr3} are bounded by
\beq
\left| \sum_k v_k^2 \left( \nu^2 ( \psi_k' )^2 + \nu \psi''_k \right) \frac{ \d t}{N} \right| \leq C \left( \frac{ \nu^2}{N}+ \frac{ \nu}{\ell_4} \right) \sum_k v_k^2.
\eeq
We now deal with the terms \eqref{eqn:fserr2}.  By rigidity we have $\psi'_k = 0$ if $k \notin \Chat_{q_5}$.  We can therefore assume $k \in \Chat_{q_5}$.  We fix a $ \delta_2 >0$ s.t.
\beq
\delta_2 < \om_{\ell,3}.
\eeq
From the definition of the $\zhat_k$ process  and the definition of $\tau_r$ we see that we can write for $k \in \Chat_{q_5}$,
\beq \label{eqn:fserr4}
\d ( \hatz_k - \tilg_a )  = \sum_{ |j - k | \leq N^{\delta_2}} \frac{1}{N} \frac{1}{ \hatz_k - \hatz_j } \d t + X_t \d t+ \sqrt{ \frac{2}{N} } \d B_k
\eeq
where we have the bound $
|X_t | \leq C \log N.
$
The first term on the RHS of \eqref{eqn:fserr4} corresponds to 
\begin{align}
\nu \sum_k v_k^2 \psi'_k \sum_{ |j-k| \leq N^{\delta_2} } \frac{1}{N} \frac{1}{ \hatz_k - \hatz_j } &= \frac{\nu}{2} \sum_{ |j-k | \leq N^{\delta_2} } \frac{ \psi'_k}{N} \frac{ v_k^2 - v_j ^2}{ \hatz_k - \hatz_j } + \frac{ \nu}{2} \sum_{ |j-k| \leq N^{\delta_2}} \frac{v_k^2}{N} \frac{ \psi'_j - \psi'_k}{ \hatz_k - \hatz_j } \label{eqn:fserr5}
\end{align}
The second term of \eqref{eqn:fserr5} is bounded by
\beq
\left|  \frac{ \nu}{2} \sum_{ |j-k| \leq N^{\delta_2}} \frac{v_k^2}{N} \frac{ \psi'_j - \psi'_k}{ \hatz_k - \hatz_j }  \right| \leq C \frac{ \nu N^{\delta_2}}{\ell_4} \sum_k v_k^2.
\eeq
We use the Schwarz inequality to bound the first term of \eqref{eqn:fserr5} by
\beq
\left| \frac{\nu}{2} \sum_{ |j-k | \leq N^{\delta_2} } \frac{ \psi'_k}{N} \frac{ v_k^2 - v_j ^2}{ \hatz_k - \hatz_j } \right| \leq \frac{1}{100} \sum_{|j-k| \leq N^{ \delta_2}} \frac{ (v_k - v_j)^2}{ N ( \hatz_k - \hatz_j )^2} + \frac{ C \nu^2}{N} N^{\delta_2} \sum_k v_k^2.
\eeq
The first term on the RHS is absorbed into the term \eqref{eqn:dissip}. Collecting everything we have proven that under the assumption $ \nu \ell_3 \leq C N$,
we have
\beq \label{eqn:ftest}
\del_t \ee [ F (t) ] \leq  C \left( \frac{ \nu^2 \ell_3}{N} + \frac{ \nu N^{\delta_2}}{\ell_4} + \nu \log (N) \right) \ee [ F (t) ].
\eeq
We can take $\ell_4 = \ell_3$ and $\nu = N / ( \ell_3 N^{\eps/2})$.  Then by Gronwall we see that for $t \leq \ell_3/N$ we get
\beq
\ee [ F (t) ] \leq C \ee [ F(0) ] \leq C
\eeq
where the second inequality follows from rigidity, the definition of $\tau$ and the initial condition $f_k(0) = \delta_{ak}$. By construction, for any $\eps >0$ we have that if $j \leq a - N^{\om_{\ell,3}+\eps}$,
\beq
\psi ( \hatz_j - \tilg_a ) > \frac{c}{N} \min\{ N^{\om_{\ell,3}+\eps}, N^{\om_{\ell,3}+\delta_1} \} = \frac{c}{N} N^{\om_{\ell,3}+\eps}
\eeq
and hence
\beq
\nu\psi ( \hatz_j - \tilg_a ) > c N^{\eps/2}.
\eeq
We conclude the claim for $b \leq a - N^{\om_{\ell,3}-\eps}$ and $a \in \Chat_{q_4}$ from Markov's inequality.  For $b > a+ N^{\om_\ell+\eps}$ the argument is similar; one just replaces $\psi$ by $-\psi$.  

We now consider the case $a \notin \Chat_{q_4}$.  Recall that we assumed $a \geq0$.   The argument is identical except one considers, instead of $\psi$ above, 
\beq
\varphi (x) := \psi (x-\tilg_d (t, \alpha) )
\eeq
where $d$ is an index chosen in the following way.  If $d_1>0$ is the largest index in $\Chat_{q_3}$ and $d_2>0$ is the largest index in $\Chat_{q_4}$ (recall $q_3 < q_4$) then $d = (d_1+d_2)/2$.   One then defines
\beq
\varphi_k := \varphi (\hatz_k), \qquad \phi_k := \e^{ \nu \varphi_k/2}, \qquad v_k = \e^{ \nu \varphi_k/2} f_k, \qquad F(t) = \sum_k v_k^2.
\eeq
With this choice rigidity implies that $\varphi'_k = 0$ for $k \notin \Chat_{q_4}$.  Rigidity also implies that the term 
\beq
\frac{\phi_k }{ \phi_j } + \frac{ \phi_j }{ \phi_k } -2 , \qquad (j, k) \in \A_3
\eeq
vanishes unless both $j, k \in \Chat_{q_4}$.  With these considerations the argument can proceed exactly as above.  We  again arrive at \eqref{eqn:ftest} and choose $\ell_4 = \ell_3$ and $\nu = N / ( \ell_3 N^{\eps/2})$.  We see that for $k \in \Chat_{q_3}$ we have
\beq
\psi ( \hatz_k - \tilg_d ) \geq c N^{ \delta_1 + \om_\ell}/N
\eeq
and so $\nu \psi (\hatz_k - \tilg_d ) \geq c N^{\eps/2}$.  We conclude as before.  Note that now we only need that $\psi( \zhat_a - \tilg_d ) < 0$ which follows by the ordering of particles and rigidity to satisfy
\beq
\ee [ F(0) ] \leq C.
\eeq
  \qed
  
  \noindent{\bf Proof of Theorem \ref{thm:finitespeed}}.  For notational simplicity we set $u=0$.  Let $\eps$ and $q_3$ be as in the statement of Theorem \ref{thm:finitespeed}. Wlog, we can assume that $\om_{\ell,3} + \eps < \om_0/2$.   We can assume that the estimates of Lemma \ref{lem:prefs} hold for $q_4$ satisfying $q_3 < q_4 < q$ and $\eps' = \eps/2$.  For any $i$ we can write
  \beq
  \UBth_{bi} (0, t)= \sum_j \UBth_{bj} (s, t) \UBth_{ji} (0, s) \geq \UBth_{ba} (s, t) \UBth_{ai} (0, s).
  \eeq
  We just need to find an $i$ s.t. the LHS is bounded above and $\UBth_{ai} (0, s)$ is bounded below.  Fix $q_5$ satisfying $q_3 < q_5 < q_4$.  As before, it suffices to assume $ a\geq 0$ and to consider the following two cases.
  \begin{enumerate}
  \item $a \in \Chat_{q_5}$, and $|b-a| > N^{\om_{\ell,3} + \eps }$.
  \item $a \notin \Chat_{q_5}$ and $b \in \Chat_{q_3}$ or $b <0$.
  \end{enumerate}
  Let us first consider the case $a \in \Chat_{q_5}$.  Since the estimates of Lemma \ref{lem:prefs} hold, and $a \in \Chat_{q_5}$, we have that
  \beq
  \UBth_{ai} (0, s) \leq \frac{1}{ N^{D_2}}, \qquad |i - a | > N^{\om_{\ell_3} + \eps/2}.
  \eeq
  Since
  \beq
  \sum_i \UBth_{ai} (0, s) = 1,
  \eeq
 this implies that there is an $i_0$ s.t. $|i_0 - a | \leq N^{\om_{\ell,3} + \eps/2}$ and $\UB_{a i_0} (0, s) \geq N^{-1}$.  Moreover, $i_0 \in \Chat_{q_4}$. Then since $| a - b | > N^{\om_{\ell,3} + \eps }$ we see that $|i_0 - b | > N^{\om_{\ell,3} + \eps/2}$.  Hence,
 \beq
 \UBth_{bi_0} (0, t) \leq \frac{1}{ N^{D_2}}.
 \eeq
  Therefore,
  \beq
  \UBth_{ba} (s, t) \leq \frac{1}{ N^{D_2-1}}.
  \eeq
  Let us now consider the case $a \notin \Chat_{q_5}$ and $b \in \Chat_{q_3}$ or $b <0$.  Wlog we can take $a>0$.   Fix $q_6$ s.t. $q_3 < q_6 < q_5$.  If $i$ is such that either $i \in \Chat_{q_6}$ or $i \leq 0$, then $\UBth_{ai} (0, s) \leq N^{-D_2}$.  Hence, there is an $i_0$ s.t. $\UBth_{ai_0 } \geq N^{-1}$ and $i_0 >0$ and $i_0 \notin \Chat_{q_6}$.  But then since $b \in \Chat_{q_3}$ or $b<0$ we get that
  \beq
  \UBth_{b i_0} (0, t) \leq \frac{1}{ N^{D_2}}
  \eeq
  and this yields the claim as before.  \qed

\subsection{Kernel estimate}
In this section we prove Lemma \ref{lem:prof1}.  It is split into two parts, an energy estimate and a Duhamel expansion.
\subsubsection{Energy estimate}
  Let $\B$ be  as in Section \ref{sec:shortrange}.  In this subsection our goal is to prove the following energy estimate for $\UB$.    The argument is very similar to that in \cite{Gap}.  The major difference is that in the duality part of Nash's argument we have to be a little careful with the support of the functions, as we do not know that rigidity holds for all particles $i$.  To compensate for this we use the finite speed estimates from the previous section.

We recall the semigroup $\UB$ for the short-range operator $\B$ associated with in short-range set $\A_{q_*}$ with parameters $q_*$ and $\om_\ell$ from Section \ref{sec:ubref}.
\bel \label{lem:energ}
Fix $0 < q_3 < q_*$.  Let $a \in \Chat_{q_3}$.  Let $\eps >0, D>0$.  Assume that Section \ref{sec:ubref} \ref{item:rho2}-\ref{item:fiji} hold.  There is an event $\F_\alpha$ which holds with probability $\pp [ \F_\alpha ] \geq 1 - N^{-D}$ on which the following estimates hold.  For every $0 \leq s \leq t \leq 10 t_1$ and every $i$,
\beq
\UB_{ia} (s, t) \leq \frac{ N^{\eps}}{ N (t-s)}.
\eeq
\eel
\proof 
Recall from \cite{Gap} the inequality
\beq
||u||_4^4 ||u||_2^{-2} \leq C \sum_{i, j \in \zz} \frac{ (u_i - u_j )^2}{ (i-j)^2}
\eeq
which holds for sequences $u : \zz \to \rr$.   Fix $q_3 < q_5 < q_*$.  Let $g (u) = \UB (s,u) g (s)$ where $g(s)$ has support only in the indices $i \in \Chat_{q_5}$ and
\beq
||g(s)||_1=1.
\eeq
Extending $g (u)$ by $0$ to all of $\zz$ we apply the above inequality to $g(u)$ and obtain, with overwhelming probability,
\begin{align}
||g (u) ||_4^4 ||g (u) ||_2^{-2} & \leq C \sum_{ i, j \in \zz} \frac{ ( g_i (u) - g_j (u) )^2}{ (i-j)^2} 
\leq C \sum_{ |i-j| \leq \ell, i, j \in \Chat_{q_*} } \frac{ (g_i (u) - g_j (u) )^2}{ (i-j)^2} + C \frac{N}{\ell} ||g (u) ||_2^2 + \frac{ 1}{N^D} \notag\\
&\leq N^\eps \langle g (u) , \B g (u) \rangle + C ||g (u)||_2^2 \frac{N}{\ell} + \frac{1}{N^{D}}.
\end{align}
We used the fact that for $0 \leq u_1 \leq u_2 \leq t_1$,
\beq \label{eqn:ubtijbd}
|\UB_{ij} (u_1, u_2) | \leq \frac{1}{N^D}, \qquad \mbox{for }(i, j) \in \{ i, j :  i \in \Chat_{q_5}, j \notin \Chat_{q_*} \} \cup \{i, j :    j \in \Chat_{q_5}, i \notin \Chat_{q_*} \}, 
\eeq
which holds due to Theorem \ref{thm:finitespeed}, with overwhelming probability.
Therefore
\beq
\del_t ||g (u) ||_2^2 = - \langle g(u), \B g(u) \rangle \leq - N^{-\eps} ||g (u) ||_2^4 +C  ||g (u)||_2^2  \frac{N}{\ell} + \frac{1}{N^{D}}.
\eeq
Above, we used the Holder inequality $||g(u)||_2\leq ||g(u)||_1^{1/3} ||g(u)||_4^{2/3} \leq ||g(s)||_1^{1/3} ||g(u)||_4^{2/3}= ||g(u)||_4^{2/3}$ for $u \geq s$.   Since $t_1 N/\ell \ll 1$ we see that this implies
\beq
||g(t) ||_2 \leq C N^\eps (t-s)^{-1/2}.
\eeq
We have therefore proven that 
\beq
|| \UB(s, t) g ||_2 \leq  (t-s)^{-1/2} N^\eps ||g||_1
\eeq
for every $g$ supported in $i \in \Chat_{q_5}$.

The above argument clearly also applies to $(\UB(s, t))^T$ (in particular note that the bound \eqref{eqn:ubtijbd} is symmetric in $i$ and $j$). 

Fix now $q_4$ satisfying $q_3 < q_4 < q_5$.   Let now $f$ be supported in $i \in \Chat_{q_4}$ and have $||f||_2 =1$.  Then we have with overwhelming probability,
\beq
|| \UB(s, t) f||_\infty = \sup_{||g||_1=1} \langle \UB(s, t) f, g \rangle \leq \sup_{ ||g||_1=1, g_i = 0, i \notin \Chat_{q_5}} \langle \UB(s, t) f, g \rangle + \frac{1}{N^D} \label{eqn:energ1}
\eeq
where we used the fact that $|( \UB(s, t) f)_i | \leq N^{-D}$ for any $D>0$ for $i \notin \Chat_{q_5}$ which holds due to Theorem \ref{thm:finitespeed}.  For $g$ as in the RHS of \eqref{eqn:energ1} we have,
\beq
\langle \UB(s, t) f, g \rangle  = \langle  f, (\UB(s, t))^T g \rangle \leq ||f||_2 || (\UB(s, t))^T g ||_2 \leq N^{\eps} (t-s)^{-1/2}.
\eeq
This proves that for $f$ supported in $i \in \Chat_{q_4}$ we have
\beq
|| \UB(s, t) f||_\infty \leq N^{\eps} (t-s)^{-1/2} ||f||_2.
\eeq
Lastly, let $f$ have support in $\Chat_{q_3}$ and $||f||_1 = 1$.  Applying what we have proved above we have with overwhelming probability, with $u = s+(t-s)/2$,
\begin{align}
|| \UB (s, t) f ||_\infty &= ||  \UB (u, t)\UB (s, u) f ||_\infty  \leq  || \UB (u,t) \1_{ \Chat_{q_4} }\UB (s, u) f ||_\infty + \frac{1}{N^D} \notag\\
 & \leq C\frac{ N^\eps}{(t-s)^{1/2} } || \1_{ \Chat_{q_4}} \UB (s, u) f ||_2 + \frac{1}{N^D} \leq C\frac{ N^\eps}{(t-s)^{1/2} } || \UB (s, u)f ||_2 + \frac{1}{N^D} \notag\\
 &\leq C \frac{ N^{2 \eps}}{ t-s} ||f||_1.
\end{align}
This yields the claim. \qed

\subsubsection{Duhamel expansion}\label{sec:duham}
We want to prove the following.
\bel \label{lem:duhamell}
Let $D>0$ and $\eps_1, \eps_2, \eps_3 >0$ and $0 < q_3 < q_*$.  Assume that Section \ref{sec:ubref} \ref{item:rho2} - \ref{item:fiji} hold.  Let $\ell_3$ be a scale satisfying
\beq
N^{\eps_3} \leq \ell_3 \leq N.
\eeq
 There exists an event $\F_\alpha = \F_\alpha (\ell_3)$ with probability $\pp [ \F_\alpha ] \geq 1 - N^{-D}$ on which the following estimates hold.  For every $0 \leq s \leq t \leq 10  t_1$ which satisfy $N (t-s) \leq N^{-\eps_1} \ell_3$ and indices $a$ and $p$ satisfying $a, p \in \Chat_{q_3}$ and 
\beq
|a-p| \geq N^{\eps_1} \ell_3
\eeq
we have
\beq
\UB_{ap} (s, t) \leq  N^{\eps_2} \frac{  N(t-s) + 1}{ (a-p)^2}.
\eeq
\eel
By taking a sequence of at most $N$ scales $\ell_3 = N^\eps, N^\eps+1, \cdots $ we easily see that Lemma \ref{lem:duhamell} implies the following estimate.
\bel \label{lem:duham}
Let $D >0$ and $\eps_1, \eps_2$ and $0 < q_3 < q_*$.  There is an event $\F_\alpha$ with probability $\pp [ \F_\alpha ] \geq 1 - N^{-D}$ on which the following holds.  For every $0 \leq s \leq t \leq 10 t_1$ and pair of indices $a, p$ satisfying $a, p \in \Chat_{q_3}$ and
\beq
|a-p | \geq N^{\eps_1} \left[ 1 \vee (N (t-s) ) \right]
\eeq
we have
\beq
\UB_{ap} (s, t) \leq N^{\eps_2} \frac{ N (t-s) + 1 }{ (a-p)^2}.
\eeq
\eel

\noindent{\bf Proof of Lemma \ref{lem:duhamell}}. We can work under the assumption that the estimates of Lemma \ref{lem:rig} hold.  We assume $a < p$.  The proof for $a > p$ is identical.  Write 
 \beq
 \B = \S + \R
 \eeq
where $\S =\B_3$ is defined as at the start of Section \ref{sec:expfs} with $\ell_3$ as in the statement of Lemma \ref{lem:duhamell},  and $\R$ is defined implicitly by the above equality. 
For notational simplicity we set $s=0$, but this has no effect on the proof.  For each $M$ we have
\begin{align}
\UB (0, t) &= \US (0, t) + \sum_{i=1}^M \int_{0 \leq s_1 \leq \cdots s_i \leq t } \US (s_i, t) \R \US (s_{i-1}, s_i) \cdots \R \US (0, s_1) \d s_1 \cdots \d s_i \notag\\
&+ \int_{0 \leq s_1 \cdots \leq s_{M+1} \leq t } \UB (s_{M+1}, t) \R \US (s_{M}, s_{M+1} ) \cdots \R \US (0, s_1) \d s_1 \cdots \d s_{M+1} \notag\\
&=: \US (0, t) + \sum_{i=1}^M \int_{0 \leq s_1 \cdots \leq t } A_i \d s_1 \cdots \d s_{i} + B_{M+1}.
\end{align}
By Theorem \ref{thm:finitespeed} we have $\US_{ap} (0, t) \leq N^{-100}$.  We next deal with the term $B_{M+1}$.  Using the estimates of Lemma \ref{lem:rig}, it is easy to check that for every $i$ we have
\beq
\sum_j | \R (i, j) | \leq C \frac{N}{\ell_3}
\eeq
and so 
\beq
|| \R ||_{\ell^p \to \ell^p } \leq C \frac{N}{\ell_3}
\eeq
for every $p$.  Hence,
\beq
|B_{M+1} (i, j) | \leq C_M N^2 \left( \frac{ N t } { \ell_3 } \right)^M \leq \frac{1}{N^{100}}
\eeq
for $M$ a large constant depending only on $\eps_1$.  By Lemma \ref{lem:aibd} below we see that
\beq
\left| \int_{0 \leq s_1 \cdots \leq t } A_i (a, p) \d s_1 \cdots \d s_{i} \right| \leq C  \frac{Nt}{ (a-p)^2} \left( \frac{ Nt}{\ell_3} \right)^{i-1}
\eeq
Note that we used that $i \leq M$ and so $a \leq L_i$.  To estimate the integral we used the bound \eqref{eqn:fsrev1} on the integrand and that the region of integration has size $C t^i$.  
This concludes the proof. \qed
\bel \label{lem:aibd} Let $A_i$, $\R$, etc. be as above.  Let $a<p$, $M$ be as above.  Define $L_i$ as
\beq
L_i := p - i\frac{ p-a}{100 M}.
\eeq
Then for $j \leq L_i$ we have
\beq \label{eqn:fsrev1}
A_i (j, p) \leq  C_M \frac{ N^i}{ (p-a)^2 \ell_3^{i-1} }.
\eeq
\eel
\remark The proof will be by induction on $i$, using the estimate from the case $i-1$.  We rewrite $A_{i+1}(j, p)$ in terms of a sum which involes matrix elements $A_i(l, p)$.  If $j \leq L_{i+1}$ and $ l \geq L_{i}$ then by definition $|l-j| \geq c |p-a|$ and so we can apply Theorem \ref{thm:finitespeed} to simplify the sum.  In the case that $l \leq L_{i}$ we can use the induction assumption. 
\proof The proof is by induction on $i$.  Define 
\beq
R:= N^{\eps_1/4} \ell_3.
\eeq
Fix $q_3 < q_5 < q_*$.  Theorem \ref{thm:finitespeed} implies that for $(a, b)$ s.t. $|a-b| > R$  and either $a$ or $b \in \Chat_{q_5}$, we have with overwhelming probability
\beq \label{eqn:fsduham}
| \US_{a b} (s, t) | \leq \frac{1}{N^D}
\eeq
for any $D>0$, and $0 \leq s \leq t \leq \ell_3$.  We have that 
\beq \label{eqn:bb1}
A_1 (j, p) = \sum_{k, l} \US_{j k } \R_{kl} \US_{l p}. 
\eeq
Since the matrix elements of $\R$ are bounded by (say) $N^2$, it suffices by \eqref{eqn:fsduham} to consider only terms in \eqref{eqn:bb1} that satisfy $|l-p| \leq R$.  
Since $j \leq L_1$ we can apply Theorem \ref{thm:finitespeed} and ignore terms satisfying
\beq
k \geq L_1+ R.
\eeq
  For such $l$ and $k$, using the fact that $j \leq L_1$ and $R \ll |p-a|$ we see that $|l-k| \geq c (p-a)$ and so 
\beq
|A_1 (j, p) | \leq \sum_{k, l} \US_{j, k} \frac{ N}{(p-a)^2} \US_{l, p} \leq \frac{N}{(p-a)^2} ,
\eeq
where in the second inequality we used the $\ell^p \to \ell^p$ boundedness of $\US$.   For $\mathcal{R}$ we used in the first inequality that for $|l - k | \geq c (p-a)$,
\beq
\mathcal{R} (k, l) = \frac{1}{N} \frac{1}{ ( \hatz_k - \hatz_{l} )^2} \leq \frac{CN}{(p-a)^2}
\eeq
with overwhelming probability. 
 Now for $A_{i+1}$ and $j \leq L_{i+1}$ we write
\beq \label{eqn:duhai1}
A_{i+1} (j, p) = \sum_{k}\sum_{  l \geq L_i} \US (j, k) \R (k, l) A_i (l, p) + \sum_{ k}\sum_{ l \leq L_i } \US (j, k) \R (k, l) A_i (l, p)
\eeq
We start with estimating the first sum.  By Theorem \ref{thm:finitespeed} we can restrict the $k$ summation to terms satisfying
\beq
k \leq L_{i+1} + R.
\eeq
Then since $l \geq L_i$ we get that $|k-l| \geq c| p-a|$ and so
\begin{align}
\left| \sum_{k}\sum_{  l \geq L_i} \US (j, k) \R (k, l) A_i (l, p)  \right| &\leq C \frac{ N}{ (p-a)^2} \sum_{k, l} \US (j, k) A_i (l, p) \leq C\frac{ N}{ (p-a)^2} || \US||_{\infty \to \infty} ||A_{i-1}||_{1\to1 } \notag\\
& \leq C \frac{N}{ (p-a)^2} \frac{ N^i}{ \ell_3^{i} }.\label{eqn:duhamai4}
\end{align}
For the second sum in \eqref{eqn:duhai1} we apply the induction assumption  and obtain
\begin{align}
\left| \sum_{ k}\sum_{ l \leq L_i } \US (j, k) \R (k, l) A_i (l, p) \right| &\leq \sum_{ k}\sum_{ l \leq L_i } \US (j, k) \R (k, l) C \frac{ N^i}{ \ell^{i-1} (p-a)^2 } \leq C || \US \R ||_{\infty \to \infty} \frac{ N^i}{ \ell_3^{i-1} (p-a)^2 } \notag\\
&\leq C \frac{N}{ \ell_3} \frac{ N^i}{ \ell_3^{i-1} (p-a)^2 } .
\end{align}
This completes the proof. \qed
\subsection{Profile of random kernel}
Combining Lemma \ref{lem:energ} and \ref{lem:duham} we easily obtain the following.
\bet \label{thm:sec4prof} Fix $0 < q_3 < q_*$.  Let $D$ and $\eps >0$.  Assume that Section \ref{sec:ubref} \ref{item:rho2}-\ref{item:fiji} hold.  There is an event $\F_\alpha$ with probability $\pp [ \F_\alpha ] \geq 1 - N^{-D}$ on which the following estimate holds.  We have for every $0 \leq s \leq t \leq t_1$ and $j, k \in \Chat_{q_3}$,
\beq
\UB_{ij} (s, t) \leq N^\eps \frac{|t-s|\vee N^{-1}}{ (( i-j)/N)^2 +( (t-s)\wedge N^{-1})^2 }
\eeq
\eet


\section{Regularity of hydrodynamic equation}\label{sec:hydro}
In this section we analyze the limiting equation \eqref{eqn: Ketadef}, deriving in particular the estimates in Lemma \ref{lem: Ketalemma}. We introduce the kernel
\begin{equation}
 K_{\eta_1} f(x) :=\int_{|x-y|\le \eta_1} \frac{f(x)-f(y)}{(x-y)^2}\mathrm{d}y. \label{eqn: K}
\end{equation}
The integral is understood in a principal value sense.
 
By \cite[Theorem 1.4]{CKK}, for the heat kernel corresponding to $K_1$, we have the estimates:
\begin{equation}\label{eqn: pt1smalldist}
p_1(t,x,y)\asymp \frac{1}{t} \wedge \frac{t}{|x-y|^2}, \quad 0< t\le 1,
\end{equation}
for $|x-y|<1$ and
\begin{equation}
\label{eqn: pt1largedist}
p_1(t,x,y) \asymp e^{-c|x-y|\log\frac{|x-y|}{t}}, \quad 0< t\le 1,
\end{equation}
for $|x-y|\ge 1$.

If $f(t,x)$ satisfies $\partial_t f(t,x) = K_1 f(t,x)$, then $g(t,x)=f(t/\eta_\ell, x/\eta_\ell )$ satisfies $\partial_t g(t,x) = K_{\eta_\ell} g(t,x)$, so from \eqref{eqn: pt1smalldist}, \eqref{eqn: pt1largedist} we obtain:
\begin{equation} \label{eqn: Ketadiag}
p_{\eta_\ell}(t,x,y)=(1/\eta_\ell)p_1(t/\eta_\ell ,x/\eta_\ell,y/\eta_\ell)\asymp \frac{1}{t} \wedge \frac{t}{|x-y|^2}, \quad 0\le t\le \eta_\ell,
\end{equation}
when $|x-y|< \eta_1$ and
\begin{equation} \label{eqn:hydro8}
p_{\eta_\ell}(t,x,y) \asymp \frac{1}{\eta_\ell}e^{-c|x-y|/\eta_\ell \log \frac{|x-y|}{t}},
\end{equation}
for $t\le \eta_\ell$ and $|x-y|\ge \eta_\ell$. \eqref{eqn: Keta-shortrange} and \eqref{eqn: Keta-expdecay} follow directly from this. In the rest of the proof, we will also use that the upper bound in \eqref{eqn: Ketadiag} remains true for all $x,y$, $t\le \eta_\ell$ (See \cite[Proposition 2.2, i)]{CKK2}.):
\begin{equation}\label{eqn: Ketadiag-global}
p_{\eta_\ell}(t,x,y)\le \frac{C}{t} \wedge \frac{t}{|x-y|^2}, \quad 0< t\le \eta_\ell.
\end{equation}

We will first estimate the spatial derivatives of $p_{\eta_\ell} (t, x, y)$.  Letting $\Lambda=(-\Delta)^{1/2}$ be the half-Laplacian with kernel
\[\int \frac{f(x)-f(y)}{(x-y)^2}\mathrm{d}y.\]
The corresponding heat kernel is 
\begin{equation}
\label{eqn: Cauchy-kernel}
p_\Lambda(t,x,y):= (e^{t\Lambda}\delta)(x) = \frac{1}{\pi}\frac{t}{t^2+(x-y)^2}.
\end{equation}
We have the Duhamel formula
\begin{align}
p_{\eta_\ell}(t,x,y)&=(e^{tK_{\eta_\ell}}\delta_y)(x)=p_\Lambda(t,x,y) +\int_0^{t} e^{(t-s)\Lambda}(K_{\eta_\ell}-\Lambda)f_{\eta_\ell}(s)\,\mathrm{d}s. \label{eqn: duhamelKeta2}
\end{align}
Here we have denoted for simplicity
\[f_{\eta_\ell}(x, s) := p_{\eta_\ell}(s,x,y).\]
Since $e^{t\Lambda}$ is smoothing, the equality \eqref{eqn: duhamelKeta2} shows in particular that $p_\eta(t,\cdot,y)$ is in $C^\infty(\mathbb{R})$. 

We will estimate the first spatial derivative by differentiating \eqref{eqn: duhamelKeta2}. The operators $\Lambda$, $K_{\eta_\ell}$ are translation invariant, so for general $k$ we have
\[\partial^k_x\int_0^{t} e^{(t-s)\Lambda}(K_{\eta_\ell}-\Lambda)f_{\eta_\ell}(s)\,\mathrm{d}s = \int_0^{t} e^{(t-s)\Lambda}\partial^k_x (K_{\eta_\ell}-\Lambda) f_{\eta_\ell}(s)\,\mathrm{d}s.\]
By direct computation we have
\[(K_{\eta_\ell}-\Lambda) f(x)= \frac{2}{\eta_\ell}f(x)-\int_{|x-z|\ge \eta_\ell}\frac{f(z)}{(x-z)^2}\,\mathrm{d}z.\]
Differentiating, we find,
\begin{equation} \label{eqn: dkKeta}
\partial^k_x (K_{\eta_\ell}-\Lambda)f(x) = \frac{2}{\eta_\ell}f^{(k)}(x)+(-1)^{k+1} (k!)\int_{|x-z|\ge \eta_\ell}{f(z)}{(x-z)^{-k-2}}\,\mathrm{d}z.
\end{equation}
We will first derive an estimate on the second term.  In order to estimate the first we later derive a H\"older estimate for $p_{\eta_\ell} (t, x, y)$.


We first derive the following estimate.
\bel \label{lem:gbd}
We have,
\beq \label{eqn:hydro1}
  \int_{ |x-z| \geq \eta_\ell}  \frac{ | \fl (z) |}{ |x-z|^{k+2} } \d z \leq \frac{C}{ \eta_\ell^{k+1}}\left( \frac{s}{ s^2 + (x-y)^2} +  \frac{\eta_\ell}{ \eta_\ell^2 + (x-y)^2} \right).
\eeq
for $0 \leq  s \leq t \leq \eta_\ell$.
\eel
\proof
 By \eqref{eqn: Ketadiag-global} we can estimate
\begin{align}
\int_{ |x-z| \geq \eta_\ell}  \frac{ | \fl (z) |}{ |x-z|^{k+2} } \d z| &\leq C \int_{ |z| \geq \eta_\ell, |z| \leq \frac{1}{2} |x-y|} \frac{1}{ |z|^{k+2}} \frac{ s}{ s^2 + (z - (x -y))^2 } \d z  \notag\\
&+ C \int_{ |z| \geq \eta_\ell, |z| \geq \frac{1}{2} |x-y|} \frac{1}{ |z|^{k+2}} \frac{ s}{ s^2 + (z - (x -y))^2 } \d z
\end{align}
For the first term we have
\beq
 \int_{ |z| \geq \eta_\ell, |z| \leq \frac{1}{2} |x-y|} \frac{1}{ |z|^{k+2}} \frac{ s}{ s^2 + (z - (x -y))^2 } \d z   \leq C \frac{s}{s^2 + (x-y)^2} \int_{ |z| \geq \eta_\ell} \frac{1}{|z|^{k+2}} \leq  \frac{C}{\eta_\ell^{k+1}} \frac{s}{s^2 + (x-y)^2} .
\eeq
For the second term we first consider the case $|x-y | \geq \eta_\ell/2$.  We have,
\beq
\int_{ |z| \geq \eta_\ell, |z| \geq \frac{1}{2} |x-y|} \frac{1}{ |z|^{k+2}} \frac{ s}{ s^2 + (z - (x -y))^2 } \d z \leq \frac{C}{ |x-y|^{k+2} } \int \frac{s}{ s^2 + z^2} \d z  \leq \frac{C}{ |x-y|^{k+2} } \leq \frac{C}{\eta_\ell^{k+1}} \frac{ \eta_\ell}{ \eta_\ell^2 + (x-y)^2}.
\eeq
In the case $|x-y| \leq \eta_\ell/2$ we have $|z - (x-y) | \geq c  |z|$ for $|z| \geq \eta_\ell$ and so
\beq
\int_{ |z| \geq \eta_\ell, |z| \geq \frac{1}{2} |x-y|} \frac{1}{ |z|^{k+2}} \frac{ s}{ s^2 + (z - (x -y))^2 } \d z \leq \frac{C s}{ \eta_\ell^{k+1}} \int_{ |z| \geq \eta_\ell} \frac{1}{ |z|^3} \d z \leq C \frac{s}{ \eta_\ell^{k+3}} \leq \frac{C}{ \eta_\ell^{k+1}} \frac{s}{ s^2 + (x-y)^2}
\eeq
because $\eta_\ell \geq s$ and $\eta_\ell \geq c |x-y|$. \qed

We now derive the following H\"older estimate for $p_{\eta_\ell}$.
\bel \label{lem:holder}
For any $0 \leq \alpha < 1$ and $0 \leq t \leq \eta_\ell$ we have
\beq
\left| p_{\eta_\ell} (t, x, y) - p_{\eta_\ell} (t, z, y)  \right| \leq \frac{C |x-z|^\alpha}{ t^{1 + \alpha} }.
\eeq
\eel
\proof We have 
\beq
p_{\eta_\ell} (t, x, y) = p_\Lambda (t, x, y) + \int_0^t ( \e^{ (t-s) \Lambda } g (s ) ) (x) \d s
\eeq
where we denoted $g(s, w) = \left[ ( K - \Lambda ) \fl (s) \right] (w)$. The first term satifies the desired bound so we nned only estimate the second term.  From \eqref{eqn: Ketadiag-global} and \eqref{eqn:hydro1} we have the estimate
\beq \label{eqn:hydro2}
|g (s, w) | \leq \frac{C}{ \eta_\ell} \left( \frac{s}{s^2 + w^2} + \frac{\eta_\ell}{ \eta_\ell^2 + w^2 } \right).
\eeq
  We estimate
\begin{align}
& \left|  ( \e^{ (t-s) \Lambda } g (s ) ) (x)- ( \e^{ (t-s) \Lambda } g (s ) ) (z) \right| \notag\\
\leq & \frac{1}{\pi} \int \left| \frac{t-s}{ (t-s)^2 + (x-w)^2} - \frac{ t-s}{ (t-s)^2 + (z-w)^2} \right| g (s, w) \d w \notag\\
\leq & C (t-s) \int \frac{ |x-z|^\alpha ( |x-w|^{2-\alpha} + |z-w|^{2-\alpha} )}{  ( ( t-s)^2 + (x-w)^2) ( (t-s)^2 + (z-w)^2)} |g (s, w) | \d w \notag\\
\leq &  \frac{ C  |x-z|^\alpha }{ (t-s)^\alpha } \int \frac{ t-s}{ (t-s)^2 + (x-w)^2} |g (s, w) | \d w \notag\\
+ &\frac{ C  |x-z|^\alpha }{ (t-s)^\alpha } \int \frac{ t-s}{ (t-s)^2 + (z-w)^2} |g (s, w) | \d w
\end{align}
Using \eqref{eqn:hydro2} and the bound
\begin{align} \label{eqn:hydro5}
\int \frac{a}{a^2 + w^2} \frac{b}{b^2 +(w-c)^2} \d w \leq  C \frac{ b \vee a}{ ( b \vee a )^2 + c^2 }
\end{align}
we get
\beq
 \left|  ( \e^{ (t-s) \Lambda } g (s ) ) (x)- ( \e^{ (t-s) \Lambda } g (s ) ) (z) \right| \leq \frac{C |x-z|^\alpha}{ (t-s)^\alpha} \frac{1}{t}.
\eeq
The claim follows after integrating in $s$. \qed

We now prove the following. 
\bel
Let  $ 0 < \alpha <1$.    We have,
\beq \label{eqn:hydro6}
\left| \del_x p_{\eta_\ell} (t, x, y) \right| \leq C \left( \frac{1 + | \log \delta | t/\eta_\ell}{ t^2 + (x-y)^2} + \frac{ \delta^\alpha}{t} \right)
\eeq
for any $t \leq \eta_\ell$ and $0 < \delta < 1$.
\eel
\proof We have
\beq
\del_x p_{\eta_\ell}(t, x, y) = \del_x p_\Lambda (t, x, y) - \int_{0}^t \e^{ (t-s) \Lambda} \del_x [ ( K - \lambda ) \fl ](x) \d s.
\eeq
We write
\begin{align} 
 \e^{ (t-s) \Lambda} \del_x [ ( K - \lambda ) \fl ] (x) &= - \int  \left( \del_w  \frac{ t-s}{ (t-s)^2 + (w-x)^2 } \right) \fl (w) \d w \label{eqn:hydro3} \\
&+ \int \frac{t-s}{ (t-s)^2 + (w-x)^2 } g_1 (s, w) \d w \label{eqn:hydro4}
\end{align}
where
\beq
g_1 (s, w) := \int_{ |w-z| \geq \eta_\ell } \frac{1}{ (w-z)^3} \fl (z) \d z.
\eeq
From \eqref{eqn:hydro1} we have
\beq
|g_1 (s, w) | \leq \frac{C}{ \eta_\ell^2} \left( \frac{s}{s^2 + (w-y)^2} + \frac{ \eta_\ell}{ \eta_\ell^2 + (w-y)^2 }\right).
\eeq
For \eqref{eqn:hydro4} we have
\begin{align}
& \left| \int \frac{t-s}{ (t-s)^2 + (w-x)^2 } g_1 (s, w) \d w \right| \leq   \frac{C}{\eta_\ell^2} \int \frac{t-s}{ (t-s)^2 + (w-(x-y))^2 }  \left( \frac{s}{s^2 + w^2} + \frac{ \eta_\ell}{ \eta_\ell^2 + w^2 }\right) \d w.
\end{align}
Using \eqref{eqn:hydro5} and $t \leq \eta_\ell$ one easily concludes
\beq
 \frac{C}{\eta_\ell^2} \int \frac{t-s}{ (t-s)^2 + (w-(x-y))^2 }  \left( \frac{s}{s^2 + w^2} + \frac{ \eta_\ell}{ \eta_\ell^2 + w^2 }\right) \d w \leq \frac{C}{t} \frac{1}{ t^2 + (x-y)^2}.
\eeq
 For the term \eqref{eqn:hydro3} we first note that
\begin{align*}
 &\int \partial_w \frac{t-s}{(t-s)^2+(w-x)^2} f_{\eta_\ell}(s,w) \,\mathrm{d}w=\int \partial_w \frac{t-s}{(t-s)^2+(w-x)^2} (f_{\eta_\ell}(s,w)-f_{\eta_\ell}(s,x)) \,\mathrm{d}w.
\end{align*}
Hence, using \eqref{eqn: Ketadiag-global} and Lemma \ref{lem:holder} we can estimate
\begin{align}
\left|\int_0^t\int \partial_w \frac{t-s}{(t-s)^2+(w-x)^2} f_{\eta_\ell}(s,w) \,\mathrm{d}w\mathrm{d}s\right| &\leq C \int_0^{t(1-\delta)}\int \frac{(t-s)|x-w|}{((t-s)^2+(w-x)^2)^2} \frac{s}{s^2+(w-y)^2} \,\mathrm{d}w\mathrm{d}s \label{eqn: dxKeta-1}\\
&\ + C \int_{t(1-\delta)}^t \int \frac{(t-s)|x-w|^{1+\alpha}}{((t-s)^2+(w-x)^2)^2} s^{-1-\alpha}\,\mathrm{d}w \mathrm{d}s \label{eqn: dxKeta-2}
\end{align}
For \eqref{eqn: dxKeta-1}, we use \eqref{eqn:hydro5} to obtain
\begin{align*}
\int_0^{t(1-\delta)}\int \frac{(t-s)|x-w|}{((t-s)^2+(w-x)^2)^2} \frac{s}{s^2+(w-y)^2} \,\mathrm{d}w\mathrm{d}s &\leq  \int_0^{t(1-\delta)} \frac{1}{t-s}\frac{t-s}{(t-s)^2+(w-x)^2}\frac{s}{s^2+(w-y)^2}\,\mathrm{d}w\mathrm{d}s\\
\leq C \frac{t}{t^2+ (x-y)^2} \int_0^{ t (1-\delta)} \frac{1}{t-s} &\leq  C | \log{\delta} | \cdot \frac{t}{t^2+(x-y)^2}.
\end{align*}
For \eqref{eqn: dxKeta-2}, we use
\[\frac{|w-x|^{1+\alpha}}{(t-s)^2+(w-x)^2}\le \frac{1}{(t-s)^{1-\alpha}}\]
to obtain
\begin{align*}
\int_{t(1-\delta)}^t \int \frac{(t-s)|x-w|^{1+\alpha}}{((t-s)^2+(w-x)^2)^2} s^{-1-\alpha}\,\mathrm{d}w \mathrm{d}s &\le \int_{t(1-\delta)}^t (t-s)^{-1+\alpha} \int \frac{t-s}{(t-s)^2+(w-x)^2} s^{-1-\alpha}\,\mathrm{d}w\mathrm{d}s\\
&\le \frac{C}{1-\alpha}\frac{\delta^\alpha}{t}.
\end{align*}
This yields the claim. \qed

We now can conclude with estimates of the spatial derivatives of $p_{\eta_\ell} (t, x, y)$.  
\bet  Fix $D_1 >0$ and $D_2 >0$ and $\delta >0$.   Let $\eta_\ell = N^{\om_\ell}/N$ for each $k$ we have
\beq \label{eqn:hydro7}
\left| \del_x^k p_{\eta_\ell} (t,x, y)  \right| \leq \frac{C}{t^k} \frac{t}{t^2 + (x-y)^2 } + N^{-D_2},
\eeq
for $N^{-D_1} \leq t \leq N^{-\delta} \eta_\ell$.  For $|x-y|> N^{\eps} \eta_\ell$ we have
\beq \label{eqn:hydro8}
\left| \del_x^k p_{\eta_\ell} (t,x, y)  \right|  \leq N^{-D_2}.
\eeq
\eet
\proof  The bound \eqref{eqn:hydro7} for $k=1$ follows by taking $\delta = N^{-D}$ in \eqref{eqn:hydro6}.  For $k\geq 2$ we have by the Chapman-Kolmogorov equation and translation invariance:
\[\partial_x^{(k-1)}p_{\eta_\ell}(t,x,y) = \int p(t/2,x,w)p^{(k-1)}(t/2,w,y)\,\mathrm{d}y.\]
Differentiating once more, we have
\begin{equation} \label{eqn: der-chapman}
\partial_x^{(k)}p_{\eta_\ell}(t,x,y) = \int \partial_x p(t/2,x,w)p^{(k-1)}(t/2,w,y)\,\mathrm{d}y.
\end{equation}
The bounds \eqref{eqn:hydro7} for higher $k$ then follow by strong induction and \eqref{eqn:hydro5}.

For \eqref{eqn:hydro8} we use translation invariance and the Chapman-Kolmogorov equation to write for $x>0$,
\begin{align}
\del_x p_{\eta_\ell} (t, x, 0 ) &= \int  \del_w p_{\eta_\ell} (t/2, x, w)  p_{\eta_\ell} (t/2, w, 0) \d w \notag\\
&= \int_{w \leq x/2}  \del_w p_{\eta_\ell} (t/2, x, w)  p_{\eta_\ell} (t/2, w, 0) \d w + \int_{w > x/2}  \del_w p_{\eta_\ell} (t/2, x, w)  p_{\eta_\ell} (t/2, w, 0) \d w \notag\\
&= -\int_{w \leq x/2}   p_{\eta_\ell} (t/2, x, w) \del_w p_{\eta_\ell} (t/2, w, 0) \d w + \int_{w > x/2}  \del_w p_{\eta_\ell} (t/2, x, w)  p_{\eta_\ell} (t/2, w, 0) \d w \notag\\
& + p_{\eta_\ell} (t/2, x, x/2) p_{\eta_\ell} (t/2, x/2, 0).
\end{align}
If $|x| > N^\eps \eta_\ell$ we see that each term is $O (N^{-D})$ using \eqref{eqn:hydro8}.  The higher order derivatives can be handled similarly. \qed

Finally we handle the time derivative.
\bet Fix $D_1, D_2 >0$ and $\delta >0$.  For $N^{-D_1} \leq t \leq N^{-\delta} \eta_\ell$ we have
\beq
\left| \del_t p_{\eta_\ell} (t, x, y) \right| \leq \frac{C}{t^2 + (x-y)^2} + N^{-D_2}.
\eeq
\eet
\proof We write
\begin{align}
\del_t p_{\eta_\ell} (t, x, 0) = \int_{ |x-z| \leq \eta_{\ell}} \frac{ p_{\eta_\ell} (t, x, 0 ) - p_{\eta_\ell} (t, z, 0) }{(x-z)^2} \d z &= \int_{ |x-z| \leq t}  \frac{ p_{\eta_\ell} (t, x, 0 ) - p_{\eta_\ell} (t, z, 0) }{(x-z)^2} \d z \label{eqn:hydro9} \\
&+ \int_{\eta_\ell > |x-z| > t}  \frac{ p_{\eta_\ell} (t, x, 0 ) - p_{\eta_\ell} (t, z, 0) }{(x-z)^2} \d z 
\end{align}
An argument similar to the proof of Lemma \ref{lem:gbd} yields 
\begin{align}
&\left| \int_{\eta_\ell > |x-z| > t}  \frac{ p_{\eta_\ell} (t, x, 0 ) - p_{\eta_\ell} (t, z, 0) }{(x-z)^2} \d z  \right| \leq C \int_{|x-z| > t} \frac{1}{ (x-z)^2} \left( \frac{t}{t^2 + x^2} + \frac{t}{t^2 + z^2 } \right) \d z \leq \frac{C}{t^2 + x^2}.
\end{align}
We now turn to \eqref{eqn:hydro9}.  Note that
\beq
\sup_{|u-x| \leq t } \frac{1}{t^2 + u^2} \leq \frac{C}{t^2 + x^2}.
\eeq
Hence by a second order Taylor expansion,
\begin{align}
& \int_{ |x-z| \leq t}  \frac{ p_{\eta_\ell} (t, x, 0 ) - p_{\eta_\ell} (t, z, 0) }{(x-z)^2} \d z \notag\\
= & \int_{ |x-z| \leq t} \frac{ \del_x p_{\eta_\ell} (t, x, 0) }{ (x-z)} + \O\left( \frac{1}{t} \frac{1}{t^2 + x^2 } + N^{-D} \right) \d z  = \O \left( \frac{1}{t^2 + x^2} + N^{-D}\right).
\end{align}
This yields the claim. \qed

\section{Mesoscopic linear statistics}\label{sec: mesosection}
Let $\varphi_N$ be a sequence of twice differentiable functions such that:
\begin{gather}
\|\varphi_N\|_{L^\infty}, \|\varphi_N'\|_{L^1} \le C, \label{eqn: varphiLinfty}\\
\mathrm{supp} \varphi_N'(x) \cap [-t^{1/4},t^{1/4}] \subset (-t_1N^r,t_1N^r), \label{eqn: offIq-decay}
\end{gather}
for some $0<r<\omega_1/100$.

We will assume throughout that $t_1=N^{\omega_1}/N$. As for the parameter $t$, we assume $t_1N^r\ll t \ll 1$, and we will denote $t= N^{\omega_0}/N$. 

Finally, since the spectrum of $H_t$ is contained in $-N^{C_V}-1,N^{C_V}+1$ with overwhelming probability, there is no loss of generality in making the following assumption on the support of $\varphi_N$:
\begin{equation}\label{eqn: varphisupport}
\mathrm{supp} \varphi_N \subset [-N^{C_V}-2,N^{C_V}+2].
\end{equation}

We can now state the main result of this section.  After introducing some notation in Section \ref{sec:resolvnotation} we give an outline of the proof in  Section \ref{sec:mesooutline}.  The majority of the remainder of Section \ref{sec: mesosection} is concerned with the proof.
\begin{thm}\label{thm: linstat}
Let $\varphi_N$ be a sequence of real-valued $C^{2}(\mathbb{R})$ functions satisfying \eqref{eqn: varphiLinfty}, \eqref{eqn: offIq-decay}, \eqref{eqn: varphisupport} in addition to the following growth conditions on the derivatives:
\begin{equation}
\|\varphi_N^{(k)}\|_{L^\infty}\le C t_1^{-k+1}, \quad k=1,2,
\end{equation}
and 
\begin{equation}\label{eqn: H12-lwrbd}
\int\int\left(\frac{\varphi_N(x)-\varphi_N(y)}{x-y}\right)^2\,\mathrm{d}x\mathrm{d}y\ge c.
\end{equation}
Let the parameters $t=N^{\omega_0}/N$ and $t_1=N^{\omega_1}/N$ satisfy
\begin{align}
\omega_0 > \omega_1 > \omega_0/2 \label{eqn: omega-condition-1}
\end{align}
Then, uniformly in $|x|\le N^{\omega_1/8-\omega_0/16}$,
\beq \label{eqn:revd1}
\mathbb{E}[e^{i x[(\mathrm{tr}\varphi_N)(H_t)-\mathbb{E}\mathrm{tr}\varphi_N(H_t)]}] = \exp\left(-\frac{x^2}{2}V(\varphi_N)\right)+ \O_\prec(N^{\omega_0/4-\omega_1/2}).
\eeq
Here $V(\varphi_N)$ is a quadratic functional in $\varphi_N$ such that
\begin{equation}
\begin{split}
V(\varphi_N) &= -\frac{1}{\pi^2}\int_{-Ct}^{Ct}\varphi_N(\tau)(H\varphi_N')(\tau)\,\mathrm{d}s\mathrm{d}\tau +\O(1)\\
&= \frac{1}{2\pi^2}\int_{-Ct}^{Ct}\int_{-Ct}^{Ct}\left(\frac{\varphi_N(\tau)-\varphi_N(s)}{\tau-s}\right)^2\,\mathrm{d}s\mathrm{d}\tau+\O(1). \label{eqn:cc5}
\end{split}
\end{equation}
Here $C>0$ is some (small) constant, and $H$ denotes the Hilbert transform (see \ref{eqn: HT-def}).
In particular,
\[V(\varphi_N)\ge  c\log t/t_1 + \O(1)\]
if $\varphi_N = \int_0^x\chi(y/(t_1N^\alpha))p_{t_1}(0,y)\,\mathrm{d}y$.

If $\mathrm{supp} \varphi_N\subset (-N^rt_1,N^rt_1)$, then we have the more precise evaluation:
\begin{equation}\label{eqn: Vvar-compact}
\begin{split}
V(\varphi_N) &= -\frac{1}{\pi^2}\int\varphi_N(\tau)(H\varphi_N')(\tau)\,\mathrm{d}\tau+\O(N^{\omega_0/2-\omega_1}N^{2r})\\
&=\frac{1}{2\pi^2}\int\int\left(\frac{\varphi_N(\tau)-\varphi_N(s)}{\tau-s}\right)^2\,\mathrm{d}s\mathrm{d}\tau+\O(N^{\omega_0/2-\omega_1}N^{2r}).
\end{split}
\end{equation}
\end{thm}

\remark We make several comments concerning Theorem \ref{thm: linstat} and the many conditions in the statement.
\begin{enumerate}
\item The first inequality of \eqref{eqn: omega-condition-1} ensures that the scale of the function is smaller than the time scale of DBM. See the remark after the statement of Theorem \ref{thm: mesostatement}. 
\item The second inequality of \eqref{eqn: omega-condition-1} is technical, but removing it requires substantial modification of some of the estimates below. It is used in particular to simplify the handling of some error terms in the proof of Proposition \ref{prop: I2prop}, which is key in deriving the theorem. 
\item The condition \eqref{eqn: H12-lwrbd} ensures the limiting random variable is non-degenerate, that is, its variance is bounded below. It is used only in the final integration at \eqref{eqn: psi-int}. Our method can be extended to cover the case of vanishing variance, but we will have no need for such an extension.
\item A typical situation in which Theorem \ref{thm: linstat} holds with the approximation \eqref{eqn: Vvar-compact} is when 
\[\varphi_N(x) = \varphi(x/t_1),\]
where $\varphi$ is some smooth   function either compactly supported or vanishing as $|x| \to \infty$. This setting has been studied extensively in the random matrix theory literature (see for example \cite{DJ},  \cite{meso1}, \cite{meso2}), and is typically what is being referred to when one speaks of ``linear statistics of mesoscopic observables''. The more general theorem above is essential for the main result of this paper.
\item If the functions $\varphi_N$ do not have spatial decay
, the variance of the linear statistics grows logarithmically. This should be compared to the well-known fact that the variance of the number of eigenvalues in an interval grows like $\log N$. See \cite{costinlebowitz, orourke}. A function $\varphi_N$ with ``large'' (compared to $t$) support, but whose derivative is supported in a region of size $t_1$ is, up to a linear transformation, an approximation on scale $t_1$ of an indicator function.
\end{enumerate}

Concerning the last remark above, we note that the fact that we allow for non-compactly supported functions (which is required for the proof) causes substantial technical difficulties. Many alternative approaches would be viable if it sufficed to consider compactly supported $\varphi$.

\subsection{Notation for resolvents} \label{sec:resolvnotation}

A central role will be played by the resolvent matrix
\begin{equation}\label{eqn: Gdef}
G(z) = (H_t-z)^{-1},
\end{equation}
where $z=\tau+\i\eta\in \mathbb{C}$.
The normalized trace of $G$ is denoted by $m_N(z)$:
\begin{equation}
m_N(z) = \frac{1}{N}\mathrm{tr} G(z).
\end{equation}
The latter quantity closely approximates the Stieltjes transform $\mfct(z)$ of the deformed semicircle law.

Let $H^{(j)}$ be the $(j,j)$-submatrix of $H_t$, that is, the $(N-1)\times(N-1)$ matrix obtained from the Wigner matrix $H$ by removing the $j$th row and column. We introduce the following notation for the resolvent of $H^{(j)}$ and its normalized trace:
\begin{align}
G^{(j)}(z) := (H_t^{(j)}-z)^{-1}, \qquad m^{(j)}_N(z) := \frac{1}{N}\mathrm{tr} (H^{(j)}_t-z)^{-1}.
\end{align}
Following \cite{shcherbina}, we reserve special symbols for two quantities involving  $G$ and $G^{(j)}$ which will play a role in the computations to come. First, we denote, for $j=1,\ldots, N$,
\begin{equation}
\label{eqn: Adef}
A_j=A_j(z):= -\frac{1}{G_{jj}(z)}.
\end{equation}
Next, we let $h^{(j)} := (h_{ji})_{1\le i \le N}$, and then define
\begin{equation} \label{eqn: Bdef}
B_j=B_j(z):= \langle (G^{(j)}(z))^2 h^{(j)}, h^{(j)}\rangle,
\end{equation}
where $\langle \mathbf{u},\mathbf{v}\rangle$ denotes the inner product of the vectors $\mathbf{u}, \mathbf{v}\in \mathbb{C}^N$. The importance of these quantities for us comes mainly through the identity \eqref{eqn: ABidentity}.

Following \cite{landonyau}, we also define
\begin{equation}\label{eqn: gj-def}
 g_i(z) = \frac{1}{V_i-z-tm_{\mathrm{fc},t}(z)},
\end{equation}
so that
\[m_{\mathrm{fc},t}(z) = \frac{1}{N}\sum_{i=1}^N g_i(z).\]
Finally, we define
\begin{align*} 
R_2(z) &= \frac{1}{N}\sum_{i=1}^N g_i(z)^2, \qquad \tilde{R}_2(z) = \frac{1}{N}\sum_{i=1}^N \frac{1}{\mathbb{E}[A_j(z)]^2}.
\end{align*}

We will often deal with centered random variables. For a random variable $X$ with $\mathbb{E}|X|<\infty$, we denote by
\begin{equation}
X^\circ := X - \ee [X],
\end{equation}
the corresponding centered random variable.

\subsection{Outline of the proof} \label{sec:mesooutline}
Let us now outline the strategy of proof and provide a guide for the reader. The main idea is to compute the derivative of the characteristic function of the random variable 
\[Z:=\mathrm{tr}\varphi_N(H)-\mathbb{E}\mathrm{tr}\varphi_N(H).\]
This approach has been previously applied to linear statistics of random matrices by Shcherbina \cite{shcherbina}. Denoting $\psi(x)=\mathbb{E}[\e^{i x Z}]$, we have
\beq \label{eqn:cc4}
\psi'(x)=-\i \mathbb{E}[Z\e^{\i x Z}],
\eeq
If we can show that the left side is close to $-x V(\varphi_N) \times \psi (x)$, then it follows by direct integration in $x$ that $Z$ is approximately normal.

Next, the problem of computing $\mathbb{E}[Ze^{\i xZ}]$ is reduced to computations involving the matrix $H$ (more precisely, the resolvent $G(z)$) through the Helffer-Sj{\"o}strand representation \eqref{eqn: HS-formula}:
\[\mathbb{E}[Ze^{\i x Z}]=\int_{\mathbb{C}} \mathbb{E}[\mathrm{tr}G(\tau+\i \eta)e(x)^\circ]\partial_z\tilde{\varphi}(z)\,\mathrm{d}z.\]
Here 
\begin{equation}
\label{eqn: outline-ecirc}
e^\circ(x)= \exp(\i x Z)-\mathbb{E}[\exp(\i x Z)],
\end{equation}
and the function $\tilde{\varphi_N} (z)$ is a quasi-analytic extension of $\varphi_N$ to the complex plane $\cc$,
\beq
\tilde{\varphi}_N (x + \i \eta ) = \chi ( \eta ) ( \varphi_N (x) + \i \eta \varphi_N' (x) )
\eeq
where $\chi ( \eta)$ is a smooth compactly supported cut-off function equal to $1$ in neighborhood of $0$.  
The quantity $\mathbb{E}[ Ze^{\i x Z}]$ is then decomposed into two pieces:
\begin{align*}
\int_{\mathbb{C}} \mathbb{E}[\mathrm{tr}G(\tau+\i \eta)e^\circ]\partial_z\tilde{\varphi}_N(z)\,\mathrm{d}z&=\int \partial_z \tilde{\varphi}_N (z)(T_1(z)+T_2(z))\,\mathrm{d}z,\\
T_1(\tau,\eta)&:=\ee \sum_{j=1}^N\left[ G_{jj}^\circ(\tau+\i \eta) e_j^\circ \right],\\
T_2(\tau,\eta) &:= \ee \sum_{j=1}^N\left[ G_{jj}^\circ(\tau+\i \eta) ( e - e_j ) \right].
\end{align*}
Above, $e_j$ is the same as $e$, but with the minor $H^{(j)}$ replacing $H$.  
Through careful resolvent expansions, it will be found that the integral involving $T_1$ is close to a multiple of $\mathbb{E}[Z e^{\i x Z}]$ itself:
	\[\int \partial_z \tilde{\varphi}_N(z)T_1(z)\,\mathrm{d}z=t\int \partial_z \tilde{\varphi}_N(z) \tilde{R}_2(z) \cdot \mathbb{E}[\mathrm{tr}G(\tau+\i \eta) e(x)^\circ]\,\mathrm{d}z+\text{error}.\]
This relatively straightforward computation appears in Section \ref{sec: comp-T1}.  The main input here is that the dominant contribution to the fluctuations of $G_{jj}$ are caused by the $j$th row and column of $H$.  Since $e_j$ is independent of these matrix entries, a resolvent expansion based on the Schur complement formula allows for the calculation of the expectation over the $j$th row and column (i.e, the expectation conditional on $H^{(j)}$), ultimately leading to the above expression.
 
The computation of $T_2$ is more involved. It results in the appearance of a deterministic kernel depending on $\mfct$ which will ultimately generate the covariance kernel $\sqrt{-\Delta}$ (the square root of the Laplacian having integral kernel $(\cdot - y)^{-2}$) appearing in the statement of the theorem. Part of this is the statement of Proposition \ref{prop: I2prop}, which is:
\begin{align} \label{eqn:cc3}
&\int_{\Omega_N} \dvarp(z) \frac{T_2(z)}{1-t\tilde{R}_2(z)}\,\mathrm{d}z\\
=&-\frac{2\i x}{\pi}\mathbb{E}[e(x)]\int_{\Omega_N}\int_{\Omega_N}\dvarp(z)\dvarp(z')\frac{1}{1-tR_2(z)} S_{2,1}(z,z')\,\mathrm{d}z\mathrm{d}z'\\
&+\frac{2\i x}{\pi}\mathbb{E}[e(x)]\int_{\Omega_N}\int_{\Omega_N}\dvarp(z)\dvarp(z')\frac{1}{1-tR_2(z)} (S_{2,2}(z,z')+S_{2,3}(z,z'))\,\mathrm{d}z\mathrm{d}z' +\text{error}.
\end{align}
The kernels $S_{2,1}$, $S_{2,2}$, $S_{2,3}$ are defined in \eqref{eqn: S21-def}, \eqref{eqn: S22-def}, \eqref{eqn: S23-def}.   They are deterministic functions of the initial data $V$ and $\mfct$.   We remark that 
 the transition between Gaussian statistics with a universal variance profile when $t\gg t_1$ and a distribution depending on $V$ when $t\ll t_1$ alluded to in the introduction to this paper can essentially be understood by looking at  the behavior of the quantities $S_{1,2}$, $S_{2,2}$ appearing in \eqref{eqn: S21-def}  and \eqref{eqn: S22-def} when $t$ depends on $N$. The computation of $T_2$ appears in Sections \ref{sec: comp-T2}-\ref{sec: comp-T22}. 

We now summarize the proof of Proposition \ref{prop: I2prop}.  In the initial step, we Taylor expand the difference $e-e_j$ in powers of $\i x (\tr \varphi_N (H) - \tr \varphi_N ( H^{(j)}))^{\circ}$.  Already, the quadratic term will be negligible.  We will use the Helffer-Sj{\"o}strand formula to evaluate the linear term.  This is the source of the second integration over $\cc$ and the term $\partial_{z'} \varphi_N (z')$ in \eqref{eqn:cc3} above, as well as the prefactor $x$ which must appear on the RHS of \eqref{eqn:cc4} in order to conclude the Gaussian statistics.  A resolvent expansion yields the following expression for $T_2$:
\begin{align}
T_2 =& -\sum_{j=1}^N \frac{1}{\mathbb{E}[A_j(\tau+\i\eta)]^2} \frac{\i x}{\pi}\int\dvarp(z')  \ee \left[e_j (x)  \left(N [ \moj - m_N ](s+\i\eta')\right)^\circ  A_{j}(\tau+\i\eta)^\circ  \right] \mathrm{d}z'  \label{eqn: T2-out}\\
&+\text{error}. \nonumber 
\end{align}
where $A_j (z) = - (G_{jj} (z) )^{-1}$. 
In Section \ref{sec: comp-T2} we expand the main term in \eqref{eqn: T2-out}. The leading terms resulting in $S_{2,1}$, $S_{2,2}$, $S_{2,3}$ are computed in Sections \ref{sec: T21-comp} and \ref{sec: comp-T2}, while the error terms are estimated in Section \ref{sec: comp-T22}.

In Section \ref{sec: variance}, 
we consider the quadratic expression in \eqref{eqn:cc3}, ultimately deriving the simplified expression \eqref{eqn:cc5} for the asymptotic variance.  This expression is approximately equal to a constant factor times the Sobolev norm $( \varphi, \sqrt{-\Delta}\varphi )^2$. The reader will note that this section could be drastically simplified if we were dealing with functions with compact support.  After some simple manipulations involving $\mfct$ (see Proposition \ref{prop: sum-gg}), the main work is in transforming, up to some errors, the area integrals over $\cc$ which appear in \eqref{eqn: I2-1}, \eqref{eqn: I2-2} into line integrals over $\rr$ (using Green's theorem) and isolating the main terms. It is found that ultimately the only non-vanishing contribution to the variance comes from the expression $S_{2,1}$ \eqref{eqn: I2-1}. Once the error terms are dealt with the variance kernel \eqref{eqn:cc5} comes out of some essentially exact computations involving $\mfct$.

\subsubsection{A simple example}

In order to illustrate the guiding principles of these calculations, let us consider the simpler case of the Stieltjes transform instead of the general test function $\varphi$ above (in fact, in a moment we will just consider the calculation of its expectation).  
This central limit theorem is best understood as an extension of the local law.  Recall that the local law states that $|m_N ( z ) - \mfct (z) | \leq N^{\eps} / (N \eta)$ with overwhelming probability.  The (imaginary and real parts of the) quantity $(N \eta ) (m_N (z) -\mfct (z) )$ is expected to satisfy a central limit theorem.  In order to prove the local law, one uses the Schur complement formula and after some simplification arrives at,
\beq \label{eqn:cc1}
m_N (z) = \frac{1}{N} \sum_{j=1}^N G_{jj} (z) = \frac{1}{N} \sum_{j=1}^N \frac{1}{ V_j - z - t \mfct (z) + \eps_j }
\eeq
where $\eps_j$ is an error term.  The local law can be viewed as finding large deviations estimates for the error terms $\eps_j$.  The central limit theorem can be viewed as a more careful consideration of the error terms $\eps_j$.  This analysis is aided by the fact that one can, e.g., consider moments of the quantity $(N \eta) (m_N (z) - \mfct (z) )$ which results in one only having to calculate the first few moments of $\eps_j$. 

As a simple example, let us consider the task of calculating the expectation of $(N \eta) (m_N (z) - \mfct (z))$ up to $o(1)$ errors.  One can Taylor expand each term on the RHS of \eqref{eqn:cc1} in powers of $\eps_j$.  The large deviations estimates on $\eps_j$ are sufficient to truncate this expansion after a few terms (in this case the $3$rd order).  The form of $\eps_j$ is roughly $\eps_j = \sum_{k, l \neq j } h_{jk}( G^{(j)}_{kl} - N^{-1} \delta_{kl} ) h_{lj}$.  The difference between this and the local law is that instead of proving large deviations estimates on $\eps_j$, we can use them as a starting point in order to calculate a few moments of $\eps_j$, thanks to the presence of the expectation infront of $\ee[ ( N \eta ) ( m_N (z) - \mfct (z) )]$.  The moments of $\eps_j$ are most easily calculated by first taking the partial expectation over the $j$th row of $H$.  The term $\ee[ \eps_j^2]$ is seen to be negligible and $\ee[ \eps_j]$ gives $\ee[ m_N^{(j)} - \mfct]$.  This expression can be rewritten as $\ee[ m_N^{(j)} - \mfct] = \ee[ m_N - \mfct] + \ee[ m_N^{(j)} - m_N ]$.  The first expression is the analogue of $T_1$ described above, and is the expression we wished to calculate in the first place.  It appears with a coefficient (in this case $R_2$) and is moved back over to the LHS of the equation.  The second term is the analogue of $T_2$, and an algebraic identity (this is \eqref{eqn: ABidentity} below) together with a further resolvent expansion allows for its calculation.    

In a full proof of a central limit theorem, one will be calculating the expectation of $(N \eta) (m_N - \mfct)$ times either its characteristic function as above, or a monomial in it and its conjugate if one is proceeding by the method of moments.  In these cases, this factor must also be expanded around the corresponding expression involving $H^{(j)}$, which coordinates well with the expansion of $G_{jj}$ using the Schur complement formula.

In summary, the local law determines $m_N$ down to $N^{\eps} / (N \eta)$.  In order to remove the $N^{\eps}$ factor, we proceed similarly the proof of the local law based around resolvent expansions, except that we use the independence structure of the matrix ensemble to calculate expectations of the first order of error terms.

\subsection{Estimates for $A_j$ and $B_j$}
The following definition will be useful.

\begin{defn}[Stochastic Domination]
Let 
\[X= (X^{(N)}(u): N\in \mathbb{N}, u \in U^{(N)}), \qquad Y=(Y^{(N)}: N\in\mathbb{N}, u\in U^{(N)}) \]
be two families of nonnegative random variables, where $U^{(N)}$ is a possibly $N$-dependent parameter set. We say that $X$ is \emph{stochastically dominated} by $Y$, uniformly in $u$, if for all small $\epsilon>0$ and all (large) $D$ we have
\[\sup_{u\in U^{(N)}}\mathbb{P}(X^{(N)}> N^\epsilon Y^{(N)}(u)) \le N^{-D}\]
for all $N\ge N_0(\epsilon,D)$. If $X$ is stochastically dominated by $Y$, uniformly in $u$, we write
\begin{equation}
X\prec Y.
\end{equation}
For complex valued $Y$, we write $Y= \O_{\prec}(X)$ if $|Y|\prec X$.
\end{defn}

Recall the definition of $\mathcal{D}_{\epsilon,q}$. We define $\Omega_N$ as the union of this region with its reflection about the real axis, with a choice $\epsilon= \xi$ to be determined.
\begin{equation}
\begin{split}
\Omega_N &= (\mathcal{D}_{\xi,q} \cup \overline{\mathcal{D}_{\xi,q}})\\
&= \left\{ z = E + \i \eta : E \in \I_q,  N^{10C_V} \geq |\eta| \geq N^\xi/N \right\} \cup \left\{ z : E + \i \eta : |E| \leq N^{2 C_V} ,  N^{C_V} \geq |\eta| \geq c \right\}.
\end{split}
\end{equation}
Since $m_N(\bar{z})= \overline{m_N(z)}$ and $\mfct(\bar{z}) = \overline{\mfct(z)}$, the local law extends to $z\in \Omega_N$.

The following estimates for the quantities $A$ and $B$ defined in \eqref{eqn: Adef} and $\eqref{eqn: Bdef}$:
\begin{thm}\label{thm: ABthm}
We have, uniformly in $z\in \Omega_N$,
\begin{align}
\mathbb{E}_jA_j(z) &= z + tm_{\mathrm{fc},t}(z) -V_j + \O_\prec(t(N|\eta|)^{-1}), \label{eqn: EA}\\
A_j^\circ &= \O_\prec( \sqrt{t}N^{-1/2} + t(N|\eta|)^{-1/2})) \label{eqn: Ahalf}\\
B_j(z) &= t\del_z m_{\mathrm{fc},t}(z) + \O_{\prec} ( t|\eta|^{-1} ( N |\eta| )^{-1/2} ).\label{eqn: Bbbound}
\end{align}
\end{thm}

\begin{proof}
Recall the definition of $A_j (x + \i \eta)$ above.
By the Schur complement formula \cite[Lemma 7.7]{landonyau},
\begin{equation}\label{eqn: justA}
\begin{split}
A_j(z) = z-h_{jj}+ \langle \Goj h^{(j)}, h^{(j)}\rangle.
\end{split}
\end{equation}
Taking the partial expectation over $h^{(j)}=(h_{ji})_{ i \neq j}$, we have
\begin{equation}\label{eqn: justEA}
\mathbb{E}_jA_j(z) =  z- V_j + \frac{t}{N}\tr \Goj (z).
\end{equation}

Using the local law, we obtain
\begin{equation*}
\mathbb{E}_j A_j(z) = z -V_j + t m_{\mathrm{fc},t}(z)+\O_\prec(t(N|\eta|)^{-1}),
\end{equation*}
which is \eqref{eqn: EA}.

The estimate \eqref{eqn: Ahalf} is proved in \cite[Lemma 7.9]{landonyau} using the local law.

For \eqref{eqn: Bbbound}, note that
\[B_j= \langle (G^{(j)})^2 h^{(j)},h^{(j)}\rangle = \del_z \langle G^{(j)} h^{(j)},h^{(j)}\rangle.\]
Taking the expectation with respect to $h^{(j)}$ first, and then using the local law, we find
\begin{equation} \label{eqn:Bexpect}
\begin{split}
\ee [B_j] &= t\ee [ \del_z \moj ]\\ 
&= t\del_z \oint_{|z-\zeta|=\eta/2} \frac{
\mathbb{E}[\moj](\zeta)}{z-\zeta}\,\mathrm{d}\zeta\\
&= t\del_z m_{\mathrm{fc},t}(z) + \O_\prec (t |\eta|^{-1} ( N |\eta|)^{-1} ).
\end{split}
\end{equation}
For the second moment, we have
\begin{align}
\ee [ | B_j^\circ |^2 ] &= \ee [ \left|  \sumoj_{i, r, k} \Goj_{ir} \Goj_{rk} ( h_{ji}h_{kj} - \delta_{ik}tN^{-1} ) \right|^2 ] + \O (t^2 |\eta|^{-2}(N |\eta|)^{-2} ) \notag \\
&= \ee \sumoj_{i, l, k, r, m, n} \Goj_{il} \Goj_{lk} \bGoj_{rm}  \bGoj_{mn}  \ee_j [ ( h_{ji } h_{kj} - N^{-1}t \delta_{ik} ) ( h_{j r } h_{nj} - N^{-1} t\delta_{rn} ) ] +  \O ( t(N |\eta|)^{-2} ) \notag \\
&= \frac{2t^2}{N^2} \ee \tr |\Goj|^4  +  \O ( t(N |\eta|)^{-2} ) = \O ( t^2|\eta|^{-2} ( N |\eta|)^{-1} ).
\end{align}
\eqref{eqn: Bbbound} now follows from the large deviation type estimates in \cite[Lemma 7.7]{landonyau}, and the local law. \qed
\end{proof}

We will repeatedly use the identity:
\beq \label{eqn: ABidentity}
 N (m_N - \moj )= G_{jj} \left( 1 + \sumoj_{i, l, k} h_{ji} \Goj_{il} \Goj_{lk} h_{kj} \right) = -A_j^{-1} ( 1 + B_j ).
\eeq
The following lemma collects the main estimates we need for this quantity.

\begin{lem}\label{lem: AinverseB-lemma}
Uniformly for  $\tau+\i\eta \in \Omega_N$,
\begin{align}
A_j^{-1} ( 1 + B_j ) &= \frac{1 + t\del_z m_{\mathrm{fc,t}} }{ \ee [A_j] } + \O_\prec ( |\eta|^{-1}  ( N |\eta|)^{-1/2} ), \label{eqn: AinverseB-estimate}\\
\left( (A_j^{-1}) ( 1 + B_j ) \right)^\circ &=  \frac{ B_j^\circ} { \ee [A_j] } - \frac{ A_j^\circ(1+ \mathbb{E}B_j)}{ \ee[A_j]^2} +\frac{1}{ \ee[A_j]^2}(A_j^\circ B_j^\circ)^\circ + \frac{1}{\mathbb{E}[A_j]^2}\left(\frac{(A_j^\circ)^2}{A_j}(1+B_j)\right)^\circ. \label{eqn: AinverseB-expand}
\end{align}

\end{lem}
\begin{proof}
The first estimate follows directly from  \eqref{eqn: Ahalf}, \eqref{eqn: Bbbound} and the stability estimate \cite[Eqn (7.8)]{landonyau}
\begin{equation}\label{eqn: stability}
|V_i-z-tm_{\mathrm{fc},t}(z)|\ge c\max(t,|\eta|).
\end{equation}

We begin by using the expansion
\begin{equation}
\label{eqn: Aexp}
\begin{split}
\frac{1}{A}&= \frac{1}{\mathbb{E}A}\cdot \frac{1}{1+\frac{A^\circ}{\mathbb{E}A}} =\frac{1}{\mathbb{E}A}\left(1-\frac{A^\circ}{\mathbb{E}A}+\frac{(A^\circ)^2}{(\mathbb{E}A)^2}-\ldots+ (-1)^k \frac{1}{\mathbb{E}[A]^k}\frac{(A^\circ)^k}{1+\frac{A^\circ}{\mathbb{E}A}}\right).
\end{split}
\end{equation}

For \eqref{eqn: AinverseB-expand}, we expand using \eqref{eqn: Aexp} with $k=2$,
\begin{equation}\label{eqn: ABexp}
\begin{split}
\left( (A_j^{-1} ) ( 1 + B_j ) \right)^\circ &=  \frac{ B_j^\circ} { \ee [A_j] } - \frac{1}{\ee[A_j]^2} \left( A_j^\circ (1 + B_j ) \right)^\circ + \left(\frac{(A_j^\circ)^2(1+B_j)}{\mathbb{E}A_j+A_j^\circ}\right)^{\circ} \\
&=  \frac{ B_j^\circ} { \ee [A_j] } - \frac{ A_j^\circ(1+ \mathbb{E}B_j)}{ \ee[A_j]^2} +\frac{1}{ \ee[A_j]^2}(A_j^\circ B_j^\circ)^\circ + \frac{1}{\mathbb{E}[A_j]^2}\left(\frac{(A_j^\circ)^2}{A_j}(1+B_j)\right)^\circ.
\end{split}
\end{equation}
\end{proof}
\qed

\subsection{Computation of the characteristic function}
\label{sec: characteristic}
We derive an equation for the derivative of the  characteristic function of the linear statistic. Let $z=\tau+\i\eta$. Recall the definition of $C_V$ in \eqref{eqn: CVdef}. Without loss of generality, we can assume $C_V\ge 5$. We let $\chi$ be a smooth cut-off function such that $\chi(x)=1$, for $|x|\le N^{10C_V}-1$ and $\chi(x)=0$, for $|x|\ge N^{10C_V}$. Next, define the almost analytic extension of $\varphi_N$ to $\mathbb{C}$.
\[\tilde{\varphi}_N(z)= \chi(\eta)(\varphi(\tau)+\i\eta\varphi_N'(\tau)).\]
The Helffer-Sj\"ostrand formula is the following representation of $\varphi_N$:
\begin{equation}
\label{eqn: HS-formula}
\begin{split}
\varphi_N(\lambda)&= \frac{1}{\pi}\int \frac{\partial_{\bar{z}}\tilde{\varphi}_N(\tau+\i\eta)}{\lambda-\tau-\i\eta}\,\mathrm{d}\tau\mathrm{d}\eta = \frac{1}{\pi}\int_{\mathbb{R}^2} \frac{\i\eta \varphi_N''(\tau)\chi(\eta)+\i(\varphi_N(\tau)+\i\eta\varphi'_N(\tau))\chi'(\eta)}{\lambda-\tau-\i\eta}\,\mathrm{d}\tau\mathrm{d}\eta.
\end{split}
\end{equation}

Define
\beq
e (x) := \exp\left( \i x  (\tr [ \varphi_N ] -\mathbb{E}\tr[\varphi_N])\right), \qquad \psi (x) := \ee[ e (x) ] \label{eqn: psidef}.
\eeq
By \eqref{eqn: HS-formula}, the derivative $\psi'(x)$ equals
\begin{gather}
\label{eqn: drunkenman}
\frac{\i}{\pi}\int_{\mathbb{R}^2} (\i\eta \varphi_N''(\tau)\chi(\eta)+\i(\varphi_N(\tau)+\i\eta\varphi'_N(\tau))\chi'(\eta)) E(z)\,\mathrm{d}\tau\mathrm{d}\eta,\\
E(z): = N\mathbb{E}[e(x)(m_N(\tau+\i\eta)-\mathbb{E}m_N(\tau+\i\eta))].
\end{gather}
The rest of this section is concerned with computing $E(z)$. Let
\[
e_j(x) := \exp\left(\i x\int_{\mathbb{R}^2} \partial_{\bar{z}}\tilde{\varphi}_N(\tau) \,\mathrm{tr} G^{(j)}(\tau+\i\eta)^\circ\,\mathrm{d}\eta \mathrm{d}\tau\right) .\]
We write
\begin{equation} \label{eqn: dec}
\begin{split}
-\i\psi'(x)=&\int_{\mathbb{C}}  \partial_{\bar{z}}\tilde{\varphi}_N(z) \ee [ \tr G (\tau+\i\eta) e^\circ ]\,\mathrm{d}z = \int_{\cc} \partial_{\bar{z}}\tilde{\varphi}_N \sum_{j=1}^N \ee \left[  G_{jj} (\tau+\i\eta) e^\circ \right]\,\mathrm{d}z \\
=&\int_{\cc} \partial_{\bar{z}}\tilde{\varphi}_N \  \sum_{j=1}^N \ee \left[ G_{jj}^\circ(\tau+\i\eta) e_j^\circ \right]\,\mathrm{d}z + \int_{\cc} \partial_{\bar{z}}\tilde{\varphi}_N\sum_{j=1}^N \ee \left[ G_{jj}^\circ ( e - e_j ) \right]\,\mathrm{d}z.
\end{split}
\end{equation}
In view of \eqref{eqn: dec}, we define
\begin{equation}\label{eqn:tdecomp}
\begin{split}
T_1(\tau,\eta) &:= \ee \sum_{j=1}^N\left[ G_{jj}^\circ(\tau+\i\eta) e_j^\circ \right], \qquad T_2(\tau,\eta) := \ee \sum_{j=1}^N\left[ G_{jj}^\circ(\tau+\i\eta) ( e - e_j ) \right].
\end{split}
\end{equation}
We compute these two terms in Propositions \ref{T1prop} and \ref{prop: I2prop}. The result of Proposition \ref{T1prop} is
\begin{align*}
T_1&=t\tilde{R}_2(z) \cdot \mathbb{E}[e(x)\mathrm{tr}G(\tau+\i\eta)^\circ] +\O_\prec(N^{-1/2}|\eta|^{-3/2} )+|x|\O_\prec\left(|\eta|^{-1}N^{-1/2}\|\varphi_N''\|^{1/2}\|\varphi_N'\|_{L^1}^{1/2}\right).
\end{align*}
By the definition of $T_1$ and $T_2$ \eqref{eqn:tdecomp}, we have, for $z\in \Omega_N$,
\begin{align*}
(1-t\tilde{R}_2(z))\mathbb{E}[e^\circ(x) \mathrm{tr} G(z)]&= T_2(z)+ \O_\prec(N^{-1/2}|\eta|^{-3/2} )+|x|\O_\prec\left(|\eta|^{-1}N^{-1/2}\|\varphi_N''\|^{1/2}\|\varphi_N'\|_{L^1}^{1/2}\right),
\end{align*}
so, since $|1-t\tilde{R}_2|\ge c$ by \cite[Eqn (7.10)]{landonyau} and Proposition \ref{thm: ABthm},
\begin{equation}\label{eqn: EetrG-exp}
\begin{split}
\mathbb{E}[e^\circ(x) \mathrm{tr} G(z)]&= \frac{T_2(z)}{1-t\tilde{R}_2(z)}+ \O_\prec(N^{-1/2}|\eta|^{-3/2} )+|x|\O_\prec\left(|\eta|^{-1}N^{-1/2}\|\varphi_N''\|^{1/2}\|\varphi_N'\|_{L^1}^{1/2}\right).
\end{split}
\end{equation}

Write:
\begin{align*}
\int \bar{\del}_{z}\tilde{\varphi}_N(z) \mathbb{E}[e^\circ(x) \mathrm{tr} G(z)]\,\mathrm{d}z &= \int_{\Omega_N} \bar{\del}_{z}\tilde{\varphi}_N(z) \mathbb{E}[e^\circ(x) \mathrm{tr} G(z)]\,\mathrm{d}z + \int_{\Omega_N^c} \bar{\del}_{z}\tilde{\varphi}_N(z) \mathbb{E}[e^\circ(x) \mathrm{tr} G(z)]\,\mathrm{d}z \\
 =:& I_1+I_2.
\end{align*}

For $I_2$, note that $\Omega_N^c\cap \supp \chi(\eta) \subset \{z: |\Im z|< N^{-1+\xi}\}$,
so  we have
\begin{equation}\label{eqn: I2error}
\begin{split}
I_2 &= 2\int_{0<\eta<N^{-1+\xi}} \i\eta \varphi''_N(\tau+i\eta)\chi(\eta)\mathbb{E}[e(x)(\Im \mathrm{tr}G(z))^\circ]\,\mathrm{d}z\\
&= \O_\prec(N^{-1+\xi}\|\varphi_N''\|_{L^1}).
\end{split}
\end{equation}
For $I_1$, we use \eqref{eqn: EetrG-exp}:
\begin{align}
I_1 &= \int_{\Omega_N} \dvarp(z) \frac{T_2(z)}{1-t\tilde{R}_2(z)}\,\mathrm{d}z \label{eqn: I1-main}\\
&+ \int_{\Omega_N} \dvarp(z) \Delta_1(z) \,\mathrm{d}z =: I_1' + \int_{\Omega_N} \dvarp(z) \Delta_1(z) \,\mathrm{d}z,
\end{align}
where 
\[\Delta_1 = \mathbb{E}[e^\circ(x) \mathrm{tr} G(z)]-  \frac{T_2(z)}{1-t\tilde{R}_2(z)}\]
is a holomorphic function in $\Omega_N$ satisfying the bounds:
\[\Delta_1=\O_\prec(N^{-1/2}|\eta|^{-3/2}) +|x|\O_\prec\left(|\eta|^{-1}N^{-1/2}\|\varphi_N''\|^{1/2}_{L^1}\|\varphi_N'\|_{L^1}^{1/2}\right).\]
Using integration by parts in $\tau=\Re z$ when $|\eta|\ge \|\varphi_N''\|_{L^1}^{-1}$ as in the proof of Lemma \ref{lem: ejminuse} (see \eqref{eqn: ibp-example}), it is easily shown that
\begin{equation}
\label{eqn: I1error}
\int_{\Omega_N} \dvarp(z) \Delta_1 \,\mathrm{d}z =(1+|x|)\O_\prec(N^{-1/2}\log N)\|\varphi_N''\|^{1/2}\|\varphi_N'\|_{L^1}^{3/2}.
\end{equation}

We compute the main term in $I_1'$. We need an expression for $T_2$. The next proposition will be proved in the following sections.
\begin{prop}\label{prop: I2prop}
Let 
\begin{equation}
\label{eqn: S21-def}
S_{2,1}(z,z') = \frac{t^2}{N}\sum_{j=1}^N g_j(z)^2g_j(z') \partial_{z'} \frac{\mfct(z)-\mfct(z')}{z-z'},
\end{equation}
\begin{equation}\label{eqn: S22-def}
S_{2,2}(z,z') := \frac{t^2}{N}\sum_{j=1}^Ng_j(z)^2 g_j(z')^2(1+t\partial_z \mfct(z'))\frac{\mfct(z)-\mfct(z')}{z-z'},
\end{equation}
and
\begin{equation}\label{eqn: S23-def}
S_{2,3}(z,z') := \frac{t}{N}\sum_{j=1}^Ng_j(z)^2 g_j(z')^2(1+t\partial_z \mfct(z')).
\end{equation}

The quantity $I_1'$ \eqref{eqn: I1-main} is given by
\begin{equation}\label{eqn: I2-display}
\begin{split}
I_1'&=-\frac{2\i x}{\pi}\mathbb{E}[e(x)]\int_{\Omega_N}\int_{\Omega_N}\dvarp(z)\dvarp(z')\frac{1}{1-tR_2(z)} S_{2,1}(z,z')\,\mathrm{d}z\mathrm{d}z'\\
&\quad +\frac{2\i x}{\pi}\mathbb{E}[e(x)]\int_{\Omega_N}\int_{\Omega_N}\dvarp(z)\dvarp(z')\frac{1}{1-tR_2(z)} (S_{2,2}(z,z')+S_{2,3}(z,z'))\,\mathrm{d}z\mathrm{d}z'\\
&\quad + |x|\O(t^{1/2}N^{-1/2+2\xi})\|\varphi''_N\|_{L^1}\|\varphi_N'\|_{L^1}\nonumber\\
&\quad +(1+|x|)^2\O(N^{-1/2}(\log N)^2)\|\varphi_N'\|^{5/2}_{L^1}\|\varphi_N''\|^{1/2}_{L^1} \nonumber \\
&\quad +|x|\O(N^{-1/2}\log N)(t^{1/2}\|\varphi_N''\|_{L^1}\|\varphi_N\|_{L^1}+ t^{-1/2}\|\varphi_N'\|_{L^1}^2). \nonumber
\end{split}
\end{equation}
\end{prop}

Recall the definition of $\psi$ in \eqref{eqn: psidef}. By Proposition \ref{prop: I2prop}, \eqref{eqn: I2error}, \eqref{eqn: I1error}, we have
\begin{equation}\label{eqn: dpsieq}
\begin{split}
\psi'(x)&= \frac{\i}{\pi} \int \dvarp(z) \mathbb{E}[\mathrm{tr} G(z) e^\circ]\,\mathrm{d}z\\
&= -xV(\varphi_N) \psi(x)\\
&\quad + |x|\O(t^{1/2}N^{-1/2+2\xi})\|\varphi''_N\|_{L^1}\|\varphi_N'\|_{L^1}\nonumber\\
&\quad +|x|(1+|x|)\O(N^{-1/2}(\log N)^2)\|\varphi_N'\|^{5/2}_{L^1}\|\varphi_N''\|^{1/2}_{L^1} \nonumber \\
&\quad +|x|\O(N^{-1/2}\log N)(t^{1/2}\|\varphi_N''\|_{L^1}\|\varphi_N\|_{L^1}+ t^{-1/2}\|\varphi_N'\|_{L^1}^2). \nonumber
\end{split}
\end{equation} 
where 
\begin{equation}\label{eqn: Vvarphidef}
\begin{split}
V(\varphi_N) &:= -\frac{2}{\pi^2}\int_{\Omega_N}\int_{\Omega_N}\dvarp(z)\dvarp(z')\frac{1}{1-tR_2(z)} S_{2,1}(z,z')\,\mathrm{d}z\mathrm{d}z'\\
&\quad +\frac{2}{\pi^2}\int_{\Omega_N}\int_{\Omega_N}\dvarp(z)\dvarp(z')\frac{1}{1-tR_2(z)} S_{2,2}(z,z')\,\mathrm{d}z\mathrm{d}z'\\
&\quad +\frac{2}{\pi^2}\int_{\Omega_N}\int_{\Omega_N}\dvarp(z)\dvarp(z')\frac{1}{1-tR_2(z)} S_{2,3}(z,z')\,\mathrm{d}z\mathrm{d}z.
\end{split}
\end{equation}
By our assumptions \eqref{eqn: offIq-decay}, \eqref{eqn: varphiLinfty}, \eqref{eqn: omega-condition-1}, the error term in \eqref{eqn: dpsieq} is bounded by
\[N^{3\xi}(1+|x|)\O\left(\frac{t^{1/2}}{N^{1/2}t_1}\right)+N^{3\xi}|x|(1+|x|)\O\left(\frac{1}{(Nt_1)^{1/2}}\right)=(1+|x|)\O(N^{\omega_0/2-\omega_1+3\xi})+|x|^2\O(N^{-\omega_1/2+3\xi}).\]
Integrating \eqref{eqn: dpsieq} from $x=0$ to $|x|\le N^{\omega_1/4-\omega_0/8-3\xi}$ using \eqref{eqn: H12-lwrbd}, we find:
\begin{equation}\label{eqn: psi-int}
\psi(x) = \exp\left(-\frac{x^2}{2}V(\varphi_N)\right) + \O(N^{\omega_0/4-\omega_1/2}),
\end{equation}
which is the assertion of Theorem \ref{thm: linstat}.

\subsection{Computation of $T_1$}\label{sec: comp-T1}
\begin{prop}\label{T1prop}
We have the estimate:
\begin{equation}\label{eqn: T1estimate}
\begin{split}
T_1&=t\tilde{R}_2(z) \cdot \mathbb{E}[e(x)\mathrm{tr}G(\tau+\i\eta)^\circ]+ \O_\prec(N^{-1/2}|\eta|^{-3/2} )\\
&\quad +|x|\O_\prec\left(|\eta|^{-1}N^{-1/2}\|\varphi_N''\|^{1/2}\|\varphi_N'\|_{L^1}^{1/2}\right).
\end{split}
\end{equation}
uniformly for $z=\mathcal{D}_{\xi,q}\cup \overline{\mathcal{D}}_{\xi,q}$.
\end{prop}
We choose $k=3$ in \eqref{eqn: Aexp} and write:
\begin{equation}
\label{eqn: T1split}
\begin{split}
T_1 =& \sum_{j=1}^N \frac{\ee [  e_j^\circ \ee_j [ A_j^\circ(\tau+\i\eta) ] ] }{ \ee [A_j(\tau+\i\eta)]^2 }- \sum_{j=1}^N \frac{\ee [ e_j^\circ \ee_j [ ( A^\circ_j(\tau+\i\eta) )^2 ]] }{ \ee [A_j(\tau+\i\eta)]^3 }  + \sum_{j=1}^N \frac{1}{\mathbb{E}[A_j(\tau+\i\eta)]^4}\ee\left[ \frac{e_j^\circ (A^\circ_j)^3 }{1 + \frac{A^\circ_j}{\ee A_j}}\right],
\end{split}
\end{equation}
where we have denoted by $\ee_j$ integation over the first row of $H$ and have used that $e_j$ is independent of this row.
The first term on the right of \eqref{eqn: T1split} will be seen to be the main term in \eqref{eqn: T1estimate}. To deal with the second term, we compute
\begin{equation}
\label{eqn: Acirc2}
\begin{split}
\ee_j [ ( A^\circ_j )^2 ] &= \ee_j [ (  - \sqrt{t}w_{jj}  + \sum_{kl}^{(j)} h_{jk} \Goj_{kl} h_{lj} - \frac{t}{N} \ee [ \tr \Goj ] )^2 ] \\
&= \frac{t}{N}  + \ee_j [ ( \sum^{(j)}_{kl}  \Goj_{kl} ( h_{jk}  h_{jl} - N^{-1}t \delta_{kl} ) )^2 ] + N^{-2} t^2\ee_j [ ( \tr \Goj - \ee \tr \Goj )^2].
\end{split}
\end{equation}
We further compute, using the local law:
\begin{equation}
\label{eqn: Acirc2var}
\begin{split}
\ee_j [ ( \sum^{(j)}_{kl}  \Goj_{kl} ( h_{jk}  h_{jl} - N^{-1}t \delta_{kl} ) )^2 ] &= \frac{t^2}{N^2} \sum^{(j)}_{kl} \Goj_{kl} \Goj_{kl}= \frac{t^2}{N} \del_z m_N+ \O_\prec ( t^2 N^{-2}|\eta|^{-2}).
\end{split}
\end{equation}

Inserting \eqref{eqn: Acirc2}, \eqref{eqn: Acirc2var} into \eqref{eqn: T1split} and using $|e^\circ_j|\le 2$, $\mathbb{E}e^\circ_j=0$, we find:
\begin{align}
T_1(\tau,\eta) &= \sum_{j=1}^N\frac{ \ee [  e_j^\circ \ee_j [ A_j^\circ(\tau+\i\eta) ] ] }{ \ee [A_j(\tau+\i\eta)]^2 } \label{eqn: T1stterm}\\
& + \sum_{j=1}^N\frac{1}{|\mathbb{E}[A_j]|^3}\cdot \O ( t^2 N^{-2}|\eta|^{-2}) + \sum_{j=1}^N \frac{1}{|\mathbb{E}[A_j]|^4}\O(t^{3/2}N^{-3/2}+t^3N^{-2}|\eta|^{-2})).
\end{align}
For the last term we have also used \eqref{eqn: Ahalf}.

Note that
\begin{align*}
\mathbb{E}_j(A_j)^\circ &= \frac{t}{N}\mathrm{tr}\Goj-\frac{t}{N}\mathbb{E}\mathrm{tr}\Goj =t(m_N^{(j)})^\circ,
\end{align*}
and so \eqref{eqn: AinverseB-estimate} implies
\beq\label{eqn: mNdiff}
N  \left| \ee [ \mathbb{E}_j[(A_j)^\circ] e_j ] - \ee [ (t m_N )^\circ e_j ] \right| \le  2t\ee [ | ( A_j^{-1} (1 + B_j ) )^\circ | ] \prec tN^{-1/2} |\eta|^{-3/2}.
\eeq

It now follows from \eqref{eqn: stability} that
\begin{equation*}
\begin{split}
T_1&=\sum_{j=1}^N \frac{t \ee [  e_j^\circ m_N ] }{ \ee [A_j(\tau+i\eta)]^2 }+ \sum_{j=1}^N \frac{1}{|\mathbb{E}[A_j(\tau+i\eta)]|^2}\cdot \O_\prec(tN^{-3/2}|\eta|^{-3/2}).
\end{split}
\end{equation*}

We now replace $e_j$ in \eqref{eqn: T1stterm} by $e(x)$. Using that $| \e^{\i a } - \e^{ \i b } |\le |a-b|$, we find
\begin{equation}\label{eqn: Eejminuse}
\begin{split}
&\left| \ee [ (m_N)^\circ e_j ] - \ee [ ( m_N )^\circ e ] \right| \leq C ( 1 + | x|  ) \ee\left[ \left|\int_{\mathbb{C}}\partial_{\bar{z}}\tilde{\varphi}_N(z')N[  ( (\moj)^\circ(z') - (m_N)^\circ(z')) ]\,\mathrm{d}z'\right||(m_N)^\circ(z)|\right].
\end{split}
\end{equation}
Here $z'=s+\i\eta'$.

To evaluate \eqref{eqn: Eejminuse}, we use $|(m_N)^\circ|\prec (N|\eta|)^{-1}$, together with the following lemma, for which we will also have use in the next section. 

\begin{lem}\label{lem: ejminuse}
We have the estimate,
\[\int_{\mathbb{C}} \dvarp(z')N[(\moj)^\circ(z') - (m_N)^\circ(z')) ]\,\mathrm{d}z' = |x|\O_\prec (\|\varphi_N''\|_{L^1}\|\varphi_N'\|_{L^1}/N)^{1/2}.\]
\end{lem}
\begin{proof}
Let $\epsilon>0$ be a parameter to be determined later. Split the integral into two regions, using the real-valuedness of $\varphi_N$:
\begin{equation}\label{eqn: ejminuse-split}
\begin{split}
\int_{\mathbb{C}} \dvarp(z')N[(\moj)^\circ(z') - (m_N)^\circ(z')) ]\,\mathrm{d}z' &=\Re \int_{\D_{\epsilon,q}} \partial_{\bar{z}}\tilde{\varphi}_N(z')N[  ( (\moj)^\circ(z') - (m_N)^\circ(z')) ]\,\mathrm{d}z'\\
&+\Re \int_{\D_{\epsilon,q}^c} \dvarp(z')  N\left((\moj)^\circ(z') - (m_N)^\circ(z')\right)\,\mathrm{d}z'.
\end{split}
\end{equation}

For the first integral, we simply estimate the real part by the full modulus. Our task is thus to estimate the sum
\begin{align}
& N\left|\int_{\D_{\epsilon,q}}\dvarp(z') ((\moj)^\circ(z') - (m_N)^\circ(z'))\,\mathrm{d}z'\right| \label{eqn: phiN-first} \\
+ & N\left|\int_{\D_{\epsilon,q}^c}\eta' \varphi_N''(s) \chi(\eta')[\Im ((\moj)^\circ(z') - (m_N)^\circ(z'))]\, \mathrm{d}z'\right| \label{eqn: phiN''-bound}\\
+& N\int_{\D_{\epsilon,q}^c}|\varphi_N(s)||\chi'(\eta')||(\moj)^\circ(z') - (m_N)^\circ(z')|\,\mathrm{d}z' \label{eqn: phiN-bound}\\
+& N\int_{\D_{\epsilon,q}^c}|\varphi_N'(s)||\eta'||\chi'(\eta')| |(\moj)^\circ(z') - (m_N)^\circ(z')|\, \mathrm{d}z'. \label{eqn: phiN'-bound}
\end{align}

Below, we will repeatedly use \eqref{eqn: ABidentity} and \eqref{eqn: AinverseB-estimate} to approximate the quantity $(\moj(z)-m_N(z))^\circ$, resulting in the bound
\begin{equation}\label{eqn: mprec}
|(\moj(z)-m_N(z))^\circ|\prec (N|\eta|)^{-3/2}.
\end{equation}
 Since  $ \{\chi'(\eta')\neq 0\}\subset \{|\eta'|\ge N^{10C_V}-1\},$
using \eqref{eqn: AinverseB-estimate}, \eqref{eqn: phiN-first} is bounded by
\begin{equation} \label{eqn: pre-2-way-split}
\left|\int_{\mathcal{D}_{\epsilon,q}}\i\eta'\varphi_N''(s) N((\moj)^\circ(z') - (m_N)^\circ(z'))\,\mathrm{d}z'\right| + N^{-4C_V}(\|\varphi_N'\|_{L^1}+\|\varphi_N'\|_{L^1}).
\end{equation}
The error term here is $\O(N^{-2})$. Introducing a new parameter $\epsilon_2$, we split the $\eta'$ integral in the first term in \eqref{eqn: pre-2-way-split} into the regions
\begin{align}
&\{N^\epsilon/N<|\eta'|\le N^{\epsilon_2}/N\}, \label{eqn: T1-split-1}\\
&\{N^{\epsilon_2}/N \le |\eta'|\le N^{10C_V}\}. \label{eqn: T1-split-2}
\end{align}
In the region \eqref{eqn: T1-split-1}, we use \eqref{eqn: mprec} to find a bound of 
\[\int_{\{N^\epsilon/N<|\eta'|\le N^{\epsilon_2}/N\}}|\eta'||\varphi_N''(s)|\O(N^{-1/2}|\eta'|^{-3/2})\,\mathrm{d}s\mathrm{d}\eta'\le CN^{\epsilon_2/2-1}\|\varphi_N''\|_{L^1}.\]
In \eqref{eqn: T1-split-2}, we integrate by parts in $s$, and combine \eqref{eqn: mprec} and analyticity, to find that the term \eqref{eqn: phiN-first} is bounded by
\begin{equation}\label{eqn: ibp-example}
\begin{split}
&N\int_{N^{-1+{\epsilon_2}}<|\eta'|<N^{10C_V}}|\varphi_N'(s)||\eta'||\partial_{z'}(A^{-1}_j(1+B_j)(s))^{\circ} |\mathrm{d}z'\\
\leq &N\int_{N^{-1+{\epsilon_2}}<|\eta'|<10}|\varphi_N'(s)|N^{-3/2}|\eta'|^{-3/2}\mathrm{d}z' \leq  N^{-\epsilon_2/2}\|\varphi_N'\|_{L^1}.
\end{split}
\end{equation}
Optimizing $\epsilon_2$, we find that \eqref{eqn: phiN-first} is bounded by $\O(N^{-1/2}\|\varphi_N'\|_{L^1}^{1/2}\|\varphi_N''\|_{L^1}^{1/2}).$  For \eqref{eqn: phiN''-bound}, we use the assumption \eqref{eqn: offIq-decay} on the support of $\varphi_N'$. The integration is over 
\[\{0<|\eta'|<N^\epsilon/N\}\cup\{10 < |\eta'| < N^{10C_V}\}.\] 
In the first region, we have by monotonicity -- see \cite[Lemma 7.19]{landonyau} for details -- $|\Im m_N^\circ(z')|, |\Im(\moj)^\circ(z')| \prec (N\eta')^{-1}$, so this term is
\[\O_{\prec}\left(\int_{0<|\eta'|<N^{\epsilon}/N}| \varphi_N''(s)|\, \mathrm{d}z' \right) =\O( N^{-1+\epsilon}\|\varphi''\|_{L^1}).\]

For the integral over $|\eta'|>10$, we integrate by parts and use $\partial_{\eta'}\Im m_N = -\partial_s \Re m_N$ to find the estimate
\begin{align*}
&N\left|\int_{|\eta'|>10} \varphi_N'(s) \partial_{\eta'}(\eta'\chi(\eta')) ((\moj)^\circ(z') - (m_N)^\circ(z'))\,\mathrm{d}z'\right|\\
+&N\left|\int \varphi_N'(s) 10\chi(10) ((\moj)^\circ(s+10\i) - (m_N)^\circ(s+10\i))\,\mathrm{d}s\right|.
\end{align*}
We use \eqref{eqn: ABidentity}, \eqref{eqn: AinverseB-estimate} to find that the expectation of both terms is bounded by  $\O_\prec(N^{-1/2}\|\varphi_N'\|_{L^1})$.

Recalling \eqref{eqn: varphisupport}, the term \eqref{eqn: phiN-bound} is estimated by
\begin{align*}
&N\int_{\Omega_N^c\cap \{N^{C_V}-1 \le|y|\le N^{C_V}\}}|\varphi_N(s)|\, \mathrm{d}z' \leq N \frac{N^{10C_V}}{N^{20C_V}}.
\end{align*}
Assuming (without loss of generality) that $C_V\ge 5$, this is $\O(N^{-2})$.
Using $|m_N^\circ(\eta')|\prec (N\eta')^{-1}$, the term \eqref{eqn: phiN'-bound} is $\O(N^{-1}\|\varphi_N'\|_{L^1})$. Combining all the above, we find that, for any $\epsilon>0$:
\begin{equation}\label{eqn: ejminuse}
\begin{split}
 &\left|\int_{\mathbb{C}} \dvarp(z')N[(\moj)^\circ(z') - (m_N)^\circ(z')) ]\,\mathrm{d}z'\right| =\O_\prec(N^{-1/2}\|\varphi_N''\|_{L^1}^{1/2}\|\varphi_N'\|_{L^1}^{1/2}+N^{-1+\epsilon}\|\varphi_N''\|_{L^1}).
\end{split}
\end{equation}
The result now follows by optimizing in $\epsilon$.
\end{proof}\qed

Using the previous lemma in \eqref{eqn: Eejminuse}, we have
\[\left| \ee [ (m_N)^\circ e_j ] - \ee [ ( m_N )^\circ e ] \right| = \O((N|\eta|)^{-1}) \|\varphi_N''\|^{1/2}_{L^1}\|\varphi_N'\|_{L^1}^{1/2}N^{-1/2}.\]

By \eqref{eqn: stability} and \eqref{eqn: EA}, we can estimate
\begin{equation}\label{eqn: R2sum-bound}
\frac{t}{N}\sum_{j=1}^N \frac{1}{|\mathbb{E}A_j|^2} \prec 1.
\end{equation}
From this we get that, for all $\tau+\i\eta \in \Omega_N$,
\begin{align} \label{eqn:t1final}
&T_1(\tau,\eta) = \ee [ e (x)  \tr G(\tau+\i\eta)^\circ ]\cdot  \frac{t}{N}\sum_{j=1}^N\frac{1}{\ee [A_j]^2 }\\
&+ \frac{t}{N}\sum_{j=1}^N\frac{1}{|\mathbb{E}[A_j]|^2}\left(\O_\prec(N^{-1/2}|\eta|^{-3/2})\right) + \frac{t}{N}\sum_{j=1}^N\frac{1}{|\mathbb{E}[A_j]|^2}|x|\O_\prec\left(|\eta|^{-1}N^{-1/2}\|\varphi_N''\|^{1/2}\|\varphi_N'\|_{L^1}^{1/2}\right) \notag \\
& =  t\tilde{R}_2(z)\cdot \ee [ e (x) \tr G(\tau+\i\eta)^\circ ]+ \O_\prec(N^{-1/2}(|\eta|^{-3/2}+t^{-1/2}) ) +|x|\O_\prec\left(|\eta|^{-1}N^{-1/2}\|\varphi_N''\|^{1/2}\|\varphi_N'\|_{L^1}^{1/2}\right).\nonumber
\end{align} 
This is the claim of Proposition \ref{T1prop}.
\qed

\subsection{Computation of $T_2$}\label{sec: comp-T2}
We now compute $T_2$. Recall from the definition \eqref{eqn:tdecomp} that
\[T_2(\tau,\eta) = \sum_{j=1}^N \ee \left[ G_{jj}^\circ(\tau+\i\eta) ( e - e_j ) \right].\]
By \eqref{eqn: Aexp}, we have:
\begin{equation}\label{eqn: expandforT2}
\begin{split}
\ee \left[ G_{jj}^\circ(\tau+\i\eta) ( e - e_j ) \right]  &= \frac{1}{\mathbb{E}[A_j]^2}\ee \left[ A_j^\circ(\tau+\i\eta) ( e - e_j ) \right]\\
&\quad - \frac{1}{\mathbb{E}[A_j]^2}\mathbb{E}\left[\frac{1}{\mathbb{E}[A_j]+A_j^\circ} (A_j^\circ(\tau+\i\eta))^2 ( e - e_j ) \right].
\end{split}
\end{equation}

Using the expansion
\begin{align*}
\exp(iX^{(j)})-\exp(\i X) &= \exp(iX^{(j)})\cdot (1-\exp(\i(X-X^{(j)})))\\
&= \exp(\i X^{(j)})(\i(X^{(j)}-X)+O(|X-X^{(j)}|)^2),
\end{align*}
we have, by Lemma \ref{lem: ejminuse},
\begin{equation}\label{eqn: eminuse1}
\begin{split}
&e_j(x) - e(x) - \frac{\i x}{\pi} e_j (x) \int \dvarp(z') (N \cdot [ \moj - m_N](s+\i\eta') )^\circ\,\mathrm{d}z'\\ 
= &|x|^2\O_\prec(\|\varphi_N''\|_{L^1}\|\varphi_N'\|_{L^1}N^{-1}),
\end{split}
\end{equation}
with overwhelming probability.

Using \eqref{eqn: eminuse1} and \eqref{eqn: Ahalf} in \eqref{eqn: expandforT2}, we get the following expression for $T_2$, which holds for $\tau+i\eta \in \Omega_N$:
\begin{align}
T_2 =& -\sum_{j=1}^N \frac{1}{\mathbb{E}[A_j(\tau+\i\eta)]^2} \frac{\i x}{\pi}\int\dvarp(z')  \ee \left[e_j (x)  \left(N [ \moj - m_N ](s+\i\eta')\right)^\circ  A_{j}(\tau+\i\eta)^\circ  \right] \mathrm{d}z' \label{eqn: T2-display-1}\\
&+ \frac{1}{N}\sum_{j=1}^N\frac{t|x|^2}{|\mathbb{E}A_j(\tau+\i\eta)|^2}\cdot \O((N|\eta|)^{-1/2}+t^{-1/2}N^{-1/2})\|\varphi_N''\|_{L^1}\|\varphi_N'\|_{L^1} \label{eqn: T2-error-1}\\
&+ \frac{1}{N}\sum_{j=1}^N\frac{|x|}{|\mathbb{E}A_j(\tau+\i\eta)|^3}\cdot \O( t^2N^{-1/2}|\eta|^{-1}+tN^{-1/2})\|\varphi_N''\|^{1/2}_{L^1}\|\varphi_N'\|^{1/2}_{L^1} \label{eqn: T2-error-2}.
\end{align}


We now compute the main term in \eqref{eqn: T2-display-1}. We begin by splitting:
\begin{align}
&\int_{\Omega_N}\dvarp(z')  \ee \left[e_j (x)  \left(N [ \moj - m_N ](s+\i\eta')\right)^\circ  A_{j}(\tau+\i\eta)^\circ  \right] \mathrm{d}z' \label{eqn: T2-insideD} \\
+&\int_{\Omega_N^c}\dvarp(z')  \ee \left[e_j (x)  \left(N [ \moj - m_N ](s+\i\eta')\right)^\circ  A_{j}(\tau+\i\eta)^\circ  \right] \mathrm{d}z'. \label{eqn: T2-outsideD}
\end{align}
The term \eqref{eqn: T2-outsideD} is estimated in the same way as the second term in \eqref{eqn: ejminuse-split}. Together with \eqref{eqn: Ahalf}, This gives a bound of $\O_\prec((t(N|\eta|)^{-1/2}+t^{1/2}N^{-1/2})N^{-1+\xi})\|\varphi_N''\|_{L^1}$.  We see that the total contribution to $T_2$ of the sum over $j$ of \eqref{eqn: T2-outsideD} is bounded by
\begin{equation}\label{eqn: T1-error-3}
 \frac{1}{N}\sum_{j=1}^N\frac{t|x|}{|\mathbb{E}[A_j(\tau+\i\eta)]|^2}(\O((N|\eta|)^{-1/2})+\O(t^{-1/2}N^{-1/2}))N^\xi \|\varphi_N''\|_{L^1}.
\end{equation}

For the first term \eqref{eqn: T2-insideD}, we use the expansion \eqref{eqn: ABexp}. The main terms are
\begin{equation}\label{eqn: T2-display}
\begin{split}
T_{2,1} &= -\sum_{j=1}^N \frac{ix}{\pi}\int_{\Omega_N}\dvarp(z')  \frac{1}{\mathbb{E}[A_j(z)]^2}\ee \left[e_j (x) \frac{B_j^\circ(z')}{\mathbb{E}[A_j(z')]}  A_{j}(z)^\circ  \right] \mathrm{d}z',\\
T_{2,2}&=\sum_{j=1}^N \frac{ix}{\pi}\int_{\Omega_N}\dvarp(z')  \frac{1}{\mathbb{E}[A_j(z)]^2}\ee \left[e_j (x) \frac{A_j^\circ(z')(1+\mathbb{E}B_j(z'))}{\mathbb{E}[A_j(z')]^2}  A_{j}(z)^\circ  \right] \mathrm{d}z'.
\end{split}
\end{equation}
The remaining terms will be shown to be error terms:
\begin{align}
T_{2,3}&=\sum_{j=1}^N\frac{x}{\mathbb{E}[A_j(z)]^2} \int_{\Omega_N}  \dvarp(z')\frac{1}{\mathbb{E}[(A_j(z')]^2}\mathbb{E}[A^\circ_j(z')B^\circ_j(z') A_j^\circ(z)] \mathrm{d}z',\\
T_{2,4}&=\sum_{j=1}^N\frac{x}{\mathbb{E}[A_j(z)]^2} \int_{\Omega_N} \dvarp(z')\frac{1}{\mathbb{E}[A_j(z')]^2}\mathbb{E}\left[\frac{(A_j^\circ(z'))^2}{A_j(z')}(1+B_j(z')) A_j^\circ(z)\right]\mathrm{d}z'.
\end{align}

Collecting the error terms obtained so far and using \eqref{eqn: stability}, we find
\begin{equation}
\begin{split}
I_1&=\int_{\Omega_N} \dvarp(z) \frac{T_2(z)}{1-t\tilde{R}(z)}\,\mathrm{d}z\\
&= \int_{\Omega_N} \dvarp(z) \frac{(T_{2,1}(z)+T_{2,2}(z)+T_{2,3}(z)+T_{2,4}(z))}{1-t\tilde{R}(z)}\,\mathrm{d}z+ \int_{\Omega_N} \dvarp(z) \Delta_{1,1}(z)\mathrm{d}z,\\
\end{split}
\end{equation}
where $\Delta_{1,1}$ is $1/(1-t\tilde{R}(z))$ times the difference between $T_2$ and the main term \eqref{eqn: T2-display-1}, restricted to the region $\Omega_N$. $|\Delta_{1,1}|$ is bounded by the sum of the errors \eqref{eqn: T2-error-1}, \eqref{eqn: T2-error-2} and \eqref{eqn: T1-error-3}.

We have:
\begin{align}
\int_{\Omega_N} \dvarp(z) \Delta_{1,1}(z) \,\mathrm{d}z &= \int_{\Omega_N} \i\varphi_N''(\tau)\eta\chi(\eta)  \Delta_{1,1}(z) \,\mathrm{d}z \label{eqn: Delta11-phi''}\\
&+ \int_{\Omega_N} \i\varphi_N(\tau)\chi'(\eta)  \Delta_{1,1}(z) \,\mathrm{d}z \label{eqn: Delta11-phi}\\
&- \int_{\Omega_N} \varphi_N'(\tau)\eta\chi'(\eta)  \Delta_{1,1}(z) \,\mathrm{d}z. \label{eqn: Delta11-phi'}
\end{align}
We first estimate \eqref{eqn: Delta11-phi''}. After integration by parts in $\tau$, and using 
$$|\partial_{z}\Delta_{1,1}(z)|\le 2|\eta|^{-1}\max_{|w-z|=|\eta|/2}|\Delta_{1,1}(w)|,$$
 this is bounded by
\begin{align}
&\int_{\{z:N^{-1+\xi}<|\eta|<N^{10C_V}\}}|\varphi_N'(\tau)\eta\chi(\eta)| |\partial_{z}\Delta_{1,1}(z)|\,\mathrm{d}z \nonumber\\
\leq & \|\varphi_N'\|_{L^1} \int_{N^{-1+\xi}<|\eta|<N^{10C_V}} \sup_{\tau}\left( \frac{1}{N}\sum_{j=1}^N\frac{tN^\xi|x|}{|\mathbb{E}[A_j(\tau+\i\eta)]|^2}\O((N|\eta|)^{-1/2}+t^{-1/2}N^{-1/2}) \|\varphi_N''\|_{L^1}\right)\,\mathrm{d}\eta \label{eqn: eta-integral-1}\\
+ & \|\varphi_N'\|_{L^1}\int_{N^{-1+\xi}<|\eta|<N^{10C_V}} \sup_\tau\left(  \frac{1}{N}\sum_{j=1}^N\frac{|x|}{|\mathbb{E}A_j(\tau+\i\eta)|^3}\cdot \O( t^2N^{-1/2}|\eta|^{-1}+tN^{-1/2} )\|\varphi_N''\|^{1/2}_{L^1}\|\varphi_N'\|^{1/2}_{L^1}\right)\,\mathrm{d}\eta \label{eqn: eta-integral-2}
\end{align}
Split the $\eta$ integral \eqref{eqn: eta-integral-1} into $\{|\eta|\le t\}$, $\{|\eta|>t\}$, and
\[\frac{1}{N}\sum_{j=1}^N \frac{1}{|\mathbb{E}[A_j]|^2}\le C\log N/(\max(t,|\eta|)).\]
This gives the estimate
\begin{equation}\label{eqn: Delta11-error}
|x|\|\varphi_N'\|_{L^1}\|\varphi_N''\|_{L^1}\O(t^{1/2}N^{-1/2+2\xi}).
\end{equation}
With $t=N^{\omega_0}/N$ and $\|\varphi_N''\|_{L^1}\le N/N^{\omega_1}$, this is $\O(N^{\omega_0/2-\omega_1+2\xi}),$ 
which is $O(N^{-c})$ if $\xi$ is small enough. By direct computation and \eqref{eqn: R2sum-bound}, the term \eqref{eqn: eta-integral-2} is
$|x|\O(tN^{-1/2}\log N)\|\varphi_N''\|_{L^1}^{1/2}\|\varphi_N'\|_{L^1}^{3/2}. $
For the terms \eqref{eqn: Delta11-phi}, \eqref{eqn: Delta11-phi'}, the integrands are supported in the region 
$\{z:N^{10C_V}-1<|\Im z|<N^{10C_V}\}.$ 
In this region, we use the bound 
$|\mathbb{E}A_j(\tau+\i\eta)|\ge c|\eta|,$
to obtain a bound of the form
$C\|\varphi_N'\|_{L^1}N^{-2}.$
The remaining terms $T_{2,1}$, $T_{2,2}$, $T_{2,3}$, $T_{2,4}$ are computed in the following sections.

\subsection{Computation of $T_{2,1}$.}  \label{sec: T21-comp}
We now compute the term $T_{2,1}$ \eqref{eqn: T2-display}. Since $e_j(x)$ is independent of $(h_{ij})_{i=1}^N$, we first compute
\begin{equation}
\ee_j [A_j^\circ(\tau+\i\eta)  \frac{ B_j^\circ(s+\i\eta') }{ \mathbb{E}[A_j(s+\i\eta')] }]. \label{eqn: EAB}
\end{equation}
For simplicity of notation, we will write $G(s)$ for $G(s+\i\eta')$ and $G(\tau)$ for $G(\tau+\i\eta)$. Similar notational simplifications apply to $\Goj(s+\i\eta')$, $\moj(s+\i\eta')$, $\moj(s+\i\eta')$, etc.

The result of the following computation is:
\begin{prop}\label{EAmscBprop}
Uniformly for $\tau+\i\eta, s+\i\eta' \in \Omega_N$,
\begin{equation}\label{eqn: EAmscB}
\begin{split}
N\ee_j [ A_j^\circ(\tau+\i\eta) \frac{ B_j^\circ(s+\i\eta')}{\mathbb{E}[A_j(s+\i\eta')]} ]=& \frac{2t^2}{\mathbb{E}[A_j(s+\i\eta')]} \cdot \partial_s \frac{m_{\mathrm{fc},t}(\tau)-m_{\mathrm{fc},t}(s)}{\tau-s+\i(\eta-\eta')}\\
&+ g_j(s)\cdot\O_\prec(t^2N^{-1}|\eta'|^{-2}|\eta|^{-1}).
\end{split}
\end{equation}
\end{prop}

\begin{proof}
We first recenter around the conditional expectations $\mathbb{E}_jA_j$, $\mathbb{E}_jB_j$ instead of the full expectations, using the identity
\begin{align*}
&\mathbb{E}_j[(A-\mathbb{E}[A](B-\mathbb{E}[B]] = \mathbb{E}_j[(A-\mathbb{E}_j[A])(B-\mathbb{E}_j[B])]+(\mathbb{E}_j[A]-\mathbb{E}[A])(\mathbb{E}_j[B]-\mathbb{E}[B]).
\end{align*}

 This produces an error $\O(t^2N^{-1}|\eta'|^{-2}|\eta|^{-1})$. We then write
\begin{align}
&N \ee_j [ (A_j(z)-\mathbb{E}_jA_j(z))\frac{(B_j(z')-\mathbb{E}_j B_j(z'))}{\mathbb{E}[A_j(z')]} ] \nonumber \\
= &N \ee_j \left( \sumoj_{i, k}  \Goj_{ik}(\tau) ( h_{ji} h_{kj} - N^{-1}t \delta_{ik} ) \right) \nonumber \times \left( \sumoj_{lkm} \frac{1}{\mathbb{E}[A_j(z')]} \Goj_{lk}(s) \Goj_{km}(s) ( h_{jl} h_{mj} - N^{-1}t \delta_{lm}) \right) \nonumber \\
= & \frac{2t^2}{N \mathbb{E}[A_j(s+\i\eta')]} \tr ( \Goj(\tau)( \Goj(s) )^2 ). \label{eqn: conjugate}
\end{align}

Now we use
\begin{align*}
\frac{1}{N}\mathrm{tr} \Goj(\tau)(\Goj(s))^2&= \partial_{z}\mathrm{tr}(\Goj(\tau)\Goj(z))|_{z=s+i\eta'}.
\end{align*}
We can write this as
\[ \frac{1}{N}\partial_s \frac{\mathrm{tr}\Goj(\tau)-\mathrm{tr}\Goj(s)}{\tau-s+\i(\eta-\eta')}= \partial_s\frac{\moj(\tau)-\moj(s)}{\tau-s+\i(\eta-\eta')} .\]
Note the identity:
\begin{equation}
\label{eqn: f-diffquot}
\partial_{z'}\frac{f(z)-f(z')}{z-z'} = \int_0^1 (1-\alpha)f''(z'+\alpha(z-z'))\,\mathrm{d}\alpha.
\end{equation}
If $\eta\eta'>0$ and $|\eta-\eta'|<\max(|\eta|,|\eta'|)/2$, we use \eqref{eqn: f-diffquot} with
\[f(z) = m_N^{(j)}(z)-\mfct(z)\]
to find
\begin{align*}
&\partial_s \frac{m_{\mathrm{fc},t}(\tau)-m_{\mathrm{fc},t}(s)}{\tau-s+\i(\eta-\eta')}+ \O_\prec(\max_{\alpha\in [z,z']} \frac{ |\alpha\eta+(1-\alpha)\eta'|^{-3}}{N})= \partial_s \frac{m_{\mathrm{fc},t}(\tau)-m_{\mathrm{fc},t}(s)}{\tau-s+\i(\eta-\eta')}+ \O_\prec(N^{-1}|\eta'|^{-2}|\eta|^{-1}).
\end{align*}
If $|\eta-\eta'|>\max(|\eta|,|\eta'|)/2$, we perform the differentiation
\[\frac{-\partial_sm_N^{(j)}(s)}{(\tau-s)+\i(\eta-\eta')}+\frac{m_N^{(j)}(z)-m_N^{(j)}(z')}{((\tau-s)+\i(\eta-\eta'))^2}.\]
Using the local law, we replace $m_N^{(j)}(s), m_N^{(j)}(\tau)$ by $\mfct(s), \mfct(\tau)$ with an error $\O(N^{-1}|\eta'|^{-2}|\eta|^{-1})$.

If $\eta\eta'<0$, applying the local law again we find
\[\partial_s \frac{m_{\mathrm{fc},t}(\tau)-m_{\mathrm{fc},t}(s)}{\tau-s+\i(\eta-\eta')}+ \O_\prec(N^{-1}|\eta'|^{-2})\frac{1}{|\eta-\eta'|}.\]
\end{proof}\qed


Using Proposition \ref{EAmscBprop} in the main term of \eqref{eqn: T2-display}, and using \eqref{eqn: EA} to replace $1/\mathbb{E}[A_j]$, $1/\mathbb{E}[A_j]^2$ by $g_j$, $g_j^2$ we find, for $\tau +i\eta\in \Omega_N$:
\begin{align}
T_{2,1}(z)=& -\frac{2\i x}{\pi}\frac{\ee [e (x)]}{N} \sum_{j=1}^N t^2 \int_{\Omega_N}g_j(z)^2g_j(z') \dvarp(z') \partial_s\frac{ m_{\mathrm{fc},t}(\tau) - m_{\mathrm{fc},t}(s) }{\tau-s +\i(\eta-\eta')}\mathrm{d}z' \label{eqn: T2-main}\\
&-\frac{2\i x}{\pi} \sum_{j=1}^N \frac{t^2}{N}(\mathbb{E}[e_j(x)-e(x)]) \int_{\Omega_N}g_j(z)^2g_j(z') \i\eta'\varphi_N''(s) \partial_s\frac{ m_{\mathrm{fc},t}(\tau) - m_{\mathrm{fc},t}(s) }{\tau-s +\i(\eta-\eta')}\mathrm{d}z' \label{eqn: T2-ej-e}\\
&+ \frac{1}{N}\sum_{j=1}^N\frac{|x|}{|\mathbb{E}[A_j(\tau+\i\eta)]|^2}\int_{\Omega_N} \frac{1}{\mathbb{E}|A_j(s+\i\eta')|}\cdot\O_\prec(t^2N^{-1}|\eta'|^{-2}|\eta|^{-1})|\eta'||\varphi_N''(s)|\,\mathrm{d}z' \label{eqn: T2-cubic}\\
&+ \frac{1}{N}\sum_{j=1}^N\frac{t^2|x|}{|\mathbb{E}[A_j(\tau+\i\eta)]|^2}\int_{\Omega_N} \frac{1}{\mathbb{E}|A_j(s+\i\eta')|}\cdot(\O_\prec((N|\eta'|)^{-1})+\O_\prec(t^{-1}(N|\eta|)^{-1})) \label{eqn: T2-thirderror}\\
&\quad \times |\eta'||\varphi_N''(s)|\left|\partial_s\frac{ m_{\mathrm{fc},t}(\tau) - m_{\mathrm{fc},t}(s) }{\tau-s +\i(\eta-\eta')}\right|\,\mathrm{d}z' 
\end{align}
Note that above, we have omitted the terms with support in the region $\{\chi(\eta')\neq 0\}$, as they are smaller than the terms displayed.

For the term \eqref{eqn: T2-cubic}, we use \eqref{eqn: R2sum-bound} and \eqref{eqn: stability} to find an estimate
\begin{equation}
|x|\O_\prec(\log N (N|\eta|)^{-1})\|\varphi_N''\|_{L^1}.
\end{equation}

To deal with the remaining terms, we use the following estimates:
\begin{prop}\label{EAmscBprop-2}
If $\eta, \eta'\in \Omega_N$ and $\eta\eta' >0$, then
\begin{equation}
\partial_s \frac{m_{\mathrm{fc},t}(\tau)-m_{\mathrm{fc},t}(s)}{\tau-s+\i(\eta-\eta')}= \O(|\eta|^{-1}|\eta'|^{-1}), \label{eqn: negligible3}
\end{equation}
If $ \eta\eta'<0$, then
\begin{equation}
\partial_s \frac{\mfct(\tau)- m_{\mathrm{fc},t}(s) }{ \tau-s +\i(\eta-\eta')}=  \O(|\eta-\eta'|^{-1}|\eta'|^{-1}). \label{eqn: negligible4}
\end{equation}
\end{prop}
\begin{proof}
By the representation \eqref{eqn: f-diffquot} the left side of \eqref{eqn: negligible3} is
\begin{equation}
\int_0^1 \alpha \mfct^{(2)}(z'+\alpha(z-z'))\,\mathrm{d}\alpha \label{eqn: mfct-1der}.
\end{equation}
This is bounded by \[\max_{\zeta\in [z,z']}|m''_{\mathrm{fc},t}(\zeta)|\le C|\eta|^{-1}|\eta'|^{-1}.\]

For \eqref{eqn: negligible4}, we simply perform the differentiation:
\begin{align*}
&\del_{s}\frac{\mfct(\tau)-m_{\mathrm{fc},t}(s) }{ \tau-s +\i(\eta-\eta')}\\
=&-\frac{\del_{s}m_{\mathrm{fc},t}(s) }{ \tau-s +\i(\eta-\eta')}+\frac{\mfct(\tau)-m_{\mathrm{fc},t}(s) }{(\tau-s +\i(\eta-\eta'))^2} = \O(|\eta'|^{-1}|\eta-\eta'|^{-1})+\O(|\eta-\eta'|^{-2}).
\end{align*}
\end{proof}\qed

By \eqref{eqn: negligible3}, \eqref{eqn: negligible4}, and \eqref{eqn: R2sum-bound} the term \eqref{eqn: T2-thirderror} is bounded by
$|x|O_\prec(\log N (N|\eta|)^{-1})\|\varphi_N''\|_{L^1}.$
For the term \eqref{eqn: T2-ej-e}, we use \eqref{eqn: ejminuse}, \eqref{eqn: negligible3}, \eqref{eqn: negligible4}, and integrate by parts in $s'$ when $|\eta'|\le \|\varphi_N''\|_{L^1}$ to find an error
\[|x|(1+|x|)\O(\log N |\eta|^{-1}\|\varphi_N''\|_{L^1}^{1/2}\|\varphi_N'\|_{L^1}^{3/2}N^{-1/2})\]

We have shown the main term in \eqref{eqn: T2-display} is
\[-\frac{2\i x}{\pi}\ee [e (x)]\int_{\Omega_N} S_{2,1}(z,z') \dvarp(z') \mathrm{d}z' + |x|(1+|x|)\O(\log N N^{-1/2}|\eta|^{-1}) \|\varphi_N''\|^{1/2}_{L^1}\|\varphi_N'\|^{1/2},\]
with
\begin{equation}
S_{2,1}(z,z') = \frac{t^2}{N}\sum_{j=1}^N g_j(z)^2g_j(z') \partial_s \frac{\mfct(z)-\mfct(z')}{z-z'}.
\end{equation}

Multiplying $T_{2,1}(z)/(1-t\tilde{R}(z))$ by $\dvarp(z)$, and integrating we have:
\begin{align}
\int_{\Omega_N} \dvarp(z) \frac{T_{2,1}(z)}{1-t\tilde{R}_2(z)}\mathrm{d}z&=-2\i x\frac{\mathbb{E}[e(x)]}{\pi} \int_{\Omega_N} \frac{1}{1-t\tilde{R}_2(z)} \dvarp(z) \int_{\Omega_N}\dvarp(z') S_{2,1}(z,z')\,\mathrm{d}z'\mathrm{d}z \nonumber \\
&+ \int_{\Omega_N} \frac{1}{1-t\tilde{R}_2(z)}\i\eta \chi(\eta)\varphi''_N(\tau) \Delta_{2,1} (z)\,\mathrm{d}z,\label{eqn: Delta}
\end{align}
where
\[\Delta_{2,1}(z):= T_{2,1}+\frac{\i\mathbb{E}[e(x)]}{\pi}\int \dvarp(z')S_{2,1}(z,z')\,\mathrm{d}z'\]
is analytic in $\Im z>0$ and $\Im z<0$ and 
\[\frac{\Delta_{2,1}(z)}{1-\tilde{R}_2(z)} = \O(|\eta|^{-1}\log N/N^{1/2})(|x|(1+|x|)\|\varphi''_N\|_{L^1}^{1/2}\|\varphi_N'\|_{L^1}^{3/2}.\]
Integrating by parts in $\tau$ in the integral \eqref{eqn: Delta} and using
\[\partial_\tau \frac{\Delta_{2,1}(z)}{1-t\tilde{R}_2(z)}=\O(|\eta|^{-2}\log N /N^{1/2})(|x|(1+|x|))\|\varphi''_N\|^{1/2}_{L^1}\|\varphi_N'\|_{L^1}^{3/2},\]
we find
\begin{align}
\int_{\Omega_N} \dvarp(z) \frac{T_{2,1}(z)}{1-t\tilde{R}_2(z)}\mathrm{d}z&=-\frac{2\i x}{\pi}\mathbb{E}[e(x)] \int_{\Omega_N} \frac{1}{1-t\tilde{R}_2(z)} \dvarp(z) \int_{\Omega_N}\dvarp(z') S_{2,1}(z,z')\,\mathrm{d}z'\mathrm{d}z \nonumber \\
&+ \O((\log N)^2/N^{1/2})|x|(1+|x|) \|\varphi''_N\|^{1/2}_{L^1}\|\varphi_N'\|^{5/2}_{L^1}. \label{eqn: Delta21-final}
\end{align}

\subsection{Computation $T_{2,2}$, $T_{2,3}$, $T_{2,4}$}\label{sec: comp-T22}

The computation of $T_{2,2}$ is almost identical (but simpler) to that in Proposition \ref{EAmscBprop}.
\begin{prop}\label{EAmscAprop}
There are constants for $t+\i\eta, s+\i\eta' \in \Omega_N$,
\begin{equation}
\begin{split}
&N\ee_j [ A^\circ_j(\tau+\i\eta) (1+\mathbb{E}B_j(z'))A_j^\circ(s+\i\eta') ]\\
=& \,  (1+t\partial_z m_{\mathrm{fc},t}(z')) \cdot \left(  2t^2  \frac{m_{\mathrm{fc},t}(\tau)-m_{\mathrm{fc},t}(s)}{\tau-s+\i(\eta-\eta')}\cdot(1+\O(tN^{-1}|\eta'|^{-2})) + t \right)\\
+&\O_\prec(t^2N^{-1}|\eta|^{-1}|\eta'|^{-1}).
\end{split}
\end{equation}
\end{prop}

We have shown that 
\begin{align*}
T_{2,2}(z) &=\frac{2\i x}{\pi}\mathbb{E}[e(x)]\int_{\Omega_N} \dvarp(z') \int(S_{2,2}(z,z')+ S_{2,3}(z,z')) \,\mathrm{d}z'\mathrm{d}z + \Delta_{2,2}(z)
\end{align*}
where 
\begin{align*}
S_{2,2}(z,z')&= \frac{t^2}{N}\sum_{j=1}^Ng_j(z)^2 g_j(z')^2(1+t\partial_z \mfct(z'))\frac{\mfct(z)-\mfct(z')}{z-z'},\\
S_{2,3}(z,z')&= \frac{t}{N}\sum_{j=1}^Ng_j(z)^2 g_j(z')^2(1+t\partial_z \mfct(z')).
\end{align*}
$\Delta_{2,2}(z)$ is analytic in $\Im z\neq 0$ and
\begin{align*}
& |\Delta_{2,2}(z)| =\frac{|x|}{N}\sum_{j=1}^N \frac{1}{|\mathbb{E}[A_j(z)]|^2}\int_{\Omega_N}\frac{1}{|\mathbb{E}[A_j(z')]^2|}\O(t^2N^{-1}|\eta|^{-1}|\eta'|^{-1})|\eta'||\varphi''(s)|\,\mathrm{d}z'\\
&+\frac{|x|}{N}\sum_{j=1}^N \frac{1}{|\mathbb{E}[A_j(z)]|^2}\int_{\Omega_N}\frac{1}{|\mathbb{E}[A_j(z')]^2|}(t^2\min(|\eta|^{-1},|\eta'|^{-1})+t)\O((tN)^{-1}(|\eta|^{-1}+|\eta'|^{-1}))|\eta'||\varphi''(s)|\,\mathrm{d}z'\\
&+ \left|\frac{x}{N}\sum_{j=1}^N \mathbb{E}[e(x)-e_j(x)]g_j(z)^2\int_{\Omega_N}g_j(z')^2\eta'\varphi_N''(s)\left( 2it^2\frac{\mfct(z)-\mfct(z')}{z-z'}+t \right)\,\mathrm{d}z'\right|\\
&= |x|(1+|x|)\O(|\eta|^{-1}\log N N^{-1/2}\|\varphi_N''\|_{L^1}^{1/2}\|\varphi_N'\|_{L^1}^{3/2}).
\end{align*}
We have used \eqref{eqn: R2sum-bound} and \eqref{eqn: stability}.

Using the derivative bound
\[\partial_\tau \frac{\Delta_{2,2}(z)}{1-t\tilde{R}_2(z)} =|x|(1+|x|)\O(|\eta|^{-2}\log N/N^{1/2})\|\varphi''_N\|^{1/2}_{L^1}\|\varphi_N'\|_{L^1}^{3/2}, \]
we have 
\begin{equation}\label{eqn: Delta22-final}
\begin{split}
\int_{\Omega_N} \dvarp(z) \frac{\Delta_{2,2}(z)}{1-t\tilde{R}_2(z)} \,\mathrm{d}z &= \int_{\Omega_N} \i\varphi_N'(\tau)\eta\chi(\eta)\partial_\tau \frac{\Delta_{2,2}(z)}{1-t\tilde{R}_2(z)}\mathrm{d}z\\
&= |x|(1+|x|)\O((\log N)^2/N^{1/2})\|\varphi'_N\|_{L^1}^{5/2}\|\varphi_N''\|_{L^1}^{1/2}.
\end{split}
\end{equation}

 For $T_{2,3}$, we use \eqref{eqn: Ahalf}, \eqref{eqn: Bbbound}, to estimate the integrand by
\begin{equation}
\label{eqn: T23-bound}
\frac{1}{N}\sum_{j=1}^N\frac{|x|}{|\mathbb{E}[A_j(z)]|^2}\frac{t}{|\mathbb{E}[A_j(z')]|^2}\O\left(t^2|\eta'|^{-1/2}|\eta|^{-1/2}+t^{3/2}|\eta'|^{-1/2}+t^{3/2}|\eta|^{-1/2}+t\right)N^{-1/2}|\eta'|^{-3/2}.
\end{equation}

The terms to estimate are
\begin{align}
T_{2,3}(z)&=\sum_{j=1}^N\frac{x}{\mathbb{E}[A_j(z)]^2} \int_{\Omega_N}  \i\eta'\chi(\eta')\varphi_N''(s)\frac{1}{\mathbb{E}[(A_j(z')]^2}\mathbb{E}[A^\circ_j(z')B^\circ_j(z') A_j^\circ(z)] \mathrm{d}z' \label{eqn: T23-ddphi-term}\\
&+\sum_{j=1}^N\frac{x}{\mathbb{E}[A_j(z)]^2} \int_{\Omega_N}  \i\varphi_N(s)\chi'(\eta')\frac{1}{\mathbb{E}[(A_j(z')]^2}\mathbb{E}[A^\circ_j(z')B^\circ_j(z') A_j^\circ(z)] \mathrm{d}z' \label{eqn: T23-phi-term}\\
&-\sum_{j=1}^N\frac{x}{\mathbb{E}[A_j(z)]^2} \int_{\Omega_N}  \eta'\varphi_N'(s)\chi'(\eta')\frac{1}{\mathbb{E}[(A_j(z')]^2}\mathbb{E}[A^\circ_j(z')B^\circ_j(z') A_j^\circ(z)] \mathrm{d}z'. \label{eqn: T23-dphi-term}
\end{align}
Use $\{z':|\Im z'|\ge N^{10C_V}-1\}$ on the support of the integrands, and \eqref{eqn: T23-bound} to estimate the terms \eqref{eqn: T23-phi-term}, \eqref{eqn: T23-dphi-term} by
\begin{equation}\label{eqn: T23-bound-z-1}
N^{-10C_V}\sum_{j=1}^N \frac{|x|}{|\mathbb{E}[A_j(z)]|^2}|\eta|^{-1/2}(\|\varphi_N\|_{L^1}+\|\varphi_N'\|_{L^1}).
\end{equation}
Inserting the bound \eqref{eqn: T23-bound} into \eqref{eqn: T23-ddphi-term}, using $|\mathbb{E}[A_j(z)]|\ge ct$, for $|\eta|\le t$ and $|\mathbb{E}[A_j(z)]|\ge|\eta|$ for $|\eta|\ge t$, we find
\begin{equation}\label{eqn: T23-bound-z-2}
\frac{1}{N}\sum_{j=1}^N \frac{t^{-1}|x|}{|\mathbb{E}{[A_j(z)]|^2}} N^{-1/2}\O(t^2|\eta|^{-1/2}\log N+t^{3/2}\log N+t^2|\eta|^{-1/2}+t^{3/2}) \|\varphi_N''\|_{L^1}.
\end{equation}
Integrating $(\dvarp(z) T_{2,3}(z))/(1-t\tilde{R}_2(z))$ over $\Omega_N$, and integrating by parts:
\begin{align*}
\int_{\Omega_N} \i\eta\chi(\eta) \varphi_N''(\tau) \frac{T_{2,3}(z)}{1-t\tilde{R}_2(z)} \,\mathrm{d}z=&\int_{\Omega_N} \i\partial_\eta(\eta\chi(\eta)) \varphi_N'(\tau) \frac{T_{2,3}(z)}{1-t\tilde{R}_2(z)} \,\mathrm{d}z\\
 -&N^{-1+\xi}\int  \i\chi(\eta) \varphi_N'(\tau) \frac{T_{2,3}(\tau+\i N^{-1+\xi})}{1-t\tilde{R}_2(\tau+\i N^{-1+\xi})} \,\mathrm{d}\tau
\end{align*}
Using \eqref{eqn: T23-bound-z-1}, \eqref{eqn: T23-bound-z-2}, these terms are bounded by
\begin{equation}\label{eqn: T23-final}
|x|\O(t^{1/2}N^{-1/2}\log N \|\varphi_N''\|_{L^1}\|\varphi_N\|_{L^1})= |x|\O(N^{-1/2+(1/2)\omega_0-\omega_1}\log N)\|\varphi_N''\|_{L^1}.
\end{equation}
The contribution to
\[\int_{\Omega_N} \dvarp(z) \frac{T_{2,3}(z)}{1-t\tilde{R}_2(z)} \,\mathrm{d}z\]
of the terms involving $\chi'(\eta)\varphi_N(\tau)$, $\chi'(\eta)\varphi_N(\tau)$ is easily estimated using the support property of $\chi'$, and found to be 
\[|x|\O(N^{-2C_V})\|\varphi_N''\|_{L^1}.\]

For $T_{2,4}$, we again have three terms
\begin{align}
T_{2,4} &= \sum_{j=1}^N\frac{x}{\mathbb{E}[A_j(z)]^2} \int_{\Omega_N}  \frac{\i \varphi_N''(s)\eta'\chi(\eta')}{\mathbb{E}[A_j(z')]^2}\mathbb{E}\left[\frac{(A_j^\circ(z'))^2}{A_j(z')}(1+B_j(z')) A_j^\circ(z)\right]\mathrm{d}z' \label{eqn: T24-ddphi-term}\\
&+ \sum_{j=1}^N\frac{x}{\mathbb{E}[A_j(z)]^2} \int_{\Omega_N}  \frac{\i\varphi_N(s)\chi'(\eta')}{\mathbb{E}[A_j(z')]^2}\mathbb{E}\left[\frac{(A_j^\circ(z'))^2}{A_j(z')}(1+B_j(z')) A_j^\circ(z)\right]\mathrm{d}z' \label{eqn: T24-phi-term}\\
&- \sum_{j=1}^N\frac{x}{\mathbb{E}[A_j(z)]^2} \int_{\Omega_N}  \frac{\i\varphi_N'(s)\eta'\chi'(\eta')}{\mathbb{E}[A_j(z')]^2}\mathbb{E}\left[\frac{(A_j^\circ(z'))^2}{A_j(z')}(1+B_j(z')) A_j^\circ(z)\right]\mathrm{d}z'. \label{eqn: T24-dphi-term}
\end{align}

We use \eqref{eqn: Ahalf}, \eqref{eqn: R2sum-bound} to find that the integrand in \eqref{eqn: T24-ddphi-term}-\eqref{eqn: T24-dphi-term} is bounded by
\begin{align}
&\frac{|x|}{N}\sum_{j=1}^N\frac{t}{\mathbb{E}[|A_j(z)]|^2}\frac{1}{|\mathbb{E}[A_j(z')]|^2}\mathbb{E}\frac{1}{|A_j(z')|}\notag\\
\times &\O\left(t^2|\eta'|^{-1}(N|\eta|)^{-1/2}+t^{3/2}|\eta'|^{-1}N^{-1/2}+t(N|\eta|)^{-1/2}+t^{1/2}N^{-1/2}\right). \label{eqn: T24-kernel-bound}
\end{align} 

For \eqref{eqn: T24-ddphi-term}, we first integrate by parts to write this term as
\begin{equation}\label{eqn: T24-bound}
\begin{split}
&\sum_{j=1}^N\frac{x}{\mathbb{E}[A_j(z)]^2} \int_{\Omega_N}  \frac{\i\varphi_N'(s)\partial_{\eta'}(\eta'\chi(\eta'))}{\mathbb{E}[A_j(z')]^2}\mathbb{E}\left[\frac{(A_j^\circ(z'))^2}{A_j(z')}(1+B_j(z')) A_j^\circ(z)\right]\mathrm{d}z'\\
&-\sum_{j=1}^N\frac{xN^{-1+\xi}}{\mathbb{E}[A_j(z)]^2} \int_{\Omega_N}  \frac{\i\varphi_N'(s)\chi(\eta')}{\mathbb{E}[A_j(s+\i N^{-1+\xi})]^2}\mathbb{E}\left[\frac{(A_j^\circ(s+\i N^{-1+\xi}))^2}{A_j(s+\i N^{-1+\xi})}(1+B_j(s+\i N^{-1+\xi})) A_j^\circ(z)\right]\mathrm{d}s.
\end{split}
\end{equation}

Using \eqref{eqn: T24-kernel-bound} in \eqref{eqn: T24-bound}, together with the estimate \eqref{eqn: stability} when $|\eta'|\le t$ and $|A_j(z')|\ge c|\eta'|$ when $|\eta'|\ge t$, \eqref{eqn: T24-ddphi-term} is bounded by
\begin{equation}\label{eqn: T24-main-estimate}
\frac{1}{N}\sum_{j=1}^N \frac{|x|}{|\mathbb{E}[A_j(z)]|^2}\O(\log N (N|\eta|)^{-1/2}+ t^{-1/2}N^{-1/2}\log N+ t^{-1/2}N^{-1/2}) \|\varphi_N'\|_{L^1}
\end{equation}
Using \eqref{eqn: T24-kernel-bound} again and $\{\eta:\chi'(\eta)\neq 0\}\subset \{|\eta|\ge N^{C_V}-1\}$, the terms \eqref{eqn: T24-phi-term} and \eqref{eqn: T24-dphi-term} are estimated by
\[|x|\O(N^{-2C_V}\|\varphi_N''\|_{L^1}).\]

Using the bound \eqref{eqn: T24-main-estimate}, we now conclude as in the case of $T_{2,3}$, by integrating by parts in $\tau$:
\begin{equation}\label{eqn: T24-final}
\begin{split}
&\int_{\Omega_N} \dvarp(z) \frac{T_{2,4}(z)}{1-t\tilde{R}_2(z)} \,\mathrm{d}z\\
=&\frac{|x|}{N}\sum_{j=1}\int_{\Omega_N} \partial_\eta(\eta\chi(\eta)) \varphi_N'(\tau)\frac{|T_{2,4}(z)|}{|1-t\tilde{R}_2(z)|}  \,\mathrm{d}z+ \O(N^{-2})\\
=& |x|\O((\log N) N^{-1/2}t^{-1/2})\|\varphi_N'\|_{L^1}^2.
\end{split}
\end{equation}

Collecting the error terms \eqref{eqn: Delta11-error}, \eqref{eqn: Delta21-final}, \eqref{eqn: Delta22-final}, \eqref{eqn: T23-final}, \eqref{eqn: T24-final}, we obtain
\begin{align}
I_1'&=-\frac{2\i x}{\pi}\mathbb{E}[e(x)]\int_{\Omega_N}\int_{\Omega_N}\dvarp(z)\dvarp(z')\frac{1}{1-tR_2(z)} S_{2,1}(z,z')\,\mathrm{d}z\mathrm{d}z' \label{eqn: I2-1}\\
&\quad +\frac{2\i x}{\pi}\mathbb{E}[e(x)]\int_{\Omega_N}\int_{\Omega_N}\dvarp(z)\dvarp(z')\frac{1}{1-tR_2(z)}(S_{2,2}(z,z')+S_{2,3}(z,z'))\,\mathrm{d}z\mathrm{d}z' \label{eqn: I2-2}\\
&\quad + |x|\O(t^{1/2}N^{-1/2+2\xi})\|\varphi''_N\|_{L^1}\|\varphi_N'\|_{L^1}\nonumber\\
&\quad +|x|(1+|x|)\O(N^{-1/2}(\log N)^2)\|\varphi_N'\|^{5/2}_{L^1}\|\varphi_N''\|^{1/2}_{L^1} \nonumber \\
&\quad +|x|\O(N^{-1/2}\log N)(t^{1/2}\|\varphi_N''\|_{L^1}\|\varphi_N\|_{L^1}+ t^{-1/2}\|\varphi_N'\|_{L^1}^2). \nonumber
\end{align}
This ends the proof of Proposition \ref{prop: I2prop}.
\qed

\subsection{Variance term}\label{sec: variance}
In this section, we give an asymptotic approximation of the expression $V(\varphi_N)$ defined in  \eqref{eqn: Vvarphidef}. This quantity represents the variance of the limiting random variable for the linear statistics of $\varphi_N$. The result is as follows
\begin{prop}\label{prop: variance-comp}
Recall the definition of $V(\varphi_N)$ in \eqref{eqn: Vvarphidef}. Then
\begin{equation}\label{eqn: Vvarphi-size}
V(\varphi_N) = -\frac{1}{\pi^2}\int_{-Ct}^{Ct}\varphi_N(\tau) (H\varphi_N')(\tau)\,\mathrm{d}\tau +\O(1).
\end{equation}
Here, $Hf$ denotes the Hilbert transform:
\begin{equation}\label{eqn: HT-def}
(Hf)(x)=\lim_{\epsilon\rightarrow 0} \int f(y) \Re \frac{1}{(x-y) +\i\epsilon}\,\mathrm{d}y.
\end{equation}
In particular, for 
\[\varphi_N(x) = \int_0^x \chi(y/(t_1N^\alpha)) p_{t_1}(0,y)\mathrm{d}y,\]
we have
\[V(\varphi_N) \ge c\log(t/t_1) \cdot (1+o(1)).\] 

Moreover, if 
\[\mathrm{supp}\varphi_N \subset (-N^rt_1,N^rt_1),\]
then
\[V(\varphi_N) = -\frac{1}{\pi^2}\int\varphi_N(\tau)(H\varphi'_N)(\tau)\,\mathrm{d}\tau+\O(N^{\omega_0/2-\omega_1+\xi})\]
for any $\xi>0$.
\end{prop}
We begin by reducing the domain of integration. Define 
\begin{align*}
\Omega_N^* = \{z=E+\i\eta: E\in \mathcal{I}_q,\, N^{-1+\xi}<|\eta|\le N^{10C_V}\}.
\end{align*}

Note that
\begin{equation}\label{eqn: high-region}
\{\partial_{\bar{z}}\tilde{\varphi}_N \neq 0\}\cap (\Omega_N\setminus \Omega_N^*)\subset \{N^{10C_V}-1<|\eta|<N^{10C_V}\}.
\end{equation}
If either $z$ or $z'$ lies in the latter region, then
\begin{equation}\label{eqn: high-region-bound}
\begin{split}
\frac{1}{1-tR_2(z)}S_{2,1}(z,z') &= \frac{t^2}{N(1-tR_2(z))}\sum_{j=1}^N g_j(z)^2g_j(z') \partial_{z'}\frac{\mfct(z)-\mfct(z')}{z-z'}= \O_\prec(t^2N^{-4C_V}).
\end{split}
\end{equation}
Similarly ,
\begin{equation}\label{eqn: high-region-bound2}
\frac{1}{1-tR_2(z)}S_{2,2}(z,z') =\O(t^2N^{-4C_V}).
\end{equation}
With \eqref{eqn: high-region-bound}, \eqref{eqn: high-region-bound2}, it is easy to show that the domain of integration $\Omega_N\times \Omega_N$ in \eqref{eqn: I2-1}, \eqref{eqn: I2-2} can be replaced by $\Omega_N^*\times \Omega_N^*$ with an error $\O(N^{-2})$.

Next, we have the following:
\begin{prop}\label{prop: sum-gg}
\begin{equation}
\frac{1}{N}\sum_{j=1}^N \frac{g_j(z)^2}{1-tR_2(z)}g_j(z') = \partial_z \frac{\mfct(z)-\mfct(z')}{z-z'+t(\mfct(z)-\mfct(z'))}.
\end{equation}
\end{prop}
\begin{proof}
By \cite[Eqn. (7.24)]{landonyau}
\begin{align}
\partial_z \mfct(z)&= \frac{R_2(z)}{1-tR_2(z)}, \quad 1+t\partial_z \mfct(z)= \frac{1}{1-tR_2(z)}.\label{eqn: t-shcherbina}
\end{align}
By partial fractions, we have
\begin{align*}
&\frac{1}{N}\sum_{j=1}^N\frac{g_j(z)^2}{1-tR_2(z)}g_j(z') =\partial_z \frac{1}{N}\sum_{j=1}^N g_j(z)g_j(z') =\frac{1}{N} \partial_z \frac{\mfct(z)-\mfct(z')}{z-z'+t(\mfct(z)-\mfct(z'))}.
\end{align*}
\end{proof}
\qed

The integrals appearing in the definition of $V(\varphi_N)$ are
\begin{align}
I_{1,1}&:=\int_{\Omega_N^*}\int_{\Omega_N^*}\dvarp(z)\dvarp(z')\frac{1}{1-tR_2(z)} S_{2,1}(z,z')\,\mathrm{d}z\mathrm{d}z', \label{eqn: I2-11}\\
I_{1,2}&:=\int_{\Omega_N^*}\int_{\Omega_N^*}\dvarp(z)\dvarp(z')\frac{1}{1-tR_2(z)} S_{2,2}(z,z')\,\mathrm{d}z\mathrm{d}z' \label{eqn: I2-22}\\
I_{1,3}&:= \int_{\Omega_N^*}\int_{\Omega_N^*}\dvarp(z)\dvarp(z')\frac{1}{1-tR_2(z)} S_{2,3}(z,z')\,\mathrm{d}z\mathrm{d}z. \label{eqn: I2-23}
\end{align}

By Proposition \ref{prop: sum-gg},
\begin{equation}\label{eqn: I2-display}
I_{1,1}= \frac{t^2}{N}\int_{\Omega_N^*}\int_{\Omega_N^*} \dvarp(z)\dvarp(z') \partial_z \sum_{j=1}^N g_j(z)g_j(z') \partial_s \frac{\mfct(z)-\mfct(z')}{z-z'}\,\mathrm{d}z\mathrm{d}z'.
\end{equation}
Similarly, we have:
\begin{align*}
\frac{S_{2,2}(z,z')}{1-tR_2(z)}&=\frac{t^2}{N}\partial_z\partial_{z'} \sum_{j=1}g_j(z)g_j(z')\frac{\mfct(z)-\mfct(z')}{z-z'},\\
\frac{S_{2,3}(z,z')}{1-tR_2(z)}&=\frac{t}{N}\partial_z\partial_{z'} \sum_{j=1}g_j(z)g_j(z').
\end{align*}

Integrating by parts in $s=\Re z'$. The boundary term is only non-zero in the region \eqref{eqn: high-region}, where we can use \eqref{eqn: high-region-bound}.
\begin{align}
I_{1,2}=&-\frac{t^2}{N}\int_{\Omega_N^*}\int_{\Omega_N^*} \dvarp(z)\partial_{s}\dvarp(z') \partial_z \sum_{j=1}^N g_j(z)g_j(z')  \frac{\mfct(z)-\mfct(z')}{z-z'}\,\mathrm{d}z\mathrm{d}z' \label{eqn: I2-d-onphi}\\
&-\frac{t^2}{N}\int_{\Omega_N^*}\int_{\Omega_N^*} \dvarp(z)\dvarp(z') \partial_z \sum_{j=1}^N g_j(z)g_j(z') \partial_s \frac{\mfct(z)-\mfct(z')}{z-z'}\,\mathrm{d}z\mathrm{d}z' \label{eqn: I2-d-onzz'}
\end{align}
Note that the second integral \eqref{eqn: I2-d-onzz'} is equal to $I_{1,1}$.

We begin by computing the $z$ integral in \eqref{eqn: I2-d-onphi}. The integrand is $\dvarp(z)$ multiplying a function analytic in each of $\{\Im z>0\}$ and $\{\Im z<0\}$. Let $\Omega\subset \mathbb{C}$ be a domain. For $F$ a $C^1(\Omega)$ function, Green's theorem in complex notation is 
\begin{equation}\label{eqn: greensthm}
\int_\Omega \bar{\partial}_zF(z) \,\mathrm{d}z = -\frac{\i}{2}\int_{\partial \Omega}F(z) \,\mathrm{d}z.
\end{equation}

 We split the integral \eqref{eqn: I2-d-onphi} into the two regions $\Omega_N^*\cap \{\Im z>0\}$, $\Omega_N^*\cap\{\Im z<0\}$ and apply Green's theorem to each. The first region is a rectangle in the upper half-plane. The integrand in the resulting line integral, $\tilde{\varphi}_N$, is zero on the ``top'' segment $[-qG+\i N^{10C_V},qG+\i N^{10C_V}]$. 

We label the terms corresponding to three other boundary line integrals by $(+)$ to denote $\Im z>0$ and number them according to the corresponding boundary segments as $(1)$ for $[-qG+\i N^{-1+\xi},qG+\i N^{-1+\xi}]$; $(2)$ for $[qG+\i N^{-1+\xi},qG+\i N^{10C_V}]$;  and $(3)$ for $[-qG+iN^{10C_V},-qG+iN^{10C_V}]$:
\begin{align}
&2\i \int_{\Omega_N^*\cap\{\Im z>0\}} \dvarp(z) \frac{t^2}{N} \partial_z \sum_{j=1}^N g_j(z)g_j(z')  \frac{\mfct(z)-\mfct(z')}{z-z'}\,\mathrm{d}z \nonumber \\
=&t^2\int_{-qG}^{qG} (\varphi_N(\tau)+\i N^{-1+\xi}\varphi_N'(\tau)) \partial_\tau\frac{\mfct(\tau+\i N^{-1+\xi})-\mfct(z')}{\tau+\i N^{-1+\xi}-z'+t(\mfct(\tau+\i N^{-1+\xi})-\mfct(z'))} \label{eqn: I2+1} \\
&\quad \times  \frac{\mfct(\tau+\i N^{-1+\xi})-\mfct(z')}{\tau+\i N^{-1+\xi}-z'}\,\mathrm{d}\tau \nonumber\\
+&t^2\varphi_N(qG)\int_{N^{-1+\xi}}^{N^{10C_V}+1}\chi(\eta) \partial_\tau \frac{\mfct(qG+\i\eta)-\mfct(z')}{qG+\i\eta-z'+t(\mfct(qG+\i\eta)-\mfct(z'))}  \frac{\mfct(qE+\i\eta)-\mfct(z')}{qG+\i\eta-z'}\,\mathrm{d}\eta \label{eqn: I2+2}\\
+&t^2\varphi_N(-qG)\int_{N^{10C_V}+1}^{N^{-1+\xi}}\chi(\eta) \partial_\tau \frac{\mfct(-qG+\i\eta)-\mfct(z')}{-qG+\i\eta-z'+t(\mfct(-qG+\i\eta)-\mfct(z'))}   \notag\\
&\times \frac{\mfct(-qG+\i\eta)-\mfct(z')}{-qG+\i\eta-z'}\,\mathrm{d}\eta \label{eqn: I2+3}\\
&:= I_{1,+,1}+I_{1,+,2}+I_{1,+,3}.
\end{align}
Similarly, the second region $\Omega_N^*\cap\{\Im z <0\}$ is labelled by $(-)$ in indices. The sides are labelled in counter-clockwise orientation as $(1)$, $[qG-\i N^{-1+\xi},-qG-\i N^{-1+\xi}]$; $(2)$, $[-qG-\i N^{-1+\xi},-qG-\i N^{10C_V}]$; $(3)$, $[qG-\i N^{10C_V},qG-\i N^{-1+\xi}]$. Applying Green's theorem to \eqref{eqn: I2-d-onphi} over this region:
\begin{align}
&2\i \int_{\Omega_N^*\cap\{\Im z<0\}} \dvarp(z) \frac{t^2}{N} \partial_z \sum_{j=1}^N g_j(z)g_j(z')  \frac{\mfct(z)-\mfct(z')}{z-z'}\,\mathrm{d}z \nonumber \\
=&t^2\int_{qG}^{-qG} (\varphi_N(\tau)-\i N^{-1+\xi}\varphi_N'(\tau)) \partial_z\frac{\mfct(\tau-\i N^{-1+\xi})-\mfct(z')}{\tau-z'+t(\mfct(\tau-\i N^{-1+\xi})-\mfct(z'))} \label{eqn: I2minus1}\\
&\quad \times  \frac{\mfct(\tau-\i N^{-1+\xi})-\mfct(z')}{\tau-\i N^{-1+\xi}-z'}\,\mathrm{d}x\\
+&t^2\varphi_N(-qG)\int_{-N^{-1+\xi}}^{-N^{10C_V}-1}\chi(\eta) \partial_z \frac{\mfct(-qG+\i\eta)-\mfct(z')}{-qE+i\eta-z'+t(\mfct(-qG+\i\eta)-\mfct(z'))} \notag\\
& \times  \frac{\mfct(-qG+\i y)-\mfct(z')}{-qG+\i\eta-z'}\,\mathrm{d}\eta \label{eqn: I2minus2}\\
+&t^2\varphi_N(qG)\int_{-N^{10C_V}-1}^{-N^{-1+\xi}}\chi(\eta) \partial_z \frac{\mfct(qG+\i\eta)-\mfct(z')}{qE+i\eta-z'+t(\mfct(qG+\i y)-\mfct(z'))}  \frac{\mfct(qE+\i\eta)-\mfct(z')}{qG+\i\eta-z'}\,\mathrm{d}\eta \label{eqn: I2minus3}\\
&:= I_{1,-,1}+I_{1,-,2}+I_{1,-,3}.\nonumber
\end{align}

We now insert $-(\i/2)I_{1,\pm,k}(z')$, $k=1,2,3$, into the integral \eqref{eqn: I2-d-onphi}, and apply Green's theorem in each of the regions $\Omega_N^*\cap \{\Im z'>0\}$ and $\Omega_N^*\cap \{\Im z'<0\}$. We label the oriented sides of that region as previously:
\begin{align}
&-(\i/2)\int_{\Omega_N^* \cap \{\Im z'>0\}} \partial_s\dvarp(z') I_{1,\pm,k}(z')\,\mathrm{d}z' \nonumber \\
&= \frac{1}{4}\int_{-qG}^{qG}(\varphi_N'(s)+\i N^{-1+\xi}\varphi_N''(s)) I_{1,\pm,k}(s+\i N^{-1+\xi})\,\mathrm{d}s \label{eqn: I2pm1} \\
&+ \frac{1}{4}\int_{N^{-1+\xi}}^{N^{10C_V}}\chi(\eta')(\varphi_N'(qG)+\i\eta'\varphi''_N(qG))I_{1,\pm,k}(qG+\i\eta')\,\mathrm{d}\eta' \label{eqn: I2pm-vanishing-1}\\
&+ \frac{1}{4}\int_{N^{10C_V}}^{N^{-1+\xi}}\chi(\eta')(\varphi_N'(-qG)+\i\eta'\varphi''_N(-qG))I_{1,\pm,k}(qG+\i\eta')\,\mathrm{d}\eta' \label{eqn: I2pm-vanishing-2}
\end{align}
By the support condition \eqref{eqn: offIq-decay}, the terms \eqref{eqn: I2pm-vanishing-1} and \eqref{eqn: I2pm-vanishing-2} are 0 for any $k$ and choice of $\pm$. We denote the remaining term \eqref{eqn: I2pm1} by $I_{1,\pm,k,+}$. Similarly, applying Green's theorem to $\Omega_N^*\cap \{\Im z'<0\}$:
\begin{align}
&-(\i/2)\int_{\Omega_N^* \cap \{\Im z'>0\}}\dvarp(z') I_{1,\pm,j}(z')\,\mathrm{d}z' \nonumber \\
&= \frac{1}{4}\int_{qG}^{-qG}(\varphi_N'(s)-\i N^{-1+\xi}\varphi_N''(s)) I_{2,\pm,k}(x'-\i N^{-1+\xi})\,\mathrm{d}x' \label{eqn: I2pm1-2}\\
&:= I_{1,\pm,k,-} \nonumber
\end{align}
To summarize, we have shown so far
\begin{align*}
&\frac{t^2}{N}\int_{\Omega_N^*}\int_{\Omega_N^*} \dvarp(z)\partial_{z'}\dvarp(z') \partial_z \sum_{j=1}^N g_j(z)g_j(z')  \frac{\mfct(z)-\mfct(z')}{z-z'}\,\mathrm{d}z\mathrm{d}z = \sum_{k=1}^3\sum_{\alpha,\beta \in \{\pm\}}I_{1,\alpha,k,\beta}.
\end{align*}

Only the terms $I_{1,\pm,1,\pm}$ contribute to the variance. This is the content of the following.
\begin{prop}\label{prop: I2-bound-prop}
Recall the parameter $\sigma>0$ in the local law, Lemma \ref{lem: gDd}. For any choice of $I_{2,\pm,k,\pm}$ with $k \neq 1$, we have
\[|I_{1,\pm,k,\pm}|\le \O(t^\sigma \log N (\|\varphi_N'\|_{L^1}+N^{-1+\xi}\|\varphi_N''\|_{L^1})).\]
\end{prop}
\begin{proof}
By \cite[Eqn (7.25)]{landonyau}, 
\begin{equation}\label{eqn: dmfct-bound}
|\partial_z \mfct(z)|\le C t^{-1}.
\end{equation}
Compute the derivative:
\begin{align*}
\partial_z\frac{\mfct(z)-\mfct(z')}{z-z'+t(\mfct(z)-\mfct(z'))}&=-\frac{\mfct(z)-\mfct(z')}{(z-z'+t(\mfct(z)-\mfct(z')))^2}(1+t\partial_z\mfct(z))\\
&+ \frac{\partial_z \mfct(z)}{z-z'+t(\mfct(z)-\mfct(z'))}.
\end{align*}
Note that $t(\mfct(z)-\mfct(z'))=\O_\prec(t)$ for $z,z'\in \mathcal{D}_{\epsilon,q}$. So if $|\Re z|\ge qG\ge t^{1/2}N^{\sigma/2}$ and $|\Re z'| =\O(t_1N^r)$, then
\begin{equation}\label{eqn: t-3/2}
\partial_z\frac{\mfct(z)-\mfct(z')}{z-z'+t(\mfct(z)-\mfct(z'))} \prec t^{-1}t^{-1/2}N^{-\sigma/2}\le t^{-3/2+\sigma/2}.
\end{equation}
Similarly:
\[\left|\frac{\mfct(z)-\mfct(z')}{z-z'}\right|\prec \min(t^{-1/2+\sigma/2},|\eta|^{-1}).\]

By \eqref{eqn: stability}, and $\frac{1}{N}\sum |g_j(z)|\le C\log N$ \cite[Eqn. (7.36)]{landonyau}, so
\[
\left|\frac{1}{N}\sum_{j=1}^N g_j(z)^2(1+t\partial_z \mfct(z))g_j(z')\right|\le C|\eta|^{-2} \log N.
\]
Combining this with \eqref{eqn: t-3/2}, we have 
\[\left|\frac{1}{N}\sum_{j=1}^N g_j(z)^2(1+t\partial_z \mfct(z))g_j(z')\right|\prec \min(|\eta|^{-2} \log N, t^{-3/2+\sigma/2}).\]
Inserting this into \eqref{eqn: I2+2}, \eqref{eqn: I2+3}, \eqref{eqn: I2minus2}, \eqref{eqn: I2minus3}, for $z'$ such that $s=\Re z'\in \mathrm{supp}\varphi_N'$ and $k=2,3$:
\begin{align*}
|I_{1,\pm,k,\pm}|&\le t^2\log N(\|\varphi_N'\|_{L^1}+N^{-1+\xi}\|\varphi_N''\|_{L^1})\int_{N^{-1+\xi}}^{N^{10C_V}}\min(t^{-2+\sigma},|\eta|^{-3})\mathrm{d}\eta\\
&\le t^{\sigma} \log N (\|\varphi_N'\|_{L^1}+N^{-1+\xi}\|\varphi_N''\|_{L^1}).
\end{align*}
\end{proof}\qed

For brevity of notation, we let $s^\pm =s\pm \i N^{-1+\xi}$, $\tau^\pm = \tau\pm \i N^{-1+\xi}$. We have so far shown that 
\begin{align}
&\int\int_{\Omega_N \times \Omega_N} \dvarp(z)\dvarp(z') \partial_z \sum_{j=1}^N g_j(z)g_j(z')  \frac{\mfct(z)-\mfct(z')}{z-z'}\,\mathrm{d}z'\mathrm{d}z \label{eqn: mainvariance}\\
=&\frac{t^2}{4}\int_{-qG}^{qG}\int_{-qG}^{qG}\tilde{\varphi}_N(\tau)\tilde{\varphi}_N'(s)\partial_\tau \frac{\mfct(\tau^+)-\mfct(s^+)}{\tau-s+t(\mfct(\tau^+)-\mfct(s^+))} \frac{\mfct(\tau^+)-\mfct(s^+)}{\tau-s}\,\mathrm{d}\tau\mathrm{d}s \label{eqn: plusvariance-1}\\
+&\frac{t^2}{4}\int_{-qG}^{qG}\int_{-qG}^{qG}\tilde{\varphi}_N(\tau)\tilde{\varphi}_N'(s)\partial_\tau \frac{\mfct(\tau^-)-\mfct(s^-)}{\tau-s+t(\mfct(\tau^-)-\mfct(s^-))} \frac{\mfct(\tau^-)-\mfct(s^-)}{\tau-s}\,\mathrm{d}\tau\mathrm{d}s \label{eqn: plusvariance-2}\\
-&\frac{t^2}{4}\int_{-qG}^{qG}\int_{-qG}^{qG}\tilde{\varphi}_N(\tau)\tilde{\varphi}_N'(s)\partial_\tau \frac{\mfct(\tau^+)-\mfct(s^-)}{\tau-s+t(\mfct(\tau^+)-\mfct(s^-))} \frac{\mfct(\tau^+)-\mfct(s^-)}{\tau-s+2iN^{-1+\xi}}\,\mathrm{d}\tau\mathrm{d}s \label{eqn: mainvariance-1a}\\
-&\frac{t^2}{4}\int_{-qG}^{qG}\int_{-qG}^{qG}\tilde{\varphi}_N(\tau)\tilde{\varphi}_N'(s)\partial_\tau \frac{\mfct(\tau^-)-\mfct(s^+)}{\tau-s+t(\mfct(\tau^-)-\mfct(s^+))} \frac{\mfct(\tau^-)-\mfct(s^+)}{\tau-s-2iN^{-1+\xi}}\,\mathrm{d}\tau\mathrm{d}s \label{eqn: mainvariance-1b}\\
+& \O(t^\sigma \log N (\|\varphi_N'\|_{L^1}+N^{-1+\xi}\|\varphi_N''\|_{L^1})).
\end{align}
The main terms are \eqref{eqn: mainvariance-1a}, \eqref{eqn: mainvariance-1b}. These are of order $\log N$ for the functions we are interested in. The other two terms are bounded by a constant:
\begin{prop}\label{prop: const-bd}
Let $s^\pm = s\pm \i N^{-1+\xi}$, $\tau^\pm = \tau \pm \i N^{-1+\xi}$.
There is a constant $C$ such that
\begin{align}
\left|t^2\int_{-qG}^{qG}\int_{-qG}^{qG}\tilde{\varphi}_N(\tau)\tilde{\varphi}_N'(s)\partial_\tau \frac{\mfct(\tau^+)-\mfct(s^+)}{\tau-s+t(\mfct(\tau^+)-\mfct(s^+))} \frac{\mfct(\tau^+)-\mfct(s^+)}{\tau-s}\,\mathrm{d}\tau\mathrm{d}s\right| &\le C \label{eqn: mfct-frac-1}\\
\left|t^2\int_{-qG}^{qG}\int_{-qG}^{qG}\tilde{\varphi}_N(\tau)\tilde{\varphi}_N'(s)\partial_\tau \frac{\mfct(\tau^-)-\mfct(s^-)}{\tau-s+t(\mfct(\tau^-)-\mfct(s^-))} \frac{\mfct(\tau^-)-\mfct(s^-)}{\tau-s}\,\mathrm{d}\tau\mathrm{d}s\right| &\le C.
\end{align}
\end{prop}
\begin{proof}
First note the estimate
\begin{equation}\label{eqn: dfracmfct-bound}
\left|\partial_\tau \frac{\mfct(\tau^+)-\mfct(s^+)}{\tau-s+t(\mfct(\tau^+)-\mfct(s^+))}\right|\le Ct^{-2},
\end{equation}
for $|\tau-s|\le Ct$, which follows from the alternate representation 
\begin{equation}\label{eqn: mfct-partialfrac-alt}
\frac{1}{N}\sum_{j=1}^N g_j(s)^2(1+t\partial_s \mfct(s))g_j(\tau),
\end{equation}
\eqref{eqn: dmfct-bound}, \eqref{eqn: stability}, and the bound \cite[Eqn (7.24)]{landonyau}
\begin{equation}\label{eqn: gj2-bound}
\frac{1}{N}\sum_{j=1}^N \frac{1}{|V_j-z-t\mfct(z)|^2}\le Ct^{-1}.
\end{equation}

Define $\zeta(z):=z+t\mfct(z)$. We begin by noting that \eqref{eqn: t-shcherbina} implies, for $z=\tau \pm \i N^{-1+\xi}, z'=s+\i N^{-1+\xi}\in \Omega_N$
\begin{equation}
\begin{split}\label{eqn: zeta-lwr-bound}
\left|\Re(\zeta(z)-\zeta(z'))\right|  &= \left|\Re \int_\tau^{s}\partial_x \zeta(x+\i N^{-1+\xi})\,\mathrm{d}x\right| = \left|\int_{\tau}^{s}\Re \frac{1}{1-tR_2(x+\i N^{-1+\xi})}\,\mathrm{d}x\right|\\
&= \left|\int\int_\tau^z \frac{\Re(1-tR_2(x-\i N^{-1+\xi})}{|1-tR_2(x+\i N^{-1+\xi})|^2}\,\mathrm{d}x\right| \geq C|\tau-s|.
\end{split}
\end{equation}
In the second to last step we have used \eqref{eqn: gj2-bound} as well as the lower bound
$\Re(1-tR_2(z))\geq c,$
(see \cite[Lemma 7.2]{landonyau}).

We then estimate the integral in \eqref{eqn: mfct-frac-1} as
\begin{equation}\label{eqn: mfct-frac-split}
\begin{split}
&\int_{-qG}^{qG}|\varphi_N'(s)|\int_{\tau:|\tau-s|< Mt}\frac{|\mfct(\tau^+)-\mfct(\tau^+)|}{|\tau-s|}\,\mathrm{d}\tau\mathrm{d}s\\
+&t^2 \int_{-qG}^{qG}|\varphi_N'(s)|\int_{\tau:|\tau-s|\ge Mt}\left|\partial_\tau \frac{\mfct(\tau^+)-\mfct(s^+)}{\tau-s+t(\mfct(\tau^+)-\mfct(s^+))}\right|\frac{|\mfct(\tau^+)-\mfct(\tau^+)|}{|\tau-s|}\,\mathrm{d}\tau\mathrm{d}s,
\end{split}
\end{equation}
where $M$ is some constant.
In the range $\{\tau: |\tau-s|< Mt\}$, we use \eqref{eqn: dmfct-bound} in the inner integral, to obtain a bound of constant order. Using that $|\mfct|\le C$ in $\mathcal{I}_q$, and \eqref{eqn: zeta-lwr-bound}, we have for $|\tau-s|\ge Mt$:
\begin{equation}
\left|\partial_\tau \frac{\mfct(\tau^+)-\mfct(s^+)}{\tau-s+t(\mfct(\tau^+)-\mfct(s^+))}\right|\le C\frac{t^{-1}}{|s-\tau|},
\end{equation}
while 
\[\frac{|\mfct(\tau^+)-\mfct(\tau^-)|}{|\tau-s|}\le \frac{C}{|s-\tau|}.\]
Integrating over $\{\tau:|\tau-s|\ge Mt\}$ then again gives a constant bound for the $\tau$ integral. Since $\|\varphi_N'\|_{L^1}\le C$ by assumption, we are done.
\end{proof}\qed

Summing the two terms \eqref{eqn: mainvariance-1a}, \eqref{eqn: mainvariance-1b}, we find a kernel multiplying $\tilde{\varphi}_N'(s)\tilde{\varphi}_N(\tau)$, equal to
\begin{align*}
&2t^2\Re\, \partial_\tau \frac{\mfct(\tau^+)-\mfct(s^-)}{\tau-s-2\i N^{-1+\xi}+t(\mfct(\tau^+)-\mfct(s^-))} \frac{\mfct(\tau^+)-\mfct(s^-)}{\tau-s-2\i N^{-1+\xi}}\\
=&2t^2\Re \, \frac{\partial_\tau \mfct(\tau^+)(\tau -s+2\i N^{-1+\xi})-\mfct(\tau^+)+\mfct(s^-)}{(\tau-s-2\i N^{-1+\xi}+t(\mfct(\tau^+)-\mfct(s^-)))^2} \frac{\mfct(\tau^+)-\mfct(s^-)}{\tau-s-2\i N^{-1+\xi}}.
\end{align*}
Recall:
\begin{equation*}
\lim_{\epsilon \rightarrow 0} \mfct(x+\i\epsilon) = (H\rhofct)(x) + \i\pi \rho(x),
\end{equation*}
so
\begin{equation}\label{eqn: taylor}
\mfct(\tau^+)-\mfct(s^-) = 2\i\pi \rhofct(\tau)+ \max_z|\mfct'(z)|\cdot \O(|\tau-s|),
\end{equation}
so the kernel is 
\[- 2\Re \frac{1}{\tau-s+2\i N^{-1+\xi}} +\max_z|\mfct'(z)|\frac{t^2}{t^2+|\tau-s|^2}\O\left(\frac{|\tau-s|}{|\tau-s+2\i N^{-1+\xi}|}\right),\]
when $|\tau-s|\le Mt$, provided $M$ is sufficiently small.
Integrating $\tilde{\varphi}_N(\tau)\tilde{\varphi}_N'(s)$ against the error term, using \eqref{eqn: dmfct-bound}, and splitting the $\tau$ integral according to $|\tau-s|\le Mt$ and $|\tau-s|> Mt$, we find an error term of 
\[C(\|\varphi_N'\|_{L^1}+N^{-1+\xi}\|\varphi_N''\|_{L^1})\|\varphi_N\|_{L^1}\le C.\]
The main term of \eqref{eqn: mainvariance} is then
\begin{equation}\label{eqn: hilbert-transform}
\begin{split}
&\frac{1}{2}\int_{-tM<\tau<tM}\tilde{\varphi}_N(\tau)\int\tilde{\varphi}_N'(s)\Re  \frac{1} {\tau-s+2\i N^{-1+\xi}}\,\mathrm{d}s\mathrm{d}\tau\\
=&\frac{1}{2}\int_{-tM<\tau<tM}\varphi_N(\tau)\int\varphi_N'(s)\Re  \frac{1} {\tau-s+2\i N^{-1+\xi}}\,\mathrm{d}s\mathrm{d}\tau + \O(1)\\
=&\frac{1}{2}\int_{-tM<\tau<tM}\varphi_N(\tau)(H\varphi_N')(\tau)\,\mathrm{d}\tau + \O(1).
\end{split}
\end{equation}
For the second step, we have used
\begin{align*}
&\Re \int_{|\tau-s|\le N^{10}} \frac{1}{\tau-s+\i 2N^{-1+\xi}}\,\mathrm{d}\tau = \int_{|\tau-s|\le N^{10}} \frac{\tau -s}{(\tau-s)^2+4N^{-2+2\xi}}\,\mathrm{d}\tau = 0,
\end{align*}
to write:
\begin{align*}
&\Re \int_{-tM<\tau<tM} \varphi_N(\tau)\int \varphi_N'(s)\frac{1}{\tau-s+\i 2N^{-1+\xi}}\,\mathrm{d}s\mathrm{d}\tau\\
=&\Re \int_{-tM<\tau<tM} \varphi_N(\tau)\int (\varphi_N'(s)-\varphi_N'(\tau))\frac{1}{\tau-s+\i 2N^{-1+\xi}}\,\mathrm{d}s\mathrm{d}\tau
\end{align*}
The difference between the last expression and $\int_{-tM<\tau<tM}\varphi_N(\tau)(H\varphi_N')(\tau)\,\mathrm{d}\tau$ is
\begin{equation*}
\Re \int_{-tM<\tau<tM} \varphi_N(\tau)\int (\varphi_N'(s)-\varphi_N'(\tau))\frac{2 \i N^{-1+\xi}}{(\tau-s+\i 2N^{-1+\xi})(\tau-s)}\,\mathrm{d}s\mathrm{d}\tau
\end{equation*}
We split the inner integral into $|\tau-s|\le \delta$ and $|\tau-s|>\delta$ to find the estimate
\[t\|\varphi_N'\|_{C^\alpha}\delta^\alpha + \frac{N^{-1+\xi}}{\delta}\|\varphi_N'\|_{L^1}.\]
Optimizing in $\delta$, and using $\|\varphi_N'\|_{L^1}\le C$ we get a bound of
\[ (t\|\varphi_N'\|_{C^\alpha})^{\frac{1}{1+\alpha}}N^{(-1+\xi)\frac{\alpha}{1+\alpha}}.\]
The right side is bounded by
\[CN^{\frac{\omega_0}{1+\alpha}-\omega_1+\frac{\alpha}{1+\alpha}\xi}.\]
Using condition \eqref{eqn: omega-condition-1}, for $\alpha>0$ small enough, this is $\O(N^{-c})$.

We now proceed to computing $I_{1,1}$ \eqref{eqn: I2-display}. Recall the definition of the constant $M$ introduced in Proposition \ref{prop: const-bd}. Note first that by the support assumption on $\varphi_N$, $\dvarp=0$ in $\{|E|> Mt\}$ we can replace the integration domain $\Omega_N^*$ by:
\[\Omega_N^{**}=\Omega_N^* \cap \{E+\i \eta: |E|\le Mt\}.\]
We use Green's theorem \eqref{eqn: greensthm} on $\Omega_N^{**}\cap \{\Im z>0\}$. The boundary of this region consists of four segments. The function $\tilde{\varphi}_N(z)$ is zero on the top segment is $[-tM+\i N^{10C_V},tM+ \i N^{10C_V}]$. We number the remaining parts of the boundary as: (1) $[-tM+\i N^{-1+\xi},tM
+\i N^{-1+\xi}]$; (2) $[-tM+\i N^{10C_V},-tM+\i N^{-1+\xi}]$; and (3) $[tM+\i N^{-1+\xi},tM+\i N^{10C_V}]$.
The result of our application of Green's theorem to
\begin{equation}
\frac{t^2}{N}\int_{\Omega_N^{**}\cap \{\Im z >0\} }\dvarp(z) \partial_z\sum_{j=1}^Ng_j(z)g_j(z')\partial_{z'}\frac{\mfct(z)-\mfct(z')}{z-z'}\,\mathrm{d}z\mathrm{d}z'
\end{equation}
is a sum of line integrals which we label according to the corresponding parts of the boundary. For (1), we have:
\begin{equation}
\begin{split}
I_{2,+,1}&=t^2 \int_{-Mt}^{Mt} (\varphi_N(\tau)+\i N^{-1+\xi}\varphi'_N(\tau))\partial_\tau\frac{\mfct(\tau+\i N^{-1+\xi})-\mfct(z')}{\tau+\i N^{-1+\xi}-z'+t(\mfct(\tau+\i N^{-1+\xi})-\mfct(z'))}\\
&\quad \times \partial_{z'}\frac{\mfct(\tau+\i N^{-1+\xi})-\mfct(z')}{\tau+\i N^{-1+\xi}-z'}\,\mathrm{d}\tau;
\end{split}
\end{equation}
for (2), we have the sum:
\begin{equation}\label{eqn: I2p2}
\begin{split}
I_{2,+,2}&= -t^2\varphi_N(-tM)\int^{N^{10C_V}}_{N^{-1+\xi}} \partial_\tau\frac{\mfct(-tM+\i\eta)-\mfct(z')}{-tM+\i\eta-z'+t(\mfct(-tM+\i\eta)-\mfct(z'))}\\
&\quad \times \partial_{z'}\frac{\mfct(-tM+\i\eta)-\mfct(z')}{-2tM+\i\eta-z'} \mathrm{d}\eta,
\end{split}
\end{equation}
for (3):
\begin{equation}
\begin{split}
I_{2,+,3}&= t^2\varphi_N(tM)\int^{N^{10C_V}}_{N^{-1+\xi}} \partial_\tau\frac{\mfct(tM+\i\eta)-\mfct(z')}{tM+\i\eta-z'+t(\mfct(tM+\i\eta)-\mfct(z'))}\\
&\quad \times \partial_{z'}\frac{\mfct(tM+\i\eta)-\mfct(z')}{tM+\i\eta-z'} \mathrm{d}\eta.
\end{split}
\end{equation}

We similarly define $I_{2,-,k}$, $k=1,2,3$ as the line integrals along the boundary of $\Omega_N^{**}\cap \{\Im z<0\}$. We now insert $I_{2,\pm,k}(z')$ into the integral \eqref{eqn: I2-d-onzz'} and apply Green's theorem to obtain
\begin{equation}
\begin{split}
&\frac{1}{2i}\int_{\Omega_N^*\cap\{\Im z'>0\}}\dvarp(z')I_{2,\pm,k}(z')\,\mathrm{d}z'\\
=&\frac{1}{4}\int_{-qG}^{qG}(\varphi_N(s)+\i N^{-1+\xi}\varphi_N'(s))I_{2,\pm,k}(s+\i N^{-1+\xi})\mathrm{d}s\\
+&\frac{1}{4}\int^{N^{-1+\xi}}_{N^{10C_V}}\varphi_N(-qG)I_{2,\pm,k}(-qG+\i\eta')\,\mathrm{d}\eta'\\
+&\frac{1}{4}\int_{N^{-1+\xi}}^{N^{10C_V}}\varphi_N(qG)I_{2,\pm,k}(qG+\i\eta')\,\mathrm{d}\eta'\\
:=&I_{2,\pm,k,+,1}+I_{2,\pm,k,+,2}+I_{2,\pm,k,+,3}.
\end{split}
\end{equation}

Applying Green's theorem to the $z'$ integral in the region $\Omega_N^*\cap \{\Im z'<0\}$, we obtain:
\begin{equation}
\begin{split}
&\frac{1}{2i}\int_{\Omega_N^*\cap\{\Im z'>0\}}\dvarp(z')I_{2,\pm,k}(z')\,\mathrm{d}z'\\
=&\frac{1}{4}\int_{qG}^{-qG}(\varphi_N(s)-\i N^{-1+\xi}\varphi_N'(s))I_{2,\pm,k}(s-\i N^{-1+\xi})\mathrm{d}s\\
+&\frac{1}{4}\int_{-N^{-1+\xi}}^{-N^{10C_V}}\varphi_N(-qG)I_{2,\pm,k}(-qG-\i\eta')\,\mathrm{d}\eta'\\
+&\frac{1}{4}\int^{-N^{-1+\xi}}_{-N^{10C_V}}\varphi_N(qG)I_{2,\pm,k}(qG-\i\eta')\,\mathrm{d}\eta'\\
:=&I_{2,\pm,k,-,1}+I_{2,\pm,k,-,2}+I_{2,\pm,k,-,3}.
\end{split}
\end{equation}
So far, we have
\begin{align*}
&\frac{t^2}{N}\int_{\Omega_N^*}\int_{\Omega_N^{**}} \dvarp(z)\dvarp(z') \partial_z \sum_{j=1}^N g_j(z)g_j(z') \partial_{z'} \frac{\mfct(z)-\mfct(z')}{z-z'}\,\mathrm{d}z\mathrm{d}z'\\
=&\sum_{k=1}^3\sum_{j=1}^3 \sum_{\alpha,\beta\in\{\pm\}}I_{2,\alpha,k,\beta,j}.
\end{align*}

The main contribution comes from the terms $I_{2,\pm,1,\pm,1}$. The remaining terms are polynomially smaller.
\begin{prop} \label{prop:revd2}
For any choice of $k, j$ with $(k,j)\neq (1,1)$, $\alpha,\beta \in \{\pm\}$,
\[|I_{2,\alpha,k,\beta,j}|=\O(1).\]
\end{prop}
\begin{proof}
We start with $k=1$, $j=2,3$. By symmetry, it suffices to deal with $j=2$. That is, we estimate
\begin{align}
&I_{2,\pm,1,\pm,2} \nonumber \\
=&\pm\frac{\varphi_N(qG)}{4}\int_{N^{-1+\xi}}^{N^{10C_V}} \int_{-Mt}^{Mt}\frac{\varphi_N(\tau)\pm \i N^{-1+\xi}\varphi_N'(\tau)}{1-tR_2(\tau\pm \i N^{-1+\xi})}S_{2,1}(\tau\pm \i N^{-1+\xi},qG+\i\eta') \,\mathrm{d}\tau\mathrm{d}\eta'. \label{eqn: I2pm1pm2}
\end{align}
Compute the kernel
\begin{align*}
&\frac{1}{1-tR_2(\tau\pm \i N^{-1+\xi})}S_{2,1}(\tau\pm \i N^{-1+\xi},qG+\i\eta')\\
=&-t^2 \frac{\mfct(z)-\mfct(z')-\mfct'(z)(z-z')}{(z-z'+t(\mfct(z)-\mfct(z'))^2}\frac{\mfct(z)-\mfct(z')-\mfct'(z')(z-z')}{(z-z')^2}
\end{align*}
where $z=\tau\pm \i N^{-1+\xi}$, $z'=qG+\i\eta'$. Since $|\tau|\le 2Mt$, we have
\[|\tau-qG|\ge (1/2)qG\ge t^{1/2}N^{\sigma/2},\]
and so, as in Proposition \ref{prop: I2-bound-prop},
\[\left|t^2 \frac{\mfct(z)-\mfct(z')-\mfct'(z)(z-z')}{(z-z'+t(\mfct(z)-\mfct(z'))^2}\right| \le Ct^{1/2+\sigma/2}.\]
Inserting this into \eqref{eqn: I2pm1pm2}, we find
\begin{align*}
&|I_{2,\pm,1,\pm,2}|\\
\le &C(1+\|\varphi_N'\|_{L^\infty}N^{-1+\xi})t^{1/2+\sigma/2}\int_{N^{-1+\xi}}^{N^{10C_V}} \int_{-Mt}^{Mt}\frac{1}{|\tau+\i N^{-1+\xi}-qG-i\eta|^2}+\frac{|\mfct'(z')|}{|\tau+\i N^{-1+\xi}-qG-i\eta|}\mathrm{d}\tau\mathrm{d}\eta'.
\end{align*}
Recalling that $|\mfct'(z)'|\le Ct^{-1}$, this quantity is bounded by
\[C\log N (1+\|\varphi_N'\|_{L^\infty}N^{-1+\xi}) t^{1/2+\sigma}N^\sigma.\]
We have shown
\begin{equation}
|I_{2,\pm,1,\pm,2}|, |I_{2,\pm,1,\pm,3}| = \O(\log N (1+\|\varphi_N'\|_{L^\infty}N^{-1+\xi}) t^{1/2+\sigma}N^\sigma).
\end{equation}

We now estimate $I_{2,\pm,2,\pm,1}$:
\begin{equation} \label{eqn: I2pm2pm1}
t^2\varphi_N(-Mt)\int^{N^{10C_V}}_{N^{-1+\xi}} \int_{-qG}^{qG}(\varphi_N(s)\pm \i N^{-1+\xi}\varphi_N'(s))\frac{1}{1-tR_2(-Mt\pm \i\eta)}S_{2,1}(-Mt+ \i\eta,s\pm \i N^{-1+\xi} )\,\mathrm{d}s \mathrm{d}\eta.
\end{equation}
The kernel is
\begin{equation}\label{eqn: I2+2+1-kernel}
\begin{split}
&\frac{1}{1-tR_2(-tM \pm i\eta)}S_{2,1}(-tM + \i\eta ,s \pm \i N^{-1+\xi})\\
=&-t^2\frac{\mfct(z)-\mfct(z')-\mfct'(z)(z-z')}{(z-z'+t(\mfct(z)-\mfct(z'))^2} \frac{\mfct(z)-\mfct(z')-\mfct'(z')(z-z')}{(z-z')^2},
\end{split}
\end{equation}
with $z=-qG \pm \i\eta$ and $z'=s \pm \i N^{-1+\xi}$. The $I_{2,\pm,2,\pm,1}$ can be performed in the same way whether $\Im z\Im z' >0$ or $\Im z\Im z'<0$, except in the region 
\begin{equation}\label{eqn: key-region}
\{(z,z'):|s+tM|\le Mt/2, N^{-1+\xi}<\eta<tM/10\}.
\end{equation}
If $|s+tM|\ge Mt/2$ and $\eta\ge Mt/10$, we use the estimate \eqref{eqn: zeta-lwr-bound}
\[\frac{1}{|z-z'+t(\mfct(z)-\mfct(z'))|}\le \frac{C}{|tM+s|},\]
and 
\begin{equation}\label{eqn: I2+2+1-eta}
|z-z'|=|tM-s+\i N^{-1+\xi}+\i\eta+\tau|\ge \eta
\end{equation}
to find the bound
\begin{equation}
\begin{split}
&t^2\int_{\{\eta:\eta \ge Mt/10\}}\int_{\{s:|tM+s|\ge Mt/2\}}\frac{(1+|\mfct'(z)||z-z'|)(1+|\mfct'(z')||z-z'|)}{|z-z'+t(\mfct(z)-\mfct(z'))|^2|z-z'|^2}\,\mathrm{d}s\mathrm{d}\eta\\
\le &Ct^2\int_{\{N^{-1+\xi}<\eta \le Mt/10\}}t^{-1}(\eta^{-2}+|\mfct'(z)|\eta^{-1}+ |\mfct'(z')|\eta^{-3/2}t^{1/2})\,\mathrm{d}s\\
 =&\O(1).
\end{split}
\end{equation}
To pass to the last line, we have used $|\mfct'(z)|\le C\eta^{-1}$. For the case $|s+tM|\le Mt/2$ and $\eta\ge Mt/10$, we have the bound
\begin{equation}
\begin{split}
&t^2\int_{\{\eta \ge Mt/10\}}\int_{\{s:|tM+s|\le Mt/2\}}\frac{(1+|\mfct'(z)||z-z'|)(1+|\mfct'(z')||z-z'|)}{\eta^2|z-z'|^2}\,\mathrm{d}s\mathrm{d}\eta\\
\le &Ct^2\int_{\{s:|s+tM|\le tM/2\}}t^{-3}\,\mathrm{d}\eta\\
 =&\O(1).
\end{split}
\end{equation}

If $|s+tM|\ge Mt/2$ and $\eta\le Mt/10$, use $|z-z'|\ge |s+tM|$ to find the estimate
\begin{equation}
\begin{split}
&t^2\int_{\{\eta:N^{-1+\xi}<\eta \le Mt/10\}}\int_{\{s:|tM+s|\ge Mt/2\}}\frac{(1+|\mfct'(z)||z-z'|)(1+|\mfct'(z')||z-z'|)}{|tM+s|^2|z-z'|^2}\,\mathrm{d}s\mathrm{d}\eta\\
\le &Ct^{-1}\int_{\{\eta:N^{-1+\xi}<\eta \le Mt/10\}}(1+t|\mfct'(z)|+t|\mfct'(z')|)\,\mathrm{d}\eta\\
 =&\O(1).
\end{split}
\end{equation}
At this point, we have obtained estimates for $I_{2,\pm,2,\pm,1}$ in the complement of \eqref{eqn: key-region}

We now estimate the contribution to $I_{2,\alpha,2,\beta,1}$ from the region \eqref{eqn: key-region}, when $\alpha$ and $\beta$ are of opposite signs. This term is somewhat delicate. It will suffice to deal with $I_{2,+,2,-,1}$.
We split the $s$ integral into the regions $\{s: |s+Mt|\le Mt/2\}$ and its complement. In the first region $\varphi_N(s) \equiv \varphi_N(-Mt)$. When $\eta < Mt/10 $ as well, we expand the kernel to second order. For this, $\Im z\Im z' <0$, so we have the expansion
\begin{equation}\label{eqn: 2nd-order-exp}
\mfct(z')= \mfct(\bar{z})+\mfct '(\bar{z})(z'-\bar{z})+\O(\max_{z}|\mfct''(z)||z-\bar{z}|^2)
\end{equation}
Using \eqref{eqn: 2nd-order-exp} and the lower bound
\[|\mfct(z)-\mfct(\bar{z})|= 2\pi \rhofct(z)\ge c >0,\]
the kernel \eqref{eqn: I2+2+1-kernel} is given by:
\begin{align}
-&\frac{1}{(z-z')^2}\Big(1-\frac{\mfct'(\bar{z})}{2\pi \i \rhofct(z)}(z-\bar{z})-\frac{\mfct'(z)}{2\pi \i \rhofct(z)}(z-z')-\frac{\mfct'(\bar{z})}{2\pi \i \rhofct(z)}(z'-\bar{z}) \label{eqn: I2+2+1-cancel1}\\
&\quad +\frac{(\mfct'(\bar{z}))^2}{4\pi^2\rhofct(z)}(z'-\bar{z})(z-\bar{z})+\frac{\mfct'(\bar{z})\mfct'(z)}{4\pi^2\rhofct(z)}(z'-\bar{z})(z-z')\Big) \nonumber \\
&\times\left(1-\frac{\mfct'(\bar{z})}{\pi \i\rhofct(z)}(z'-\bar{z})+\frac{z-z'}{\pi \i t\rhofct(z)}\right) \label{eqn: I2+2+1-cancel2} \\
+&\O\left(\frac{\max_z|\mfct''(z)|(|z-z'|^2+|\bar{z}-z'|^2)}{|z-z'|^2}\right), \label{eqn: I2+2+1-kernel-error}
\end{align}
for $|tM-s|\le Mt/2$ and $\eta\le Mt/10$.

The cancellation that arises from performing the $s$ integral first in \eqref{eqn: I2+2+1-cancel1}, \eqref{eqn: I2+2+1-cancel2} is crucial.
For example, the contribution to $I_{2,+,2,-,1}$ from $\{|s-Mt|\le Mt/2\}$, $\eta <Mt/10$ of the term $1/(z-z')^2$ is
\begin{align*}
&\varphi_N(-Mt)^2\int^{Mt/10}_{N^{-1+\xi}} \int_{s:|s+Mt|\le Mt/2}\frac{1}{(-Mt+ \i\eta-s+  \i N^{-1+\xi})^2}\,\mathrm{d}s\mathrm{d}\eta\\
\le&C\varphi_N(-Mt)^2\int^{Mt/10}_{N^{-1+\xi}}\frac{tM}{(tM/2)^2+(\eta+N^{-1+\xi})^2} \mathrm{d}\eta.\\
=&\O(1).
\end{align*}
To estimate the remaining terms, letting $z=-tM+\i\eta$ and $z'=s-\i N^{-1+\xi}$, we compute:
\begin{align}
\int_{s: |s+Mt|\le Mt/2} \frac{z'-\bar{z}}{(z-z')^2}\mathrm{d}s &= -\int_{s:|s-Mt|\le Mt/2} \frac{1}{z-z'}\mathrm{d}z + \int_{s:|s-Mt|\le Mt/2}\frac{z-\bar{z}}{(z-z')^2}\,\mathrm{d}s,\\
\int_{s:|s+Mt|\le Mt/2} \frac{1}{z-z'}\mathrm{d}s &= \log\left(\frac{\frac{tM}{2}+\i\eta+\i N^{-1+\xi}}{\frac{tM}{2}-\i\eta-\i N^{-1+\xi}}\right)= \O(\eta/t),\quad  |\eta|\le Mt/10, \label{eqn: I2+2+1-log-term}\\
\int_{s:|s+Mt|\le Mt/2}\frac{z-\bar{z}}{(z-z')^2}\,\mathrm{d}s&=  \frac{2\i\eta tM}{(tM/2)^2+(\eta+N^{-1+\xi})^2}. \label{eqn: I2+2+1-sq-term}
\end{align}
We have used the principal determination of the logarithm in \eqref{eqn: I2+2+1-log-term}. 

Using \eqref{eqn: I2+2+1-log-term}, \eqref{eqn: I2+2+1-sq-term}, and \eqref{eqn: dmfct-bound},
\begin{align*}
&\int_{N^{-1+\xi}}^{Mt/10}\int_{s:|s+tM/2|\le Mt/2} \left(1+\frac{1}{2\pi \i \rhofct(z)}\right) \frac{\mfct'(\bar{z})(z'-\bar{z})}{(z-z')^2}\mathrm{d}s\mathrm{d}\eta\\
=&\int_{N^{-1+\xi}}^{Mt/10} \frac{\O(t^{-1}|\eta|tM)}{(tM/2)^2+(\eta+N^{-1+\xi})^2}+\O(t^{-2}|\eta|)\mathrm{d}\eta=\O(1),\\
&\int_{N^{-1+\xi}}^{Mt/10}\int_{s:|s+tM/2|\le Mt/2}\frac{\mfct'(\bar{z})(z-\bar{z})}{(z-z')^2}\mathrm{d}s\mathrm{d}\eta\\
= &\int^{Mt/10}_{N^{-1+\xi}}\O(\eta)\frac{tM}{(tM/2)^2+(\eta+N^{-1+\xi})^2} \mathrm{d}\eta=\O(1),\\
&\int_{N^{-1+\xi}}^{Mt/10}\int_{s:|s+tM/2|\le Mt/2}\frac{1}{t(z-z')}\mathrm{d}s\mathrm{d}\eta=\O(1).\\
\end{align*}
Moreover, for $z=-tM+\i\eta$, $z'=s-\i N^{-1+\xi}$, we have
\begin{equation}\label{eqn: I2+2+1-quad}
|\bar{z}-z'|, |z-\bar{z}| \le 2|z-z'|,
\end{equation}
so
\begin{equation}
\begin{split}
&\left|\int_{N^{-1+\xi}}^{Mt/10}\int_{s:|s+tM/2|\le Mt/2}\frac{(\mfct'(\bar{z}))^2}{4\pi^2\rhofct(z)}\frac{(z'-\bar{z})(z-\bar{z})}{(z-z')^2}\left(1+\mfct'(\bar{z})(z'-\bar{z})-\frac{z-z'}{t}\right) \mathrm{d}s\mathrm{d}\eta\right|\\
\le& \int_{N^{-1+\xi}}^{Mt/10}\int_{s:|s+tM/2|\le Mt/2} t^{-2}(1+t^{-1}(|z-z'|+|\bar{z}-z'|)\,\mathrm{d}s\mathrm{d}\eta=\O(1).
\end{split}
\end{equation}
Similar estimates hold for the other terms containing a quadratic expression in $z'-z$, $\bar{z}-z'$ or $z-\bar{z}$ in \eqref{eqn: I2+2+1-cancel1}, \eqref{eqn: I2+2+1-cancel2}.

For the error term \eqref{eqn: I2+2+1-kernel-error}, we use \eqref{eqn: I2+2+1-quad} and the estimate
\begin{equation}\label{eqn: mfct-2nd-derivative}
\begin{split}
|\partial_z^2 \mfct(z)|&\le \left|\frac{1}{N}\sum_{j=1}^N \frac{g_j(z)^3(1+t\partial_z \mfct(z))}{(1-tR_2(z))^2}\right|\\
&\le Ct^{-2}.
\end{split}
\end{equation}
The last step in \eqref{eqn: mfct-2nd-derivative} follows from \eqref{eqn: gj2-bound} and \eqref{eqn: stability}. (See also \cite[Lemma 7.2]{landonyau}.) The result is
\begin{align*}
&\int_{N^{-1+\xi}}^{Mt/10}\int_{s:|s+tM/2|\le Mt/2}\frac{\max_z|\mfct''(z)|(|z-z'|^2+|\bar{z}-z'|^2)}{|z-z'|^2} \mathrm{d}s\mathrm{d}\eta\\
\le  & C\int_{N^{-1+\xi}}^{Mt/10}\int_{s:|s+tM/2|\le Mt/2}\mathrm{d}s\mathrm{d}\eta=\O(1).
\end{align*}
At this point all terms in the expansion  \eqref{eqn: I2+2+1-cancel1}, \eqref{eqn: I2+2+1-cancel2} are accounted for.

To estimate the contribution from the region \eqref{eqn: key-region} to $I_{2,\alpha,2,\beta,1}$ when $\alpha$ and $\beta$ have the same sign,
we use (see \eqref{eqn: dfracmfct-bound})
\[\left|\partial_z \frac{\mfct(z)-\mfct(z')}{z-z'+t(\mfct(z)-\mfct(z'))}\right|\le Ct^{-2}, \]
and the estimate
\begin{equation}\label{eqn: dmfct-diffquot}
\left|\partial_{z'} \frac{\mfct(z)-\mfct(z')}{z-z'}\right|\le Ct^{-2},
\end{equation}
which follows from \eqref{eqn: f-diffquot} and the estimate \eqref{eqn: mfct-2nd-derivative}. We have:
\begin{align*}
|I_{2,+,2,+,1}|&\le Ct^{-2}\varphi(-Mt)\int_{N^{-1+\xi}}^{Mt/10}\int_{s:|s+tM|\le Mt/2}|\tilde{\varphi}(s)|\,\mathrm{d}\eta\mathrm{d} \leq C.
\end{align*}
The same bound holds for $I_{2,-,2,-,1}$.

Replacing $z=-Mt+\i\eta$ by $z=Mt+ \i\eta$, we obtain the bounds
$|I_{2,\pm,3,\pm,1}|=\O(1). $
Turning to $I_{2,\pm,2,\pm,2}$, we have to estimate:
\[\pm \varphi(-Mt)\varphi_N(-qG)\int_{-N^{C10_V}}^{-N^{-1+\xi}}\int_{N^{-1+\xi}}^{N^{10C_V}}\frac{1}{1-tR_2(-Mt+ \i\eta)}S_{2,1}(-tN^\sigma\pm \i\eta, -qG \pm \i\eta')\mathrm{d}\eta\mathrm{d}\eta'.\]
Note that
\begin{equation}\label{eqn: I2pm2pm2-lwr-bd}
|z-z'-t(\mfct(z)-\mfct(z'))| \ge cqG
\end{equation}
for $z=-Mt+i\eta$ and $z'=-qG+i\eta'$, so by \eqref{eqn: t-3/2}, \eqref{EAmscBprop-2} the integrand is bounded by
$ Ct^{1/2+\sigma/2}|\eta|^{-1}|\eta'|^{-1}.$ 
Performing the double integration, we obtain a bound of
\[Ct^{1/2+\sigma/2}(\log N)^2.\]
This last estimate depended only on the lower bound \eqref{eqn: I2pm2pm2-lwr-bd}, so we have the same estimate for $I_{2,\pm,2,\pm,3}$, $I_{2,\pm,3,\pm,2}$, $I_{2,\pm,3,\pm,3}$. \qed
\end{proof}

Denote $\tau^\pm = \tau\pm \i N^{-1+\xi}$ and $s^\pm = s\pm \i N^{-1+\xi}$. We have shown
\begin{align}
&\int\int_{\Omega_N \times \Omega_N} \dvarp(z)\dvarp(z') \frac{1}{1-tR_2(z)}S_{2,1}(z,z')\,\mathrm{d}z'\mathrm{d}z \label{eqn: mainvariance-2}\\
=&\frac{t^2}{4}\int_{-qG}^{qG}\int_{-Mt}^{Mt}\tilde{\varphi}_N(\tau)\tilde{\varphi}_N(s)\partial_\tau \frac{\mfct(\tau^+)-\mfct(s^+)}{\tau-s+t(\mfct(\tau^+)-\mfct(s^+))}\partial_s \frac{\mfct(\tau^+)-\mfct(s^+)}{\tau-s}\,\mathrm{d}\tau\mathrm{d}s \label{eqn: plusvariance-1-2}\\
+&\frac{t^2}{4}\int_{-qG}^{qG}\int_{-Mt}^{Mt}\tilde{\varphi}_N(\tau)\tilde{\varphi}_N(s)\partial_\tau \frac{\mfct(\tau^-)-\mfct(s^-)}{\tau-s+t(\mfct(\tau^-)-\mfct(s^-))} \partial_s \frac{\mfct(\tau^-)-\mfct(s^-)}{\tau-s}\,\mathrm{d}\tau\mathrm{d}s \label{eqn: plusvariance-2-2}\\
-&\frac{t^2}{4}\int_{-qG}^{qG}\int_{-Mt}^{Mt}\tilde{\varphi}_N(\tau)\tilde{\varphi}_N(s)\partial_\tau \frac{\mfct(\tau^+)-\mfct(s^-)}{\tau-s+t(\mfct(\tau^+)-\mfct(s^-))} \partial_s \frac{\mfct(\tau^+)-\mfct(s^-)}{\tau-s+2iN^{-1+\xi}}\,\mathrm{d}\tau\mathrm{d}s \label{eqn: mainvariance-1-2}\\
-&\frac{t^2}{4}\int_{-qG}^{qG}\int_{-Mt}^{Mt}\tilde{\varphi}_N(\tau)\tilde{\varphi}_N(s)\partial_\tau \frac{\mfct(\tau^-)-\mfct(s^+)}{\tau-s+t(\mfct(\tau^-)-\mfct(s^+))} \partial_s \frac{\mfct(\tau^-)-\mfct(s^+)}{\tau-s-2iN^{-1+\xi}}\,\mathrm{d}\tau\mathrm{d}s \label{eqn: mainvariance-2-2}\\
+& \O(1).
\end{align}
The main terms here are \eqref{eqn: mainvariance-1-2} and \eqref{eqn: mainvariance-2-2}. For the remaining terms we have
\begin{prop}
We have the estimate:
There is a constant $C$ such that
\begin{align}
\left|t^2\int_{-qG}^{qG}\int_{-Mt}^{Mt}\tilde{\varphi}_N(\tau)\tilde{\varphi}_N(s)\partial_\tau \frac{\mfct(\tau^+)-\mfct(s^+)}{\tau-s+t(\mfct(\tau^+)-\mfct(s^+))} \partial_s \frac{\mfct(\tau^+)-\mfct(s^+)}{\tau-s}\,\mathrm{d}\tau\mathrm{d}s\right| &\le C, \label{eqn: mainvariance-2-same-side} \\
\left|t^2\int_{-qG}^{qG}\int_{-Mt}^{Mt}\tilde{\varphi}_N(\tau)\tilde{\varphi}_N(s)\partial_\tau \frac{\mfct(\tau^-)-\mfct(s^-)}{\tau-s+t(\mfct(\tau^-)-\mfct(s^-))} \partial_s\frac{\mfct(\tau^-)-\mfct(s^-)}{\tau-s}\,\mathrm{d}\tau\mathrm{d}s\right| &\le C.
\end{align}
\end{prop}
\begin{proof}
We deal with the first quantity. The second quantity is estimated similarly. The kernel part of the integrand is
\[-t^2\frac{\mfct(\tau^+)-\mfct(s^+)-\partial_\tau \mfct(\tau^+)(\tau-s)}{(\tau-s+t(\mfct(\tau^+)-\mfct(s^+)))^2}\frac{\mfct(\tau^+)-\mfct(s^+)-\partial_s\mfct(s^+)(\tau-s)}{(\tau-s)^2}.\]
In the region $\{\tau: |\tau-s|\le Mt\}$, we use \eqref{eqn: dfracmfct-bound}, and \eqref{eqn: dmfct-diffquot}.
So \eqref{eqn: mainvariance-2-same-side} is bounded by
\begin{equation}\label{eqn: constant-bd}
t^2\int_{-Mt}^{Mt}\left(\int_{\tau:|\tau-s|\le Mt} \frac{C}{t^4}\mathrm{d}\tau+ \int_{\tau:|\tau-s|\ge Mt} \frac{(1+|\mfct'(\tau)||\tau-s|)(1+|\mfct'(s)||\tau-s|)}{|\tau-s|^4}\,\mathrm{d}\tau\right)\mathrm{d}s\le C.
\end{equation}
\end{proof}\qed

The sum of the remaining terms \eqref{eqn: mainvariance-1-2}, \eqref{eqn: mainvariance-2-2} is
\begin{equation*}
-\frac{t^2}{2}\int_{-qG}^{qG}\int_{-Mt}^{Mt}\tilde{\varphi}_N(\tau)  \tilde{\varphi}_N(s) \Re\,\partial_\tau \frac{\mfct(\tau^-)-\mfct(s^+)}{\tau-s+t(\mfct(\tau^-)-\mfct(s^+))} \partial_s \frac{\mfct(\tau^-)-\mfct(s^+)}{\tau-s-2iN^{-1+\xi}}\,\mathrm{d}\tau\mathrm{d}s.
\end{equation*}
Using the expansion \eqref{eqn: I2+2+1-cancel1}, \eqref{eqn: I2+2+1-cancel2} in the region $\{\tau: |\tau|\le 2Mt\}$:
\begin{align*}
&t^2\partial_\tau \frac{\mfct(\tau^-)-\mfct(s^+)}{\tau-s+t(\mfct(\tau^-)-\mfct(s^+))}\partial_s \frac{\mfct(\tau^-)-\mfct(s^+)}{\tau-s-2\i N^{-1+\xi}}\\
=&-\frac{1}{(\tau-s- 2\i N^{-1+\xi})^2}+\Delta_3,
\end{align*}
where $\Delta_3(z,z')$ is an error term. The most serious terms in $\Delta_3$ are handled using the computation
\begin{equation}
\label{eqn: hilbert-fourier}
\begin{split}
&\int_{-2tM}^{2tM} \int_{-tM}^{tM} \frac{\mfct'(z)}{\rhofct(z)}\tilde{\varphi}_N(s)\tilde{\varphi}_N(\tau)\frac{1}{s-\tau-2\i N^{-1+\xi}}\,\mathrm{d}s\mathrm{d}\tau\\
=&c\int\int f(\xi)\overline{g(\lambda)}K(\xi-\lambda)\,\mathrm{d}\xi\mathrm{d}\lambda,
\end{split}
\end{equation} 
where $f$ and $g$ are the inverse Fourier transforms of $\mathbf{1}_{[-2tM,2tM]}(\tau)\frac{\mfct'(z)}{\rhofct(z}\tilde{\varphi}(\tau)$ and $\mathbf{1}_{[-tM,tM]}(s)\tilde{\varphi}_N(s)$, respectively, and
\begin{align*}
K(\xi,\lambda)&= K(\xi-\lambda)\\
&:=\i\mathbf{1}_{(-\infty,0]}(\xi-\lambda)\e^{-2N^{-1+\xi}|\xi-\lambda|},
\end{align*}
so that
\[\widehat{K}(x)=\frac{1}{x-2\i N^{-1+\xi}}.\]
From the Fourier representation, the Plancherel theorem and the simple estimates
\[\|f\|_{L^2}=\O(t^{-1/2}), \|g\|_{L^2}=\O(t^{1/2}),\]
the term \eqref{eqn: hilbert-fourier} is $O(1)$. All other error terms are then easily estimated, using $z-\bar{z}=\i 2N^{-1+\xi}$ and the trivial bound
\[\int_{-2tM}^{2tM}\int_{-tM}^{tM}|\tilde{\varphi}_N(s)||\tilde{\varphi}_N(\tau)|\frac{1}{|s-\tau-2\i N^{-1+\xi}|^2}\,\mathrm{d}s\mathrm{d}\tau=\O(\log N).\]

As in \eqref{eqn: constant-bd} contribution from the region $\{|\tau|\ge 2Mt\}\subset \{\tau: |\tau-s|\ge Mt\}$, as well as the error terms, are $\O(1)$.
Adding the contributions from the two main terms, we find:
\begin{align}
&\int\int_{\Omega_N \times \Omega_N} \dvarp(z)\dvarp(z') \frac{1}{1-tR_2(z)}S_{2,2}(z,z')\,\mathrm{d}z'\mathrm{d}z  \nonumber \\
=&\frac{1}{2}\int_{-tM}^{tM}\int_{-2tM}^{2tM}\varphi_N(\tau)\varphi_N(s)\frac{(\tau-s)^2-N^{-2+2\xi}}{((\tau-s)^2+N^{-2+2\xi})^2}\mathrm{d}\tau\mathrm{d}s+\O(1) \nonumber \\
=&\frac{1}{2} \int_{-tM}^{tM}\int_{-2tM}^{2tM}\varphi_N(\tau)\varphi_N(s) \Re \partial_s \frac{1}{\tau-s+\i N^{-1+\xi}}\,\mathrm{d}\tau\mathrm{d}s+\O(1) \nonumber\\
=&-\frac{1}{2} \int_{-tM}^{tM}\int\varphi_N(\tau)\varphi_N'(s) \Re \frac{1}{\tau-s+\i N^{-1+\xi}}\,\mathrm{d}\tau\mathrm{d}s+\O(1). \label{eqn: hilbert-transform2}
\end{align}
This is the same quantity as in \eqref{eqn: hilbert-transform}, and so this ends the computation of the term \eqref{eqn: I2-d-onzz'}.

It remains to estimate $I_{1,3}$. Integrating by parts in $z$ and $z'$, we have
\begin{equation}
I_{1,3}= \frac{t}{N}\sum_{j=1}^N\int_{\Omega_N^*}\int_{\Omega_N^*} g_j(z)g_j(z') \partial_\tau\tilde{\varphi}(z)\partial_s \tilde{\varphi}(z') \,\mathrm{d}z\mathrm{d}z'+\O(N^{-2}).
\end{equation}
As for $I_{1,1}$ and $I_{1,2}$, we use Green's theorem to the domains $\Omega_N \cap \{\Im z > 0\}$, $\Omega_N \cap \{\Im z < 0\}$, $\Omega_N \cap \{\Im z' > 0\}$, $\Omega_N \cap \{\Im z' < 0\}$. By the support properties of $\partial_s \tilde{\varphi}(\tau)$, we only find contributions from the segments $[-qG\pm \i N^{-1+\xi},qG \pm \i N^{-1+\xi}]$. Denoting $\tau^\pm = \tau\pm \i N^{-1+\xi}$, $s^\pm = s\pm \i N^{-1+\xi}$, the result is
\begin{equation}\label{eqn: I13-4terms}
\begin{split}
I_{1,3}&= \frac{t}{4N}\sum_{j=1}^N \int_{-qG}^{qG}\int_{-qG}^{qG}g_j(\tau^+)g_j(s^+) \partial_s \tilde{\varphi}_N(\tau^+) \partial_\tau \tilde{\varphi}_N(s^+)\,\mathrm{d}s\mathrm{d}\tau\\
&+ \frac{t}{4N}\sum_{j=1}^N \int_{-qG}^{qG}\int_{-qG}^{qG}g_j(\tau^-)g_j(s^-) \partial_s \tilde{\varphi}_N(\tau^-) \partial_\tau \tilde{\varphi}_N(s^-)\,\mathrm{d}s\mathrm{d}\tau\\
&- \frac{t}{4N}\sum_{j=1}^N \int_{-qG}^{qG}\int_{-qG}^{qG}g_j(\tau^+)g_j(s^-) \partial_s \tilde{\varphi}_N(\tau^+) \partial_\tau \tilde{\varphi}_N(s^-)\,\mathrm{d}s\mathrm{d}\tau\\
&- \frac{t}{4N}\sum_{j=1}^N \int_{-qG}^{qG}\int_{-qG}^{qG}g_j(\tau^-)g_j(s^+) \partial_s \tilde{\varphi}_N(\tau^-) \partial_\tau \tilde{\varphi}_N(s^+)\,\mathrm{d}s\mathrm{d}\tau+\O(N^{-2}).
\end{split}
\end{equation}
By \eqref{eqn: gj2-bound}, we have
\[\left|\frac{t}{N}\sum_{j=1}^Ng_j(z)g_j(z')\right|\le C,\]
so from \eqref{eqn: I13-4terms}, we obtain
\begin{equation}\label{eqn: I13-final}
|I_{1,3}|\le C\|\varphi_N'\|_{L^1}^2\le C.
\end{equation}

Combining the results \eqref{eqn: hilbert-transform}, \eqref{eqn: hilbert-transform2}, \eqref{eqn: I13-final} we find
\begin{align*}
V(\varphi_N)&=\frac{2}{\pi^2}(-I_{1,1}+I_{1,2}+I_{1,3})\\
&= \frac{2}{\pi^2}(-2\cdot \eqref{eqn: I2-display} + \eqref{eqn: I2-d-onzz'})\\
&= \frac{1}{\pi^2}\int_{-Mt}^{Mt}\varphi_N(\tau)H\varphi_N'(\tau)\,\mathrm{d}\tau+\O(1).
\end{align*}

If the function $\varphi_N$ is compactly supported:
\[\mathrm{supp} \varphi_N \subset (-t_1N^1,t_1N^r),\]
the terms $I_{1,2}$, $I_{1,3}$ are small for large $N$. Indeed, by \eqref{eqn: gj2-bound}, \eqref{eqn: stability}:
\begin{align*}
|I_{1,2}|&\le Ct^{-1}\int_{\Omega_N}\int_{\Omega_N}|\dvarp(z)||\dvarp(z')|\left|\frac{\mfct(z)-\mfct(z')}{z-z'}\right|\mathrm{d}z\mathrm{d}z'\\
&\le Ct^{-1}\|\varphi_N\|_{L^1}\log N\\
&\le CN^rN^{\omega_1-\omega_0}\log N,
\end{align*}
with a similar bound holding for $I_{1,3}$. For $I_{1,1}$, the support of $\tilde{\varphi}_N$ means that we can apply Green's theorem to find
\begin{equation}
\label{eqn: I11-compact}
\begin{split}
I_{1,1}=&\frac{t^2}{2}\Re \int_{-N^rt_1}^{N^rt_1}\int_{-N^rt_1}^{N^rt_1}\tilde{\varphi}_N(s)\tilde{\varphi}_N(\tau)\partial_\tau \frac{\mfct(\tau^+)-\mfct(s^+)}{\tau-s+t(\mfct(\tau^+)-\mfct(s^+))} \partial_s \frac{\mfct(\tau^+)-\mfct(s^+)}{\tau-s}\,\mathrm{d}\tau\mathrm{d}s\\
&-\frac{t^2}{2}\Re \int_{-N^rt_1}^{N^rt_1}\int_{-N^rt_1}^{N^rt_1}\tilde{\varphi}_N(s)\tilde{\varphi}_N(\tau) \frac{\mfct(\tau^+)-\mfct(s^-)}{\tau-s+t(\mfct(\tau^+)-\mfct(s^-))} \partial_s \frac{\mfct(\tau^+)-\mfct(s^-)}{\tau-s+2 \i N^{-1+\xi}}\,\mathrm{d}\tau\mathrm{d}s
\end{split}
\end{equation}
By \eqref{eqn: dfracmfct-bound}, \eqref{eqn: dmfct-bound}, the first term in \eqref{eqn: I11-compact} is bounded by
\[Ct^{-1}\|\varphi_N\|_{L^1}^2.\]
By \eqref{eqn: 2nd-order-exp}, the kernel in the second term in \eqref{eqn: I11-compact} is
\[-\frac{1}{(\tau-s+2iN^{-1+\xi})^2}+\O(\max_z|\mfct'(z)|+1/t)\frac{1}{|\tau-s+2\i N^{-1+\xi}|},\]
so that
\begin{equation}\label{eqn: I11-compact-main}
I_{1,1}=\frac{1}{2}\Re\int_{-N^rt_1}^{N^rt_1}\int_{-N^rt_1}^{N^rt_1}\varphi_N(\tau)\varphi_N(s)\frac{1}{(\tau-s+2\i N^{-1+\xi})^2}\,\mathrm{d}s\mathrm{d}\tau +\O(t^{-1}\log N\|\varphi_N\|_{L^1}).
\end{equation}
After integration by parts in $s$, the main term in \eqref{eqn: I11-compact-main} is
\begin{equation} \label{eqn: I11-ibp}
-\frac{1}{2}\Re\int \int_{-N^rt_1}^{N^rt_1}\varphi_N'(s)\frac{\varphi_N(\tau)-\varphi_N(s)}{\tau-s+2\i N^{-1+\xi}}\,\mathrm{d}s\mathrm{d}\tau.
\end{equation}
We have added in the term
\[-\lim_{R\rightarrow \infty} \Re \int \varphi_N'(s)\varphi_N(s) \int_{\{\tau:1/R<|\tau-s|\le R\}}\frac{1}{\tau-s+2\i N^{-1+\xi}}\,\mathrm{d}\tau\mathrm{d}s=0.\]

By the same computation as for \eqref{eqn: hilbert-transform}, this quantity is
\[-\frac{1}{2}\int\varphi_N(\tau)H(\varphi_N')(\tau)\,\mathrm{d}\tau + (t\|\varphi_N'\|_{C^\alpha})^{\frac{1}{1+\alpha}}N^{(-1+\xi)\frac{\alpha}{1+\alpha}} .\]
Reversing the integration by parts in $s$ in \eqref{eqn: I11-ibp}, we obtain the second expression on the right side in \eqref{eqn: Vvar-compact}.

\subsection{Mean}
In this section, we compute the next order correction to the deformed (average) semicircle law:
\begin{thm}
Let $\varphi_N$ be a sequence of functions as in Theorem \ref{thm: linstat}. Let $\lambda_i$ denote the eigenvalues of the deformed model $H_t=V+\sqrt{t}W$. Then
\begin{equation}\label{eqn: mean-difference}
\begin{split}
\mathbb{E}\sum_{j=1}^N \varphi_N(\lambda_i) - N\int \varphi_N(x)\,\rhofct(x)\,\mathrm{d}x&= -t^2\int\varphi_N(x) ((R_2\mfct')(x+i0)-(R_2\mfct')(x-i0))\,\mathrm{d}x \\
&+ \O(tN^{-\xi}+N^{-1/2}t^{-1/2})\|\varphi_N'\|_{L^1} + \O(\|\varphi_N''\|_{L^1}/N)^{1/2}.
\end{split}
\end{equation}
\end{thm}
\begin{proof}
Using the Helffer-Sj\"ostrand representation \eqref{eqn: HS-formula}, the difference \eqref{eqn: mean-difference} can be rewritten as
\[\frac{N}{2\pi}\int_{\mathbb{C}} \dvarp(z) \mathbb{E}[m_N(z)-\mfct(z)]\mathrm{d}z.\]
Proceeding as in Section \ref{sec: characteristic}, we replace the domain of integration by $\Omega_N$:
\begin{equation}
\begin{split}\label{eqn: mean-domain}
N\int_{\mathbb{C}} \dvarp(z) \mathbb{E}[m_N(z)-\mfct(z)]\mathrm{d}z=&N\int_{\Omega_N} \dvarp(z) \mathbb{E}[m_N(z)-\mfct(z)]\mathrm{d}z+\O( \|\varphi_N''\|_{L^1}/N)^{1/2}.
\end{split}
\end{equation}
We now compute
\begin{equation}\label{eqn: mean-mN-diff}
N(m_N(z)-\mfct(z))= \sum_{j=1}^N \left(G_{jj}(z) - \frac{1}{V_j-z-t\mfct(z)}\right)
\end{equation}
for $z\in \Omega_N$.
By \eqref{eqn: justA}, we have
\begin{align*}
G_{jj}(z)& = \frac{1}{V_j-z-\mfct(z)}-\frac{t(\mfct(z)-\moj(z))+\sqrt{t}w_{jj}}{V_j-z-t\mfct(z)}\\
&+ \frac{(t(\mfct(z)-\moj(z))+\sqrt{t}w_{jj})^2}{(V_j-z-t\mfct(z))^2} - \frac{(t(\mfct(z)-\moj(z))+\sqrt{t}w_{jj})^3}{(V_j-z-t\mfct(z))^2A_j}.
\end{align*} 
Putting this in \eqref{eqn: mean-mN-diff} and taking expectations, we obtain:
\begin{align}
N\mathbb{E}[m_N(z)-\mfct(z)] &= -t\sum_{j=1}^N g_j(z) \mathbb{E}[\moj(z)-\mfct(z)] \label{eqn: mean-expansion-1}\\
&+ \frac{t}{N}\sum_{j=1}^N g_j(z)^2 +\Delta_{\text{mean}},
\end{align}
where the error
\begin{align*}
\Delta_{\text{mean}}(z) &= \sum_{j=1}^N \mathbb{E}\frac{(t(\mfct(z)-\moj(z)))^2}{(V_j-z-t\mfct(z))^2}\\
&- \sum_{j=1}^N\mathbb{E}\frac{(t(\mfct(z)-\moj(z)+\sqrt{t}w_{jj})^3}{(V_j-z-t\mfct(z))^2A_j}
\end{align*}
is analytic in $\Omega_N$ with
\begin{equation} \label{eqn: mean-error}
\begin{split}
|\Delta_{\text{mean}}(z)|&\le C\sum_{j=1}^N \frac{t^2|\mfct(z)-\moj(z)|^2}{|V_j-z-t\mfct(z)|^2 }\\
&+ C\sum_{j=1}^N\mathbb{E}\frac{t^3|\mfct(z)-\moj(z)|^3}{|V_j-z-t\mfct(z)|^2|A_j|}\\
&+ C\sum_{j=1}^N \mathbb{E}\frac{t^{3/2}|w_{jj}|^3}{|V_j-z-\mfct(z)|^2|A_j|}
\end{split}
\end{equation}

Next, we write
\begin{align*}
t\sum_{j=1}^N g_j(z) \mathbb{E}[\moj(z)-\mfct(z)]&=t\sum_{j=1}^Ng_j(z) \mathbb{E}[\moj(z)-m_N(z)] +Nt\mfct(z)\mathbb{E}[m_N(z)-\mfct(z)].
\end{align*}

Using this in \eqref{eqn: mean-expansion-1}, we find
\begin{equation*}
N(1+t\mfct(z))\mathbb{E}[m_N(z)-\mfct(z)]= -t\sum_{j=1}^N g_j(z) \mathbb{E}[\moj(z)-m_N(z)] + \frac{t}{N}\sum_{j=1}^N g_j(z)^2 +\Delta_{\text{mean}}(z).
\end{equation*}
By \eqref{eqn: ABidentity}, \eqref{eqn: AinverseB-estimate}:
\begin{equation}
N\mathbb{E}[\moj(z)-m_N(z)]= \frac{1+t\partial_z\mfct(z)}{\mathbb{E}[A_j(z)]} +\O(|\eta|^{-1}(N|\eta|)^{-1}).
\end{equation}
Together with \eqref{eqn: EA}, this shows
\begin{equation}\label{eqn: mean-main-term}
\begin{split}
N(1+t\mfct(z))\mathbb{E}[m_N(z)-\mfct(z)]&= -\frac{t^2}{N}\sum_{j=1}^N g_j^2(z)\partial_z \mfct(z) +\Delta_{\text{mean}}.
\end{split}
\end{equation}
Since $|\mfct(z)|\le Ct^{-1/2}$, we have
$1+t\mfct(z)= 1+\O(t^{1/2}), $
and so we may divide both sides of \eqref{eqn: mean-main-term} by $1-t\mfct(z)$:
\begin{equation}
N\mathbb{E}[m_N(z)-\mfct(z)]= -t^2\partial_z \mfct(z)R_2(z)+\frac{\Delta_{\text{mean}}(z)}{1+\mfct(z)}
\end{equation}
for $z\in \Omega_N$.
We have so far shown that
\begin{align}
N\int_{\Omega_N}\dvarp(z)\mathbb{E}[m_N(z)-\mfct(z)] \,\mathrm{d}z &= -t^2\int_{\Omega_N}\dvarp(z) R_2(z)\partial_z\mfct(z)\,\mathrm{d}z\\
&+\int_{\Omega_N} \dvarp(z) \frac{\Delta_{\text{mean}}(z)}{1+\mfct(z)}\,\mathrm{d}z. \label{eqn: mean-final-error}
\end{align}
By the local law, on $\Omega_N$, we have
\begin{equation}
\begin{split}
\sum_{j=1}^N\frac{t^2|\mfct(z)-\moj(z)|^2}{|V_j-z-\mfct(z)|^2}&=\O(t(N|\eta|)^{-1}|\eta|^{-1}),\\
\sum_{j=1}^N\mathbb{E}\frac{t^3|\mfct(z)-\moj(z)|^3}{|V_j-z-t\mfct(z)|^2|A_j|}&= \O(t (N|\eta|)^{-2}|\eta|^{-1}),\\
\sum_{j=1}^N\mathbb{E}\frac{t^{3/2}|w_{jj}|^{3/2}}{|V_j-z-t\mfct(z)|^2||A_j|} &= \frac{t^{3/2}}{N}\sum_{j=1}^N |g_j(z)|^3 \O(N^{-1/2}).
\end{split}
\end{equation}
It follows:
\begin{align*}
&\int_{\Omega_N}\i\eta\chi(\eta)\varphi''(\tau) \frac{\Delta_{\text{mean}}(z)}{1+t\mfct(z)} \,\mathrm{d}z = -\int_{\Omega_N}\i\eta\chi(\eta)\varphi'(\tau) \partial_\tau \frac{\Delta_{\text{mean}}(z)}{1+t\mfct(z)} \,\mathrm{d}z
\end{align*}
The last quantity is bounded by
\begin{equation}\label{eqn: mean-varphi''-bound}
\|\varphi'_N\|_{L^1} \int_{N^{-1+\xi}}^{10}(t|\eta|^{-1}(N|\eta|)^{-1}+\frac{t^{3/2}}{N}\sum_{j=1}^N |g_j(z)|^3 \O(N^{-1/2}))\mathrm{d}\eta \le C(tN^{-\xi}+ N^{-1/2}t^{-1/2})\|\varphi_N'\|_{L^1}.
\end{equation}
The remaining part of the error term \eqref{eqn: mean-final-error} is
\begin{equation}\label{eqn: mean-error-2}
\int_{\Omega_N} \i\chi'(\eta)(\varphi_N(\tau)+\i\eta\varphi_N'(\tau))\frac{\Delta_{\mathrm{mean}}(z)}{1+t\mfct(z)}\,\mathrm{d}z
\end{equation}
Since $\{|\eta|:\chi'\neq 0\}\subset [N^{10C_V}-1, N^{10C_V}]$, \eqref{eqn: mean-error-2} gives an error of
\begin{equation}\label{eqn: mean-error-3}
C\|\varphi_N\|_{L^1}N^{-20C_V}+\|\varphi_N'\|_{L^1}N^{-10C_V} = CN^{-2}\|\varphi_N'\|_{L^1}.
\end{equation}
Adding the errors \eqref{eqn: mean-domain}, \eqref{eqn: mean-varphi''-bound} and \eqref{eqn: mean-error-3}, we find
\begin{align}
N\int\dvarp(z)\mathbb{E}[m_N(z)-\mfct(z)] \,\mathrm{d}z &= -t^2\int_{\Omega_N}\dvarp(z) R_2(z) \partial_z\mfct(z)\,\mathrm{d}z\\
&+\O(tN^{-\xi}N^{-1/2}t^{-1/2})\|\varphi_N'\|_{L^1}+\O( \|\varphi_N''\|_{L^1}/N)^{1/2}.
\end{align}
A simple computation using $|t\mfct'(z)|\le C$ shows that
\[t^2\int_{\Omega_N^c}\dvarp(z)R_2(z) \partial_z\mfct(z)\,\mathrm{d}z =\O(N^{-1+\xi})\|\varphi_N''\|_{L^1}+\O(tN^{-C_V})(1+\|\varphi_N'\|_{L^1}).\]
The desired result is now obtained by applying Green's theorem to 
\begin{equation*}
t^2\int \dvarp(z) R_2(z)\partial_z\mfct(z)\,\mathrm{d}z = t^2\int\varphi_N(x) ((R_2\mfct')(x+\i 0)-(R_2\mfct')(x-\i 0))\,\mathrm{d}x.
\end{equation*}
\end{proof}\qed

\subsection{$\beta$-ensembles}

In this section we consider the mesoscopic central limit theorem for $\beta$-ensembles.  
The proof is close to the argument that appears in \cite[Theorem 5.4]{homogenization}, except that we consider general potentials $V$. Our statement is somewhat simpler because we are considering functions with small support. The main input is a result of the loop equation \eqref{eqn: first-loop}, first introduced in this context by K. Johansson \cite{johansson}. This provides us with a quadratic relation for the difference between a deformed resolvent $m_{N,\varphi_N}(z)$ 
and the resolvent of the limiting density associated to $V$, with a precision of order $o(1/N)$. Combined with the Helffer-Sj\"ostrand formula \eqref{eqn: HS-formula}, this allows us to obtain a relation for the characteristic function for the linear statistics (see \eqref{eqn: dZ}).

\begin{thm} \label{thm: linstat-beta}
Let $\supp \rho_V=[A,B]$. Let $\varphi_N$ be a sequence of real-valued $C^2$ functions on $\mathbb{R}$ with support in $[2A-B,2B-A]$, satisfying \eqref{eqn: varphiLinfty}, \eqref{eqn: offIq-decay}, in addition to the following growth conditions on the derivatives:
\begin{align*}
\|\varphi_N^{(k+1)}\|_{L^\infty}&\le C_kt_1^{-k}, \quad k= 0,1.
\end{align*}
Let the parameters $t,t_1$ be chosen as in Theorem \ref{thm: linstat}. There is an $\eps>0$ such that uniformly in $|x|\le N^\epsilon$:
\[\mathbb{E}[\e^{\i x[(\mathrm{tr}\varphi_N)(H_t)-N\int \varphi_N(x)\rho_V(x)\mathrm{d}x]}] = \exp\left(-\frac{x^2}{2}V(\varphi_N)+ \i x\delta(\varphi_N)\right)+ \O_\prec(N^{-1+20\epsilon})\|\varphi_N''\|_{L^1}+\O(N^{-2\epsilon}),\]
where the variance $V(\varphi_N)$ is given by
\begin{equation*}
V(\varphi_N) = \frac{1}{2\beta\pi^2}\int_A^B\int_A^B \left(\frac{\varphi_N(x)-\varphi_N(y)}{x-y}\right)^2\frac{-AB-xy-\frac{1}{2}(A+B)(x+y)}{\sqrt{(x-A)(B-x)}\sqrt{(y-A)(B-y))}}\,\mathrm{d}x\mathrm{d}y,
\end{equation*}
and 
\begin{equation*}
\delta(\varphi_N) = \frac{1}{2\pi^2}\left(\frac{2}{\beta}-1\right) \int_A^B \varphi_N(x)  \left(\frac{(H\rho_V)'(\tau+\i 0)}{\rho_V(\tau+\i 0)}- \frac{(H\rho_V)'(\tau-\i 0)}{\rho_V(\tau-\i 0)}\right)\,\mathrm{d}x.
\end{equation*}
\end{thm}
\begin{proof}
The proof uses the loop-equation computation in \cite{homogenization}. It was carried out there for the special case $V(x)=\frac{x^2}{2}$.

For a sequence of functions $\varphi_N$, consider the complex weighted measures
\begin{align*}
\mu_{V,\varphi_N}(\mathrm{d}\mathbf{x}) &= \frac{1}{Z(x)}\e^{\i xS_N(\varphi_N)}\mu_V(\mathrm{d}\mathbf{x}),\\
Z(x) &= \mathbb{E}^{\mu_V}[\e^{\i xS_N(\varphi_N)}],\\
S_N(\varphi_N) &= \sum_{j=1}^N \varphi_N(x_j) - N\int \varphi_N(x)\,\rho_V(x)\,\mathrm{d}x.
\end{align*}
We denote by $\rho_1^{(N,h)}$ the 1-point function of $\mu_{V,h}$.

Define the Stieltjes transforms
\begin{align}
m_{N,h}(z) &= \int \frac{\rho_1^{(N,h)}(x)}{x-z}\,\mathrm{d}x, \qquad m_V(z) = \int \frac{\rho_V(x)}{x-z}\,\mathrm{d}x.
\end{align}

We study the asymptotic behavior of $m_{N,\varphi_N}(z)-m_V(z)$. Define the quantities:
\begin{align}
b(z) &= 2m_V(z)- \partial_E \tilde{V}(z):=2m_V(z)-\partial_E\left(V(E)+\i\eta V'(E)-\frac{\eta^2}{2}V''(E)\right)\\
c_N(z) &= -\frac{2\i x}{\beta N}\int \frac{\varphi_N'(s)}{z-s}\,\rho_V(s)\mathrm{d}s+\frac{1}{N}\left(\frac{2}{\beta}-1\right)m'_{N, \varphi_N}(z)\\
&\quad+ \int \frac{\partial_E \tilde{V}(z)-V'(s)}{z-s}(\rho^{(N,\varphi_N)}_1(s)-\rho_V(s)) \,\mathrm{d}s.
\end{align}

\begin{lem}
Let $\supp \rho_V = [A,B]$, and let $\kappa>0$ small, $\xi>0$ be arbitrary.
Uniformly in 
\[\Omega_N:=\{z= E+\i \eta: N^{-1+\xi}\le |\eta| \le N^{-\xi}, E\in (A+\kappa,B-\kappa)\},\]
we have 
\begin{equation}\label{eqn: betam}
m_{N,\varphi_N}(z)-m_V(z)=-\frac{c_N(z)}{b(z)}+\O\left(\frac{N^\xi \omega_N(z)}{|b_N(z)|}\right),
\end{equation}
where the error $\omega_N(z)$ is given by
\begin{equation}\label{eqn: error-omega}
\omega_N(z) = \frac{N^{-2+\xi/20}}{|Z(x)|^2}\left(\frac{1}{|\eta|}\|\varphi_N''\|_{L^1}+|\eta|^{-2}\|\varphi_N'\|_{L^1}\right).
\end{equation}
\end{lem}
\begin{proof}
The proof in the case of quadratic $V$ is given in \cite[Theorem 5.4]{homogenization}. The same proof can be applied in our case, with minor modifications. Here, we merely point out these differences. 

The main difference with the argument in \cite{homogenization} is that we are dealing with the case of general $V$, not just the Gaussian case. In particular, for non-analytic $V$, we use the analytic extension $\tilde{V}$. From the equilibrium relation, we have
\[V'(E) = -2\Re m_V(E+\i 0),\]
so 
\begin{equation}
|\Im b(z)| = |2\Im m(z)|+\O(N^{-\xi})>c \label{eqn: b-lower-bound}
\end{equation}
in $\Omega_N$. The lower bound \eqref{eqn: b-lower-bound} is essential to the rest of the argument, and explains our choice of upper bound for $|\eta|$ in the definition of the region $\Omega_N$.

First, we have the rigidity estimate
\begin{equation}\label{eqn: wgted-betarigidity}
|\mu_{V,\varphi_N}|\left(|x_k-\gamma_k^{(V)}|>N^{-\frac{2}{3}+\xi}(\widehat{k}^{-1/3}\right)\le \frac{e^{-cN}}{|Z(x)|},
\end{equation}
where $\gamma_k$ is the $k$-th classical location for $\rho_V$. This follows from the result for general potential in \cite{bourgade2014edge}. From this, we have the following estimates as in \cite[Lemma 5.3]{homogenization}: for each $\epsilon>0$, and $0<|\eta|<1$,
\begin{equation}\label{eqn: rho_1torhoV}
\begin{split}
\int \frac{\varphi_N'(s)}{z-s}\,\rho_1^{(N,\varphi_N)}(s)\mathrm{d}s -\int \frac{\varphi_N'(s)}{z-s}\,\rho_V(s)\mathrm{d}s=\,\frac{N^{-1+\epsilon}}{|Z(x)|}\O\left(\int \frac{|\varphi_N''(s)|}{|z-s|}\,\mathrm{d}s+\int\frac{|\varphi_N'(s)|}{|z-s|^2}\,\mathrm{d}s\right),
\end{split}
\end{equation}
\begin{align}
m_{N,\varphi_N}(z)-m_V(z)&= \O_\prec\left(\frac{N^{-1+\epsilon}}{\eta|Z(x)|^2}\right), \label{eqn: m-diff-bound}\\
m_{N,\varphi_N}'(z)-m_V'(z)&= \O_\prec\left(\frac{N^{-1+\epsilon}}{\eta^2|Z(x)|^2}\right),\label{eqn: mprime-bound}\\
\frac{1}{N^2}\mathrm{Var}_{\mu^{V,\varphi_N}}\left(\sum_{k=1}^N\frac{1}{z-x_k}\right) &= \O\left(\frac{N^{-2+2\epsilon}}{\eta^2|Z(x)|^2}\right). \label{eqn: VarV}
\end{align}
The proof given there depends only on \eqref{eqn: wgted-betarigidity}. 

Next, we have the \emph{loop equation} \cite[Eqn. (6.18)]{bourgade2014edge}:
\begin{equation}\label{eqn: first-loop}
\begin{split}
&(m_{N,\varphi_N}(z)-m_V(z))^2+b_N(z)\cdot (m_{N,\varphi_N}(z)-m_V(z))\\
+&\frac{2\i x}{\beta N}\int \frac{\varphi_N'(s)}{z-s}\,\rho_1^{(N,\varphi_N)}(s)\mathrm{d}s-\frac{1}{N}\left(\frac{2}{\beta}-1\right)m'_{N, \varphi_N}(z) \\
+& \int \frac{\partial_E \tilde{V}(z)-V'(s)}{z-s}(\rho^{(N,\varphi_N)}_1(s)-\rho_V(s)) \,\mathrm{d}s \\ 
=&\,\frac{1}{N^2}\mathrm{Var}_{\mu^{V,\varphi_N}}\left(\sum_{k=1}^N\frac{1}{z-x_k}\right)+ \O(e^{-cN}).
\end{split}
\end{equation}

Using the estimates \eqref{eqn: rho_1torhoV}, \eqref{eqn: VarV} in \eqref{eqn: first-loop} leads to following equation for $X_N:= m_{N,\varphi_N}(z)-m_V(z)$:
\begin{equation}\label{eqn: quadratic}
X_N^2(z)+b_N(z)X_N(z) +c_N(z) = \O(\omega_N(z)).
\end{equation}
We can now argue as in \cite{homogenization} that $X_N(z)$ is the root of \eqref{eqn: quadratic} corresponding to $X_N(z)\rightarrow 0$. We choose $\epsilon < \xi/100$ to obtain the error bound \eqref{eqn: error-omega}.
\end{proof}
\qed

We use the Helffer-Sj\"ostrand formula:
\begin{equation}\label{eqn: HS-V-section}
\varphi_N(\lambda) = \frac{1}{\pi}\int_{\mathbb{R}^2} \frac{\bar{\partial}_z \tilde{\varphi}_N(z)}{\lambda-\tau-\i \eta}\,\mathrm{d}z, \quad z=\tau+\i\eta
\end{equation}
where $\bar{\partial}_z = \frac{1}{2}(\partial_\tau +i\partial_\eta)$ and $\tilde{\varphi}_N$ is the almost-analytic extension of $\varphi_N$:
\[\tilde{\varphi}_N(z) = (\varphi_N(\tau)+\i\eta\varphi_N'(\tau))\chi(\eta),\]
where $0\leq \chi(y) \leq 1$ is a cutoff function with $\chi(\eta)=1$, $|\eta|\leq N^{-\xi}/2$, and $\chi(\eta)=0$, $|\eta|>N^{-\xi}$.

The representation \eqref{eqn: HS-V-section} allows us to derive the following
\begin{prop}  We have,
\begin{equation}
\begin{split}
\i\mathbb{E}[\e^{\i xS(\varphi_N)}S(\varphi_N)]&= -x(V(\varphi_N)+\i\delta(\varphi_N))\mathbb{E}[\e^{\i xS(\varphi_N)}]\\
&+|x|\O\left(\frac{N^{-1+20\xi}}{|Z(x)|^2}\right)(1+\|\varphi_N'\|_{L^1})\|\varphi_N''\|_{L^1}+|x|\O\left(\frac{N^{-2\xi}}{|Z(x)|^2}\right)\|\varphi_N'\|_{L^1}^2
\end{split}
\end{equation}
\end{prop}
\begin{proof}

We compute the quantity
\begin{equation}\label{eqn: Svarphi-int}
S(\varphi_N)= \frac{1}{\pi}\int_{\mathbb{R}^2} (\i\eta\varphi_N''(\tau)\chi(\eta) +\i(\varphi_N(\tau)+\i\eta\varphi_N'(\tau))\chi'(\eta))N(m_{N,\varphi_N}(z)-m_V(z))\,\mathrm{d}\eta\mathrm{d}\tau.
\end{equation}
We let $\epsilon> 10 \xi$, and split the integral \eqref{eqn: Svarphi-int} into two regions, $\Omega_N\cap\{|\eta|>N^{-1+\epsilon}\}$ and its complement.

For the integral over $(\Omega_N\cap\{|\eta|>N^{-1+\epsilon}\})^c$, note that $\chi'=0$ in this region, and use \eqref{eqn: m-diff-bound}:
\begin{equation}
\label{eqn: m-diff-outside}
\begin{split}
&\int_{(\Omega_N\cap\{|\eta|>N^{-1+\epsilon})^c} i\eta\varphi_N''(\tau)\chi(\eta)N(m_{N,\varphi_N}(z)-m_V(z))\,\mathrm{d}\tau\mathrm{d}\eta\\
=&\int_{|\eta|<N^{-1+\epsilon} }\int |\eta||\varphi_N''(\tau)|\chi(\eta)\O_\prec\left(\frac{N^{\xi}}{|\eta||Z(x)|^2}\right)\,\mathrm{d}\tau\mathrm{d}\eta\\
=&\O_\prec\left(\frac{N^{-1+20\xi}}{|Z(x)|^2}\right)\|\varphi_N''\|_{L^1}
\end{split}
\end{equation}

We let 
\[\Delta_V(z) = m_N(z)-m_{V,\varphi_N}(z)+\frac{c(z)}{b(z)}.\]
That is, $\Delta_V(z)$ is the error term in \eqref{eqn: betam}. It is analytic in $\Omega_N$, and $|\Delta_V(z)|=\O_\prec(N^{\xi}\omega_N(z)/|b_N(z)|)$.

Thus, integrating by parts:
\begin{equation}\label{eqn: Vclt-error-1}
\begin{split}
&\int_{\Omega_N\cap \{|\eta|>N^{-1+\epsilon}\}} \eta\varphi_N''(\tau)\chi(\eta)\Delta_V(z)\,\mathrm{d}z\\ 
=&\int_{\Omega_N\cap \{|\eta|>N^{-1+\epsilon}\}} \partial_\eta(\eta \chi(\eta)) \varphi_N'(\tau) \Delta_V(z)\,\mathrm{d}z+\O(N^{-1+\xi})\int |\varphi_N'(\tau)||\Delta_V(\tau+\i N^{-1+\xi})|\,\mathrm{d}\tau\\
=&\frac{\|\varphi_N'\|_{L^1}}{|Z(x)|^2}\O_\prec(N^{-1+\xi+\xi/20})\int_{N^{-1+\epsilon}}^{N^{-\xi}} |\partial_\eta(\eta\chi(\eta))|(\|\varphi''\|_{L^1}|\eta|^{-1}+\|\varphi_N'\|_{L^1}|\eta|^{-2}) \,\mathrm{d}\eta\\
=&\O_\prec\left(\frac{N^{-1+2\xi}}{|Z(x)|^2}\right)\|\varphi_N'\|_{L^1}\|\varphi_N''\|_{L^1}+\O_\prec\left(\frac{N^{-2\xi}}{|Z(x)|^2}\right)\|\varphi_N'\|_{L^1}^2.
\end{split}
\end{equation}

For the remainder of the error term, we have
\begin{equation}\label{eqn: Vclt-error-2}
\begin{split}
&\int_{\Omega_N\cap \{|\eta|>N^{-1+\epsilon}\}}(\varphi_N(\tau)+\i\eta\varphi_N'(\tau))\chi'(\eta)\Delta_V(z) \mathrm{d}z\\
=& \int_{\Omega_N\cap \{|\eta|>N^{-1+\epsilon}\}} (|\varphi_N(x)|+|\varphi_N'(x)|)|\chi'(y)|\frac{N^{1+\xi}\omega_N(z)}{|b(z)|}\,\mathrm{d}z\\
=&\O_\prec\left(\frac{N^{-1+3\xi}}{|Z(\lambda)|^2}\right)\|\varphi_N''\|_{L^1}(1+\|\varphi_N'\|_{L^1}).
\end{split}
\end{equation}

Finally, 
\begin{equation}\label{eqn: coverb-outside}
\begin{split}
&\int_{(\Omega_N\cap\{|\eta|>N^{-1+\epsilon}\})^c}\dvarp(z)\frac{c(z)}{b(z)}\,\mathrm{d}z =\int_{(\Omega_N\cap\{|\eta|>N^{-1+\epsilon}\}}\eta\varphi''_N(\tau) \frac{c(z)}{b(z)}\,\mathrm{d}z =\O(N^{-1+\xi}|x|)\|\varphi_N\|_{L^1}.
\end{split}
\end{equation}

Combining the estimates \eqref{eqn: m-diff-outside}, \eqref{eqn: Vclt-error-1}, \eqref{eqn: Vclt-error-2}, \eqref{eqn: coverb-outside}, we have:
\[S(\varphi_N) = -\frac{1}{2\pi}\int_{\mathbb{C}} \dvarp(z)\frac{c(z)}{b(z)}\,\mathrm{d}z+\O\left(\frac{N^{-1+20\xi}}{|Z(x)|^2}\right)(1+\|\varphi_N'\|_{L^1})\|\varphi_N''\|_{L^1}+\O\left(\frac{N^{-2\xi}}{|Z(x)|^2}\right)\|\varphi_N'\|_{L^1}^2.\]

To compute the main term, we use:
\begin{align}
\lim_{\epsilon\rightarrow 0} b(\tau\pm \i\epsilon) &= \pm 2\pi \i \rho_V(\tau),\\
\lim_{\epsilon\rightarrow 0}\int \frac{\varphi_N'(s)}{\tau\pm \i \epsilon -s}\,\rho_V(s)\mathrm{d}s&=\mp \i\pi \varphi_N'(\tau)+ (H_V\varphi_N')(\tau),\\
\end{align}
where
\begin{equation}
H_Vf(\tau) := \mathrm{p.v.} \int \frac{f(s)}{\tau-s}\,\rho_V(s)\mathrm{d}s.
\end{equation}
Applying Green's theorem to the region $\{z: |\Im z|>\epsilon\}$, we obtain
\begin{equation}\label{eqn: V-after-green}
\begin{split}
&\lim_{\epsilon \rightarrow 0} -\frac{1}{\pi}\int_{\{|\Im z|>\epsilon\}} \dvarp(z)\frac{c(z)}{b(z)}\,\mathrm{d}z\\
=&-\frac{x}{\beta \pi^2}\int \varphi_N(\tau) (H_V\varphi_N')(\tau) \frac{1}{\rho_V(x)}\,\mathrm{d}\tau\\
&- \frac{1}{2\pi^2 \i}\int\varphi_N(\tau) \int \frac{V'(\tau)-V'(s)}{\tau-s} (\rho_1^{(N,\varphi_N)}(s)-\rho_V(s)) \,\mathrm{d}s\frac{1}{\rho_V(\tau)}\,\,\mathrm{d}x\\
+&\frac{\i}{2\pi^2}\left(\frac{2}{\beta}-1\right) \int \varphi_N(\tau) \left((H\rho_V)'(\tau+\i 0)\frac{1}{\rho_V(\tau+\i 0)}- (H\rho_V)'(\tau-\i 0)\frac{1}{\rho_V(\tau-\i 0)}\right)\,\mathrm{d}\tau.
\end{split}
\end{equation}

To simplify this expression, write
\[\rho_V(\tau) = \sigma(\tau)\omega(\tau),\]
where
$ \sigma(\tau)=\sqrt{(\tau-A)(B-\tau)}\mathbf{1}_{[A,B]}(\tau), $
and $\omega(\tau)>0$ on $\mathrm{supp}\rho_V$. Using this notation, we write:
\begin{equation}\label{eqn: finite-hilbert}
\begin{split}
(H_V\varphi_N')(\tau) &= \mathrm{p.v.} \int \frac{\varphi_N'(s)}{\tau-s} \rho_V(s) \,\mathrm{d}s = \omega(\tau)\mathrm{p.v.}\int_A^B\frac{\varphi_N'(s)}{\tau-s}\sigma(s)\,\mathrm{d}s - \int_A^B \varphi_N'(s) \frac{\omega(\tau)-\omega(s)}{\tau-s}\sigma(s)\,\mathrm{d}s.
\end{split}
\end{equation}
For a $C^1$ function on $[A,B]$, we define the finite Hilbert transform by
\[H_{AB}f(\tau) = \mathrm{p.v.}\int \frac{f(s)}{\tau-s}\,\mathrm{d}s.\]
The quantity $\omega(\tau)$ can be expressed in terms of $V$ using Tricomi's inversion formula for $H_{AB}$ \cite[p. 179]{tricomi} (see also \cite[Eqn. (3.9)]{johansson}):
\begin{equation} \label{eqn: tricomi-inversion}
\omega(\tau)=\frac{1}{2\pi}\int_A^B \frac{V'(\tau)-V'(s)}{\tau-s}\frac{1}{\sigma(s)}\,\mathrm{d}s.
\end{equation}
An alternative formulation of the relation \eqref{eqn: tricomi-inversion} is
\begin{equation}\label{eqn: HAB-V}
H_{AB}(\omega\sigma)(s)= -\frac{V'(s)}{2}.
\end{equation}
By the equilibrium relation, we have
\begin{align*}
 &H_{AB}(\rho_1^{(N,\varphi_N)}-\rho_V)(\tau)\\
=&\mathrm{p.v.}\int_A^B \frac{1}{\tau-s}(\rho_1^{(N,\varphi_N)}(s)-\rho_V(s))\,\mathrm{d}s\\
=&-\frac{\i x}{\beta N}\varphi'_N(\tau).\\
\end{align*}
This relation can be inverted as \cite[p. 178]{tricomi}:
\[(\rho_1^{(N,\varphi_N)}-\rho_V)(\tau)= \frac{1}{\sigma(\tau)}\frac{\i x}{N\beta}\int_A^B \frac{\varphi_N'(s)}{\tau-s}\sigma(s)\,\mathrm{d}s.\]

With this notation, we write \eqref{eqn: finite-hilbert} as:
\begin{align}
&\omega(\tau)\sigma(\tau)\cdot \frac{1}{\sigma(\tau)}\mathrm{p.v.}\frac{\i x}{N\beta}\int_A^B\frac{\varphi_N'(s)}{\tau-s}\sigma(s)\,\mathrm{d}s-\frac{\i x}{\beta N}\int_A^B \frac{\varphi_N'(s)\omega(s)}{\tau-s}\sigma(s)\,\mathrm{d}s \nonumber \\
=&\omega(\tau)\sigma(\tau)\cdot (\rho_1^{(N,\varphi_N)}-\rho_V)(\tau)+H_{AB}[\omega\sigma H_{AB}(\rho_1^{(N,\varphi_N)}-\rho_V)](\tau). \label{eqn: hilbert-conv-1}
\end{align}
Returning to \eqref{eqn: V-after-green}, we compute, using \eqref{eqn: HAB-V}
\begin{align}
&\frac{1}{2}\int \frac{V'(\tau)-V'(s)}{\tau-s} (\rho_1^{(N,\varphi_N)}(s)-\rho_V(s)) \,\mathrm{d}s \nonumber\\
=& \frac{V'(\tau)}{2}H_{AB}(\rho_1^{(N,\varphi_N)}-\rho_V)(\tau) - H_{AB}[(V'/2)(\rho_1^{(N,\varphi_N)}-\rho_V)](\tau) \nonumber\\
=& -H_{AB}(\omega\sigma)(\tau)H_{AB}(\rho_1^{(N,\varphi_N)}-\rho_V)(\tau)+ H_{AB}[H_{AB}(\omega\sigma)(\rho_1^{(N,\varphi_N)}-\rho_V)](\tau). \label{eqn: hilbert-conv-2}
\end{align}
Using the general \emph{convolution relation} \cite[Eqn. (4), p. 174]{tricomi}:
\[H[\phi_1 H\phi_2 + \phi_2H\phi_1]=H\phi_1 H\phi_2 -\phi_1\phi_2,\]
the two terms \eqref{eqn: hilbert-conv-1} and \eqref{eqn: hilbert-conv-2} sum to zero. Using this in \eqref{eqn: V-after-green}, we find:
\begin{align}
&\lim_{\epsilon \rightarrow 0} -\frac{1}{\pi}\int_{\{|\Im z|>\epsilon\}} \dvarp(z)\frac{c(z)}{b(z)}\,\mathrm{d}z \nonumber\\
=&-\frac{x}{\beta \pi^2}\int_A^B \varphi_N(\tau)\mathrm{p.v.}\int_A^B \frac{\varphi_N'(s)}{\tau-s}\sigma(s) \,\mathrm{d}s \frac{1}{\sigma(\tau)}\,\mathrm{d}\tau \label{eqn: V-quad-term}\\
+&\frac{\i}{2\pi^2}\left(\frac{2}{\beta}-1\right) \int_A^B \varphi_N(\tau)  \left((H\rho_V)'(\tau+\i 0)\frac{1}{\rho_V(\tau+\i 0)}- (H\rho_V)'(\tau-\i 0)\frac{1}{\rho_V(\tau-\i 0)}\right)\,\mathrm{d}\tau. \nonumber
\end{align}
For the term \eqref{eqn: V-quad-term}, a final simplification is possible. Using
\[\mathrm{p.v.}\int_A^B \frac{1}{\tau-s}\frac{1}{\sqrt{(\tau-A)(B-\tau)}}\,\mathrm{d}\tau =0, \quad s\in [A,B],\]
the integral in this term is rewritten as
\[\int_A^B \varphi_N'(s)\int_A^B \frac{\varphi_N(\tau)-\varphi_N(s)}{\tau-s}\sigma(s) \,\mathrm{d}s \frac{1}{\sigma(\tau)}\,\mathrm{d}\tau.\]
Integrating by parts in $s$, we obtain the expression $V(\varphi_N)$.
\end{proof}\qed

To calculate the characteristic function, differentiate $Z(x)$ and use 
\begin{equation}\label{eqn: dZ}
\begin{split}
\frac{\mathrm{d}}{\mathrm{d}x}Z(x) &= \mathbb{E}^{\mu_V}[\e^{\i x S(\varphi_N)} \i S(\varphi_N)] \\
&= -x(V(\varphi_N)+\i\delta(\varphi_N)) Z(x)\\
&+ \O\left(\frac{N^{-1+20\xi}}{|Z(x)|^2}\right)(1+\|\varphi_N'\|_{L^1})\|\varphi_N''\|_{L^1}+\O\left(\frac{N^{-2\xi}}{|Z(x)|^2}\right)\|\varphi_N'\|_{L^1}^2.
\end{split}
\end{equation}

To avoid the extra factors of $|Z(x)|$ appearing in the error term, we consider
\[g(x) = \e^{x^2(V(\varphi_N)+\i\delta(\varphi_N))}Z(x)^2.\] 
Then from \eqref{eqn: dZ}, and recalling the assumptions \eqref{eqn: offIq-decay}, \eqref{eqn: varphiLinfty} on $\varphi_N$, we obtain
\[g'(x) = \e^{x^2(V(\varphi_N)+\i\delta(\varphi_N))}\O(N^{-1+20\xi}\|\varphi_N''\|_{L^1}+N^{-2\xi}).\]
Integrating $g'(x)$ from $0$ to $x$ with $|x|\leq N^{\xi}$, we obtain the theorem.
\end{proof}\qed

\section{Proof of main results} \label{sec:fixed}
\subsection{Proof of Theorem \ref{thm:fixed}} \label{sec:mainproof}
In this section we prove Theorem \ref{thm:fixed},  fixed energy universality for Dyson Brownian motion.  We follow closely Section 4 of \cite{homogenization}, taking advantage of the new input of the mesoscopic CLT of Section \ref{sec: mesosection}.  Let $V$ be $(g , \G)$-regular and fix $t_0 = N^{\om_0}/N$ with $N^{\sigma} \g \leq t_0 \leq N^{ - \sigma} \G^2$.  Let $t_1 = N^{\om_1}/N$ with $\om_1 < \om_0/3$.  Let $|E| < q \G$ be given and let $i_0$ be the index s.t. the classical eigenvalue $\gamma_{i_0} (t_0 + t_1)$ is closest to $E$.

Let $x_i (t)$ denote Dyson Brownian motion with initial data $V$.  Consider the auxilliary process
\beq
\hat{x}_i (t) := a ( x_i (t/a^2 ) - b)
\eeq
where
\beq
a = \rhosc (0) / \rho_{\fc, t_0 } (\gamma_{i_0} (t_0 ) ) , \qquad b= \gamma_{i_0} (t_0 ).
\eeq
Then the process $\hat{x}_i (t)$ is DBM started from initial data $a( V-b)$.  Note that since $V$ is $(g, \G)$-regular, the initial data $a (V-b)$ is $(c g, c \G)$-regular for some $c>0$.  Define $\hat{t}_0 := t_0 / a^2$ and $\hat{t}_1 := t_1 / a^2 $.  At time $\hat{t}_0$ we have that the free convolution law for $\hat{x}_i$ satisfies $\hat{\rho}_{\fc, \hat{t}_0} (0) = \rhosc (0)$ and $\hat{\gamma}_{i_0 } ( \hat{t}_0 ) = 0$. The process $\{ \hat{x}_i (t)\}_i$ satisfies the hypotheses of Theorem \ref{thm:mainhomog}.   Therefore, we have a coupling to a process $\{ y_i (t) \}_i$ that is a DBM started from initial data $y_i (0)$ a GOE matrix independent from $\{ x_i (t) \}_i$.  By definition of $a$, and the rigidity estimates of Theorem \ref{thm:rigidity} we have for $|j| \leq N^{62 \om_1/60}$ that
\beq \label{eqn:revb2}
 |a ( x_{i_0+j} (t_0)  - \gamma_{i_0} (t_0 )) - y_{N/2+j } (0)| \leq \frac{N^{\eps}}{N}
\eeq
with overwhelming probability, where $y_i$ is  GOE ensemble at time $0$.
We further have the estimate that
\begin{align}
&\hat{x}_{i_0+i} (\hat{t}_0 + \hat{t}_1) - \hat{\gamma}_{i_0} ( \hat{t}_0 + \hat{t}_1 ) \notag\\
=& y_{N/2+i} ( \hat{t}_1 ) + \frac{1}{N}\sum_{ |j| \leq N^{ 61 \om_1/60 } } p((i-j)/N) ( \hat{x}_{i_0+j} (\hat{t}_0 ) - y_{N/2+j} (0)  ) + \frac{1}{N}\O ( N^{-c_1} ),
\end{align}
for some $c_1 >0$.  We defined,
\beq
p((i-j)/N ) = \zeta ((i-j)/N, \hat{t}_1 )
\eeq
for notational simplicity, and $\zeta$ is defined in \eqref{eqn:zetarev1}.
 From Proposition \ref{prop:zetaprop} we see that the function $p (x)$ satisfies
\beq \label{eqn:revb1}
|p (x) | \leq C \frac{ t_1}{ t_1^2 + x^2}, \qquad |p' (x) | \leq C \frac{1}{ t_1^2 + x^2}.
\eeq
The constant $c_1 >0$ is fixed for the purposes of this section and depends only on $\om_0$ and $\om_1$.  Hence, we see that for $|i| \leq N^{\om_1/2}$,
\begin{align}
&a \left( x_{i_0 + i} (t_0 + t_1 ) - \gamma_{i_0 } (t_0 + t_1 ) \right) \notag\\
= & y_{N/2 +i  } ( t_1 / a^2 ) + \frac{1}{N} \sum_{|j| \leq N^{ 61 \om_1/60 } } p (  j / N) \left( a ( x_{i_0+j} (t_0)  - \gamma_{i_0} (t_0 )) - y_{N/2+j } (0) \right) + \frac{1}{N} \O ( N^{-c_1} ).
\end{align}
We used \eqref{eqn:revb1} and \eqref{eqn:revb2} to replace $p ((i-j)/N)$ by $p(i/N)$, i.e., 
\begin{align}
&\left| \frac{1}{N} \sum_{|j| \leq N^{ 61 \omega_1/50 } } (p ( (i-j)/N) - p (i/N))  \left( a ( x_{i_0+j} (t_0)  - \gamma_{i_0} (t_0 )) - y_{N/2+j } (0) \right) \right| \notag\\
 \leq & \frac{N^{\eps+ \om_1/2}}{N^3 t_1} \sum_j \frac{t_1}{ t_1^2 + (j/N)^2} \leq C \frac{N^{ \eps-\omega_1/2}}{N}
\end{align}
for $|i| \leq N^{\omega_1/2}$ with overwhelming probability.   

Note that
\beq
a ( \gamma_{i_0 + k} (t_0 ) - \gamma_{i_0} (t_0 ) ) = \gsc_{N/2+k} - \gsc_{N/2} + \frac{1}{N} \O ( N^{-\om_0/10} ) = \frac{k}{N} \frac{1}{ \rhosc (0) } +  \frac{1}{N} \O ( N^{-\om_0/10} )
\eeq
for $|k| \leq N^{62 \om_1/60}$.  With overwhelming probability we have 
\begin{align}
 &  \sum_{|j| \leq N^{ 61 \om_1/60 } } p (  j / N) \left( a ( x_{i_0+j} (t_0)  - \gamma_{i_0} (t_0 )) - y_{N/2+j } (0) \right) \notag\\
=& \sum_{ |j| \leq N^{ \om_1 ( 1+ 1/300)} }  p (  j / N) \left( a ( x_{i_0+j} (t_0)  - \gamma_{i_0} (t_0 )) - y_{N/2+j } (0) \right) + \O ( N^{-\om_1/350} ) \notag\\
=& \sum_{ j }  p ( j / N) \chi_1 \left[ (j/N) / ( t_1 N^{2 \om_1/300} ) \right]  \left( a ( x_{i_0+j} (t_0)  - \gamma_{i_0} (t_0 )) - y_{N/2+j } (0) \right) + \O ( N^{-\om_1/350 } ) ,
\end{align}
where $\chi_1 (x)$ is a smooth cut-off function identically $1$ for $|x| \leq 1$ and $0$ for $|x| > 2$.

Arguing as in \cite{homogenization} with overwhelming probability we can rewrite
\beq
 \sum_{ j } \chi_1 \left[ (j/N) / ( t_1 N^{2 \om_1/300} ) \right]  p (  j / N) \left( a ( x_{i_0+j} (t_0)  - \gamma_{i_0} (t_0 )) - y_{N/2+j } (0) \right) = \left( \zeta_x - \zeta_y + \O ( N^{-c_2} ) \right)
\eeq
for a constant $c_2 >0$ and some mesoscopic linear statistics $\zeta_x$ and $\zeta_y$.  The functions $\zeta_x$ and $\zeta_y$ are of the form 
\beq \label{eqn:zetadeff1}
\zeta_x = \sum_j G ( a( x_j (t_0 ) - \gamma_{i_0 } ( t_0 ) ) ) - G (a ( \gamma_j (t_0 ) - \gamma_{i_0} (t_0)) ), \qquad \zeta_y = \sum_j G ( y_j (0 ) ) - G ( \gsc_j )
\eeq
for a function $G$ which is defined by
\beq \label{eqn:zetadeff2}
G (x) := \int_{0}^x \chi_1 (  s / ( t_1 N^{2 \om_1/300} )) p (s) \d s .
\eeq
Note that by rigidity we have
\beq
| \zeta_x | \leq C  \frac{N^\eps}{N} \sum_j \frac{t_1}{ t_1^2 + (j/N)^2 } \leq C N^\eps
\eeq
for any $\eps >0$, with overwhelming probability.

For simplicity we will only consider the $2$-point function.  It suffices to calculate
\beq
\sum_{i, j} \ee [ Q (  N a(x_i (t_0 + t_1 ) - E), N a( x_j (t_0 + t_ 1)  - x_i (t_0 + t_1 ) ) ) ]
\eeq
for compactly supported smooth $ Q : \rr^2 \to \rr$. 

\subsubsection{Reduction to observables with small Fourier support}
 By the homogenization result and rigidity we can write for any sufficiently small $\delta_R >0$,
\begin{align}
&\sum_{i, j} \ee [ Q (  N a(x_i (t_0 + t_1 ) - E), N a  ( x_j (t_0 + t_ 1)  - x_i (t_0 + t_1 ) ) ) ] \notag\\
=& \sum_{ |i|, |j| \leq N^{\delta_R} }\ee [ Q ( N (y_{N/2+i} ( \hat{t}_1 ) - a (E - \gamma_{i_0} (t_0 + t_1 ) ) ) + \zeta_x - \zeta_y , N (y_{N/2+j} (\hat{t}_1 ) - y_{N/2+i} (\hat{t}_1 )  ))] \notag\\
+& \O ( N^{2 \delta_R-c_3} ),
\end{align}
for $c_3 = \min \{ c_1 , c_2, \om_1/350 \}$.  For simplicity denote
$
d = a(  E - \gamma_{i_0} (t_0 + t_1 ) ) .
$
Note $|d| \leq C$. 
We now make a Fourier cut-off of $Q$.  We denote by $\hat{Q} (\lambda, y)$ the Fourier transform of $Q$ in the first variable.  We let $\psi ( \lambda)$ be the Fourier transform of $\zeta_x$.  Since $\zeta_x$ is independent of $\{ y_i (t) \}_i$, we can write 
\begin{align}
& \sum_{ |i|, |j| \leq N^{\delta_R} }\ee [ Q ( N (y_{N/2+i} (\hat{t}_1 ) - d ) + \zeta_x - \zeta_y , N (y_{N/2+j} (\hat{t}_1) - y_{N/2+i} (\hat{t}_1 )  ))] \notag\\
= & \sum_{ |i|, |j| \leq N^{\delta_R} } \int \d \lambda \psi ( \lambda ) \ee \left[ \hat{Q} ( \lambda, N ( y_{N/2+j} ( \hat{t}_1  ) - y_{N/2+i} ( \hat{t}_1  ) ) ) \e^{ \i \lambda [ N ( y_{N/2+i } - d ) + \zeta_y ] } \right] .
\end{align}
The particles $x_i (t_0 )$ are distributed as the eigenvalues of
\beq
V + \sqrt{t_0} W \equald V + \sqrt{t_0-t_3}W + \sqrt{t_3} W' =: \hat{V} + \sqrt{t_3} W',
\eeq
where $W$ and $W'$ are independent GOE matrices.  We choose $t_3 = N^{\om_3}/N$ with $\om_1 < \om_3 < 2 \om_1$.   We take $\om_3 = \om_1 (2 - 1/5)$.  We calculate
\beq
\left| \psi (\lambda )  \right| = \left| \ee [ \e^{ \i \lambda \zeta_x } ] \right| = \left| \ee \left[ \ee \left[ \e^{ \i \lambda \zeta_x } \vert \hat{V} \right] \right] \right|.
\eeq
The matrix $\hat{V}$ is $(q g, q \G)$-regular with overwhelming probability for any $0 < q < 1$. By our choice of $\om_3$ we can apply Section \ref{sec: mesosection} and conclude that with overwhelming probability,
\beq \label{eqn:fixbd1}
\left|  \ee \left[ \e^{ \i \lambda \zeta_x } \vert \hat{V} \right] \right| \leq \e^{ - c_x \lambda^2 \log (N) } + N^{ - c_4}, \qquad |\lambda | \leq N^{c_4}
\eeq
for some constants $c_4, c_x >0$.   Fix  $\delta_F >0$ and let $\chi_2$ be a smooth compactly supported function s.t. $\chi_2 ( \lambda ) = 1$ for $| \lambda | \leq \delta_F/2$ and $\chi_2 (\lambda ) = 0$ for $|\lambda | > \delta_F$.  For any $\eps >0$ we have,
\begin{align}
&\sum_{ |i|, |j| \leq N^{\delta_R} } \int \d \lambda \psi ( \lambda ) \ee \left[ \hat{Q} ( \lambda, N ( y_{N/2+j} (\hat{t}_1  ) - y_{N/2+i} (\hat{t}_1  ) ) ) \e^{ \i \lambda [ N ( y_{N/2+i } - d ) + \zeta_y ] } \right] \notag\\
= &\sum_{ |i|, |j| \leq N^{\delta_R} } \int \d \lambda \psi ( \lambda ) \ee \left[ \chi_2 ( \lambda ) \hat{Q} ( \lambda, N ( y_{N/2+j} (\hat{t}_1  ) - y_{N/2+i} (\hat{t}_1  ) ) ) \e^{ \i \lambda [ N ( y_{N/2+i } - d ) + \zeta_y ] } \right] \notag\\
 +& \O ( N^{2 \delta_R + \eps} ( N^{ - c_x \delta_F^2/4} + N^{-c_4} ) ).
\end{align}
Above we estimated the region $\delta_F /2 < | \lambda | \leq N^\eps$ using \eqref{eqn:fixbd1}.  The region $| \lambda | > N^\eps$ is estimated using the fact that $Q$ is Schwartz and so $|\hat{Q} ( \lambda, y ) | \leq C_M / ( 1 + | \lambda |^M)$ for any $M >0$.  

Let $Q_1 : \rr^2 \to \rr$ be the function with Fourier transform in the first variable $\hat{Q}_1 ( \lambda, y) = \hat{Q} ( \lambda, y) \chi_1 (\lambda )$.  We see that we have proven 
\begin{align}
&\sum_{i, j} \ee [ Q (  N ( x_i (t_0 + t_1 ) - E ), N ( x_j (t_0 + t_1 ) - x_i ( t_0 + t_ 1 ) ) ) ] \notag\\
=&  \sum_{ |i| , |j| \leq N^{ \delta_R} } \ee [ Q_1 ( N ( y_{N/2+i} (\hat{t}_1 ) - d ) + \zeta_x - \zeta_y , N ( y_{N/2+j} (\hat{t}_1  ) - y_{N/2+i } (\hat{t}_1 ) ) ) ] \notag\\
+ & \O \left(  N^{ 2 \delta_R + \eps} ( N^{-c_3} +  N^{-c_4} + N^{ - c_x \delta_F^2 /4 } ) \right). \label{eqn:Qx}
\end{align}

Note that we have the following bound for $Q_1$. For any $M>0$ there is a constant $C (M, Q, \delta_F )$ s.t.
\beq \label{eqn:Q1fixbd}
| \del_x^\alpha Q_1 (x, y) | + | \del_x^\alpha \del_y Q_1 (x, y ) | \leq C (Q, M, \delta_F ) \frac{1}{ 1 + |x|^M } ( \delta_F)^\alpha
\eeq
for every $\alpha$.

\subsubsection{Reduction to constancy of $G$}
If we repeat the same argument with an auxillliary DBM started from a GOE ensemble, which we denote by $z_i (t)$ then we see that
\begin{align} \label{eqn:Qz}
&\sum_{i, j} \ee [ Q (  N ( z_i ( \hat{t}_1  )  ), N ( z_j (  \hat{t}_1  ) - z_i (  \hat{t}_1  ) ) ) ] \notag\\
=&  \sum_{ |i| , |j| \leq N^{ \delta_R} } \ee [ Q_1 ( N ( y_{N/2+i} (\hat{t}_1 ) ) + \zeta_z - \zeta_y , N ( y_{N/2+j} ( \hat{t}_1  ) - y_{N/2+i } (\hat{t}_1 ) ) ) ] \notag\\
+ & \O \left(  N^{ 2 \delta_R+\eps} ( N^{-c_3} +  N^{-c_4} + N^{ - c_z \delta_F^2 /4 } ) \right),
\end{align}
for a constant $c_z >0$.  Define the function 
\beq
F(s) :=  \sum_{ |i| , |j| \leq N^{ \delta_R} } \ee [ Q_1 ( N  y_{N/2+i} (\hat{t}_1 )  + s - \zeta_y , N ( y_{N/2+j} ( \hat{t}_1  ) - y_{N/2+i } (\hat{t}_1 ) ) ) ].
\eeq
From \eqref{eqn:Qx} and \eqref{eqn:Qz} and the fact that $|\zeta_x| + | \zeta_z | \leq N^\eps$ with overwhelming probability we see that in order to prove fixed energy universality it suffices to show that $|F(s) - F(0) | = o(1)$ for $|s| \leq N^{\delta_R/2}$.

\subsubsection{Preliminary estimates for reverse heat flow}
Define
\beq
F_h (s) := F(s+h) - F(s).
\eeq
We will eventually expand $F_h (s)$ in a power series.  In this section we establish estimates on the terms in that power series.   We follow closely \cite{homogenization}.

For $\alpha \in \nn_{\geq 0 }$ define
\beq
F^\alpha (s) := \sum_{ |i|, |j| \leq N^{ \delta_R } } \ee [ \left( \del_x^\alpha Q_1 \right) ( N  y_{N/2 + i } (\hat{t}_1  )  +s - \zeta_y , N ( y_{j} (\hat{t}_1 ) - y_i (\hat{t}_1  ) ) ) ],
\eeq
and 
\beq
F^\alpha_h (s) := F^\alpha (s+h) - F^\alpha (s).
\eeq

First, the argument of \cite{homogenization} using the translation invariance of the GOE statistics gives
\begin{align} \label{eqn:fixa1}
&\bigg| \sum_{i, j} \ee [ ( \del_x^\alpha  Q_1 ) ( N z_{i} ( \hat{t}_1  ) + s + h , N ( z_j (\hat{t}_1  ) - z_i (\hat{t}_1  ) ) ) ] \notag\\
-& \bigg| \sum_{i, j} \ee [ ( \del_x^\alpha  Q_1 ) ( N z_{i} ( \hat{t}_1  ) + s  , N ( z_j (\hat{t}_1  ) - z_i (\hat{t}_1  ) ) ) ] \bigg| \leq C(Q, \delta_F ) ( \delta_F )^\alpha N^{-1/4}
\end{align}
for $|s|, |h| \leq N^{1/4}$.

By rigidity or the local law for the GOE (see, e.g., Theorem 2.2 of \cite{erdos2012rigidity}) with overwhelming probability, only the eigenvalues $z_i (\hat{t}_1)$ with $|i-N/2| \leq N^{\delta_R}$ close to the point $0$, and so using \eqref{eqn:Q1fixbd} we have
\begin{align} \label{eqn:fixa2}
&\bigg| \sum_{i, j} \ee [ ( \del_x^\alpha  Q_1 ) ( N z_{i} (\hat{t}_1  ) + s  , N ( z_j (\hat{t}_1  ) - z_i (\hat{t}_1  ) ) ) ] \notag\\
-& \sum_{ |i|, |j| \leq N^{ \delta_R} } \ee [ ( \del_x^\alpha  Q_1 ) ( N z_{N/2+i} ( \hat{t}_1  ) + s  , N ( z_{N/2+j} (\hat{t}_1  ) - z_{N/2+i} (\hat{t}_1  ) ) ) ] \bigg| \leq C(Q, \delta_F , M) ( \delta_F )^\alpha N^{- M \delta_R}
\end{align}
for any $M >0$ and $|s| \leq 10 N^{\delta_R/2}$.  By Theorem \ref{thm:mainhomog} applied to $z_i$ and $y_i$,
\begin{align} \label{eqn:fixa3}
&\bigg| \sum_{ |i|, |j| \leq N^{ \delta_R} }  \ee [ ( \del_x^\alpha  Q_1 ) ( N z_{N/2+i} (\hat{t}_1  ) + s  , N ( z_{N/2+j} (\hat{t}_1  ) - z_{N/2+i} (\hat{t}_1 2 ) ) ) ] \notag\\
-& \sum_{ |i|, |j| \leq N^{ \delta_R} } \ee [ ( \del_x^\alpha  Q_1 ) ( N y_{N/2+i} ( \hat{t}_1  ) + s + \zeta_z - \zeta_y  , N ( y_{N/2+j} (\hat{t}_1  ) - y_{N/2+i}(\hat{t}_1  ) ) ) ] \bigg|  \notag\\
\leq & C (Q, \delta_F ) ( \delta_F)^\alpha N^{2 \delta_R - c_3}.
\end{align}
We have the estimate
\beq
\hat{ \zeta_z}  ( \lambda ) = \e^{ - c_z \lambda^2 \log(N) + \i \lambda b_N } + \O ( N^{ - c_5} ), \qquad |\lambda| \leq N^{c_5}
\eeq
for constants $c_z , c_5 > 0$ and $b_N$ satisfying $|b_N| \leq N^{\delta_R/100}$.  Let $\zeta$ be a Gaussian with variance $c_z \log(N)$ and mean $b_N$.  We see that
\beq \label{eqn:fixa4}
\left| \ee[ F^\alpha_h ( s + \zeta ) ] - \ee[ F^\alpha_h (s + \zeta_z ) ] \right| \leq C(Q, \delta_F ) ( \delta_F )^\alpha N^{ 2 \delta_R - c_5 }.
\eeq
Collecting \eqref{eqn:fixa1} \eqref{eqn:fixa2} \eqref{eqn:fixa3} and \eqref{eqn:fixa4}, we see that
\beq
\left| \ee[ F^\alpha_h (s + \zeta ) \right| \leq C(Q, M, \delta_F ) ( \delta_F)^\alpha N^{2 \delta_R} ( N^{-c_5} + N^{-c_3} + N^{-M \delta_R} )
\eeq
for any $M>0$ and $|s|, |h| \leq 5 N^{\delta_R/2}$.

\subsubsection{Reverse heat flow}
Following \cite{homogenization} we can write
\beq
F_h (0) = \sum_{ \alpha=1}^\infty \frac{   ( c_z \log (N)  )^\alpha}{ \alpha!} \ee[ F_h^\alpha ( \zeta + b_N ) ].
\eeq
Hence,
\beq
| F_h (0) | \leq  C N^{ \delta_F^2 c_z} N^{ 2 \delta_R}  ( N^{-c_5} + N^{-c_3} + N^{-M \delta_R} ),
\eeq
for $|h| \leq N^{ \delta_R/2}$.

Take $\delta_F >0$  and small enough so that $\delta_F^2 c_z  \leq  \min\{ c_5, c_3, 1/10 \} / 10$.  Now take $\delta_R >0$ as
\beq
\delta_R = \min\{ c_3, c_4, c_x \delta_F^2/4, c_z \delta_F^2/4 , c_5 \}/10.
\eeq
Now take $M$ large enough so that $M \delta_R > 10$.  We see that we have proven that there is an $\fa >0$ so that
\beq
|F_h (0) | \leq N^{ - \fa}
\eeq 
for $|h| \leq \fa $, and 
\begin{align}
&\sum_{i, j} \ee [ Q (  N ( x_i (t_0 + t_1 ) - E ), N ( x_j (t_0 + t_1 ) - x_i ( t_0 + t_ 1 ) ) ) ] \notag\\
=&  \sum_{ |i| , |j| \leq N^{ \delta_R} } \ee [ Q_1 ( N ( y_{N/2+i} (t_1 / a^2) - d ) + \zeta_x - \zeta_y , N ( y_{N/2+j} ( t_1 / a^2 ) - y_{N/2+i } (t_1 / a^2) ) ) ]  + \O \left(  N^{-\fa} \right), \notag
\end{align}
and
\begin{align}
&\sum_{i, j} \ee [ Q (  N ( z_i ( t_1/a^2 ) - E ), N ( z_j ( t_1 /a^2 ) - z_i (  t_ 1 / a^2 ) ) ) ] \notag\\
=&  \sum_{ |i| , |j| \leq N^{ \delta_R} } \ee [ Q_1 ( N ( y_{N/2+i} (t_1/a^2) ) + \zeta_z - \zeta_y , N ( y_{N/2+j} ( t_1/a^2 ) - y_{N/2+i } (t_1 / a^2) ) ) ]  +  \O \left(  N^{-\fa} \right). \notag
\end{align}
This yields fixed energy universality.

\subsection{Multitime correlation functions}

The proof of Theorem \ref{thm:multitime} is nearly identical to that of Theorem \ref{thm:fixed}.  It suffices to calculate observables of the form, e.g.,
\beq \label{eqn:Qmultitime}
\sum_{i, j} \ee [ Q ( x_i( t_a ) - E (t_a ), (  ( x_j (t_b ) - E(t_b ) ) - ( x_i (t_a ) - E (t_a ) ) ) ].
\eeq
with the energies $E (t)$ defined as in the theorem statement.  Since the mesoscopic part $\zeta_x, \zeta_y$ of the homogenization estimates
\beq
x_i (t)  - E (t) = y_i (t) + \zeta_x - \zeta_y + o (1)
\eeq
are the same for $t_a$ and $t_b$, the proof given above applies to observables of the form \eqref{eqn:Qmultitime}.

\section{General $\beta$-ensembles} \label{sec:beta}
In this section we prove fixed energy universality for $\beta$-ensembles, $\beta \geq 1$.  The strategy is similar to the case of classical DBM; however the coupling must change as we lack a suitable matrix model representation for the DBM flow on $\beta$-ensembles.

\subsection{DBM flow for general $\beta$}
We let $x_i$ be a general $\beta$-ensemble with potential $V$ satisfying the hypotheses in Section \ref{sec:betaresult}.  We let $y_i$ be an independent Gaussian $\beta$-ensemble.  We consider the coupled flows
\beq
\d x_i = \frac{ \sqrt{2} \d B_i } { \sqrt{ \beta N}} + \frac{1}{N} \sum_{j} \frac{1}{ x_i - x_j } \d t - \frac{V' (x_i )}{2} \d t
\eeq
and
\beq
\d y_i = \frac{ \sqrt{2} \d B_i } { \sqrt{ \beta N }} + \frac{1}{N} \sum_j \frac{1}{ y_i - y_j } \d t - \frac{ y_i }{2} \d t.
\eeq
These flows leave the distribution of $\{ x_i\}_i$ and $\{ y_i\}_i$ invariant (however, they obviously do not leave the joint distribution of $\{x_i, y_i \}_i$ invariant).  For notational simplicity we only consider eigenvalues near the index $i_0 = N/2$; the general case proceeds via the same proof.  Note that we do not need to perform the re-indexing argument of Section \ref{sec:resc} because we are in the one-cut case:  we only ever consider $i_0 \in [[ \alpha N,  (1 - \alpha ) N]]$ for a fixed $\alpha >0$, where both the $\beta$-ensemble $V$ and the Gaussian $\beta$-ensemble both exhibit bulk statistics.

We can re-scale and translate the $x_i$ so that the equilibrium density satisfies
\beq
\int_{-\infty}^0 \rho_V (x) \d x = 1/2, \qquad \rho_V (0) = \rhosc (0).
\eeq
If $\gV$ and $\gsc$ are the classical eigenvalue locations of $\rho_V$ and $\rhosc$ respectively, then we have
\beq
| \gV_i - \gsc_i | \leq \frac{C}{N}, \qquad \mbox{for } |i| \leq \sqrt{N}
\eeq
where we have once again shifted indices so that $\gV_0 = \gsc_0=0$ and the indices run over $[[-N/2, N/2]]$.

Recall the equilibrium equation
\beq
\frac{V'(x)}{2} = - \int \frac{ \rho_V (y) \d y }{ y-x}, \qquad \frac{x}{2} = - \int \frac{ \rhosc (y) \d y }{ y-x} .
\eeq

Fix now a parameter $\ell = N^{\om_\ell }$ with
$
\om_\ell < 1/2 
$
 and a $t_1 = N^{\om_1}/N$ with 
$
\om_1 \leq \om_\ell / 4
$. 
We define the following short-range index set $\E$ by
\beq
\E := \{ |i-j| \leq \ell \} \cup \{ ij > 0, |i| \geq N/4,  |j| \geq N/4 \}.
\eeq
We introduce the notation
\beq
\sumEi_j := \sum_{j : (i, j ) \in \E }, \qquad \sumEic_j := \sum_{ j : (i, j) \notin \E }.
\eeq
We consider the process $\xhat_i$ defined by
\beq
\d \xhat_i = \frac{ \sqrt{2} \d B_i } { \sqrt{ \beta N }} + \frac{1}{N} \sumEi_j \frac{1}{ x_i - x_j + \eps_{ij}} \d t + \1_{\{ |i| > N/5\}} \left( \sumEic_j \frac{1}{ x_i - x_j }  - \frac{ V' (x_i ) }{2} \right) \d t
\eeq
where
\beq
\eps_{ij} = N^{-500}, i >j \qquad \eps_{ij} = - N^{-500}, i < j.
\eeq
We will need some level repulsion estimates and the following event $\F_\eps$.  For $\eps >0$ we let $\F_\eps (t)$ be the event that for all $i$ we have
\beq
| x_i (t) - \gV_i | + | y_i (t) - \gsc_i | \leq \frac{ N^\eps}{ N^{2/3} (N/2-|i| +1)^{1/3} }
\eeq
and let
\beq
\F_\eps := \bigcap_{ 0 \leq t \leq 1 } \F_\eps (t).
\eeq
By the rigidity estimates from \cite{bourgade2014universality} and the argument in Appendix \ref{a:cont} we see that $\F_\eps$ holds with overwhelming probability.   The following level repulsion estimates follow from \cite{Gap}.
\bel \label{lem:lr}
There is an $\eps_{LR} >0$ and a small $\delta >0$ so that the following holds.  For $ s \geq \e^{ - N^{\eps_{LR}}}$ we have for any $\eps >0$,
\beq
\pp [ \F_\delta (t) \cap |x_i (t) - x_{i+1} (t) | \leq s/N ] \leq C N^\eps s^2, \qquad\pp [ \F_\delta ( t) |y_i (t) - y_{i+1} (t)  | \leq s/N ] \leq C N^\eps s^2.
\eeq
For any $s >0$ we have
\beq
\pp [ \F_\delta (t) \cap |x_i (t) - x_{i+1} (t)  | \leq s/N ] \leq C N^3 s^2, \qquad\pp [ \F_\delta (t)  |y_i (t) - y_{i+1} (t)  | \leq s/N ] \leq C N^3 s^2.
\eeq
\eel
We have the following estimate.  It is proven by working on the event $\F_\eps$ and combining the proof of (3.7) of \cite{homogenization} (see also the related Lemma 4.4 of \cite{landonyau}) with the proof of Lemma \ref{lem:shortrange}.
\bel  \label{lem:xhatxest} Let $\eps >0$. There is an event with probability at least $1- N^{-300}$ on which
\beq
\sup_{ 0 \leq t \leq 10 t_1} \sup_i | \xhat_i - x_i | \leq t_1N^\eps \left( \frac{\ell}{N} + \frac{1}{ \ell } \right).
\eeq
\eel

\subsubsection{Finite speed estimates and profile of $\UB$}
We define now the operator $\B$ by 
\beq
( \B u )_i := \frac{1}{N} \sumEi_j \frac{u_j - u_i }{ ( x_i - x_j + \eps_{ij} ) ( y_i - y_j + \eps_{ij} ) }.
\eeq
Let
\beq
B_{ij} := \frac{1}{N} \frac{1}{ ( x_i - x_j + \eps_{ij} ) ( y_i - y_j + \eps_{ij} ) }.
\eeq
We need a finite speed estimate analogous to Theorem \ref{thm:finitespeed}.  The main difference is that we can only prove the estimate on an event of polynomially high probability, instead of overwhelming probability.
\bel \label{lem:betashort}
Let $0 < q < 1$ and let $|a| \leq q N /4$.  Let $\delta >0$.  There is an event $\F_a$ with $\pp [ \F_a ] \geq 1 - N^{ -c \delta }$ on which the following holds.  For all $0 \leq s \leq t \leq \ell/N$, we have
\beq
\UB_{ij} (s, t) + \UB_{ji} (s, t) \leq \e^{ - N^{ c' \delta } }, \qquad i \leq a - \ell N^\delta, \mbox{ and } j \geq a + \ell N^{\delta}.
\eeq
\eel
\proof It is more convenient to adapt the proof of \cite{Gap} instead of the proof of Theorem \ref{thm:finitespeed}.   We can assume that the event $\F_{\eps_1}$ holds with a small $\eps_1 >0$.  Fix $\delta>0$.  Let 
$
b := a - \ell N^{  \delta }. 
$
Define
\beq
\phi_j := \e^{  \nu \psi_j }, \qquad \psi_j := \frac{1}{N} \min \{ (j-b)_+, (a-b) \}.
\eeq
%
Note that
\beq
| \psi_j - \psi_k | \leq \frac{ |j-k|}{N}
\eeq
for any $j, k$.  Let $r_j (u)$ satisfy
\beq
\del_u r = \B r
\eeq
with initial condition $r_j (s) = \1_{ \{ j \leq b \} }$, i.e.,
\beq
r_j (u) = \sum_{ k \leq b } \UB_{jk} (s, u).
\eeq
Define $f(u)$ by
\beq
f (u) = \sum_j \phi_j r_j^2 (u).
\eeq
Following \cite{Gap}, we differentiate $f$ and obtain
\begin{align}
f' (u) &= 2 \sum_i \phi_i \sumEi_j r_j (u) B_{ij} (r_j-r_i )  = \sum_{ (i, j) \in \E} B_{ij} (r_j - r_i ) [ r_i \phi_i - r_j \phi_j ] \notag\\
&= \sum_{ (i, j) \in \E } B_{ij} (r_j - r_i ) \phi_i (r_i - r_j ) + \sum_{ (i, j) \in \E } B_{ij} (r_j - r_i ) [ \phi_i-\phi_j ] r_j.
\end{align}
As in \cite{Gap}, we use Schwarz on the second sum, absorbing the part quadratic in $r$ into the first sum which is negative.  We obtain
\beq
f' (u) \leq C \sum_i \sumEi_j B_{ij} \phi_i^{-1} r^2_i (\phi_i - \phi_j )^2.
\eeq
By definition of $\psi$ we see that $\psi_j = \psi_i$ if $|i- b | \geq 2 \ell N^\delta$ and $(i, j) \in \E$.  Hence,
\beq
f' (u) \leq C \sum_{ i : |i-b| \leq 2 \ell N^\delta }  \sum_{j : |i-j | \leq \ell }  B_{ij} \phi_i r_i^2 \left( \frac{ \phi_j }{ \phi_i } -1 \right)^2.
\eeq
On the condition that $ \nu \ell \leq C N$ 
we get for any $\eps >0$
\begin{align}
f' (u) &\leq C \sum_{ i : |i-b| \leq 2 \ell N^\delta } \sum_{j : |i-j| \leq \ell } B_{ij} \frac{ |i-j|^2 }{ N^2} \phi_i r_i^2 \nu^2 \notag\\
&= C \nu^2 \sum_{ i : |i-b| \leq 2 \ell N^\delta}  \phi_i r_i^2 \sum_{ j : |j-i | \leq N^\eps } B_{ij} \frac{ |i-j|^2}{N^2} +  C \nu^2 \sum_{ i : |i-b| \leq 2 \ell N^\delta}  \phi_i r_i^2 \sum_{ j : \ell \geq |j-i | > N^\eps } B_{ij} \frac{ |i-j|^2}{N^2}.
\end{align}
The second term we can estimate by rigidity and obtain
\beq
 \nu^2 \sum_{ i : |i-b| \leq 2 \ell N^\delta}  \phi_i r_i^2 \sum_{ j : \ell \geq |j-i | > N^\eps } B_{ij} \frac{ |i-j|^2}{N^2} \leq C \frac{ \nu^2 \ell}{N} f(u).
\eeq
The first term we estimate by
\beq
\nu^2 \sum_{ i : |i-b| \leq 2 \ell N^\delta}  \phi_i r_i^2 \sum_{ j : |j-i | \leq N^\eps } B_{ij} \frac{ |i-j|^2}{N^2} \leq f (u) \nu^2 \left( \sum_{ i : |i-b| \leq 2 \ell N^\delta }  \sum_{ |i-j| \leq N^\eps }\frac{ B_{ij} |i-j|^2}{N^2} \right).
\eeq
We obtain by Gronwall that, using $t \leq \ell/N$,
\beq
f(u) \leq \exp\left[ C \frac{ \nu^2 \ell^2}{N^2}  + C \nu^2 \int_{0}^{\ell/N} \sum_{ i : |i-b| \leq 2 \ell N^\delta }  \sum_{ |i-j| \leq N^\eps }\frac{ B_{ij} |i-j|^2}{N^2} \d t \right] f(s).
\eeq
By the level repulsion estimates of Lemma \ref{lem:lr}, the fact that we are working on $\F_{\eps_1}$, and Markov's inequality there is an event with probability at least $1 - N^{-\eps/2}$ on which
\beq
\int_{0}^{\ell/N} \sum_{ i : |i-b| \leq 2 \ell N^\delta }  \sum_{ |i-j| \leq N^\eps }\frac{ B_{ij} |i-j|^2}{N^2}  \leq C \frac{N^{4 \eps}  \ell^2 N^{ \delta}}{N^2}.
\eeq
Hence, if we take
\beq
\nu = \frac{ N}{ \ell N^{2 \eps} N^{\delta/2} }
\eeq
we see that 
\beq
f(u) \leq C f(s) \leq C N
\eeq
with probability at least $1 - N^{-\eps/2}$.  By definition we have for $ j \geq b$ and by our choice of $\nu$,
\beq
\nu \psi_j = \nu \frac{ \ell N^\delta}{N} =  N^{\delta/2 - 2 \eps }
\eeq 
Taking $\eps = \delta / 8$ we get the claim. \qed

It is not too hard to adapt the arguments of Section \ref{sec:duham} to prove the following.
\bel 
Let $\eps >0$ and $\delta >0$.  Fix $a, b$ with $a \leq b$ and $|a|, |b| \leq q_1 N /4$ with $0 < q_1 < 1$.  Assume $N^{\eps_1} < b - a < N^{1- \eps_1}$.  There is an event $\F_{ab}$ with $\pp[ \F_{ab}] \geq 1 - N^{ - c \eps} + N^{ - c \delta}$ on which the following holds.
\beq
\UB_{ij} (s, t) + \UB_{ji} (s, t) \leq \frac{N^\eps}{N} \frac{ t -s + 1/N}{  ((b-a)/N )^2 }
\eeq
for every $i \leq a$ and $j \geq b$ and $0 \leq s \leq t \leq \min\{ |a-b| N^{-\delta}/N, 10 t_1 \} $.
\eel
Arguing in a dyadic fashion, this implies the following estimate.
\bel \label{lem:betaprofile3}
Let $\eps >0$ and $\delta >0$.  Let $|a| < q_1 N/4$ with $0 < q_1 < 1$.  There is an event $\F_a$ with $\pp[ \F_a] \geq 1 - N^{ - c \eps} + N^{ - c \delta }$ on which the following estimates hold.
\beq
\UB_{ij} (s, t) + \UB_{ji} (s, t) \leq \frac{N^\eps}{N} \frac{ t -s + 1/N}{  ((a-j)/N )^2 }
\eeq
for any $(i, j, s, t)$ satisfying $i \leq a$ and $j \geq a + N^\eps$ and $0 \leq s \leq t \leq \min\{ N^{-\delta} |i-j| / N, 10 t_1 \}$.
\eel
The proof of the following is a straightforward modification of Lemma \ref{lem:energ}.
\bel \label{lem:betasupbd}
Let $\eps >0$.  Let $0 < q < 1$.  There is an event $\F$ with $\pp [ \F ] \geq 1 - N^{- c \eps }$ on which we have the following estimates.  For all $|a| \leq q N/4$ and $|b| \leq q N/ 4$, and $0 \leq s \leq t \leq \ell/N$,
\beq
\UB_{ab} (s, t) \leq \frac{ N^\eps}{ N (t-s) }.
\eeq
\eel
Combining the previous two lemmas yields the following estimate.
\bel \label{lem:betaprofile}
Let $\eps >0$ and $0 <q_1 < q_2 <1$.  Fix $a$ with $|a| \leq q_1 N/4$.  Fix $0 \leq s \leq 10 t_1$.  There is an event $\F_a (s)$ with probability $\pp [ \F_a (s)] \geq 1 - N^{- c \eps}$ on which the following estimates hold. 
\beq
\UB_{ja} (s, t) + \UB_{aj} (s, t) \leq \frac{N^\eps}{N} \frac{ t-s+1/N}{ (( a-j )/N)^2 + ( t -s )^2 + 1/N^2 } 
\eeq
for every $|j| \leq q_2 N /4$ and $t$ satisfying $s \leq t \leq 10 t_1$.  
\eel
\remark Alternatively, one may fix $t$ and let $s$ vary instead.
We will also later need the following slight variant.
\bel \label{lem:betaprofile2}
Let $\eps >0$, and fix $ 0 < q_1 < q_2 < 1$.  Fix $a$ with $|a| \leq q_1 N/4$.  Fix a scale $t_3 = N^{\om_3}/N$, with $\om_3 \leq \om_1$.  Fix $s$.  There is an event with probability at least $1-N^{-c \eps}$ on which
\beq
\UB_{ja} (u, t) + \UB_{aj} (u, t) \leq \frac{N^\eps}{N} \frac{ t_3}{ (( a-j )/N)^2 + t_3^2 } 
\eeq
for every $j$ and $u, t$ satisfying $|j| \leq q_2 N/4$ and  $s \leq u \leq t \leq s + 10  t_3$, and $|t-u | \geq t_3/10$.
\eel

\subsubsection{Parabolic equation}
Let $u_i = \xhat_i - \yhat_i$.  Then $u_i$ satisfies the equation
\beq
\del_t u_i = \sumEi_j B_{ij} (u_j - u_i ) + \xi_i + A_i
\eeq
where
\beq
\xi_i = \sumEi_j B_{ij} (  ( x_i - x_j ) - ( \xhat_i - \xhat_j ) - (y_i - y_j ) + ( \yhat_i - \yhat_j ) ) 
\eeq
and
\beq
A_i := \1_{\{ |i| > N/5\}} \left( \sumEic_j \frac{1}{ x_i - x_j }  - \frac{ V' (x_i ) }{2} \right)  - \1_{\{ |i| > N/5\}} \left( \sumEic_j \frac{1}{ y_i - y_j }  - \frac{ y_i }{2} \right) .
\eeq
Write $\xi_i$ as
\beq
\xi_i = \xione_i + \xitwo_i
\eeq
where 
\beq
\xione_i := \sumEi_{ j : |i-j| \leq N^\eps} B_{ij}(  ( x_i - x_j ) - ( \xhat_i - \xhat_j ) - (y_i - y_j ) + ( \yhat_i - \yhat_j ) ) ,
\eeq
and
\beq
\xitwo_i := \sumEi_{ j : |i - j | > N^\eps } B_{ij} (  ( x_i - x_j ) - ( \xhat_i - \xhat_j ) - (y_i - y_j ) + ( \yhat_i - \yhat_j ) ) .
\eeq
Due to rigidity and Lemma \ref{lem:xhatxest} we have with probability at least $1 - N^{-290}$ that
\beq \label{eqn:xi2estbeta}
|| \xitwo ||_\infty \leq N^\eps N t_1 \left( \frac{\ell}{N} + \frac{1}{\ell} \right).
\eeq
Using Lemmas \ref{lem:betaprofile3} and \ref{lem:betasupbd} we see that with probability at least $1 - N^{- c \eps } $, 
\beq
\sup_{ 0 \leq t \leq 10 t_1 } \left| \int_0^{t} \sum_j \UB_{aj} (s, t) \xi_j (s) \d s \right| \leq \sum_{ |j-a| \leq N t_1 N^\eps } \int_0^{ 10 t_1} | \xi_j (s)| \d s  + \sum_{|j-a | > N t_1 N^\eps} \int_0^{ 10 t_1} \frac{ N t_1 | \xi_j (s) | }{  ( j -a)^2 }.
\eeq
Hence, using Lemma \ref{lem:xhatxest}, the inequality \eqref{eqn:xi2estbeta} and Markov inequality to deal with the $\xione$ part we see that there is, for each index $a$, an event with probability at least $\pp [ \F_a ] \geq 1 - N^{-c \eps}$ on which
\beq \label{eqn:ubxi}
\sup_{0 \leq t \leq 10 t_1} \left| \int_0^{t} \sum_j \UB_{aj} (s, t) \xi_j (s) \d s \right| \leq \frac{N^{4 \eps}}{N} ( N t_1)^3 \left( \frac{\ell}{N} + \frac{1}{ \ell } \right).
\eeq
Define $v_i$ by
\beq
\del_t v = \B v, \qquad v(0) = u (0).
\eeq
For each $|a| \leq N/6$ we see that by \eqref{eqn:ubxi}  (using Lemma \ref{lem:betashort} to deal with the $A_i$ term) and the Duhamel formula that there is an event $\F_a$ with probability at least $1 - N^{- c \eps}$ on which
\beq
\sup_{0 \leq t \leq t_1} |v_a (t) - u_a (t) | \leq \frac{N^\eps}{N} ( N t_1 )^3 \left( \frac{ \ell}{N} + \frac{1}{ \ell } \right).
\eeq

\subsubsection{Initial data cut-offs}
As in Section \ref{sec:homog}, we can perform initial data cut-offs.  Let $\eps_a > 0$.  We have the following.
\bel \label{lem:betainitcutoff}
Let $\eps >0$. Let $|a| \leq \sqrt{N}$.  Let $\eps > 0$.   There is an event $\F_a$ with probability $\pp [ \F_a ] \geq 1 - N^{-c \eps}$ on which the following estimates hold.  For all $ 0 \leq t \leq t_1$,
\beq
x_a (t) - y_a (t) = \sum_{ |j - a | \leq N t_1 N^{\eps_a} } \UB_{aj} (0, t) ( x_j (0) - y_j (0) ) +\frac{ N^\eps}{N} \O \left( \frac{1}{ N^{\eps_a}} +  ( N t_1 )^3 \left( \frac{\ell}{N} + \frac{1}{ \ell } \right) \right).
\eeq
\eel

\subsubsection{Homogenization}
We now proceed as in Section \ref{sec:actualhomog}.  Fix an $\eps_B >0$.  Fix an index $a$ s.t. 
\beq
|a| \leq N^{1/2-\eps_B}.
\eeq
Define $w_i$ by
\beq
\del_t w = \B w , \qquad w_i (0) = N \delta_a (i).
\eeq
Let $p_t (x, y)$ be as in Section \ref{sec:actualhomog}.  Fix 
\beq
N^{-1} \ll s_0 \ll s_1 \ll t_1 .
\eeq
Recall the flat eigenvalue locations $\gf_j$.  We define $f (x, t)$ by
\beq
f (x, t) = \sum_j \frac{1}{N} p_{s_0 + t - s_1 -s_2 } (x, \gf_j ) w_j (s_1  )
\eeq
and $f_i (t)$ by
$
f_i (t) := f ( y_i (t), t ). 
$
Note that here it ends up being more convenient to use $y_i$ and not $\yhat_i$ above, as opposed to in Section \ref{sec:homog}.

The proof of Lemma \ref{lem:initell2} goes through without change and we have the following.
\bel \label{lem:betainit}  Let $\eps >0$ and $\eps_B > 0$ and $\eps_1 >0$.  There is an event $\F_a$ with $\pp [ \F_a ] \geq 1 - N^{-c \eps}$ on which the following estimate holds.
\begin{align} \label{eqn:initell2}
& ||w (s_1) - f (s_1) ||_2^2 \notag\\
\leq &s_0C  \sum_{|i|\leq N^{1/2-\eps_B}+N^{\om_\ell+\eps_1}} \sum_{|i-j| \leq \ell} \frac{ (w_i (s_1) - w_j (s_1 ))^2}{ (i-j)^2} + N^{\eps} \left( \frac{1}{ (N s_0)^2} + \frac{ (N s_0)^2}{\ell^2} \right) \frac{1}{s_1}.
\end{align}
\eel
Due to the lack of overwhelming probability in our events, we need to argue somewhat differently in order to prove the analog of Lemma \ref{lem:maincalc}.  In particular, we will take the expectation of a martingale, and so we need to introduce the following stopping time denoted by $\tau$.   It is constructed as the minimum of the following $6$ stopping times.    The definition is a little complicated as we need a variety of estimates to hold for the calculations of Lemma \ref{lem:maincalc}.

  Let $\eps_\tau >0$ and $\eps_1 >0$.  First we define the stopping time $\tau_1$ by
\beq
\tau_1 := \inf \{ u \geq s_1  : \exists i, |i| \geq N^{1/2} : w_i > \e^{ - N^{\eps_1/2}} \mbox{ or } f_i > \e^{ - N^{ \eps_1/2}} \}
\eeq
and then the stopping time $\tau_2$ by
\beq
\tau_2 := \inf \left\{ u \geq s_1  : \exists i : w_i (u) \geq \frac{ N^{\eps_\tau} (u - s_1 + 1/N ) }{ ( (a-i)/N)^2 + ( u-s_1 )^2 + 1/N^2 }  \right\}.
\eeq
We define the stopping time
\begin{align}
\tau_3 := &\inf \bigg\{ u \geq s_1 : \exists i :  | \yhat_i (u) - \gsc_i |  \mbox{ or } | y_i (u) - \gsc_i| \mbox{ or } \notag\\
& | \xhat_i (u) - \gV_i | \mbox{ or } | x_i (u) - \gV_i | >\frac{ N^{\eps_\tau/10}} {  ( N/2 - |i|)^{1/3} N^{2/3} }  \bigg\}
\end{align}
and the stopping time 
\begin{align}
\tau_4 := &\inf \bigg\{ u \geq s_1  : \int_{s_1}^u \frac{1}{ N (t - s_1 + s_0 )^2 (t + s_1 ) } \frac{1}{N^2} \sum_{ |j-i| \leq N^{\epst}, |i| \leq N^{1/2} } \frac{1}{ |y_i - y_j + \eps_{ij} | } + \frac{1}{ |x_i -x_j +\eps_{ij} | }\notag\\
 &\times\left( \frac{ t }{ t^2 + (( i-a)/N)^2 } + \frac{ ( s_1  ) \vee ( t -s_1  + s_0 ) }{ ( (i-a)/N)^2 +( s_1 \vee ( t-s_1  + s_0 ) )^2 } \right) \d t \geq \frac{ N^{2 \epst}}{ N s_0 s_1 }  \bigg\}
\end{align}
and the stopping time
\beq
\tau_5 = \inf \bigg\{ u:   u \geq s_1  : \int_{s_1}^u \sum_{ i, j}  \frac{1}{ | x_i - x_j  | } + \frac{1}{ |y_i - y_j | }  + \frac{1}{ |y_i - y_j + \eps_{ij} ||x_i - x_j + \eps_{ij} | } \d t \geq N^{10}   \bigg\},
\eeq
and finally
\beq
\tau_6 = \inf\bigg\{ u : u \geq s_1 : \int_{s_1 } \sum_{i,j} \left|  \frac{1}{ y_i - y_j }  - \frac{1}{ y_i - y_j +\eps_{ij} } \right| \d t \geq \frac{1}{ N^{50}} \bigg\}
\eeq
We set
\beq
\tau := \tau_1 \wedge \tau_2 \wedge \tau_3 \wedge \tau_4 \wedge \tau_5 \wedge \tau_6 \wedge 10 t_1.
\eeq
Lemmas \ref{lem:betashort} and \ref{lem:betaprofile} (for $\tau_6$ see the proof of (3.7) of \cite{homogenization}) imply that
\beq
\pp [ \tau < 10 t_1 ] \leq N^{ -c \eps_1} + N^{ -c \eps_\tau }.
\eeq

For $s_1  \leq t \leq \tau$ we have
\begin{align}
f (x, t) &\leq \frac{C N^\epst }{N} \sum_j \frac{ (t + s_0 ) } { (j/N - x)^2 + ( t + s_0 )^2 } \frac{ s_1  }{ j^2/N^2 + (s_1  )^2 } \notag\\
&\leq C N^\epst\int \frac{ t+s_0 } { (y-x)^2 + ( t+ s_0 )^2 } \frac{ s_1 }{ y^2 + (s_1  )^2 } \d y  \leq C N^\epst \frac{ ( t+s_0 ) \vee (s_1  ) }{ x^2 + ( ( t + s_0 )\vee (s_1  ) )^2 }.
\end{align}
Here we used that for $s \leq t$ and $x >0$, that if $x > t$,
\begin{align}
\int \frac{ t}{ (x-y)^2 + t^2} \frac{ s}{ y^2 + s^2} \d y &= \int_{ y \leq x/2} + \int_{ y \geq 3x/2} + \int_{ x/2 < y < 3x/2}  \frac{ t}{ (x-y)^2 + t^2} \frac{ s}{ y^2 + s^2} \d y \notag\\
&\leq C \frac{t}{x^2} \left( \int_{ y < x/2} + \int_{y > 3x/2} \frac{s}{y^2+s^2} \d y \right) + \frac{s}{x^2} \int_{ x/2 < y < 3x/2} \frac{ t}{ (y-x)^2 + t^2 } \d y \notag\\
&\leq C \frac{t}{x^2 + t^2}
\end{align}
and if $|x| \leq t$,
\beq
\int \frac{ t}{ (x-y)^2 + t^2} \frac{ s}{ y^2 + s^2} \d y \leq \frac{C}{t} \int \frac{s}{y^2 + s^2} \d y \leq \frac{ C t}{ t^2 + x^2}.
\eeq

\bel \label{lem:betamain}
We have
\beq
\d \frac{1}{N} \sum_i (w_i - f_i )^2 = - \frac{1}{2} \langle w -f, \B ( w - f ) \rangle + X_t \d t + \d M_t.
\eeq
The term $M_t$ is a martingale.  For $X_t$ we have
\beq
\int_{s_1 }^{ u} X_t \d t \leq \frac{1}{5} \int_{s_1}^u \langle (w-f), \B (w-f ) \rangle \d t + \frac{N^{8 \epst}}{ s_1} \left( \frac{1}{ (N s_0 )^{1/2} }\right),
\eeq
for any $u$  with $ s_1 \leq u \leq \tau$.
\eel
\proof For $t \leq \tau$ the bounds \eqref{eqn:fders} to \eqref{eqn:finf} hold, and we also have 
\beq
|f_i | + w_i \leq \e^{ - N^{\eps_1/2}}, \qquad |i| \geq N^{1/2}.
\eeq
We will use these tacitly in the proof.  Using the Ito formula we calculate
\begin{align}
&\d \frac{1}{N} \sum_i (w_i - f_i )^2 = \frac{1}{N} \sum_i \left[ (w_i - f_i ) \left( \del_t w_i \d t - \del_t f_i \d t - f_i' \d \haty_i - f''_i \frac{ \d t }{N} \right) - (f_i')^2 \frac{ \d t }{N} \right] .
\end{align}
The Ito terms are handled as before and we get
\beq
 \int_{s_1+s_2}^{\tau}  \frac{1}{N} \sum_i \left| (w_i - f_i ) f''_i  + (f'_i)^2 \right| \frac{ \d t}{N}  \leq N^{\epst} \frac{1}{ N s_0 s_1 }.
\eeq
We make the same calculation as in Lemma \ref{lem:maincalc} and write
\begin{align} 
\frac{1}{N} \sum_i (w_i - f_i ) ( \del_t w_i - ( \del_t f)_i ) &= - \frac{1}{2} \langle w-f, \B (w-f ) \rangle \notag\\
&+ \frac{1}{N} \sum_i (w_i - f_i ) \Bigg( \sumEi_j B_{ij} (f_j - f_i ) - \int_{ |y - \yhat_i  | \leq \eta_\ell } \frac{ f(y ) - f ( y_i ) }{ (y_i - y )^2} \rhosc (0) \d y \Bigg).\label{eqn:bet1}
\end{align}
Fix $0 < \om_{\ell,2} < \om_\ell$ and define
\beq
\E_2 := \{ (i, j) : |i - j | \leq N^{ \om_{\ell,2}} \} \cup \{ (i, j) : ij > 0, |i| \geq N/4, |j| \geq N/4 \}.
\eeq
We then write the second term in \eqref{eqn:bet1} as 
\begin{align}
&\frac{1}{N} \sum_i (w_i - f_i ) \left( \sumEi_j B_{ij} (f_j - f_i ) - \int_{ |y - \yhat_i  | \leq \eta_\ell } \frac{ f(y ) - f ( y_i ) }{ (y_i - y )^2} \rhosc (0) \d y \right) \notag\\
=& \frac{1}{N} \sum_i (w_i - f_i ) \left( \sumEti_j B_{ij} (f_j - f_i ) \right) \label{eqn:bet2} \\
+& \frac{1}{N} \sum_i (w_i - f_i ) \left( \sumEthi_j B_{ij} (f_j - f_i ) - \int_{ | \yhat_i - y | \leq \eta_\ell } \frac{ f (y) - f ( y_i ) }{ ( y_i - y)^2} \rhosc (0) \d y \right).\label{eqn:bet3}
\end{align}
Let
$
v_i := w_i - f_i.
$
We first turn to \eqref{eqn:bet2}.  For $|i| \leq N^{1/2}$ we have
\begin{align}
  \sumEti_j B_{ij} (f_j - f_i ) &= \frac{1}{N} \sumEti_j \frac{ f'_i}{ (x_j - x_i + \eps_{ij} ) } -  \sumEti_j B_{ij} \eps_{ij} f'_i \notag\\
  &+ \frac{N^{2\epst}}{N (t-s_1  + s_0 )^2 (t+s_1 ) } \O \left(  N^{\om_{\ell,2}} + \frac{1}{N} \sum_{ |j-i| \leq N^{\epst} } \frac{  1}{ |x_i -x_j + \eps_{ij} |}\right) . \label{eqn:bet4}
\end{align}
For $|i| > N^{1/2}$ we use
\begin{align}
\sumEti_j B_{ij} (f_j - f_i ) &= \frac{1}{N} \sumEti_j \frac{ f'_i}{ x_j - x_i }  - \sumEti_j B_{ij} \eps_{ij} f'_i + N^2 \O \left(  N+ \sum_{ |j-i| \leq N^{\epst} } \frac{  1}{ |x_i -x_j + \eps_{ij} |}\right) \label{eqn:bet5}.
\end{align}
Using \eqref{eqn:bet4} and \eqref{eqn:bet5}, and the definition of $\tau$ we see that for the term \eqref{eqn:bet2} we have
\begin{align}
\int_{s_1}^{u} \frac{1}{N} \sum_i (w_i -f_i ) \left( \sumEti_j B_{ij} (f_j - f_i ) \right) \d t &= \int_{s_1}^{u} \frac{1}{N^2} \sum_i v_i \sumEti_j \frac{ f'_i}{ x_i -x_j + \eps_{ij} } \d t \label{eqn:bet6} \\
&+ \O \left( \frac{ N^{\om_{\ell,2} } N^{4 \epst}}{ N s_1 s_0 } \right),
\end{align}
for any $u \leq \tau$.  We now handle \eqref{eqn:bet6}.  We rewrite it as
\begin{align}
\frac{1}{N^2} \sum_i v_i \sumEti_j \frac{ f_i'}{ x_i - x_j + \eps_{ij} } = \frac{1}{2} \frac{1}{N^2} \sum_{ (i, j) \in \E_2} \frac{ (v_i -v_j ) f_i' + v_j ( f_i' -f_j' ) }{ x_j -x_i + \eps_{ij} }  \label{eqn:bet7}.
\end{align}
The second term on the RHS of \eqref{eqn:bet7} is bounded by
\begin{align}
&\left| \frac{1}{2} \frac{1}{N^2} \sum_{ (i, j) \in \E_2} \frac{ v_j (f_i' - f_j')}{ x_i - x_j + \eps_{ij} }  \right| \leq \left| \frac{1}{2} \frac{1}{N^2} \sum_{ | i - j  | \leq N^{ \om_{\ell,2}}, |i| \leq N^{1/2} } \frac{ v_j (f_i' - f'_j ) }{ x_j - x_i + \eps_{ij} } \right|\notag\\
& + \e^{ - N^{\eps_1/2} } N^4 \left(1 + \sum_{|i-j| \leq N^{\epst}} \frac{ |y_i - y_j |}{ |x_i -x_j + \eps_{ij}|} \right) \notag\\
&\leq \frac{  N^{ 2 \epst}}{ N (t - s_1 + s_0 )^2 (t+s_1 ) } \bigg\{ N^{\om_{\ell,2}}  \notag\\
&+\frac{1}{N} \sum_{ |i-j| \leq N^{\epst} , |i| \leq N^{1/2} } \frac{ |y_i -y_j |}{ |x_i -x_j + \eps_{ij} | }  \left( \frac{t}{( (a-i)/N)^2 + t^2 } + \frac{ (t+s_0) \vee (s_1  ) }{ ((a-i)/N)^2 + ( (s_1  )\vee ( t + s_0 ) )^2 } \right) \bigg\} \notag\\
&+ \e^{ - N^{\eps_1/2} } N^4 \left(1 + \sum_{|i-j| \leq N^{\epst}} \frac{ |y_i - y_j |}{ |x_i -x_j + \eps_{ij}|} \right).
\end{align}
Hence, by the definition of $\tau$ we see that for any $u \leq \tau$
\beq
\int_{s_1}^u \left| \frac{1}{2} \frac{1}{N^2} \sum_{ (i, j) \in \E_2} \frac{ v_j (f_i' - f_j')}{ x_i - x_j + \eps_{ij} }  \right| \d t \leq \frac{ N^{5 \epst} N^{\om_{\ell,2}}}{ N s_0 s_1 }.
\eeq
The first term on the RHS of \eqref{eqn:bet7} is bounded using the Schwarz inequality.  We obtain
\begin{align}
&\left| \frac{1}{2} \frac{1}{N^2} \sum_{ (i, j) \in \E_2 } \frac{ (v_i -v_j ) f_i'}{ x_j - x_i + \eps_{ij} } \right| \leq \frac{1}{N} \frac{1}{10} \sum_{ (i, j) \in \E_2 } B_{ij} (v_i - v_j )^2 + \frac{C}{N^2} \sum_i ( f_i')^2 \sumEti_j \frac{ |y_i - y_j + \eps_{ij} | }{ | x_i - x_j + \eps_{ij} | } \notag\\
&\leq \frac{1}{N} \frac{1}{10} \sum_{ (i, j) \in \E_2 } B_{ij} (v_i - v_j )^2 + C \frac{ N^{ 2 \eps_\tau} N^{\om_{\ell,2}}}{ N ( t - s_1  + s_0 )^2 ( t + s_1 ) } \notag\\
&+ \frac{N^{ 3 \epst}}{ N^3 (t -s_1 + s_0)^2 (t + s_1 ) } \sum_{ |i - j | \leq N^{ \epst} , |i| \leq N^{1/2} } \frac{ 1 }{ | x_i - x_j + \eps_{ij} | }   \frac{ (t+s_0) \vee (s_1  ) }{ ((a-i)/N)^2 + ( (s_1  )\vee ( t + s_0 ) )^2 } \notag\\
&+ \e^{ - N^{\eps_1/10}} \sum_{ |i - j | \leq N^{ \epst} , |i| > N^{1/2} } \frac{1 }{ | x_i - x_j + \eps_{ij} | } .
\end{align}
Hence by the definition of $\tau$ we see that for any $u \leq \tau$,
\begin{align}
\int_{s_1}^u \left| \frac{1}{2} \sum_{ (i, j) \in \E_2 } \frac{ (v_i -v_j ) f_i'}{ x_j - x_i + \eps_{ij} } \right| \d t &\leq \frac{1}{10} \int_{s_1}^u \langle w-f, \B (w-f ) \rangle \d t + \frac{ N^{3 \epst} N^{ \om_{\ell,2}}}{ N s_1 s_0 }.
\end{align}
This finishes the estimate for \eqref{eqn:bet2}.  The estimate for \eqref{eqn:bet3} is handled using rigidity in the same manner as the term \eqref{eqn:timederer4} is handled in the proof of Lemma \ref{lem:maincalc} ( Note that for $|i| \leq N^{1/2}$ we have that $| \gsc_i - \gV_i | \leq C/N$ so the change from $(\hatz_i - \hatz_j )^{-2}$ to $(x_i -x_j + \eps_{ij} )^{-1}( y_i - y_j + \eps_{ij} )^{-1}$ does not affect anything; we use the exponential bound for the terms $|i| > N^{1/2}$ so we can discard them).  We obtain
\beq
\int_{s_1+s_2}^\tau \left| \frac{1}{N} \sum_i (w_i - f_i ) \left( \sumEthi_j B_{ij} (f_j - f_i ) - \int_{ | \yhat_i - y | \leq \eta_\ell } \frac{ f (y) - f ( y_i ) }{ ( y_i - y)^2} \rhosc (0) \d y \right) \right| \d t \leq \frac{ N^{2 \epst}}{ s_1 N^{\om_{\ell,2}}}.
\eeq
In summary we have so far proven that
\begin{align}
\frac{1}{N} \sum_i w_i ( \del_t w_i - ( \del_t f ) _i ) = - \frac{1}{2} \langle w-f, \B (w -f ) \rangle + Y_t
\end{align}
where
\beq
\int_{s_1}^u |Y_t| \d t \leq  \int_{ s_1}^u \frac{1}{10} \langle w-f, \B (w-f ) \rangle \d t + N^{6 \epst} \left(  \frac{1}{ N^{\om_{\ell,2}} s_1 } + \frac{N^{\om_{\ell,2}}}{ N s_0 s_1 } \right).
\eeq
The remaining term to deal with is 
\begin{align}
\frac{1}{N} \sum_i (w_i - f_i ) f_i' \d y_i &= \d M_t + \frac{1}{N} \sum_i (w_i - f_i ) f_i' \left( \frac{1}{N} \sum_j \frac{1}{ y_i - y_j } - \frac{ y_i}{2} \right)\d t \notag\\
&= \d M_t + \frac{1}{N} \sum_i (w_i - f_i) f_i' \bigg\{ \frac{1}{N} \sumEti_j \frac{1}{ y_i -y_j + \eps_{ij} }  \label{eqn:bet8} \\
&+ \frac{1}{N} \sumEti_j \frac{1}{ y_i - y_j} - \frac{1}{ y_i - y_j + \eps_{ij} } \label{eqn:bet1a} \\
&+ \frac{1}{N} \sumEtci_j \frac{1}{ y_i - y_j } - \frac{ y_i}{2} \bigg\} \d t. \label{eqn:bet9}
\end{align}
where the Martingale term is
\beq
\d M_t = \frac{1}{N} \sum_i (w_i - f_i ) f_i' \frac{ \d B_i }{ \sqrt{N}}.
\eeq
The term \eqref{eqn:bet1a} can be discarded due to the definition of $\tau_6$.  The term \eqref{eqn:bet8} is similar to \eqref{eqn:bet7}, and a similar argument gives
\beq
\int_{s_1}^u \left| \frac{1}{N} \sum_i (w_i - f_i) f_i'  \frac{1}{N} \sumEti_j \frac{1}{ y_i -y_j  + \eps_{ij}}  \right| \d t \leq \int_{s_1}^u \langle w-f, \B ( w- f ) \rangle \d t + C \frac{ N^{\om_{\ell,2}} N^{5 \epst}}{ N s_0 s_1 } .
\eeq
The term \eqref{eqn:bet9} is easily handled with rigidity and we obtain 
\beq
\int_{s_1}^u \left| \frac{1}{N} \sum_i (w_i - f_i ) f_i' \left( \sumEtci_j \frac{1}{ y_i - y_j } - \frac{ y_i }{2} \d t \right) \right| \d t \leq \frac{N^{3 \epst}}{ s_1 N^{\om_{\ell,2}}} + \frac{ N^{ 3 \epst} N^{\om_{\ell,2}}}{ s_1 N}
\eeq
This completes the proof after taking $ N^{\om_{\ell,2}} = ( N s_0 )^{1/2}$. \qed

With this in hand the completion of the homogenization theorem is very similar to Section \ref{sec:homog}. Recall that we need to average over $[s_1 , 2 s_1]$.  Therefore, as in Section \ref{sec:homog} we introduce $s_1' \in [s_1, 2s_1 ]$ and define the object the function $f$ with $s_1'$ instead of $s_1$. Fix a small $\eps_3 >0$.   First of all, let $\F$ be the event that the bounds of Lemma \ref{lem:betaprofile2} hold at the scales $t_3 = s_1$ and also $t_3 = 8 t_1$, with exponent $\eps_3/2$.  Let us also demand that on the event $\F$ we have  $|\UB_{ia} | \leq \e^{ - N^{\eps_1/10}}$ for $|i| > N^{1/2-\eps_1/4}$.   We have that $\F $ holds with probability $\pp [ \F ] \geq 1- N^{-c \eps_3}$.

By the proof of Theorem \ref{thm:techhomog} we have for $|i-a| \leq \ell/10$, 
\begin{align}
& \frac{1}{t_1} \int_0^{t_1} \1_{ \F } \left( \UB_{t_1+u} (i, a) - \frac{1}{N} p_{t_1+u} ( \gf_i, \gf_a ) \right)^2 \d u \leq \frac{ N^{\eps_3}}{ ( N t_1 )^4 } + \frac{ N^{\eps_3}}{ ( N t_1 )^2} \frac{ s_1^2}{t_1^2} \notag\\
&+ \frac{1}{s_1} \int_{s_1}^{ 2 s_1} \d s_1' \frac{1}{t_1} \left( \1_{\F } \int_{0}^{ t_1 } \d t \left( \frac{1}{N} w_{t_1+u} (i) - \frac{1}{N} f_{t_1+u} (i) \right)^2 \right).
\end{align}
We used the bounds Lemma \ref{lem:betaprofile2} at the scale $t_3 = s_1$ (which hold on $\F$) to bound the term analogous to \eqref{eqn:up2}.   
 Let $\F_\tau$ be the event $\{ \tau \geq 10 t_1 \}$, where $\tau$ is the same stopping time as above.  Here the definition of $\tau$ is with $s_1'$ instead of $s_1$.  Taking expectations we bound
\begin{align}
&\ee\left[ \1_{\F } \frac{1}{t_1} \int_{0}^{ t_1 } \d t \left( \frac{1}{N} w_{t_1+u} (i) - \frac{1}{N} f_{t_1+u} (i) \right)^2 \right] \notag\\
\leq & \ee\left[ \frac{1}{t_1}  \1_{\F \cap \F_\tau }  \int_{t_1}^{ \tau } \d t \left( \frac{1}{N} w_{u} (i) - \frac{1}{N} f_{u} (i) \right)^2 \right] + C \frac{ N^{2 \eps_3}} { ( N t_1)^2 } \pp [ \{ \tau < 10 t_1 \} ] \notag\\
\leq& \ee\left[ \frac{1}{t_1} \1_{ \F } \int_{t_1}^{ \tau } \d t \left( \frac{1}{N} w_{u} (i) - \frac{1}{N} f_{u} (i) \right)^2 \right] + C \frac{ N^{2 \eps_3}} { ( N t_1)^2 } N^{ - c \epst}.
\end{align}
We then bound, as in the proof of Theorem \ref{thm:techhomog}
\begin{align}
 & \ee\left[ \frac{1}{t_1}  \1_{   \F} \int_{t_1}^{ \tau } \d t \left( \frac{1}{N} w_{u} (i) - \frac{1}{N} f_{u} (i) \right)^2 \right] \notag\\
\leq & \ee \left[ \frac{1}{t_1} \1_{ \F } \1_{ \{ \tau > s_1'\} }\int_{s_1'}^{\tau} \d t \frac{ N^\eps}{N^2} \langle w-f, \B (w-f) \rangle \d t \right] + \frac{C N^{ 2 \eps_3}}{ \ell^2 ( N t_1 )^2 } + C \frac{ N^{2 \eps_3} ( N t_1 )^2}{ \ell^4 } .
\end{align}
for any small $\eps >0$.   Let $\F_2$ be the event that $||w(s_1) ||_2^2 \leq N^{\eps_3}/s_1$.   We have that $\F \subseteq \F_2$.  Hence,
\begin{align}
 & \ee \left[ \frac{1}{t_1} \1_{ \F } \1_{ \{ \tau > s_1' \} }\int_{s_1'}^{\tau} \d t \frac{ N^\eps}{N^2} \langle w-f, \B (w-f) \rangle \d t \right] \notag\\
\leq &  \ee \left[ \frac{1}{t_1} \1_{ \F_2 } \1_{ \{ \tau > s_1' \} }\int_{s_1'}^{\tau} \d t \frac{ N^\eps}{N^2} \langle w-f, \B (w-f) \rangle \d t \right]
\end{align}
One can check that on the event that $\{ \tau > s_1' \}$ the bound \eqref{eqn:initell2} holds with error $\eps_\tau$.  Applying first Lemma \ref{lem:betamain} and then the bound \eqref{eqn:initell2} we get,
\begin{align}
 & \ee \left[ \frac{1}{t_1} \1_{ \F_2 } \1_{ \{ \tau > s_1'  \} } \int_{s_1'}^{\tau} \d t \frac{ N^\eps}{N^2} \langle w-f, \B (w-f) \rangle \d t \right] \notag\\
\leq & \ee \left[ \frac{N^\eps}{t_1} \1_{ \F_2} \1_{ \{ \tau > s_1' \} }  \int_{s_1'}^{\tau} d M_t \right] + C \frac{N^\eps}{N^2 t_1} \ee[ \1_{ \F_2} \1_{ \{ \tau > s_1' \} } || (w-f) (s_1' ) ||_2^2 ] + \frac{N^{ 8\epst}}{N^2 t_1 s_1} \frac{1}{ ( N s_0 )^{1/2} } \notag\\
\leq  & \frac{ C N^{2 \eps} s_0}{ N^2 t_1 } \ee\left[ \1_{ \F_2} \langle w (s_1'), \B w (s_1') ) \rangle \right] + N^{2\epst} \left( \frac{1}{ ( N s_0 )^2} + \frac{ ( N s_0 )^2}{ \ell^2} \right) \frac{1}{N^2 t_1 s_1} +  \frac{N^{ 8\epst}}{N^2 t_1 s_1} \frac{1}{ ( N s_0 )^{1/2} } .
\end{align}
Above we used that
\beq
\ee \left[ \frac{N^\eps}{t_1} \1_{ \F_2}  \1_{ \{ \tau > s_1'\} } \int_{s_1'}^{\tau} d M_t \right] =\ee \left[ \frac{N^\eps}{t_1} \1_{ \F_2}   \int_{s_1'}^{\tau} d M_t \right] = 0
\eeq
due to the fact that $\F_2$ is measureable wrt $\sigma ( B_i (s), s \leq s_1 )$.  We now have due to the energy inequality and the definition of $\F_2$,
\beq
\frac{1}{s_1} \int_{s_1}^{2s_1} \ee\left[ \1_{ \F_2} \langle w (s_1'), \B w (s_1' ) \rangle \right] \d s_1' = \ee\left[\1_{\F_2} \frac{1}{s_1} \int_{s_1}^{2 s_1} \langle w (s_1'), \B w (s_1' ) \rangle \d s_1' \right] \leq C N^{\eps_3} \frac{1}{ s_1^2}.
\eeq
Collecting everything we see that
\begin{align}
&\ee\left[ \1_{ \F }  \frac{1}{t_1} \int_0^{t_1} \left( \UB_{t_1+u} (i, a) - \frac{1}{N} p_{t_1+u} (\gf_i, \gf_a ) \right)^2 \right] \notag\\
\leq & \frac{1}{ ( N t_1 )^2 } \bigg\{ \frac{ N^{2 \eps_3}}{ ( N t_1 )^2 } + N^{ 2 \eps_3 } \frac{ s_1^2}{t_1^2 } + N^{2 \eps_3 - c\epst} + \frac{ N^{2 \eps_3}}{ \ell^2} + \frac{ N^{2 \eps_3} ( N t_1)^4}{ \ell^4 } + \frac{ N^{2 \epst} t_1}{s_1 ( N s_0 )^2 } + \frac{ t_1 N^{2 \epst} ( N s_0 )^2}{ s_1 \ell^2 } \notag\\
+ &\frac{ t_1 N^{ 8 \epst}}{ s_1 ( N s_0 )^{1/2} } + \frac{N^{ 2 \eps_3} s_0 t_1}{ s_1^2 }\bigg\}.
\end{align}
This proves the following lemma (first fix $\epst$ small, then $\eps_3$ smaller depending on $\epst$).
\bel \label{lem:betatimeav}
Let $\eps >0$ be small enough.  Let $i$ and $a$ satisfy $|a| \leq N^{1/2-\eps_1}$ and $|i-a| \leq \ell/10$.  There is an event $\F_{ia}$ with probability $\pp [ \F_{ia} ] \geq 1 - N^{ - c \eps } $ on which
\begin{align}
&\frac{1}{t_1} \int_{t_1}^{2 t_1} | \UB_{u} (i, a) - \frac{1}{N} p_{u} ( \gf_i, \gf_a ) | \d u \notag\\
\leq & \frac{ N^{-\eps}}{ N t_1 } + \frac{N^{C \eps}}{  N t_1  } \left( \frac{1}{ N t_1 } + \frac{t_1^{1/2} ( N s_0 )}{ s_1^{1/2} \ell}  + \frac{ t_1^{1/2} }{ s_1^{1/2} ( N s_0)^{1/4} } + \frac{ s_1}{t_1} + \frac{ s_0^{1/2} t_1^{1/2}}{ s_1} \right)
\end{align}
\eel
We can remove the time average in the same way as in Section \ref{sec:homog}.  
\bel
Let $\eps >0$ and $\eps_2 >0$ and $\eps_3 >0$.  Let $t_2 = t_1 N^{-\eps_2}$.  There is an event with probability $\pp [\F_{ia} ] \geq 1 - N^{ - c \eps } - N^{ - c \eps_3}$ on which
\begin{align}
&\left| \UB_{t_1+2t_2} (i, a) - \frac{1}{N} p_{t_1 + 2 t_2} ( \gf_i, \gf_a ) \right| \notag\\
\leq& \frac{ N^{\eps_2} N^{-\eps}}{ N t_1 } + \frac{ N^{\eps_2} N^{C \eps}}{  N t_1  } \left( \frac{1}{ N t_1 } + \frac{t_1^{1/2} ( N s_0 )}{ s_1^{1/2} \ell}  + \frac{ t_1^{1/2} }{ s_1^{1/2} ( N s_0)^{1/4} } + \frac{ s_1}{t_1} + \frac{ s_0^{1/2} t_1^{1/2}}{ s_1} \right) \notag\\
+ & \frac{N^{\eps_3} N^{-\eps_2}}{ N t_1 }.
\end{align}
\eel
The proof is similar to that of Theorem \ref{thm:notime}.   One needs to introduce an additional event $\F_3$ on which the bounds of Lemma \ref{lem:betaprofile} (see also the remark immediately subsequent to it) hold with $\eps_3 >0$.  To get around the fact that the event of Lemma \ref{lem:betatimeav} depends on $\F_{ia}$, one applies Markov inequality on the event $\F_3$ to the term \eqref{eqn:time1}.

By choosing the scales $s_0, s_1$ appropriately, we obtain the following.
\bel \label{lem:betahomog}
Let $\sigma >0$ and let $t_1$ satisfy
\beq
\frac{ N^\sigma}{N} \leq t_1 \leq \frac{ N^{1/4-\sigma}}{N}.
\eeq
Let $\eps >0$.  Let $\eps >0$ be sufficiently small.  There is an event $\F_{ia}$ with $\pp [ \F_{ia} ] \geq 1 - N^{-c \eps}$ on which we have, for every $|t-t_1 | \leq N^{ c \sigma}/N$,
\beq
\left| \UB_{ai} (0, t) - \frac{1}{N} p_{t_1} (\gf_a, \gf_i ) \right| \leq \frac{1}{ N t_1} \left( N^{ -  \eps} + N^{ - c \sigma } \right).
\eeq
\eel

Let now $\F_1$ be the event on which the estimates of Lemma \ref{lem:betasupbd} hold with $\eps_1 >0$, and $\F_2$ the event of Lemma \ref{lem:betainitcutoff} holds with $\eps_2 >0$.  Denote by $\F_{ai}$ the event of Lemma \ref{lem:betahomog}  with $\eps_3 >0$.  We then write
\begin{align}
x_a (t) - y_a (t) &= \sum_{ |j-a| \leq N t_1 N^{\eps_a } } \UB_{aj} (0, t) (x_j (0) - y_j (0) ) + \frac{ N^{\eps_2}}{N} \O \left( \frac{1}{ N^{\eps_a}} + N^{ - c \sigma } \right) \notag\\
&=  \sum_{ |j-a| \leq N t_1 N^{\eps_a } } p_{t_1} (\gf_a, \gf_i ) (x_j (0) - y_j (0) ) \notag\\
&+  \sum_{ |j-a| \leq N t_1 N^{\eps_a } } \1_{ \F_{aj} } \left( \UB_{aj} - p_{t_1} (\gf_a, \gf_j ) \right) (x_j (0) - y_j (0) )  \label{eqn:beta1} \\
&+ \sum_{ |j-a| \leq N t_1 N^{\eps_a } } \1_{ \F_{aj}^c } \left( \UB_{aj} - p_{t_1} (\gf_a, \gf_j ) \right) (x_j (0) - y_j (0) ) \label{eqn:beta2} \\
&+ \frac{ N^{\eps_2}}{N} \O \left( \frac{1}{ N^{\eps_a}} + N^{ - c \sigma } \right) 
\end{align}
For the term \eqref{eqn:beta1} we have
\beq
\left|  \sum_{ |j-a| \leq N t_1 N^{\eps_a } } \1_{ \F_{aj} } \left( \UB_{aj} - p_{t_1} (\gf_a, \gf_j ) \right) (x_j (0) - y_j (0) )  \right| \leq N^{\eps_a+\eps} \frac{1}{N} \left( N^{- c \eps_3} + N^{ - c \sigma } \right).
\eeq
For the term \eqref{eqn:beta2} we have
\begin{align}
\ee \left[ \1_{ \F_1} \left| \sum_{ |j-a| \leq N t_1 N^{\eps_a } } \1_{ \F_{aj}^c } \left( \UB_{aj} - p_{t_1} (\gf_a, \gf_j ) \right) (x_j (0) - y_j (0) )  \right| \right] \leq \frac{ N^\eps}{N} N^{\eps_a} N^{\eps_1} N^{ - c \eps_3}.
\end{align}
Hence, first fixing $\eps_3$ small, then taking $\eps_a$ and $\eps_1$ and $\eps$ small depending on $\eps_3$ and $\sigma$, and then taking $\eps_2$ small, we obtain the following.
\bet  \label{thm:betafinal}
Let $\eps >0$ sufficiently small.  Let $t_1$, $t$ be as above.  There is an event $\F_a$ with $\pp [ \F_a ] \geq 1 - N^{ - c \eps }$ on which
\beq
x_a (t) - y_a (t) = \sum_{ |j-a| \leq N t_1 N^\eps } p_{t_1} ( \gf_a, \gf_j ) ( x_j (0) - y_j (0) ) + \frac{1}{N} \O \left( N^{ - c \eps } + N^{ - c \sigma } \right).
\eeq
\eet

\subsection{Proof of fixed energy universality}
We now prove fixed energy universality for $\beta$-ensembles.  For simplicity we just consider $E=0$.  Let $\eps_h >0$ be as in Theorem \ref{thm:betafinal}.  Suppose that the events $\F_a$ defined in Theorem \ref{thm:betafinal} hold with probablity $\pp[ \F_a] \geq 1 - N^{ c_h \eps_h}$. Denote the error in the estimate by $\O ( N^{-\eps_1-1})$.  Let
\beq
\delta_1 = c_h \eps_h /10.
\eeq
Let $\F = \bigcup_{ |a| \leq N^{ \delta_1 } } \F_a$.  Then $\pp [ \F ] \geq 1 - N^{ c_h \eps_h/2}$.   As in the proof of Theorem \ref{thm:fixed} we consider 
\beq
\sum_{i, j} \ee [ Q ( N x_i (t), N (x_j(t) - x_i (t)) ].
\eeq
We assume that the eigenvalues are labelled by $[[-N/2, N/2]]$ and that $\gV_{N/2} = 0$. Let $\delta_R =  \min\{ \delta_1/10, \eps_1/10 \}$.
 Arguing as in the proof of Theorem \ref{thm:fixed} we can apply rigidity and the homogenization theorem and write
\begin{align}
& \sum_{i, j} \ee [ Q ( N x_i (t), N (x_j(t) - x_i (t)) ]  = \sum_{ |i|, |j| \leq N^{\delta_R} } \ee [ Q ( N x_i (t), N (x_j(t) - x_i (t)) ] + \O ( N^{-D} ) \notag\\
= & \sum_{ |i|, |j| \leq N^{\delta_R} } \ee [ Q ( N y_i (t) + \zeta_x- \zeta_y, N (y_j(t) - y_i (t)) ]  + \O  ( N^{ - c } )
\end{align}
for some $c>0$.  Similarly if $z_i (t)$ is an auxilliary GOE ensemble we can write
\begin{align}
& \sum_{i,j} \ee[ Q ( N (z_i (t) - E), z_j (t) - z_i (t) ) ] \notag\\
= & \sum_{ |i|, |j| \leq N^{\delta_R} } \ee [ Q ( N ( y_i (t) - E) + \zeta_z- \zeta_y, N (y_j(t) - y_i (t)) ]  + \O  ( N^{ - c } ),
\end{align}
for any $|E| \leq N^{ \delta_R/2} /N$.  Denote
\beq
F_1(s) := \sum_{ |i|, |j| \leq N^{\delta_R} } \ee [ Q ( N ( y_i (t)  + s- \zeta_y, N (y_j(t) - y_i (t)) .
\eeq
We have
\begin{align}
& \sum_{ |i|, |j| \leq N^{\delta_R} } \ee [ Q ( N ( y_i (t)  + \zeta_x- \zeta_y, N (y_j(t) - y_i (t)) ] \notag\\
&-  \sum_{ |i|, |j| \leq N^{\delta_R} } \ee [ Q ( N ( y_i (t) - E) + \zeta_z- \zeta_y, N (y_j(t) - y_i (t)) ] \notag\\
&=  \int \d \lambda \hat{F_1} ( \lambda )( \psi_x ( \lambda) - \psi_z ( \lambda ) \e^{ \i \lambda E } )
\end{align}
where $\psi_x$ and $\psi_z$ are the Fourier transforms of $\zeta_x$ and $\zeta_y$.  If we let $E = \ee[\zeta_x] - \ee[ \zeta_z] = \O ( N^\eps)$ for any $\eps >0$ then by Section \ref{sec: mesosection},
\beq
\left|  \psi_x ( \lambda) - \psi_z ( \lambda ) \e^{ \i \lambda E } ) \right| \leq N^{ - c_2}, \qquad |\lambda | \leq N^{c_2}
\eeq
for some $c_2 >0$.  If we take $\delta_R < c_2 /2$, then we see that 
\beq
\left| \int \d \lambda \hat{F_1} ( \lambda )( \psi_x ( \lambda) - \psi_z ( \lambda ) \e^{ \i \lambda E } ) \right| \leq N^{-c_2/3}.
\eeq
This proves fixed energy universality for $\beta$ ensembles.

\appendix

\section{Local laws and properties of free convolution}  \label{a:free}

\subsection{Free convolution properties}

In this section we summarize the local laws for DBM as well as derive some properties of the free convolution.  Recall that $H_t$ is defined as
\beq
H_t : V  +  \sqrt{t} W.
\eeq
where $W$ is a GOE matrix.  The Stieltjes transform of the free convolution is defined as the solution to the fixed point equation
\beq
\mfct (z) = \frac{1}{N} \sum_{i=1}^N \frac{1}{   V_i - z - t \mfct (z) }.
\eeq 
We call $V$ $(g, \G)$-regular if
\beq
c \leq \Im [ m_V (E + \i \eta )] \leq C
\eeq
for $|E| \leq \G$ and 
\beq
g \leq \eta \leq 10
\eeq
and $||V||\leq N^{C_V}$ for some $C_V>0$.  We collect some properties of the free convolution in the following lemma.  It can be found in Section 7 of \cite{landonyau}.
\bel  Let $V$ be $(g, \G)$-regular.  
Let $0 < q < 1$.  Let $\sigma >0$ and let $t$ satisfy
\beq
\g N^\sigma \leq t \leq N^{- \sigma } \G^2.
\eeq
For $|E| \leq q \G$ we have
\beq
c \leq \Im [ \mfct (z) ] \leq C
\eeq
for $0 \leq \eta \leq 10$.  We have
\beq \label{eqn:oldmder}
| \del_z \mfct (z) | \leq \frac{ C}{ t + \eta }.
\eeq
\eel
We also require an estimate for the second derivative of $\mfct (z)$.  We easily calculate
\begin{align}
\del_z \mfct (z) 
&= \left( 1 - t \int \frac{1}{ (x - z - t \mfct (z) )^2} \d \mu_V (x) \right)^{-1} \notag\\
&\times \int \frac{1}{ (x - z - t \mfct (z) )^2} \d \mu_V (x) 
\end{align}
from which we see that
\begin{align}
\del_z^2 \mfct (z) 
 &= \left( 1 - t R_2 \right)^{-2} t R_3 R_2 \left[1 + t \del_z \mfct (z) \right] +\left( 1 - t  R_2 \right)^{-1} R_3 \left[1 + t \del_z \mfct (z) \right],
\end{align}
where
\beq
R_k := 
\int \frac{1}{ ( x - z - t \mfct (z) )^k} \d \mu_V (x) .
\eeq
Since $|R_2| \leq C t^{-1}$ and $|R_3| \leq C t^{-2}$ we see, using that $| 1 - t R_2 | \geq c$ (see \cite{landonyau}) and \eqref{eqn:oldmder} that
\beq \label{eqn:m2der}
| \del_z^2 \mfct (z) | \leq \frac{C}{t^2}.
\eeq
We have the following local law from Section 7 of \cite{landonyau}.
\bel \label{lem: gDd} Let $V$ be $(g, \G)$-regular and let $\sigma >0$, and $\eps, \delta >0$ and let $0 < q < 1$.  Let $t$ satisfy $\g N^{\sigma} \leq t \leq N^{-\sigma} \G^2$.   With overwhelming probability we have for $|E| \leq q \G$ and $\eta \geq N^{\delta}/N$,
\beq
| \mfct (z) - m_N (z) | \leq \frac{ N^\eps}{ N \eta}.
\eeq
The above estimate also holds if $\eta \geq 10$ and for any $E$.
\eel

We have the following rigidity result.
\bel
Fix $0 < q < 1$. Let $\eps >0$. For every $i$ s.t. $| \gamma_i  | \leq q \G$ we have with overwhelming probability
\beq
| x_i - \gamma_i | \leq \frac{ N^\eps}{N}.
\eeq
We have also the estimate 
\beq \label{eqn:vixi}
|x_i - V_i | \leq C \sqrt{t}
\eeq
 for every $i$.
\eel
The estimate \eqref{eqn:vixi} is just a consequence of the perturbation bound $\lambda_j ( A - B) \leq ||A-B||$ and the fact that for a GOE matrix, $||W|| \leq 3$ with overwhelming probability.  We will not use this estimate in this paper.

\subsection{Rescaling and relabelling set-up}
Recall that in Section \ref{sec:homog} we fixed an index $i_0$ with $|\gamma_{i_0} (t_0) | \leq q \G$, and that we have assumed that
\beq
\rho_{\fc, t_0 } (0) = \rhosc (0), \qquad \gamma_{i_0} (t_0 ) = 0.
\eeq
Moreover, we assume $N$ is odd and. that $i_0 = (N+1)/2$.  Finally, we have relabelled the eigenvalues so that $i_0 = 0$ and they run over the index set $[[ - (N-1)/2, (N-1)/2 ]]$.

\subsection{Construction of law of interpolating ensembles}
We now wish to discuss rigidity for the interpolating ensemble.  Define
\beq
z_i ( \alpha ) = (1 - \alpha ) y_i (t_0) + \alpha x_i (t_0)
\eeq
where $x_i (t_0)$ are the eigenvalues of $H_{t_0}$ and $y_i (t_0)$ are the eigenvalues of an independent GOE matrix.   Fix now $0 < q^* < 1$.   Let $k_0$ be the largest possible natural number so that
\beq
| \gamma_{k_0} (t_0) | \leq q^* \G , \qquad |\gamma_{-k_0} (t_0) | \leq q^* \G , \qquad | \gamma_{k_0}^{(\mathrm{sc})} | \leq q^* \G.
\eeq
First of all, extend $\gamma_x (t_0)$ to all $|x| \leq (N-1)/2$ by
\beq
\frac{x}{N} = \int_{0}^{\gamma_x (t_0) } \rho_{\fc, t_0} (E) \d E,
\eeq
and similarly for $\gsc_x$.  We define a function $f(x, \alpha) : [-k_0, k_0 ] \times [0, 1] \to \rr$ as follows.  For $|x| \leq k_0$,
\beq
f (x, \alpha ) = \alpha \gamma_x (t_0) + (1-\alpha ) \gamma^{(sc)}_x.
\eeq
Then $f(0, \alpha ) = 0$.  For each $\alpha$,  $f ( \cdot, \alpha )$ is a bijection from
\beq
f (  \cdot  , \alpha ) : [ - k_0, k_0 ] \to  [ \alpha \gamma_{-k_0} ( t_0 ) + (1 - \alpha ) \gamma^{(\sc)}_{-k_0} , \alpha \gamma_{k_0} ( t_0 ) + (1 - \alpha ) \gamma^{(\sc)}_{k_0} ] =: \Galp
\eeq
Since
\beq
\frac{ \d } { \d x} \gamma_x (t_0) = \frac{1}{N}\frac{1}{ \rhofcto (\gamma_x (t_0 ) )},
\eeq
we have
\beq
\frac{c}{N} \leq f'(x, \alpha ) \leq \frac{C}{N}.
\eeq
Let $g$ be inverse of $f$,
\beq
g ( f (x, \alpha ), \alpha ) = x.
\eeq
For each $\alpha$ we now the function $h(y, \alpha)$ on the interval $\Galp$ by
\beq
h(y, \alpha ) := \frac{1}{N} \frac{ \d }{ \d y} g ( y, \alpha ).
\eeq
By elementary calculations we have the following explicit formula for $h$:
\beq
h (y, \alpha ) = \frac{1}{N} \frac{1}{ f' ( g ( y, \alpha ), \alpha ) } = \left(  \frac{ \alpha}{ \rhofcto ( \gamma (g (y, \alpha ), t_0 ) ) } + \frac{1- \alpha }{ \rhosc ( \gamma^{(\sc)} ( g (y, \alpha ) )}\right)^{-1}
\eeq
In particular we see that
\beq
c \leq h (y, \alpha ) \leq C, \qquad | h' (y, \alpha ) | \leq \frac{C}{t_0}, \qquad | h'' (y, \alpha ) | \leq \frac{C}{ t_0^2} \qquad h (0, \alpha ) = \rhosc (0),
\eeq
and by definition,
\beq
\int_{ \Galp} h (y, \alpha ) \d y = \frac{ 2 k_0}{N}.
\eeq
Note that by definition,
\beq
\int_{0}^{f (x, \alpha ) } h ( x, \alpha ) \d x = \frac{x}{N},
\eeq
and so we can immediately see that
\beq
\left| \frac{1}{N} \sum_{|i| \leq k_0} \frac{1}{ z_i ( \alpha ) - z } - \int_{ \Galp} \frac{1}{ x - z } h (x, \alpha ) \d x \right| \leq \frac{ N^\eps}{N \eta}
\eeq
for any $\eps >0$ and $\eta \geq N^{\delta}/N$ for any $\delta >0$ with overwhelming probability.  We now define the following probability measure on $\rr$.  Let 
\beq
\nu (\d x, \alpha ) = h (x, \alpha ) \d x + \frac{1}{N} \sum_{ |i| > k_0 } \delta_{z_i ( \alpha ) }.
\eeq
Then $\nu (\alpha ) $ is a probability measure and denoting,
\beq
m (z, 0, \alpha ) = \int \frac{1}{x-z} \nu ( \d x)
\eeq
we immediately see that
\beq
\left| \frac{1}{N} \sum_{ i} \frac{1}{ z_i ( \alpha ) - z } - m (z, 0, \alpha ) \right| \leq \frac{ N^\eps}{ N \eta}, \qquad \eta \geq \frac{ N^\delta}{N}
\eeq
for any $\eps , \delta >0$ and with overwhelming probability.  We denote the free convolution of $\nu ( \alpha )$ with the GOE at time $t$ by $\rho (x, t, \alpha ) \d x$.  It is defined through its Stieltjes transform which satisfies the fixed point equation
\beq
m (z, t, \alpha ) = \int \frac{1}{ x - z - t m (z, t, \alpha ) } \d \nu (x, \alpha ).
\eeq
Let now $\eps > 0$ be given and let $0 \leq t_1 \leq N^{-\eps} t_0$.   The proofs in \cite{landonyau} yield the following lemma.
\bel
Let $0 < q < 1$.  Let $\eps >0$.  Let $0 \leq t \leq t_1$.  For $E \in q \Galp$  and $0 \leq \eta \leq 10$ we have the following estimates with overwhelming probability.  First we have,
\beq
c \leq \Im [ m (z, t, \alpha ) ] \leq C.
\eeq
We have
\beq
c \leq \left| 1 - t \int \frac{1}{ (x - z - t m (z, t, \alpha ) )^2 } \d \nu ( x, \alpha ) \right| \leq C
\eeq
and
\beq
\left| t^2 \int \frac{1}{  ( x - z - t ) m (z, t, \alpha ) )^3 } \d \nu ( x, \alpha ) \right| \leq C.
\eeq
\eel
Note that we always have following the a-priori bound for $m (z, t, \alpha )$.  By Cauchy-Schwarz,
\beq
| m (z, t, \alpha ) |^2 \leq \int \frac{1}{ | x -  z - t m (z, t, \alpha ) |^2} \d \nu (x, \alpha ) = \frac{ \Im [ m (z, t, \alpha ) ] }{ \eta + t \Im [ m (z, t, \alpha ) ] } \leq \frac{1}{t}.
\eeq
We have the following improved regularity of $m (z, t, \alpha )$.  
\bel
Let $0 < q < 1$ and $0 \leq t \leq t_1$.  For $E \in q \Galp$ we have with overwhelming probability,
\beq
\left| \del_z m (z, t, \alpha ) \right| \leq \frac{C }{ t_0 + \eta }.
\eeq
\eel
\proof We calculate
\begin{align}
\del_z m (z, t, \alpha  ) &= \left( 1 - t \int \frac{1}{ (  x - z - t m (z, t, \alpha ) )^2} \d \nu (x, \alpha ) \right)^{-1} \notag\\
&\times  \int \frac{1}{ ( x - z - t )^2} \d \nu (x, \alpha ).
\end{align}
Since $|  t m (z, t, \alpha ) | \leq \sqrt{t} \ll \G$, we see that
\beq
\Re [  z + t m (z, t, \alpha )  ] \in q_1 \Galp
\eeq
for any $q_1$ s.t. $q < q_1 < 1$ and $N$ large enough.  It therefore suffices to prove
\beq
\left| \del_z m (z, 0, \alpha ) \right| \leq \frac{C }{t_0}
\eeq
for any $E \in q_1 \Galp$.  We write
\beq
\del_z m (z, 0, \alpha ) = \int \frac{1}{ (x - z )^2} h (x, \alpha ) + \frac{1}{N} \sum_{|i| > k_0 }  \frac{1}{ (z_i ( \alpha ) - z )^2} .
\eeq
Optimal rigidity guarantees that for $|i| > k_0$, $|z_i ( \alpha ) - z | > c \G$ and so with overwhelming probability,
\beq
\frac{1}{N} \sum_{|i| > k_0 }  \frac{1}{ |z_i ( \alpha ) - z |^2} \leq \frac{C}{ \G} \Im [ m (0 + \i \G , 0, \alpha ) ] \leq \frac{C}{ \G}\leq \frac{C}{t_0}.
\eeq
We write the other term as
\begin{align}
\int \frac{1}{ (x - z )^2} h (x, \alpha ) = \int_{ |x - E | \leq t_0 } \frac{1}{ (x - z )^2} h (x, \alpha ) \d x+ \int_{|x-E | > t_0 } \frac{1}{ (x - z )^2} h (x, \alpha ) \d x .
\end{align}
We clearly have
\beq
\left| \int_{|x-E | > t_0 } \frac{1}{ (x - z )^2} h (x, \alpha ) \d x \right| \leq \frac{C}{ t_0}.
\eeq
We integrate the other term by parts and obtain
\begin{align}
 \int_{ |x - E | \leq t_0 } \frac{1}{ (x - z )^2} h (x, \alpha ) \d x = \frac{ h (E + t_0, \alpha )}{ -t_0 - \i \eta } - \frac{ h (E - t_0, \alpha )}{  t_0 - \i \eta } - \int_{ |x - E | \leq t_0 } \frac{1}{ (x - z )} h' (x, \alpha ) \d x.
\end{align}
Clearly
\beq
\left| \frac{ h (E \pm t_0, \alpha )}{  t_0 - \i \eta }  \right| \leq \frac{C}{t_0}
\eeq
We split the last term into its real and imaginary parts
\beq
\int_{ |x - E | \leq t_0 } \frac{1}{ (x - z )} h' (x, \alpha ) \d x = \i \int_{ |x - E | \leq t_0 } \frac{\eta}{ (x- E)^2 + \eta^2} h' (x, \alpha ) \d x + \int_{ |x - E | \leq t_0 } \frac{ x-E}{ (x-E)^2 + \eta^2} h' (x, \alpha ) \d x.
\eeq
The imaginary part is easily bounded by
\beq
\left| \int_{ |x - E | \leq t_0 } \frac{\eta}{ (x- E)^2 + \eta^2} h' (x, \alpha ) \d x  \right| \leq \frac{C}{t_0} \int_\rr \frac{ \eta}{ x^2 + \eta^2 } \d x \leq \frac{C}{t_0}.
\eeq
We use \eqref{eqn:m2der} to bound the real part by
\begin{align}
\left|\int_{ |x - E | \leq t_0 } \frac{ x-E}{ (x-E)^2 + \eta^2} h' (x, \alpha ) \d x \right| &=\left|\int_{ |x - E | \leq t_0 } \frac{ x-E}{ (x-E)^2 + \eta^2} (h' (x, \alpha ) - h' (E, \alpha ) )\d x \right| \notag\\
&\leq \frac{C}{t_0^2} \int_{ |x - E| \leq t_0 } \frac{ (x-E)^2}{ (x-E)^2 + \eta^2 } \d x  \leq \frac{C}{t_0}.
\end{align}
This yields the claim. \qed

This allows us to conclude a few things about the densities $\rho (t, \alpha )$.  

\bel Let $0  < q < 1$.  
We have for $0 \leq t \leq t_1$ and $E \in q \Galp$,
\beq
c \leq \rho (E, t, \alpha ) \leq C, \qquad | \rho' (E, t, \alpha ) | \leq \frac{C}{t_0}, \qquad | \del_t \rho(E, t, \alpha ) | \leq \frac{ C \log (N) }{t_0} 
\eeq
Hence,
\beq
\rho (E, t, \alpha ) = \rhosc (0) + \O \left( \frac{ t \log (N) }{ t_0 } + \frac{ |E| }{ t_0 } \right) .
\eeq
\eel
\proof We only need to prove the statement about the time derivative.  This follows immediately from the equation
\beq
\del_t m (z, t, \alpha ) = \frac{1}{2} m (z, t, \alpha ) \del_z m (z, t, \alpha ). 
\eeq
\qed

The proof of the following result is a minor modification of Section 7 of \cite{landonyau}.  Denote
\beq
m_N (z, t, \alpha ) = \sum_{i=1}^N \frac{1}{ z_i (t, \alpha ) - z }.
\eeq
\bel
Let $0 \leq \alpha \leq 1$ and fix $0 < q < 1$.  Let $\eps, \delta >0$.  Let $0 \leq t \leq t_1$.  The following estimates hold with overwhelming probability.  For $E \in q \Galp$ and $\eta \geq N^{\delta}/N$ we have 
\beq
\left| m_N (z, t, \alpha ) - m (z, t, \alpha ) \right| \leq \frac{N^\eps}{N \eta}.
\eeq
The above estimates also hold if $\eta \geq 10$. 
\eel
From the above estimates we conclude the following rigidity result.  We define the classical eigenvalue locations by
\beq
\frac{i}{N} = \int_0^{\gamma_{i} (t, \alpha ) } \rho (x, t, \alpha ) \d x.
\eeq
Note that they satisfy
\beq
\del_t \gamma_i (t, \alpha ) = \Re [ m ( \gamma_i (t, \alpha ), t, \alpha ) ]
\eeq
and so 
\beq
\left| \del_t \gamma_i (t, \alpha ) \right| \leq C \log (N).
\eeq
\bel \label{lem:wrig} Let $0 < q < 1$ and let $\eps >0$.  Let $i$ be such that $\gamma_i (0, \alpha ) \in q \Galp$.  Then
\beq
| z_i (t, \alpha ) - \gamma_i (t, \alpha ) | \leq \frac{ N^\eps}{N}
\eeq
with overwhelming probability.  We have also with overwhelming probability,
\beq
\left| z_i (t, \alpha ) - ( \alpha V_i + ( 1 - \alpha ) \gamma^{(sc)}_i ) \right| \leq C \sqrt{t_0}.
\eeq
\eel

We now make a slight digression on the classical eigenvalue locations of $\rho (t, \alpha)$ which we denote by $\gamma_i (t, \alpha)$.  We want to elucidate the connecton with the function $f (x, \alpha )$.  Fix $0  < q < 1$.  With overwhelming probability the eigenvalues $\{ z_i (0, \alpha ) : i < - k_0 \}$ are all to the left of the interval $q \Galp$, and the eigenvalues  $\{ z_i (0, \alpha ) : i > k_0 \}$ are all to the right of $q \Galp$.  Hence, with overwhelming probability $\gamma_i (0, \alpha ) = f (i, \alpha )$ for any $i$ s.t. $\gamma_i (, \alpha ) \in q \Galp$.  We therefore also have $\gamma_0 (0, \alpha ) = 0$ with overwhelming probability. 

We have the following lemma.
\bel
For $t \leq 10 t_1$ and $\eps >0$ we have with overwhelming probability
\beq \label{eqn:agamma1}
\sup_{ 0 \leq t \leq 10 t_1 } | \gamma_0 (t_0 + t) - \gamma_0 (t, 1 ) | \leq \frac{1}{N} \frac{ N^\eps N^{\om_1}}{ N^{\om_0/2}}.
\eeq
and
\beq \label{eqn:agamma0}
\sup_{ 0 \leq t \leq 10 t_1 } | 0 - \gamma_0 (t, 0 ) | \leq \frac{1}{N} \frac{ N^\eps N^{\om_1}}{ N^{\om_0/2}}.
\eeq
With overwhelming probability,
\beq \label{eqn:agap}
\gamma_k (t, \alpha ) - \gamma_j (t, \alpha ) = \frac{ j-k}{ N \rhosc (0) } + \O \left( \frac{1}{N} \right)
\eeq
for $|j| + |k| \leq N^{\om_0/2}$ and $t \leq 10 t_1$ with $\om_1 \leq \om_0/2$.
\eel
\proof  We start with \eqref{eqn:agamma1} and \eqref{eqn:agamma0}.  For $L>0$ and small $ c>0$ we define
\beq
I_1 := [ \gamma_{-L} (t, 1) , \gamma_{L} (t, 1) ], \quad I_3 := \rr \backslash [ - \gamma_{ -c \G N } (t, 1), \gamma_{ c \G N } (t, 1 ) ] \quad I_2 := I_3^c \backslash I_1.
\eeq
Take $c$ small enough so that rigidity holds for $|i| \leq c \G N$.
We write,
\begin{align}
\del_t \gamma_0 (t, 1) &=  - \Re [ m ( \gamma_0 (t, 1), t, 1 ) ] = \frac{1}{N}  \sum_{ |i| > L } \frac{1}{ x_0 (t + t_0 ) - x_i (t + t_0 ) } \notag\\
&+  \left( \frac{1}{N}  \sum_{ |i| > L } \frac{1}{ \gamma_0 (t, 1) - x_i (t + t_0 ) } -\frac{1}{N}  \sum_{ |i| > L } \frac{1}{ x_0 (t + t_0 ) - x_i (t + t_0 ) } \right) \notag\\
&+ \left(  \int_{I_3} \frac{ \rho (x, t, 1 ) \d x }{ \gamma_0 (t, 1) - x } - \frac{1}{N} \sum_{ |i| > c \G N }  \frac{1}{ \gamma_0 (t, 1) - x_i (t + t_0 ) } \right) \notag\\
&+ \left( \int_{I_2} \frac{ \rho (x, t, 1 ) \d x }{ \gamma_0 (t, 1) - x } - \frac{1}{N} \sum_{ L < |i| < c \G N }  \frac{1}{ \gamma_0 (t, 1) - x_i (t + t_0 ) }  \right) + \left(   \int_{I_1} \frac{ \rho (x, t, 1 ) \d x }{ \gamma_0 (t, 1) - x } \right) \notag\\
&=:  \frac{1}{N}  \sum_{ |i| > L } \frac{1}{ x_0 (t + t_0 ) - x_i (t + t_0 ) } + A_1 + A_2 + A_3 + A_4
\end{align}
By rigidity we have $|A_1 |  +  |A_3| \leq N^{\eps} / L$ for any $\eps >0$ with overwhelming probability.  The same argument handling the error term $E_4$ in the proof of Lemma \ref{lem:shortrange} yields $|A_4 | \leq N^\eps L / ( N t_0 )$ with overwhelming probability.  The error term $A_2$ is handled in the same way as $E_1$ in the proof of Lemma \ref{lem:shortrange} and we see that $|A_2| \leq N^\eps / \sqrt{ N \G} $ with overwhelming probability.  Choosing $L = \sqrt{ N t_0}$ we see that
\beq
\del_t \gamma_0 (t, 1 ) =  \frac{1}{N}  \sum_{ |i| > \sqrt{ N t_0 } } \frac{1}{ x_0 (t + t_0 ) - x_i (t + t_0 ) } + N^\eps \O \left( \frac{1}{ \sqrt{ N t_0 } } \right)
\eeq
with overwhelming probability.  Similarly we see that
\beq
\del_t \gamma_0 (t_0 + t ) = \frac{1}{N}  \sum_{ |i| > \sqrt{ N t_0 } } \frac{1}{ x_0 (t + t_0 ) - x_i (t + t_0 ) } + N^\eps \O \left( \frac{1}{ \sqrt{ N t_0 } } \right)
\eeq
with overwhelming probability.  Hence, for $|t | \leq 10 t _1$,
\beq
| \gamma_0 (t, 1) - \gamma_0 (t_0 + t ) | \leq \frac{1}{N} \frac{ N^\eps N^{\om_1 }}{ N^{\om_0/2}}.
\eeq
The same argument applies to $\gsc_0$ and $\gamma_0 (t, 0)$.  

We now prove \eqref{eqn:agap}.  We have,
\begin{align}
\frac{k-j}{N} &= \int_{ \gamma_{j} (t, \alpha ) }^{ \gamma_k (t, \alpha ) } \rho (E, t, \alpha ) \d E \notag\\
&= \rhosc (0) ( \gamma_k (t, \alpha ) - \gamma_j (t, \alpha ) ) + \int_{ \gamma_{j} (t, \alpha ) }^{ \gamma_k (t, \alpha ) } \O ( |E| / t_0 ) \d E + \O ( |  \gamma_k (t, \alpha ) - \gamma_j (t, \alpha )  | t/t_0 ).
\end{align}
Hence,
\beq
\gamma_k (t, \alpha ) - \gamma_j (t, \alpha ) = \frac{ j-k}{ N \rhosc (0) } + \O \left( \frac{1}{N} \right)
\eeq
for $|j| + |k| \leq N^{\om_0/2}$, and $t \leq 10 t_1$. \qed

\section{Stochastic continuity} \label{a:cont}
In Appendix \ref{a:free} we proved rigidity for each fixed time $t \geq t_0$ and each fixed $\alpha$ - i.e., Lemma \ref{lem:wrig}.  In this appendix we go from the estimates of Lemma \ref{lem:wrig} to estimates for all time $t$ and $\alpha$ simultaneously.  We continue with the notation of Appendix \ref{a:free}. 
Recall the definition of $\Chat_q$ from Section \ref{sec:homog}.
\bel \label{lem:strrig}
Let $D>0$ and $\eps >0$.  Let $0 < q < q_*$.  We have
\beq
\pp \left[ \sup_{ i \in \Chat_q } \sup_{ 0 \leq \alpha \leq 1 } \sup_{ 0 \leq t \leq 10  t_1 }  \left| z_i (t, \alpha ) - \gamma_i (t, \alpha ) \right| \geq \frac{ N^\eps}{N} \right] \leq \frac{1}{N^D}.
\eeq
\eel
Given $\alpha_1$ and $\alpha_2$, the difference
\beq
u_i := z_i (t, \alpha_1 ) - z_i (t, \alpha_2 )
\eeq
satisfies the parabolic equation
\beq
\del_t u =   L u  
\eeq
where
\beq
( L u)_i := \frac{1}{N} \sum_j \frac{ u_j - u_i }{  ( z_i (t, \alpha_1 ) - z_j (t, \alpha_1 ) ) ( z_i (t, \alpha_2 ) - z_j (t, \alpha_2 ) ) }.
\eeq
Hence,
\beq
\sup_t ||u (t) ||_\infty \leq C || z (0, \alpha_1 ) - z (0, \alpha_2 ) ||_\infty \leq C | \alpha_1 - \alpha_2 | \sup_i \{ |x_i (t_0) | + |y_i (t_0 ) | \}.
\eeq
With overwhelming probability we have that
\beq
\sup_i \{ |x_i ( t_0 ) | + |y_i (t_0 ) | \} \leq  C N^{C_V}
\eeq
for a fixed $C_V>0$ by our assumptions on $V$.  Hence in order to prove Lemma \ref{lem:strrig} we can just prove it for a set of $\alpha $ of at most size $N^{2 C_V}$; i.e., we need only prove the following.
\bel \label{lem:fixrig} Fix $0 \leq \alpha \leq 1$.  We have,
\beq \label{eqn:fixrig1}
\pp \left[ \sup_{ i \in \Chat_q } \sup_{ 0 \leq t \leq 10 t_1 }  \left| z_i (t, \alpha ) - \gamma_i (t, \alpha ) \right| \geq \frac{ N^\eps}{N} \right] \leq \frac{1}{N^D}.
\eeq
\eel 
\proof Consider the equation for $m_N$ for general $\beta \geq 1$,
\beq
\d m_N = \del_z \left( m_N ( m_N ) \right) + \frac{\beta-2}{N} \del_z^2 m_N + \sum_i \frac{1}{ N^{3/2} ( \lambda_i - z)^2} \sqrt{2  \beta} \d B_i
\eeq
For $\eta \geq N^{-2}$ we see that for any $t$ we have by the BDG inequality,
\beq
\pp \left[ \sup_{ 0 \leq s \leq N^{-100} } | m_N (z, t+s) - m_N (t) | \geq N^{-5} \right] \leq N^{-D}
\eeq
for any $D>0$.  From Appendix \ref{a:free} we know that the local law holds on the domain $\D_{\eps, q}$ defined in Section \ref{sec:rigid} on a set of times $t = k N^{-100}$, for $0 \leq k \leq N^{100} 10 t_1$ with overwhelming probability.  Hence we can extend the local law to all times $t$ on the domain $\D_{\eps, q}$.  Rigidity is a consequence of this. \qed

We also want to prove that
\beq
\sup_{ 0 \leq \alpha \leq 1 } \sup_i \sup_{ 0 \leq t \leq 10 t_1 } | z_i ( \alpha, t) | \leq N^C
\eeq
with overwhelming probability for some $C>0$.  This follows from the above argument again.  First it suffices to prove it for fixed $\alpha$.  We can assume that it holds for a mesh of times $t$ with overwhelming probability.  Consider then $m_N (z)$ with $z= N^{2 C} + \i N^{-2}$.  At each time in the mesh we have that $|m_N (z) | \leq N^{-4}$.  By the above argument it then holds for all times $t$.  Therefore a particle cannot cross $E = N^{2C}$ or at some time $t$, we would have $\Im [ m_N ( N^{2 C} + \i N^{-2} )] \geq 1$.

\section{Re-indexing argument}\label{a:reindex}
Recall our set-up in Section \ref{sec:homog}.  We have the process $x_i (t)$ that satisfies
\beq
\d x_i = \sqrt{ \frac{ 2}{N}} \d B_i + \frac{1}{N} \sum_j \frac{1}{ x_i - x_j } \d t
\eeq
We fixed an index $i_0$.  In this appendix we want to show that we can assume that $N$ is odd and $i_0 =(N+1)/2$.  Our method is to construct an auxiliary DBM process $\xstar$ of $2N-1$ particles s.t. for every index $j$ s.t. $i_0 + j \in [[1, N]]$,
\beq \label{eqn:xstarxest}
\sup_{0 \leq t \leq 1 } | \xstar_{N+j} (t) - x_{i_0 + j} (t) | \leq \frac{1}{N^{100}},
\eeq
with overwhelming probability.  Similarly, we construct a process $\ystar$ of $2N$ particles s.t.
\beq
\sup_{t_0 \leq t \leq 1} | \ystar_{N+j}(t) - y_{N/2+j} (t) | \leq \frac{1}{N^{100}}.
\eeq
Then the argument of Section \ref{sec:homog} goes through using the processes $\xstar$ and $\ystar$ instead of $x$ and $y$.  We will also see that the estimate of Lemma \ref{lem:strrig} also holds for the $\xstar$ and $\ystar$ (with an appropriate modification of $\gamma^{(\sc)}$ for the particles added to the GOE flow that has no effect on the rest of the paper).

We now construct $\xstar$.  Let $[[1, 2N-1]] = \C_1 \cup \C_2 \cup \C_3$ where $\C_1 = [[1, N - i_0 ]]$, $\C_2 = [[N-i_0+1, 2 N - i_0 ]]$ and $\C_3 = [[2N - i_0 + 1, 2N]]$.   Recall that by assumption there is a $C_V$ s.t. $|V_i | \leq N^{C_V}$.  Let $B_i$ be the Brownian motions for the $x_i$ as above.  Let $\tilB_i$ be independent standard Brownian motions except that
\beq
\tilB_i = B_{i- ( N - i_0 )}, \qquad i \in \C_2.
\eeq
We let $\xstar$ be the solution to
\beq \label{eqn:xstardef}
\d \xstar_i = \sqrt{ \frac{ 2}{N}} \d B_i + \frac{1}{N} \sum_j \frac{1}{ \xstar_i - \xstar_j } \d t 
\eeq
with initial condition 
\beq
\xstar_i (0) = \tilV_i
\eeq
where
\beq
\tilV_i = \begin{cases} - 2 N^{C_V+200} + i N^5, & i \in \C_1 \\  V_{i-(N-i_0 ) }, & i \in \C_2 \\  N^{2 C_V + 200} + i N^5 ,& i \in \C_3 \end{cases}.
\eeq
Define now the process $\xhat$ by
\beq
\d \xhat_i = \frac{ \d \tilB_i}{ \sqrt{N}} + \1_{ \{ i \in \C_1 \cup \C_3 \} } \frac{1}{N} \sum_{j\in\C_1  \cup \C_3}  \frac{1}{ \xhat_j - \xhat_i } \d t +\1_{ \{ i \in \C_2 \} } \frac{1}{N} \sum_{j\in\C_2}  \frac{1}{ \xhat_j - \xhat_i } \d t - \frac{ \xhat_i}{2} \d t,
\eeq
with initial data
\beq
\xhat_i (0) = \tilV_i.
\eeq
Then the process $\xhat_i$ decomposes into two independent processes $\{ \xhat_i \}_{ i \in \C_1 \cup \C_3 }$ and $\{ \xhat_i \}_{ i \in \C_2}$, and $\xhat_i = x_{i - ( N - i_0 )} (t)$ a.s. for all times $t$.  
For the process $\{ \xhat_i \}_{ i \in \C_1 \cup \C_3 }$  an easy argument using standard large deviations bounds on Brownian motion and the fact that initially the particles are far apart, so the interaction term is negligible gives that
\beq \label{eqn:xhatrig}
\pp [ \sup_{ 0 \leq t \leq 1 } | \xhat_i - \tilV_i | \geq 10 ] \leq N^{-D}
\eeq
for any $D>0$ and $i \in \C_1 \cup \C_3$.  In particular we see that with overwhelming probability, the $\xhat_i$ retain their ordering $\xhat_{i} < \xhat_{i+1}$.

Since with overwhelming probability the $\xhat_i$ retain their ordering, the difference $u_i = \hatx_i - \xstar_i$ satisfies a parabolic equation with overwhelming probability,
\beq
\del_t u_i = \sum_{j=1}^{2N-1} \frac{ u_j - u_i } { ( \xhat_i - \xhat_j ) ( \xstar_i - \xstar_j ) }  + \xi_i
\eeq
where
\beq
\xi_i := \1_{ \{ i \in \C_1 \cup \C_3 \} } \frac{1}{N} \sum_{j\in\C_2}  \frac{1}{ \xhat_j - \xhat_i } \d t +\1_{ \{ i \in \C_2 \} } \frac{1}{N} \sum_{j\in \C_1 \cup \C_3} \frac{1}{ \xhat_j - \xhat_i }.
\eeq
By the estimate \eqref{eqn:xhatrig} we see that if $i \in \C_1 \cup \C_3$ and $j \in \C_2$ then 
\beq
\sup_{0 \leq t \leq 1} \frac{1}{ |\xhat_i (t) - \xhat_j(t) | } \leq \frac{C}{N^{200}}.
\eeq
Therefore
\beq
\sup_{ 0 \leq t \leq 1 } || \xi (t) ||_\infty \leq \frac{1}{ N^{190}}.
\eeq
By the Duhamel formula we conclude that
\beq
\sup_{0 \leq t \leq t_1} || \xstar (t) - \xhat (t) ||_\infty \leq \frac{1}{ N^{180}}
\eeq
with overwhelming probability.  This yields \eqref{eqn:xstarxest}.  We can make a similar construction for $y$.  The new process $\ystar$ satisfies the estimates of Lemma \ref{lem:strrig}, but replacing $\gamma_j^{(\sc)}$ by $\tilV_j$ for the indices $j$ for which $\gamma_j^{(\sc)}$ is undefined. 

We remark that the $N$ appearing in \eqref{eqn:xstardef} and the corresponding definition of $\ystar$ no longer equals the number of particles which is $2N-1$.  However, since $N \asymp 2N-1$ and the same factor appears in the definition of $\xstar$ and $\ystar$, this will not affect any of our methods.

\section{Sobolev inequality} \label{a:sobolev}

Let $M_1 < M_2$ be natural numbers.  Let $u_i : \zz \to \rr$ be a sequence.  By Cauchy-Schwarz,
\begin{align}
&\left( \frac{1}{2M_1+1} \sum_{|i| \leq M_1} u_i - \frac{1}{ 2M_2+1} \sum_{ |j| \leq M_2 } u_j \right)^2 \notag\\
= &\frac{1}{ ( 2 M_1+1)^2 ( 2 M_2 +1 )^2 } \left( \sum_{ \substack{ |i| \leq M_1, |j| \leq M_2 ,\\ i \neq j}} \frac{ (u_i - u_j ) (i-j)}{ (i-j) }\right)^2 \notag\\
\leq & \frac{1}{ ( 2 M_1 +1 )^2 ( 2 M_2 + 1)^2 } \left( \sum_{ \substack{ |i| \leq M_1, |j| \leq M_2 ,\\ i \neq j}}  \frac{ ( u_i - u_j )^2}{ (i - j )^2 } \right) \left( \sum_{ \substack{ |i| \leq M_1, |j| \leq M_2 ,\\ i \neq j}}  (i-j)^2 \right).
\end{align}
Clearly,
\beq
\left( \sum_{ \substack{ |i| \leq M_1, |j| \leq M_2 ,\\ i \neq j}}  (i-j)^2 \right) \leq C \left( \sum_{  |i| \leq M_1}  M_2^3 \right) \leq C M_1 M_2^3.
\eeq
Therefore,
\beq \label{eqn:avgdif}
\left( \frac{1}{2M_1+1} \sum_{|i| \leq M_1} u_i - \frac{1}{ 2M_2+1} \sum_{ |j| \leq M_2 } u_j \right)^2 \leq C \frac{ M_2}{M_1} \left( \sum_{|i| \leq M_1 , |j| \leq M_2 } \frac{ (u_i - u_j )^2}{ ( i - j )^2 } \right).
\eeq
We can iterate the above inequality to prove the following lemma.
\bel \label{lem:sobrev} Let $u_i : \zz \to \rr$ be a sequence and let $N$ be a natural number.  There is a universal constant $C$ s.t.
\beq
\left( u_0 - \frac{1}{ 2N +1} \sum_{ |i| \leq N } u_i \right)^2 \leq C \log (N)^2 \left( \sum_{|i| \leq N, |j| \leq N } \frac{ (u_i - u_j )^2}{ ( i - j )^2 } \right).
\eeq
\eel
\proof Choose  $n \leq C \log (N)$ numbers $M_i$ s.t. $M_{0} = 1$, $M_n = N$ and $M_{i+1} \leq 3 M_i$.  Define $A_i$ by
\beq
A_i := \frac{1}{ 2 M_i +1} \sum_{ |j| \leq M_i } u_j - \frac{1}{ 2 M_{i+1} +1} \sum_{ |j| \leq M_{i+1} } u_j.
\eeq
Then,
\begin{align}
\left( u_0 - \frac{1}{ 2 N + 1 } \sum_{ |i| \leq N } u_i \right)^2 = \left( \sum_{i=1}^{n-1} A_i \right)^2 \leq n^2 \sup_i (A_i )^2.
\end{align}
By \eqref{eqn:avgdif}, 
\beq
(A_i)^2 \leq C \left( \sum_{ |i| \leq N, |j| \leq N} \frac{ (u_i - u_j )^2}{(i-j)^2} \right)
\eeq
which yields the claim. \qed

\section{Fubini lemma} 
\bel \label{lem:fubini}
Let $X ( u)$ for $0 \leq u \leq 1$ s.t. there is an event $\F$ with $\pp [ \F ] \geq 1 - \eps$ on which $\sup_u |X (u) | \leq A$.  Suppose that for each $u$ there is an event $\F_u$ with $\pp [ \F_\alpha ] \geq 1 - \eps$ on which $|X (u) | \leq Y (u)$.  Then,
\beq
\pp\left[ \int_0^1 X (u) > \int_0^1 Y(u) \d u + \delta \right] \leq \eps \left( 1 + \frac{A}{ \delta} \right)
\eeq
\eel
\proof We write 
\begin{align}
\1_{ \F} \int_0^1 | X(u) | \d u &= \1_{ \F} \int_0^1 | X (u) | \1_{ \F_u } \d u + \int_0^1 | X( u) | \1_{ \F \cap \F_u^c } \d u \notag \\
&\leq \int_0^1 Y (u) \d u + A \int_0^1  \1_{  \F_u^c } \d u .
\end{align}
By Markov's inequality,
\beq
\pp \left[ A \int_0^1 \1_{ \F_u^c } \d u > \delta \right] \leq \frac{ A \eps}{ \delta}.
\eeq
The claim follows.
\qed 

\remark We usually apply this with $A = N^C$ for some fixed $C$ and $\eps = N^{-D}$ for a large $D$ and $\delta  = N^{-D/2}$.

\section{Fixed energy universality under relaxed assumptions at intermediate scales}

The purpose of this appendix is to indicate a proof of fixed energy universality under a relaxation on the behavior of the initial data at intermediate scales $1 \geq \eta \geq N^{ - \delta_2}$.  The purpose of this kind of relaxation is in an application to the universality of band matrices \cite{newband}. 

As before, we will denote the initial data by $V$ and consider the matrix ensemble
\beq
H_{t} = V + \sqrt{t} G
\eeq
where $G$ is an independent GOE matrix.  We have the following theorem.
\bet \label{thm:new}
Consider $H_s = V + \sqrt{s} G$ as above.  Let $s = N^{\om_s}/N$, and assume that
\beq
1 > \om_s > \frac{2}{3}.
\eeq
Assume that
\beq
||V||\leq N^{C_V}
\eeq
for some $C_V>0$.  Assume  for some $E_0$ and constants $c_1$ and $C_1$, we have the estimates
\beq \label{eqn:regrev}
c_1 \leq \Im [ m_V (E + \i \eta ) ] \leq C_1
\eeq
for all $z =E + \i \eta$ obeying
\beq
|E_0 - E| \leq N^{ \delta_3 } s, \qquad N^{ - \delta_1} s \leq \eta \leq N^{- \delta_2} =: \eta^*.
\eeq
For the exponents $\delta_1, \delta_2, \delta_3 >0$ we assume that,
\beq \label{eqn:deltas}
0 < \delta_2 < \delta_3, \qquad \delta_2 < \frac{ 1 - \om_s}{4}, \qquad \delta_2 < \mathfrak{a}
\eeq
where $\mathfrak{a}$ is some small universal constant.  Under the above assumptions, the conclusion of Theorem \ref{thm:fixed} holds for $H_{s}$.
\eet

As in Theorem \ref{thm:fixed}, there are three components to the proof of Theorem \ref{thm:new}.  They are:
\begin{enumerate}
\item  Local law and eigenvalue rigidity - the estimates summarized in Section \ref{sec:rigid}.  
\item The homogenization theory of Section \ref{sec:homog}.
\item The mesoscopic central limit theorem of Section \ref{sec: mesosection}. 
\end{enumerate}
Compared to the assumptions of Theorem \ref{thm:fixed}, the assumptions of Theorem \ref{thm:new} have been weakened in two different ways but strengthened in the assumption that $s \gg N^{-1/3}$.  The two weakenings are
\begin{enumerate}[label=(\roman*)]
\item  \label{it:sq}The size of the window of regularity of $V$ has been reduced from about size $\sqrt{s}$ to $s$.
\item \label{it:etast}No bounds on the behavior of $\Im [ m_V (E + \i \eta)]$ have been assumed for intermediate scales $1 \geq \eta \geq N^{-\delta_2}$. 
\end{enumerate}
Let us first discuss how to deal with \ref{it:sq} as it is quite simple.  The work \cite{landonyau} proves the local law and rigidity for initial data $V$ regular in a window of size $s$ and down to scales $\eta \sim s$; that is, the proof of \cite{landonyau} already handles the weakening in \ref{it:sq}.   We will show that the proof given there extends to the assumptions of Theorem \ref{thm:new}; i.e., we will show that for the rigidity and local law,  \ref{it:etast} does not affect the proof.  In particular, with overwhelming probability, the matrix
\beq
V' := H_{s} = V + \sqrt{s} G
\eeq
will obey the estimate \eqref{eqn:regrev} for $|E-E_0| \leq N^{\delta_3} s/2$ and $N^{-\delta_2} \geq \eta \geq N^{\eps}/N$ for any $\eps >0$.  We then fix a time $t_0 = N^{\om_0}/N$ for a fixed
\beq
0 < \om_0 < \frac{1}{10},
\eeq
and consider the fixed energy universality of 
\beq
H'_{t_0} := V' + \sqrt{t_0} G'
\eeq
where $G'$ is an independent GOE matrix.   That is, we will condition on $V'$ and try to apply Theorem \ref{thm:fixed} to $H'_{t_0}$.  By the assumption that $\om_s > 2/3$ we see that
\beq
c_2 \leq \Im [ m_{V'} ( E + \i \eta ) ] \leq C_2
\eeq
for
\beq
|E - E_0 | \leq N^{1/10} \sqrt{ t_0} , \qquad N^{ - \delta_1' } t_0 \leq \eta \leq N^{- \delta_2} 
\eeq
where $\delta_2$ is as above, and $\delta_1'>0$.  The first estimate holds since $s \gg N^{1/10} \sqrt{t_0}$ by assumption.   We are therefore now in the set-up of Theorem \ref{thm:fixed} except for the weakening \ref{it:etast}, and we only need to check how this affects  the proof of homogenization and the mesoscopic central limit theorem.  
This is outlined below, as well the proof of the rigidity and local law under these assumptions.

\subsection{Rigidity and local law}

In this section we give the proof of the following theorem.  We define the spectral domains,
\beq
\D_{1, \eps} := \{ E + \i \eta : |E-E_0| \leq t N^{ \delta_3}/2, N^{\eps}/N \leq \eta \leq 10 \}
\eeq
and
\beq
\D_2 := \{ E + \i \eta : |E-E_0| \leq N^{ 10 C_V} , \frac{1}{10} \leq \eta \leq N^{ 10 C_V} \}.
\eeq
and $\D_\eps = \D_1 \cup \D_2$. 

\bet \label{thm:rigidrev} Let $H_{t} = V + \sqrt{t} G$, and $t = N^{\om}/N$, where $V$ satisfies $||V|| \leq N^{C_V}$ for some $C_V >0$ and
\beq
c_1 \leq \Im [ m_V ( E + \i \eta ) ] \leq C_1
\eeq
for
\beq
|E- E_0 | \leq N^{\delta_3} t, \qquad N^{-\delta_1} t \leq \eta \leq N^{-\delta_2}
\eeq
where 
\beq \label{eqn:deltasrev}
0 < \delta_2 < \delta_3, \qquad \delta_2 < \frac{ 1 - \om}{4}.
\eeq
 Then the local law of Theorem \ref{thm:localrev} holds in the spectral domain $\D_\eps$ for any $\eps >0$, with overwhelming probability.  The following rigidity estimates hold
\beq \label{eqn:rig1}
| \lambda_i - \gamma_i (t) | \leq \frac{N^{\eps}}{N}
\eeq
for $| \gamma_i - E_0 | \leq s N^{\delta_3}/2$, with overwhelming probability and any $\eps >0$.
\eet
\remark Compared to Theorem \ref{thm:new} we have crucially dropped the assumption that $t \gg N^{-1/3}$. \qed

Inspecting the proof of Lemma 7.19 of \cite{landonyau}, we see that in order to prove the estimate \eqref{eqn:rig1} we need only the local law in the spectral domain $\D_\eps$ as well as the estimate $c \leq \rho_t (E) \leq C$ for $E$ near $E_0$.  The latter is the content of Lemma \ref{lem:mrev} below.  The local law in the spectral domain $\D_2$ is proved identically as to \cite{landonyau} as the behavior of $\Im [m_V]$ for $\eta \ll 1 $ is not used.  The local law in $\D_{1, \eps}$ is discussed after the statement and proof of Lemma \ref{lem:mrev}.  This lemma is our version of Lemma 7.2 of \cite{landonyau} and is proven similarly.  The changes are given below.
\bel \label{lem:mrev}
Suppose that $V$ satisfies the assumptions of Theorem \ref{thm:new}.  We have the estimate
\beq \label{eqn:immfctrev}
c \leq \Im [ \mfct (z) ] \leq C, \qquad |E| \leq t N^{\delta_3}/2, \quad 0 \leq \eta \leq N^{ -   \delta_2  }/2.
\eeq
We also have
\beq \label{eqn:R2rev}
c \leq |1 - t  R_2 (z) | \leq C,  \qquad |E| \leq t N^{ \delta_3}/2, \qquad 0 \leq \eta \leq 10.
\eeq
\eel
\remark This immediately implies the estimates
\beq \label{eqn:revm1}
c \leq \rho_t (E) \leq C , \quad | \rho_t' (E) | \leq C/t, \qquad |E| \leq t N^{ \delta_3}/2,
\eeq
as well as
\beq \label{eqn:revm2}
|\del_z \mfc | \leq \frac{C}{t+ \eta}, \qquad | \del_z^2 \mfc | \leq \frac{C}{(t+ \eta)^2}
\eeq
\proof The proof is similar to that of Lemma 7.2 of \cite{landonyau}, and so we address only how the proof changes.  First we prove \eqref{eqn:immfctrev}.  The key point is that instead of the bounds (7.12) of \cite{landonyau} we have only the weaker bound
\beq
| m_V (E + \i \eta ) | \leq C_1 \log(N) N^{\delta_2}, \qquad |E| \leq  3 t N^{ \delta_3}/4, \quad N^{-\delta_1} t \leq \eta \leq 10.
\eeq
Consequently, the definition of $\eta_*$ that is used in the proof in \cite{landonyau} is modified to
\beq
\eta_* := \inf \{ \eta \leq N^{ - \delta_2 }/2 : | \mfct (E + \i \eta ) | \leq 2 C_1 \log(N) N^{ \delta_2} , c \leq \Im [ \mfct (E + \i \eta ) ] \leq C \}.
\eeq
The assumptions \eqref{eqn:deltasrev} assure that $\eta_* < N^{ - \delta_2}/2$ (as $t|\mfc | \leq t \eta^{-1} \ll t N^{\delta_3}$ for $\eta$ close to $N^{-\delta_2}/2$), and then the rest of the argument goes through, proving \eqref{eqn:immfctrev}.

The proof of \eqref{eqn:R2rev} is unchanged in the small $\eta \leq N^{ -\delta_2 }/2$ region.  In the larger region  $\eta \geq N^{- \delta_2}$ we use $|t R_2 | \leq t \eta^{-2} \ll 1$ due to the second assumption of \eqref{eqn:deltasrev}.  \qed 

Finally, we prove the local law in $\D_1$.  For this it is simpler to rely on Theorem 3.1 of \cite{hlmeso}.   This result will    give the local law in the domain
\beq
\D' = \{ E + \i \eta : \eta \Im [ \mfc  (E + \i \eta ) ] \geq \log(N)^{10}/N \}.
\eeq
Since $\eta \Im [ \mfc (E + \i \eta ) ]$ is an increasing function of $\eta$ we see that $\D'$ contains $\D_1$.  

The caveat is that the work \cite{hlmeso} operates under the assumption of bounded initial data. 
This assumption  is due to the fact that \cite{hlmeso} deals with the more general case of DBM flows with a general potential $U(x)$ on the RHS of \eqref{eqn:dlambda} (in the paper \cite{hlmeso} this is denoted $V(x)$ but we instead use here $U(x)$ to avoid conflict with our use of $V$ as the initial data).  Specifically, in the main calculation (Section 3.2 of \cite{hlmeso}) there are error terms involving the potential $U(x)$ that are handled using the assumption of bounded support.  In our case $U(x) =0$, and so these error terms are not present, and moreover the calculations involving the limiting hydrodynamic equation are simplified (i.e., the characteristics $z_t (u)$ of (2.25) of \cite{hlmeso} satisfy $u = z_t(u) + t \mfct (z_t(u))$).  

The changes to the argument of \cite{hlmeso} to the present set-up are otherwise notational.  One has to replace to the $\mathfrak{b}$ of the domain $\D_t$ defined in (3.2) of \cite{hlmeso} with $\mathfrak{b} = 3 N^{C_V}$.  Since $U$ is $0$, it is straightforward to prove the estimates of Proposition 2.7 of \cite{hlmeso}.  The arguments of Section 3.1 of \cite{hlmeso} go through without change.   The key use of the bounded support of the initial data is in Proposition 3.9 of \cite{hlmeso}.  This proposition bounds an error term that is not present in the case that $U(x) = 0$.   The remaining arguments of Section 3.2 directly apply and we find the local law in the domain $\D_{1, \eps}$ as required.

\subsection{Homogenization theory under weakened assumptions at intermediate scales}

Inspecting the arguments of Section \ref{sec:homog} which lead to Theorem \ref{thm:mainhomog}, we find the following theorem.

\bet \label{thm:homorev}
Consider $H'= V'+ \sqrt{t_0} G$.  Assume that $V'$ satisfies $||V'||\leq N^{C_V}$ and
\beq
c_1 \leq \Im [ m_{V'} (E + \i \eta )  ]\leq C_1
\eeq
for some $c_1$, $C_1$ and,
\beq
|E-E_0 | \leq \sqrt{t_0} N^{\delta_3} , \qquad t_0 N^{-\delta_1} \leq \eta \leq N^{-\delta_2},
\eeq
where
\beq \label{eqn:homodelt}
\delta_2 < \frac{1 - \om_0}{2} + \delta_3 , \qquad \delta_2 < \frac{1-\om_0}{4}
\eeq
Then under the assumptions \eqref{eqn:oldrev1} we have that the estimate \eqref{eqn:oldrev2} holds.
\eet
\remark The first assumption of \eqref{eqn:homodelt} is of course redundant but we keep it to match up with \eqref{eqn:deltasrev}. \qed

The key inputs to the proof of Theorem \ref{thm:mainhomog} are the rigidity estimates, the local law and some estimates on the behavior of the free convolution law $\mfct$.  Due to Theorem \ref{thm:rigidrev} we know that the local law and rigidity will hold for the model under consideration (note that \eqref{eqn:homodelt} is the same assumption as \eqref{eqn:deltasrev} after accounting for the $t_0 \to \sqrt{t_0}$ change).  The required behavior of the free convolution law follows from Lemma \ref{lem:mrev} and the estimates \eqref{eqn:revm1},\eqref{eqn:revm2}.  Note that we do not have the boundedness of $\Im [ \mfct]$ all the way up to $\eta \sim 1$.  However, $\mfct$ only enters into the proof with  an argument of $\mfct (E + \i \eta)$ with $\eta \leq N^{-1/2}$ (specifically, Lemma \ref{lem:shortrange} and Appendix \ref{a:free}).  For $\eta \leq N^{-1/2} \ll N^{-\delta_2}$, the estimates for $\mfct$ are unchanged.


\subsection{Mesoscopic estimate under weakened assumptions at intermediate scales}

The estimate we need for the proof of Theorem \ref{thm:new} is the following.
\bet \label{thm:mesorev}
Let $H_{t_m} = V + \sqrt{t_m} G$ where $t_m = N^{\om_m}/N$.  Assume
\beq
\frac{1}{20} < \om_m < \frac{1}{10}. 
\eeq
Assume that $V$ obeys the estimate $||V||\leq N^{C_V}$ for some $C_V >0$ and for some $c_1, C_1 >0$ we have,
\beq \label{eqn:revc1}
c_1 \leq \Im [ m_V (E + \i \eta ) ] \leq C_1
\eeq
for 
\beq
|E - E_0 | \leq \sqrt{t_m} N^{ \delta_3}, \qquad t_m N^{-\delta_1  } \leq \eta \leq N^{- \delta_2}
\eeq
where
\beq
\delta_2 < \frac{ 1 - \om_m}{2} + \delta_3, \qquad \delta_2 < \frac{1- \om_m}{4}.
\eeq
Let $\phi$ be a test function as described in Section \ref{sec: mesosection} on the scale $t_1 = N^{\om_1}/N$ where
\beq
\om_1 = 10^{-5}.
\eeq
There is a constant $\mathfrak{a} >0$ so that if 
\beq \label{eqn:revc3}
\delta_2 < \mathfrak{a}
\eeq
then for some $c, \eps>0$ we have the estimate
\beq
\left| \ee \left[ \e^{ \i \lambda \tr \phi } \right] \right| \leq \e^{ - c \lambda^2 \log(N) } + N^{- \eps}.
\eeq
\eet

The first thing to do is apply the preliminary argument near \eqref{eqn:fixbd1}.  Fix an auxilliary time $t_2 = N^{\om_2}/N$ with $\om_2 = \om_1 (1 + K^{-1})$ where $K$ is constant so large that
\beq
\frac{\om_2}{2} < \om_1 < \om_2 < \om_m.
\eeq
This is done so that the constraints \eqref{eqn: omega-condition-1} of Theorem \ref{thm: linstat} are satisfied with $\om_0=\om_2$.  We then write
\beq
H_{t_m} = V + \sqrt{t_m } G \stackrel{d}{=} (V + \sqrt{t_m-t_2} G) +\sqrt{t_2} G'=: V'' + \sqrt{t_2} G' =: H'_{t_2}
\eeq
where $G'$ is another independent GOE matrix.  As in \eqref{eqn:fixbd1} we will use Theorem \ref{thm: linstat} to calculate 
\beq \label{eqn:revc2}
\ee[ \e^{ \i \lambda \tr \phi } | V'' ],
\eeq
after, of course, checking that the argument goes through under the relaxation of item \ref{it:etast}.  Due to Theorem \ref{thm:rigidrev} we know that with overwhelming probability that $V''$ obeys \eqref{eqn:revc1} for $|E-E_0 | \leq \sqrt{t_2} N^{\delta_3}$ and $t_2 N^{-\delta_1 } \leq \eta \leq N^{-\delta_2}$.   

Conditioning on $V''$ we now repeat the arguments of Section \ref{sec: mesosection} to calculate \eqref{eqn:revc2}.   The main input into the arguments of Section \ref{sec: mesosection} are the local law and the behavior of the free convolution $m_{\fc, t_2}$.  Due to Theorem \ref{thm:rigidrev} we know that the local law holds for $H'_{t_2}$ and we will say nothing further of this.  

On the other hand, the estimates for $m_{\fc, t_2}$ in the regime $1 \geq \eta \geq N^{-\delta_2}$ do in fact degenerate somewhat, and so we need to show how to handle this.  This is where the condition \eqref{eqn:revc3} is required.

First it is important to note that the estimates for denominators appearing in the argument
\beq \label{eqn:revc4}
| 1 - t_2  R_2 | \geq c, \qquad |V''_i -z - t_2 m_{\fc, t_2} | \geq  c \max \{t, \eta \}
\eeq
still hold.  The first was established in Lemma \ref{lem:mrev}.   For the second, we always have that the imaginary part  of $z + t_2 m_{\fc, t_2}$ is larger than $\eta$.  For $\eta> N^{-\delta_2}$, $\eta \geq t$ by assumption.  For $\eta < N^{-\delta_2}$, we have that $\Im [ m_{\fc, t_2 } (E+ \i \eta) ] \geq c$ by Lemma \ref{lem:mrev}.  Therefore, the second estimate of \eqref{eqn:revc4} also holds.

  The estimate \eqref{eqn:revc4} ensures that all denominators involved in the calculations in Section \ref{sec: mesosection} in the present setting obey the same estimates as before, when we did not have the weakening of item \ref{it:etast}.

There are two components of the mesoscopic central limit theorem.  The first is a calculation of the characteristic function, i.e. the arguments of Sections \ref{sec: characteristic}-\ref{sec: comp-T22}, and the second is the calculation of the variance term, Proposition \ref{prop: variance-comp}.  Let us first examine the calculation of the characteristic function. 

In this part of the proof the worst that will happen is that some of the estimates may degenerate, in that the error terms becomes multiplied by a some factor $(\eta^*)^{-k}$ where $k$ is a bounded  constant, typically $k \leq 5$, arising from some of our resolvent expansions.  These resolvent expansions are  always to some low order.  As an example, the estimate \eqref{eqn: Ahalf} will degenerate to
\beq
A_j^{\circ} = \O ( \sqrt{ t} N^{-1/2} + (\eta^*)^{-1} t ( N \eta)^{-1/2} ).
\eeq
Then, this combined with the fact that \eqref{eqn:revc4} still holds will give an error of $\O ( (\eta^*)^{-1} | \eta|^{-1} ( N | \eta | )^{-1/2} )$ in \eqref{eqn: AinverseB-estimate}.  Other estimates in the calculation of the characteristic function will behave similarly.  

With these considerations, the error we will find in the estimate \eqref{eqn:revd1} will be 
\beq
\O ( N^{\om_2 /4  - \om_1/2} (\eta_*)^{-C} )
\eeq
for some constant $C>0$.  If $\mathfrak{a}$ is sufficiently small, then the error here will still be $\O (N^{-c} )$ for some $c>0$, since we have chosen $\om_2/4-\om_1/2 > 0$.  

We have to be more careful in calculating the variance, as this quantity is $\O ( \log (N))$ and so losing a polynomial factor $N^\eps$ would cause trouble.   First, let us examine the proof of Proposition \ref{prop: I2-bound-prop}.  Here, the estimates are unchanged until the last part of the proof, in which the region $\eta^* \leq \eta \leq 1$ is integrated over; here is the only place where our estimates behave differently than before.   In this region, we may simply estimate the integrand by some power of $( \eta^*)^{-1}$ and absorb it into the prefactor $t^2$.  The contribution is then $o(1)$ and so the statement of the proposition is unchanged.  Proposition \ref{prop:revd2} also involves an integration over the region $\eta^* < \eta < 1$, and similar considerations apply here.  Prior to this proposition, one must handle the terms \eqref{eqn: plusvariance-1}-\eqref{eqn: mainvariance-1b}.  However, they involve only the behavior of the free convoluton for small $\eta \ll \eta^*$, and so these terms can be treated in the same manner as before.  The remainder of the proof, i.e., the handling of \eqref{eqn: mainvariance-2}-\eqref{eqn: mainvariance-2-2} is also unchanged, involving only the behavior of the free convolution at small $\eta$.


\bibliography{mybib}{}

\begin{thebibliography}{10}

\bibitem{ziliang}
B.~Adlam and Z.~Che.
\newblock Spectral statistics of sparse random graphs with a general degree
  distribution.
\newblock {\em preprint, arXiv:1509:03368}, 2015.

\bibitem{ajankigaus}
O.~Ajanki, L.~Erd\H{o}s, and T.~Kr{\"u}ger.
\newblock Local spectral statistics of {G}aussian matrices with correlated
  entries.
\newblock {\em J. Stat. Phys.}, 163(2):280--302, 2016.

\bibitem{ajanki1}
O.~Ajanki, L.~Erd\H{o}s, and T.~Kr{\"u}ger.
\newblock Singularities of solutions to quadratic vector equations on complex
  upper half-plane.
\newblock {\em Comm. Pure. Appl. Math.}, 70(9):1672--1705, 2017.

\bibitem{ajanki3}
O.~Ajanki, L.~Erdos, and T.~Kr{\"u}ger.
\newblock Quadratic vector equations on complex upper half-plane.
\newblock {\em preprint, arXiv:1506.05095}, 2015.

\bibitem{ajanki2}
O.~Ajanki, L.~Erdos, and T.~Kr{\"u}ger.
\newblock Universality for general {W}igner-type matrices.
\newblock {\em Probab. Theory Related Fields}, 169(3-4):667--727, 2017.

\bibitem{ajankicor}
O.~Ajanki, L.~Erdos, and T.~Kr{\"u}ger.
\newblock Stability of the matrix {D}yson equation and random matrices with
  correlations.
\newblock {\em Probab. Theory and Related Fields}, pages 1--81, 2018.

\bibitem{roland2}
R.~Bauerschmidt, J.~Huang, A.~Knowles, and H.-T. Yau.
\newblock Bulk eigenvalues statistics for random regular graphs.
\newblock {\em Ann. Probab.}, 45(6A):3626--3663, 2015.

\bibitem{roland1}
R.~Bauerschmidt, A.~Knowles, and H.-T. Yau.
\newblock Local semicircle law for random regular graphs.
\newblock {\em Comm. Pure Appl. Math.}, 70(10):1898--1960, 2015.

\bibitem{transport}
F.~Bekerman, A.~Figalli, and A.~Guionnet.
\newblock Transport maps for {B}eta-matrix models and universality.
\newblock {\em Comm. Math. Phys.}, 338(2):589--619, 2015.

\bibitem{ben2005universality}
G.~Ben~Arous and S.~P{\'e}ch{\'e}.
\newblock Universality of local eigenvalue statistics for some sample
  covariance matrices.
\newblock {\em Comm. Pure. Appl. Math.}, 58(10):1316--1357, 2005.

\bibitem{biane}
P.~Biane.
\newblock On the free convolution with a semi-circular distribution.
\newblock {\em Indiana Univ. Math. J.}, 46(3):705--718, 1997.

\bibitem{bourgadeprep}
P.~Bourgade.
\newblock Extreme gaps between eigenvalues of {W}igner matrices.
\newblock {\em preprint, arXiv:1812.10376}, 2018.

\bibitem{bourgadeband}
P.~Bourgade, L.~Erd\H{o}s, H.-T. Yau, and J.~Yin.
\newblock Universality for a class of random band matrices.
\newblock {\em preprint, arXiv:1602.02312}, 2016.

\bibitem{bourgade2012bulk}
P.~Bourgade, L.~Erd{\H{o}}s, and H.-T. Yau.
\newblock Bulk universality of general $\beta$-ensembles with non-convex
  potential.
\newblock {\em J. Math. Phys.}, 53(9):095221, 2012.

\bibitem{bourgade2014edge}
P.~Bourgade, L.~Erd{\H{o}}s, and H.-T. Yau.
\newblock Edge universality of $\beta$-ensembles.
\newblock {\em Comm. Math. Phys.}, 332(1):261--353, 2014.

\bibitem{bourgade2014universality}
P.~Bourgade, L.~Erd{\H{o}}s, and H.-T. Yau.
\newblock Universality of general $\beta $-ensembles.
\newblock {\em Duke Math. J.l}, 163(6):1127--1190, 2014.

\bibitem{homogenization}
P.~Bourgade, L.~Erd{\H{o}}s, H.-T. Yau, and J.~Yin.
\newblock Fixed energy universality for generalized {W}igner matrices.
\newblock {\em Comm. Pure Appl. Math}, 2015.

\bibitem{que}
P.~Bourgade and H.-T. Yau.
\newblock The eigenvector moment flow and local quantum unique ergodicity.
\newblock {\em Comm. Math. Phys.}, 350(1):231--278, 2017.

\bibitem{newband}
P.~Bourgade, H.-T. Yau, and J.~Yin.
\newblock Random band matrices in the delocalized phase, {I}: Quantum unique
  ergodicity and universality.
\newblock {\em preprint arXiv:1807.01559}, 2018.

\bibitem{bvw}
J.~Bourgain, V.~H. Vu, and P.~M. Wood.
\newblock On the singularity probability of discrete random matrices.
\newblock {\em J. Funct. Anal.}, 258(2):559--603, 2010.

\bibitem{boutet1}
A.~Boutet~de Monvel and A.~Khorunzhy.
\newblock Asymptotic distribution of smoothed eigenvalue density. {I}.
  {G}aussian random matrices.
\newblock {\em Random Oper. Stoch. Equ.}, 7(1):1--22, 1999.

\bibitem{boutet2}
A.~Boutet~de Monvel and A.~Khorunzhy.
\newblock Asymptotic distribution of smoothed eigenvalue density. {II}.
  {W}igner random matrices.
\newblock {\em Random Oper. Stoch. Equ.}, 7(2):149--168, 1999.

\bibitem{ziliang2}
Z.~Che.
\newblock Universality of random matrices with correlated entries.
\newblock {\em Electron. J. Probab.}, 22, 2017.

\bibitem{CKK2}
Z.-Q. Chen, P.~Kim, and T.~Kumagai.
\newblock On heat kernel estimates and parabolic {H}arnack inequality for jump
  processes on metric measure spaces.
\newblock {\em Acta Math. Sin. (Engl. Ser.)}, 25(7):1067--1086, 2009.

\bibitem{CKK}
Z.-Q. Chen, P.~Kim, and T.~Kumagai.
\newblock Global heat kernel estimates for symmetric jump processes.
\newblock {\em Trans. Amer. Math. Soc.}, 363(9):5021--5055, 2011.

\bibitem{cook}
N.~A. Cook.
\newblock On the singularity of adjacency matrices for random regular digraphs.
\newblock {\em Probab. Theory Related Fields}, pages 1--58, 2015.

\bibitem{ctv}
K.~P. Costello, T.~Tao, and V.~Vu.
\newblock Random symmetric matrices are almost surely nonsingular.
\newblock {\em Duke Math. J.}, 135(2):395--413, 2006.

\bibitem{cv}
K.~P. Costello and V.~H. Vu.
\newblock The rank of random graphs.
\newblock {\em Random Structures Algorithms}, 33(3):269--285, 2008.

\bibitem{costinlebowitz}
O.~Costin and J.~L. Lebowitz.
\newblock Gaussian fluctuation in random matrices.
\newblock {\em Phys. Rev. Lett.}, 75(1):69, 1995.

\bibitem{courtade}
T.~A. Courtade and R.~D. Wesel.
\newblock Efficient universal recovery in broadcast networks.
\newblock In {\em Communication, Control, and Computing (Allerton), 2010 48th
  Annual Allerton Conference on}, pages 1542--1549. IEEE, 2010.

\bibitem{DJ}
M.~Duits and K.~Johansson.
\newblock On mesoscopic equilibrium for linear statistics in {D}yson's
  {B}rownian motion.
\newblock {\em Mem. Amer. Math. Soc.}, 255(1222), 2018.

\bibitem{ICM}
L.~Erdos.
\newblock Random matrices, log-gases and holder regularity.
\newblock {\em preprint, arXiv:1407.5752}, 2014.

\bibitem{erdos2012spectral}
L.~Erd{\H{o}}s, A.~Knowles, H.-T. Yau, and J.~Yin.
\newblock Spectral statistics of {E}rd{\H{o}}s-{R}{\'e}nyi graphs {II}:
  Eigenvalue spacing and the extreme eigenvalues.
\newblock {\em Comm. Math. Phys.}, 314(3):587--640, 2012.

\bibitem{renyione}
L.~Erd{\H{o}}s, A.~Knowles, H.-T. Yau, and J.~Yin.
\newblock Spectral statistics of {E}rd{\H{o}}s--{R}{\'e}nyi graphs {I}: local
  semicircle law.
\newblock {\em Ann. Probab.}, 41(3B):2279--2375, 2013.

\bibitem{erdos2010bulk}
L.~Erd{\H{o}}s, S.~P{\'e}ch{\'e}, J.~A. Ramirez, B.~Schlein, and H.-T. Yau.
\newblock Bulk universality for {W}igner matrices.
\newblock {\em Comm. Pure Appl. Math.}, 63(7):895--925, 2010.

\bibitem{yautaovu}
L.~Erd{\H{o}}s, J.~Ramirez, B.~Schlein, T.~Tao, V.~Vu, and H.-T. Yau.
\newblock Bulk universality for {W}igner {H}ermitian matrices with
  subexponential decay.
\newblock {\em Math. Res. Lett.}, 17(4):667--674, 2010.

\bibitem{erdos2011universality}
L.~Erd{\H{o}}s, B.~Schlein, and H.-T. Yau.
\newblock Universality of random matrices and local relaxation flow.
\newblock {\em Invent. Math.}, 185(1):75--119, 2011.

\bibitem{localrelaxation}
L.~Erd{\H{o}}s, B.~Schlein, H.-T. Yau, and J.~Yin.
\newblock The local relaxation flow approach to universality of the local
  statistics for random matrices.
\newblock {\em Ann. Inst. Henri Poincar{\'e} Probab. Stat.}, 48(1):1--46, 2012.

\bibitem{ES}
L.~Erd{\H{o}}s and K.~Schnelli.
\newblock Universality for random matrix flows with time-dependent density.
\newblock {\em Ann. Inst. H. Poincar{\'e} Probab. Statist.}, 53(4):1606--1656,
  2017.

\bibitem{Gap}
L.~Erd{\H{o}}s and H.-T. Yau.
\newblock Gap universality of generalized {W}igner and {$\beta$}-ensembles.
\newblock {\em J. Eur. Math. Soc.}, 17(8):1927--2036, 2015.

\bibitem{erdos2012bulk}
L.~Erd{\H{o}}s, H.-T. Yau, and J.~Yin.
\newblock Bulk universality for generalized {W}igner matrices.
\newblock {\em Probab. Theory Related Fields}, 154(1-2):341--407, 2012.

\bibitem{erdos2012rigidity}
L.~Erd{\H{o}}s, H.-T. Yau, and J.~Yin.
\newblock Rigidity of eigenvalues of generalized {W}igner matrices.
\newblock {\em Adv. Math.}, 229(3):1435--1515, 2012.

\bibitem{frieze}
A.~Frieze.
\newblock Random structures and algorithms.
\newblock {\em Proceedings of the {ICM}}.

\bibitem{meso2}
Y.~He and A.~Knowles.
\newblock Mesoscopic eigenvalue statistics of {W}igner matrices.
\newblock {\em Ann. Appl. Probab.}, 27(3):1510--1550, 2017.

\bibitem{hl}
J.~Huang and B.~Landon.
\newblock Spectral statistics of sparse {E}rd{\H{o}}s-{R}enyi graph
  {L}aplacians.
\newblock {\em preprint, arXiv:1510,06390}, 2015.

\bibitem{hlmeso}
J.~Huang and B.~Landon.
\newblock Rigidity and a mesoscopic central limit theorem for {D}yson
  {B}rownian motion for general beta and potentials.
\newblock {\em Probab. Theory Related Fields}, pages 1--45, 2018.

\bibitem{huang}
J.~Huang, B.~Landon, and H.-T. Yau.
\newblock Bulk universality of sparse random matrices.
\newblock {\em J. Math. Phys.}, 56(12):123301, 19, 2015.

\bibitem{johansson}
K.~Johansson.
\newblock On fluctuations of eigenvalues of random {H}ermitian matrices.
\newblock {\em Duke Math. J.}, 91(1):151--204, 1998.

\bibitem{johansson2001universality}
K.~J. Johansson.
\newblock Universality of the local spacing distribution in certain ensembles
  of hermitian {W}igner matrices.
\newblock {\em Comm. Math. Phys.}, 215(3):683--705, 2001.

\bibitem{kks}
J.~Kahn, J.~Koml{\'o}s, and E.~Szemer{\'e}di.
\newblock On the probability that a random$\pm$1-matrix is singular.
\newblock {\em J. Amer. Math Soc.}, 8(1):223--240, 1995.

\bibitem{kyprep}
A.~Knowles and J.~Yin.
\newblock {\em in preparation}, 2016.

\bibitem{komlosdet}
J.~Koml{\'o}s.
\newblock On the determinant of $(0, 1)$ matrices.
\newblock {\em Studia Sci. Math. Hunger}, 2(1):7--21, 1967.

\bibitem{generic}
A.~B.~J. Kuijlaars and K.~T.-R. McLaughlin.
\newblock Generic behavior of the density of states in random matrix theory and
  equilibrium problems in the presence of real analytic external fields.
\newblock {\em Comm. Pure Appl. Math.}, 53(6):736--785, 2000.

\bibitem{LLM}
B.~Landon, P.~Lopatto, and J.~Marcinek.
\newblock Comparison theorem for some extremal spectral statistics.
\newblock {\em preprint, arXiv:1812.10022}, 2018.

\bibitem{landonyau}
B.~Landon and H.-T. Yau.
\newblock Convergence of local statistics of {D}yson {B}rownian motion.
\newblock {\em preprint, arXiv:1504.03605}, 2015.

\bibitem{Kevin1}
J.~O. Lee and K.~Schnelli.
\newblock Local deformed semicircle law and complete delocalization for
  {W}igner matrices with random potential.
\newblock {\em J. Math. Phys.}, 54(10):103504, 2013.

\bibitem{Kevin2}
J.~O. Lee, K.~Schnelli, B.~Stetler, and H.-T. Yau.
\newblock Bulk universality for deformed {W}igner matrices.
\newblock {\em Ann. Probab.}, 44(3):2349--2425, 2016.

\bibitem{tiki}
A.~E. Litvak, A.~Lytova, K.~Tikhomirov, N.~Tomczak-Jaegermann, and P.~Youssef.
\newblock Adjacency matrices of random digraphs: singularity and
  anti-concentration.
\newblock {\em J. Math. Anal. Appl.}, 445(2):1447--1491, 2017.

\bibitem{meso1}
A.~Lodhia and N.~Simm.
\newblock Mesoscopic linear statistics of {W}igner matrices.
\newblock {\em preprint, arXiv:1503.03533}, 2015.

\bibitem{mehta2004random}
M.~L. Mehta.
\newblock {\em Random matrices}, volume 142.
\newblock Academic press, 2004.

\bibitem{nguyensurvey}
H.~H. Nguyen and V.~H. Vu.
\newblock Small ball probability, inverse theorems, and applications.
\newblock In {\em Erd{\H{o}}s Centennial}, pages 409--463. Springer, 2013.

\bibitem{orourke}
S.~O’Rourke.
\newblock Gaussian fluctuations of eigenvalues in {W}igner random matrices.
\newblock {\em Journal of Statistical Physics}, 138(6):1045--1066, 2010.

\bibitem{rvsurvey}
M.~Rudelson and R.~Vershynin.
\newblock Non-asymptotic theory of random matrices: extreme singular values.
\newblock {\em preprint arXiv:1003.2990}, 2010.

\bibitem{shcherbina}
M.~Shcherbina.
\newblock Central limit theorem for linear eigenvalue statistics of the
  {W}igner and sample covariance random matrices.
\newblock {\em Zh. Mat. Fiz. Anal. Geom.}, 7(2):176--192, 197, 199, 2011.

\bibitem{shcherbinabeta}
M.~Shcherbina.
\newblock Change of variables as a method to study general $\beta$-models: bulk
  universality.
\newblock {\em J. Math. Phys.}, 55(4):043504, 2014.

\bibitem{taosing}
T.~Tao and V.~Vu.
\newblock On the singularity probability of random {B}ernoulli matrices.
\newblock {\em J. Amer. Math Soc.}, 20(3):603--628, 2007.

\bibitem{taovuleast}
T.~Tao and V.~Vu.
\newblock Random matrices: The distribution of the smallest singular values.
\newblock {\em Geom. Funct. Anal}, 20(1):260--297, 2010.

\bibitem{tao2010random}
T.~Tao and V.~Vu.
\newblock Random matrices: Universality of local eigenvalue statistics up to
  the edge.
\newblock {\em Comm. in Math. Phys.}, 298(2):549--572, 2010.

\bibitem{tao2011random}
T.~Tao and V.~Vu.
\newblock Random matrices: universality of local eigenvalue statistics.
\newblock {\em Acta Math.}, 206(1):127--204, 2011.

\bibitem{tricomi}
F.~G. Tricomi.
\newblock {\em Integral equations}, volume~5.
\newblock Courier Corporation, 1957.

\bibitem{vershsymmetric}
R.~Vershynin.
\newblock Invertibility of symmetric random matrices.
\newblock {\em Random Structures Algorithms}, 44(2):135--182, 2014.

\bibitem{vninvert}
J.~Von~Neumann and H.~H. Goldstine.
\newblock Numerical inverting of matrices of high order.
\newblock {\em Bull. Amer. Math Soc.}, 53(11):1021--1099, 1947.

\bibitem{vusurvey}
V.~Vu.
\newblock Random discrete matrices.
\newblock In {\em Horizons of combinatorics}, pages 257--280. Springer, 2008.

\bibitem{wigner}
E.~Wigner.
\newblock Characteristic vectors of bordered matrices with infinite dimensions.
\newblock {\em Ann. Math}, 62:548--564, 1955.

\end{thebibliography}
\bibliographystyle{abbrv}

\end{document}